%% file: thesis.tex
\documentclass[11pt]{ociamthesis}  % default square logo 
\usepackage{amsmath,paralist,amsfonts,stmaryrd, amssymb,amscd,amsthm,amsbsy,upref,mathrsfs, yfonts}

\newtheorem{thm}{Theorem}[section]
\newtheorem{lem}[thm]{Lemma}%[section]
\newtheorem{cor}[thm]{Corollary}%[section]
\newtheorem{prop}[thm]{Proposition}%[section]
\theoremstyle{definition}

\newtheorem{defn}[thm]{Definition}
\newtheorem{assump}[thm]{Assumption}

\newtheorem{rem}[thm]{Remark}

\newtheorem*{remark}{Remark}

\newtheorem{Step}{Step}
\newtheorem{nclaim}{Claim}

\newcommand{\ve}{\varepsilon}
\newcommand{\vphi}{\varphi}
\newcommand{\supp}{\mathrm{\rm supp\,}}
\newcommand{\dist}{\mathrm{dist}}
\newcommand{\ball}{\mathrm{\rm ball\,}}
\newcommand{\ran}{\mathrm{\rm ran\,}}
\newcommand{\rank}{\mathrm{\rm rank\,}}
\newcommand{\age}{\mathrm{\rm age\,}}

\newcommand{\weight}{\mathrm{\rm weight\,}}
\newcommand{\base}{\mathrm{\rm base\,}}

\newcommand{\im}{\mathrm{\rm im\,}}
\newcommand{\C}{\mathbb C}
\newcommand{\R}{\mathbb R}
\newcommand{\N}{\mathbb N}
\newcommand{\Z}{\mathbb Z}
\newcommand{\Q}{\mathbb Q}
\newcommand{\half}{{\textstyle\frac12}}
\newcommand{\txtfrac}[2]{{\textstyle\frac{#1}{#2}}}
\newcommand{\XK}{\mathfrak X_{\mathrm{AH}}}
\newcommand{\Xk}{\mathfrak{X}_{k}}
\newcommand{\X}{\mathfrak{X}}

\newcommand{\Xgm}{\mathfrak X_{\mathrm {gm}}}
\newcommand{\norm}[1]{\left\| #1 \right\|}
\newcommand{\conv}{\mathrm{\rm conv}}
\DeclareMathOperator{\Codim}{codim}

\DeclareMathOperator{\Ker}{Ker}

\DeclareMathOperator{\lin}{span}
\newcommand{\wsto}{\rightharpoonup^{w^{*}}}

\title{Operators on Banach Spaces of\\[1ex]     %your thesis title,
        Bourgain--Delbaen Type}   %note \\[1ex] is a line break in the title

\author{Matthew Tarbard}             %your name
\college{St John's College}  %your college

\degree{Doctor of Philosophy}     %the degree
\degreedate{Michaelmas 2012}         %the degree date

\begin{document}

%this baselineskip gives sufficient line spacing for an examiner to easily
%markup the thesis with comments
\baselineskip=18pt plus1pt

%set the number of sectioning levels that get number and appear in the contents
\setcounter{secnumdepth}{3}
\setcounter{tocdepth}{3}

\maketitle                  % create a title page from the preamble info

\include{dedication}        % include a dedication.tex file
\include{acknowledgements}   % include an acknowledgements.tex file
\include{abstract}          % include the abstract

\begin{romanpages}          % start roman page numbering
\tableofcontents            % generate and include a table of contents
%\listoffigures              % generate and include a list of figures
\end{romanpages}            % end roman page numbering

%now include the files of latex for each of the chapters etc

\input{Intro.tex}

\input{BDExposition.tex}

\input{MainResult.tex}

\input{MainResultExt.tex}

\input{ShiftInvariant.tex}

%now enable appendix numbering format and include any appendices
%\appendix
%\include{appendix1}
%\include{appendix2}

%next line adds the Bibliography to the contents page
\addcontentsline{toc}{chapter}{References}
%uncomment next line to change bibliography name to references
\renewcommand{\bibname}{References}
\bibliography{refs}        %use a bibtex bibliography file refs.bib
\bibliographystyle{amsplain}  %use the plain bibliography style

\end{document}

%% file: dedication.tex
\thispagestyle{empty}

\begin{center}

{\large{\em To my family, for all their love and support.} }\\[1cm]

\end{center}

%% file: acknowledgements.tex
\thispagestyle{empty}

\begin{center}

{\Large{\bf Acknowledgements} }\\[1cm]

\end{center}

\noindent Firstly, I would like to thank my supervisor, Professor Richard Haydon, whose invaluable suggestions, help and support made the work in this thesis possible. 
\\[0.4cm]
I especially thank the hard work of my examiners, Professor Charles Batty and Dr Matthew Daws, who provided a very intellectually stimulating viva. Their careful reading of the thesis also led to the correction of several errors. 
\\[0.4cm]
I have been fortunate enough to present some of the results in this thesis at various seminars and conferences. These opportunities, and the people I have met through these events, helped shape my thoughts and ideas, which led to some of the results in this thesis. I would therefore like to thank Dr Bunyamin Sari for inviting me to the BIRS Banach Space Theory Workshop 2012, where I learnt a great deal about Bourgain-Delbaen spaces and other interesting problems. I would also like to thank Professor William B. Johnson, who brought to my attention at the BIRS workshop a corollary of my result that I had not originally spotted. I would also like to thank Professor Charles Batty, Dr Matthew Daws and Dr Niels Laustsen for inviting me to speak at seminars at the Universities of Oxford, Leeds and Lancaster respectively. 
\\[0.4cm]
Special thanks also go to Dr Richard Earl, who has provided me with fantastic teaching opportunities at Worcester College throughout my DPhil, as well as introducing me to a number of great people, notably, Dr Brian King and Dr Martin Galpin. I am thankful to Richard, Brian and Martin for introducing me to bridge, and for the considerable generosity they have shown me (particularly at the Worcester College Bar)!
\\[0.4cm]
I am of course very grateful to all my friends and family, who have supported me throughout my DPhil. I am particularly thankful to Victoria Mason, who has provided endless encouragement and support whilst writing this thesis, and contributed significantly to the non-academic aspect of my life. Special thanks also go to Stephen Belding, David Hewings, Tanya Gupta and Carly Leighton for always being around to socialise with and making Oxford a more interesting place to live.
\\[0.4cm]
Finally, I'd like to acknowledge that my doctoral studies were funded by the EPSRC (Engineering and Physical Sciences Research Council). 

%% file: abstract.tex
\thispagestyle{empty}

\begin{center}

{\Large{\bf Abstract} }\\[1cm]

\end{center}

\noindent The research in this thesis was initially motivated by an outstanding problem posed by Argyros and Haydon. They used a generalised version of the Bourgain-Delbaen construction to construct a Banach space $\XK$ for which the only bounded linear operators on $\XK$ are compact perturbations of (scalar multiples of) the identity; we say that a space with this property has very few operators. The space $\XK$ possesses a number of additional interesting properties, most notably, it has $\ell_1$ dual. Since $\ell_1$ possesses the Schur property, weakly compact and norm compact operators on $\XK$ coincide. Combined with the other properties of the Argyros-Haydon space, it is tempting to conjecture that such a space must necessarily have very few operators.  Curiously however, the proof that $\XK$ has very few operators made no use of the Schur property of $\ell_1$. We therefore arrive at the following question (originally posed in \cite{AH}): must a HI, $\mathcal{L}_{\infty}$, $\ell_1$ predual with few operators (every operator is a strictly singular perturbation of $\lambda I$) necessarily have very few operators? 
\\[0.4cm]
We begin by giving a detailed exposition of the original Bourgain-Delbaen construction and the generalised construction due to Argyros and Haydon. We show how these two constructions are related, and as a corollary, are able to prove that there exists some $\delta > 0$ and an uncountable set of isometries on the original Bourgain-Delbaen spaces which are pairwise distance $\delta$ apart. 
\\[0.4cm]
We subsequently extend these ideas to obtain our main results. We construct new Banach spaces of Bourgain-Delbaen type, all of which have $\ell_1$ dual. The first class of spaces are HI and possess few, but not very few operators. We thus have a negative solution to the Argyros-Haydon question. We remark that all these spaces have finite dimensional Calkin algebra, and we investigate the corollaries of this result. We also construct a space with $\ell_1$ Calkin algebra and show that whilst this space is still of Bourgain-Delbaen type with $\ell_1$ dual, it behaves somewhat differently to the first class of spaces. 
\\[0.4cm]
Finally, we briefly consider shift-invariant $\ell_1$ preduals, and hint at how one might use the Bourgain-Delbaen construction to produce new, exotic examples. 

%% file: Intro.tex
\chapter{Introduction}

\section{Historical background}
Knowledge about the types of bounded linear operator that exist from a Banach space into itself can reveal much about the structure of the underlying Banach space. In particular, it is possible to infer a great deal about the structure of the space $X$ when its operator algebra, $\mathcal{L}(X)$, is `small'. The first substantial results in this direction are those of Gowers and Maurey, presented in \cite{GowMau} and \cite{GM}. As a motivational example, we consider the space $\Xgm$ constructed by Gowers and Maurey in \cite{GowMau}. Here it was shown that all (bounded linear) operators defined on a subspace $Y$ of $\Xgm$ (and mapping into $\Xgm$) are strictly singular perturbations of the inclusion operator $i_Y \colon Y \to \Xgm$. More precisely, every such operator is expressible in the form $\lambda i_Y + S$, where $\lambda$ is a scalar and $S \colon Y \to \Xgm$ is a strictly singular operator.

We shall give a precise definition of a strictly singular operator in Section \ref{SSoperators}. For now, it suffices to think of the strictly singular operators as those which are, in some sense, small. Indeed, it is a well known result of Fredholm theory that strictly singular perturbations of Fredholm operators are still Fredholm, with the same index. 

The representation of operators on subspaces of $\Xgm$ just discussed allows us to infer some remarkable structural properties of the space $\Xgm$. We obtain the following:

\begin{enumerate}
\item $\Xgm$ is not decomposable, that is, it cannot be written as a (topological) direct sum of two of its infinite dimensional subspaces. This is because a non-trivial projection is not expressible as a strictly singular perturbation of the identity. 
\item In fact, we conclude by the same argument that $\Xgm$ is hereditarily indecomposable, that is to say, no closed infinite dimensional subspace of $\Xgm$ is decomposable. It follows that no subspace of $\Xgm$ has an unconditional basis, i.e. $\Xgm$ has no unconditional basic sequence. Indeed, if $(e_i)_{i=1}^{\infty}$ were an unconditional basis for some subspace $Y \subseteq \Xgm$, then we could decompose $Y$ as $Y = [ e_{2i} ]_{i=1}^{\infty} \oplus [e_{2i-1} ]_{i=1}^{\infty}$. (We remark, and contrast this to, the well known fact that every Banach space contains a basic sequence.)
\item $\Xgm$ is not isomorphic to any of its proper subspaces. Indeed, it follows from the operator representation and elementary results from Fredholm theory that every operator from $\Xgm$ to itself is either strictly singular (and thus not an isomorphism), or Fredholm with index 0.  In the latter case, if $\Xgm$ were isomorphic to a proper subspace $Y$, then there would be an isomorphic embedding $T \colon \Xgm \to \Xgm$ that maps onto $Y$ (we simply take an isomorphism from $\Xgm \to Y$ and then compose with the inclusion operator sending $Y$ into $\Xgm$).  In particular $T$ is not strictly singular and so must be Fredholm with index 0. However, this is clearly not the case as T is injective but not onto $\Xgm$.
\end{enumerate}

One could of course consider the relationship between a Banach space and its operator algebra from a different perspective to that just described. Instead of assuming we have some well behaved properties of the operator algebra and asking what consequence this has for the structure of the underlying Banach space, we may choose to impose some kind of structural conditions on a Banach space and see what affect this has on the associated operator algebra. In 1980, Bourgain and Delbaen \cite{BD80, B81} introduced two classes of separable Banach spaces which have `well-behaved' finite dimensional structure, specifically they are $\mathcal{L}_{\infty}$-spaces (we refer the reader to Definition \ref{Linftyspace} for more details). These spaces have many interesting properties and recently, Argyros and Haydon (see \cite{AH}) have managed to modify the original construction of Bourgain and Delbaen to construct a space, $\XK$, which solves the scalar-plus-compact problem. More precisely $\XK$ is a (hereditarily indecomposable) $\mathscr{L}_{\infty}$ space with a Schauder basis, $\ell_1$ dual and $\mathcal{L}(\XK) = \R I \oplus \mathcal{K}(\XK)$ (where, as usual, $\mathcal{K}(\XK)$ denotes the subspace of $\mathcal{L}(\XK)$ of compact operators). The proof that all operators are compact perturbations of the identity made essential use of the finite dimensional structure of the space and the specific structure of Bourgain-Delbaen spaces, which embeds some very explicit finite dimensional $\ell_{\infty}-$spaces into $\XK$. Interestingly, the Schur property of $\ell_1$ plays no role in the proof that $\XK$ has the scalar-plus-compact property; this will be the main subject of Chapter \ref{MainResult}, and so we will defer any further discussion of this until then. 

As well as solving the scalar-plus-compact problem, the space $\XK$ is interesting for many reasons. It is well known that the space of compact operators $\mathcal{K}(X)$ is a separable subspace of $\mathcal{L}(X)$ whenever $X$ is a Banach space which has the approximation property and a separable dual space. It follows that the operator algebra $\mathcal{L}(\XK)$ is separable. Moreover, it is elementary to show that the compact operators on $\XK$ are a complemented subspace of $\mathcal{L}(\XK)$. 

Another interesting property of the space $\XK$ is that every operator in $\mathcal{L}(\XK)$ has a proper closed invariant subspace, the first space for which such a result is known. Indeed, Aronszajn and Smith showed in \cite{ASm} that every compact operator on a Banach space has a proper closed invariant subspace. Clearly if a subspace is invariant under some operator, then it is invariant under that operator perturbed by some scalar multiple of the identity. In particular, every operator of the form $\lambda I + K$, with $K$ compact admits a proper, closed invariant subspace and the claim about $\XK$ follows. (One may also prove the claim by appealing to a result of Lomonosov, \cite{Lom}, which states that every operator that commutes with a non-zero compact operator must have a proper invariant subspace.)

\section{Overview of the thesis}
In this thesis, we will continue to investigate the interplay between the structure of a Banach space and its associated operator algebra. The question originally motivating the research that follows was whether or not the $\ell_1$ duality of the Argyros-Haydon space $\XK$ forces the scalar-plus-compact property. Stated slightly more precisely, must every strictly singular operator on a space with $\ell_1$ dual necessarily be compact if the operator algebra of the space is small? Standard facts about basic sequences and strictly singular operators will play a prominent role in this work, so we begin by recalling in Section \ref{prelims} some well known results from Banach space theory and operator theory which will be used throughout the thesis. Whilst we present nothing new in this chapter, many of the results and proofs are difficult to find in the existing literature. For completeness, results we will frequently rely on are stated and we either provide references to their proofs or outline them here.

{\em Chapter \ref{ChBDConst}:} Given the success of Argyros and Haydon in modifying the Bourgain-Delbaen construction to solve the scalar-plus-compact problem, the spaces of `Bourgain-Delbaen type' are a  sensible starting place to look for examples of exotic Banach spaces having $\ell_1$ dual and well behaved operator algebras. The construction employed by Argyros and Haydon in \cite{AH} appears to be somewhat different to the original construction of Bourgain and Delbaen in, for example, \cite{BD80, B81}. We begin Chapter \ref{ChBDConst} by investigating the relationship between the two constructions. Specifically, we show that the Argyros-Haydon construction in \cite{AH} is essentially a generalised `dualised version' of the original Bourgain-Delbaen construction. Of course, this is certainly not a new result, however, to the best knowledge of the author, the precise details of this connection cannot be found in any of the existing literature. 

The Argyros-Haydon approach to the Bourgain-Delbaen construction enables us to see how we might go about constructing interesting operators on spaces of Bourgain-Delbaen type.  Having understood how the original Bourgain-Delbaen construction fits into that of the Argyros-Haydon construction, we present our first new result at the end of Chapter \ref{ChBDConst}: for any of the original Bourgain-Delbaen spaces, $X$, constructed in \cite{B81}, there exists a constant $c > 0$ and an uncountable set of isometries from $X$ to itself, such that any two (non-equal) isometries are separated by distance at least $c$ with respect to the operator norm.  In particular, $\mathcal{L}(X)$ is non-separable for all of the original Bourgain-Delbaen spaces. This answers a question of Beanland and Mitchell (\cite{KB2010}).

{\em Chapter \ref{MainResult}:} In Chapter \ref{MainResult}, we address a question left over from the work of Argyros and Haydon in \cite{AH}, namely, must the strictly singular and compact operators on a space with $\ell_1$ dual coincide if every operator on the space is a strictly singular perturbation of the identity? The main result of this chapter is a negative solution to this question. In fact, we exhibit for each natural number $k$, a Banach space $\Xk$ with $\ell_1$ dual such that the Calkin algebra $\mathcal{L}(\Xk) / \mathcal{K}(\Xk)$ is $k$-dimensional, with basis $\{ I + \mathcal{K}(\Xk), S +  \mathcal{K}(\Xk), \dots S^{k-1} +  \mathcal{K}(\Xk) \}$. Here $S$ is some strictly singular operator on $\Xk$. The results are, in some sense, continuations of the work of Gowers and Maurey and Argyros and Haydon already mentioned; indeed, for each space $\Xk$, we have a representation of a general operator as some polynomial of $S$, possibly perturbed by a compact operator. We immediately obtain from this an interesting corollary concerning the norm-closed ideal structure of $\mathcal{L}(\Xk)$; it is a finite, totally ordered chain of ideals. We also observe that the space $\X_2$ provides a counterexample for an open conjecture of Johnson concerning the form of commutators in the operator algebra. The material in this chapter is an expanded version of the published paper \cite{me}.

{\em Chapter \ref{l1Calk}  :}  Having shown it is possible to construct Banach spaces with Calkin algebras of any finite dimension, the natural generalisation is to consider what infinite dimensional Calkin algebras can be obtained. In Chapter \ref{l1Calk} we exhibit a Banach space, $\X_{\infty}$, which has Calkin algebra isometric (as a Banach algebra) to the algebra $\ell_1(\N_0)$.  This generalises a result of Gowers (\cite{GM}) where a Banach space $X$ was constructed with the property that the quotient algebra $\mathcal{L}(X) / \mathcal{S}\mathcal{S}(X)$ is isomorphic as a Banach algebra to $\ell_1(\Z)$. (Here $\mathcal{S}\mathcal{S}(X)$ denotes the strictly singular operators on $X$.) The ideas used to construct $\X_{\infty}$ are very similar to those of Chapter \ref{MainResult}, though the proof is a little harder, requiring a combination of the arguments from Chapter \ref{MainResult} and the Gowers' paper \cite{GM}. 

{\em Chapter \ref{Conclusion} : } The reader will notice that all the Banach spaces constructed in this thesis have $\ell_1$ dual. We conclude the thesis by discussing the exoticness of $\ell_1$ preduals and some open research problems. This leads us to discuss some recently published work by Daws, Haydon, Schlumprecht and White on so-called `shift-invariant preduals of $\ell_1$' (see \cite{Shift} for more details). We show that one of the spaces constructed in the aforementioned paper can in fact be considered (in some sense) as being obtained from a specific Bourgain-Delbaen construction. It is hoped by the author that this insight may eventually lead to new, exotic $\ell_1$ preduals. 

\section{Notation and Elementary Definitions} \label{Notation}

Generally notation will be introduced as and when needed. Nevertheless, in this section we remind the reader of some standard notation and state the conventions and definitions that will be used throughout this thesis without further explanation.

\begin{itemize}
\item All Banach spaces are assumed to be over the real scalar field unless explicitly stated. Given a subset of vectors of a Banach space $X$,  $\{ x_i \}_{i \in I}$ , we denote by $[ x_i ]_{i \in I}$ the smallest closed subspace of $X$ generated by the $(x_i)_{i \in I}$. In other words, $[x_i]_{i \in I}$ is the closed linear span of the $(x_i)_{i\in I}$.
\item The continuous dual space of a Banach space $X$ is denoted by $X^*$ as usual, unless explicitly stated otherwise. (In Chapters \ref{MainResult} and \ref{l1Calk} we introduce a `star notation' in an attempt to be consistent with this notation.)
 \item We occasionally make use of the `angle bracket notation': for vectors $x^* \in X^*$ and $x \in X, \langle x^* , x \rangle = x^*(x)$, the evaluation of the functional $x^*$ at $x$.
 \item Given a closed subspace $M$ of a Banach space $X$, the quotient space $X / M$ is the Banach space which as a set consists of the cosets $\{ x + M \}$. We equip it with the obvious vector space operations and use the norm $\| x + M \| = \dist (x, M  ) = \inf \{ \| x - m \| : m \in M \}$.
\item All operators between Banach spaces are assumed to be continuous and linear. For an operator $T \colon X \to Y$ between Banach spaces $(X, \| \|_X )$ and $(Y, \| \|_Y)$, we denote by $\|T\|$ the usual operator norm of $T$, that is, $\|T\| = \sup \{ \|Tx\|_Y : \| x\|_X \leq 1 \}.$
\item With the same notation as above, $T^* : Y^* \to X^*$ denotes the dual operator defined by $T^*y^* (x) = y^*(Tx)$ for all $x \in X$ and $y^* \in Y^*. $
\item An operator $T: X \to Y$ is said to be an {\em isomorphism } if there exist strictly positive constants $a, b >0 $, such that \[
 a\| x \|_X \leq \|Tx\|_Y \leq b \|x\|_X
 \]
 Note we do not require an isomorphism to be onto. When we wish to be explicit about an isomorphism that is not surjective, we use the term {\em isomorphic embedding }.
\item An onto operator $T \colon X \to Y$ between normed spaces $X$ and $Y$ is said to be a {\em quotient operator} if the (bounded) linear operator $\tilde{T} \colon X / \Ker T \to Y$, defined by $\tilde{T}( x + \Ker T) = Tx$ is an isomorphism, i.e. it has a continuous inverse. It is easy to check that this is equivalent to the condition `there exists a constant $M > 0$ such that $B_Y(0;1) \subset MT(B_X(0;1))$'.
 \end{itemize}

\section{Preliminary Results} \label{prelims}

\subsection{Basic sequence techniques} \label{BasicSequenceTechniques}

It would be impossible to include all the known results about Schauder bases and basic sequences. Nevertheless, the notion of a Schauder basis and selected results from basic sequence techniques in Banach space theory feature prominently in this thesis. In this section, we state the results we will need throughout the rest of this work.

\begin{defn} Let $X$ be an infinite dimensional Banach space. 
\begin{itemize}
\item A sequence $(x_n)_{n=1}^{\infty} \subset X$ is said to be a {\em Schauder basis} for $X$ if whenever $x \in X$, there exists a unique sequence of scalars $(\alpha_n)_{n=1}^{\infty}$ such that  $x = \sum_{n=1}^{\infty} \alpha_n x_n$ (where the sum converges in the norm topology). We say $(x_n)_{n=1}^{\infty}$ is a {\em basic sequence} in $X$ if it is a Schauder basis for the closed linear span $[x_n : n \in \N]$. 
\item More generally, a sequence $(F_n)_{n=1}^{\infty}$ of non-trivial subspaces of $X$ is said to form a Schauder decomposition of $X$, written $X = \oplus_{n\in \N} F_n$, if every $x \in X$ can be written uniquely as $\sum_{n=1}^{\infty} x_n$ with $x_n \in F_n$ for all $n$. 
\end{itemize}
\end{defn} 

In a Banach space $X$ we say that vectors $y_j$ are {\em successive
linear combinations}, or that $(y_j)$ is a {\em block sequence} of a
basic sequence $(x_i)$ if  there exist $0=q_1<q_2<\cdots$ such that,
for all $j\ge 1$, $y_j$ is in the linear span $[x_i :q_{j-1}<i\le
q_j]$. If we may arrange that $y_j\in [x_i :q_{j-1}<i< q_j]$ we say
that $(y_j)$ is a {\em skipped block sequence.} More generally, if
$X$ has a Schauder decomposition $X=\bigoplus_{n\in \N} F_n$ we
say that $(y_j)$ is a block sequence (resp. a skipped block
sequence) with respect to $(F_n)$ if there exist $0=q_0<q_1<\cdots$
such that $y_j$ is in $\bigoplus_{q_{j-1}<n\le q_j}F_n$ (resp.
$\bigoplus_{q_{j-1}<n< q_j}F_n)$. A {\em block subspace} is the
closed subspace generated by a block sequence.

Finally,  a basic sequence $(x_n)$ is called {\em unconditional} if every permutation of the sequence is a Schauder basis for $[x_n : n \in \N]$.

%\item If $(x_n)_{n=1}^{\infty}$ is a basic sequence in $X$ and $(p_n)_{n=1}^{\infty}$ is a strictly increasing sequence of natural numbers with $p_0 = 0$, then a sequence of non-zero vectors $(u_n)_{n=1}^{\infty}$ of the form \[
%u_n = \sum_{j = p_{n-1} + 1}^{p_n} a_j x_j
%\]
%is called a {\em block basic sequence}. It is well known that such a sequence is in fact a basic sequence.
%\item A closed subspace of $X$ is called a {\em block subspace} if it is the closed linear span of some block basic sequence. 

It is well known (see, for example, \cite[Proposition 1.1.9]{Kalton}) that a sequence $(x_n)_{n=1}^{\infty}$ of non-zero vectors in a Banach space $X$ is a basic sequence if, and only if, there exists a constant $M \geq 1$ such that whenever $m < n$ are natural numbers and $(\lambda_i)_{i=1}^{n}$ are scalars, \[
\| \sum_{i=1}^m \lambda_i x_i \| \leq M \| \sum_{i=1}^n \lambda_i x_i \|.
\]
From this, it follows easily that there exist uniformly bounded linear projections $(P_k)_{k=1}^{\infty}$ defined by $P_k : [x_n : n \in \N] \to [x_n : n \leq k]$, $\sum_{n=1}^{\infty} a_n x_n \mapsto \sum_{n=1}^k a_n x_n$; the constant, $\sup_{k \in \N} \|P_k\|$, is known as the {\em basis constant}. We note that the co-ordinate functionals, $x_n^* \colon [x_n : n \in \N] \to \mathbb{K}$ defined by $\sum_{m=1}^{\infty} a_m x_m \mapsto a_n$ (also known in the literature as {\em biorthogonal functionals}) are continuous and refer the reader to \cite[Theorem 4.1.14]{Meg} for further details of these facts.

When the Banach space has an unconditional basis, $(x_n)_{n=1}^{\infty}$, it admits a number of non-trivial projections and the associated operator algebra has a rich structure. Indeed, it is known that for any $(\xi_n)_{n=1}^{\infty} \in \ell_{\infty}$, $T_{(\xi_n)} \colon X \to X$ defined by $\sum_{n=1}^{\infty} a_n x_n \mapsto \sum_{n=1}^{\infty} \xi_n a_n x_n$ is a bounded operator. In particular, for every $A \subset \N$, the linear projection $P_A : X \to X$ defined by $\sum_{n=1}^{\infty} a_n x_n \mapsto \sum_{n \in A} a_n x_n$ is bounded; in fact, $\sup_{A \subset \N} \|P_A\| < \infty$. More details and proofs of these facts can be found in \cite[Proposition 3.1.3, Remark 3.1.5]{Kalton}. 

Frequently we will have some condition that needs to be checked for all closed infinite dimensional subspaces of a Banach space. When the Banach space has a Schauder basis, it is often possible to show that it is sufficient to check the condition only on block subspaces. We will make use of this idea later in the thesis. The proofs of such results generally rely on the following technical result.

\begin{prop} \label{BasicSequencesTechLemma}
Let $X$ be a Banach space with Schauder basis $(e_n)_{n=1}^{\infty}$ and suppose $Y$ is a closed, infinite dimensional subspace of $X$. Denote by $K$ the basis constant of the basis $(e_n)_{n=1}^{\infty}$. For every $0 <\theta < 1$, we can find a sequence $(y_n)_{n=1}^{\infty}$ in $Y$ with $\| y_n \| = 1$ for all $n$ and a block sequence $(x_n)_{n=1}^{\infty}$  in $X$ satisfying 
\[
2K \sum_{n=1}^{\infty} \frac{\| y_n - x_n \|}{\|x_n \| }  < \theta
\]
The sequence $(y_n)_{n=1}^{\infty}$ is a basic sequence equivalent to $(x_n)_{n=1}^{\infty}$, i.e. given a sequence of scalars, $(a_n)_{n=1}^{\infty}$, $\sum_{n=1}^{\infty} a_n y_n$ converges if and only if $\sum_{n=1}^{\infty} a_n x_n$ converges.
\end{prop}

\begin{proof}
Note that for any $n \in \N$, there must exist a non-zero vector $y \in Y$ with $e_m^*(y) = 0$ for all $m \leq n$ (here, the $e_m^*$ are the biorthogonal functionals of the basis $(e_n)$). Indeed, if not, the (bounded) projection $P_n : Y \to [ e_m : m \leq n]$ is injective; this contradicts the fact that $Y$ is infinite dimensional. 

It follows from the above observation that we can pick a sequence $(y'_n)_{n=1}^{\infty} \subseteq Y$ with $\|y'_n \| = 1$ for every $n$ and $e_m^*(y'_n) = 0$ for all $m \leq n$. By the proof of the Bessaga-Pelczy\'nski Selection Principle given in \cite[Proposition 1.3.10]{Kalton}, there is a subsequence of the $y'_n$, $(y'_{n_k})$ say, and a block basic sequence $(x_k)$ (with respect to the basis $(e_n)$) satisfying \[
2K \sum_{k=1}^{\infty} \frac{ \|y'_{n_k} - x_k \| }{\|x_k\|} < \theta
\]
The first part of the proposition is proved by simply setting $y_k = y'_{n_k}$ and choosing the block basic sequence $(x_k)_{k=1}^{\infty}$ as above. 

To see that $(y_n)_{n=1}^{\infty}$ is a basic sequence equivalent to the basic sequence $(x_n)_{n=1}^{\infty}$, it is enough, by standard results, to show that there is an isomorphism from $X$ to $X$ which sends $x_n$ to $y_n$ for every $n$. This is proved in \cite[Theorem 1.3.9]{Kalton}.  
\end{proof}

The Banach spaces constructed later in this thesis rely heavily on the following well-known proposition:

\begin{prop} \label{biorthogonalsarebasicsequence}
Suppose that $(x_n^*)_{n=1}^{\infty}$ is the sequence of biorthogonal functionals associated to a basis $(x_n)_{n=1}^{\infty}$ of a Banach space $X$. Then $(x_n^*)_{n=1}^{\infty}$ is a basic sequence in $X^*$ (with basis constant no bigger than that of $(x_n)_{n=1}^{\infty}$).

If we denote by $H:= [x_n^* : n \in \N]$ the closed subspace of $X^*$ formed by taking the closed linear span of the biorthogonal vectors, then $X$ embeds isomorphically into $H^*$ under the natural mapping $j \colon X \to H^*$, $x \mapsto j(x)|_H$ (where $j : X \to X^{**}$ is the canonical embedding of $X$ into its bidual).
 \end{prop}

\begin{proof}
See \cite[Proposition 3.2.1, Lemma 3.2.3 and Remark 3.2.4]{Kalton}.
\end{proof}

We conclude this brief overview of basic sequences by recalling the notions and some basic facts about shrinking and boundedly-complete bases.

\begin{defn}
Let $X$ be a Banach space.
\begin{itemize}
\item A basis $(x_n)_{n=1}^{\infty}$ of $X$ is said to be {\em shrinking} if the sequence of its biorthogonal functionals $(x_n^*)_{n=1}^{\infty}$ is a basis for $X^*$. 
\item A basis $(x_n)_{n=1}^{\infty}$ of $X$ is said to be {\em boundedly-complete} if whenever $(a_n)_{n=1}^{\infty}$ is a sequence of scalars such that \[
\sup_N \| \sum_{n=1}^{\infty} a_n x_n \| < \infty
\]
the series $\sum_{n=1}^{\infty} a_n x_n$ converges.
\end{itemize} 
\end{defn}

We will use the following results repeatedly.

\begin{prop} \label{Kalton3.2.7}
A basis $(x_n)_{n=1}^{\infty}$ of a Banach space $X$ is shrinking if and only if every bounded block basic sequence of $(x_n)_{n=1}^{\infty}$ is weakly null. 
\end{prop}

\begin{proof}
See \cite[Proposition 3.2.7]{Kalton}.
\end{proof}

\begin{thm} \label{boundedlycompleteshrinkingbasis}
Let $(x_n)_{n=1}^{\infty}$ be a basis for a Banach space $X$ with biorthogonal functionals $(x_n^*)_{n=1}^{\infty}$. The following are equivalent. 
\begin{enumerate}[(i)]
\item $(x_n)_{n=1}^{\infty}$ is a boundedly-complete basis for $X$,
\item $(x_n^*)_{n=1}^{\infty}$ is a shrinking basis for $H := [ x_n^* : n \in \N]$,
\item The map $j: X \to H^*$ defined by $j(x)(h) = h(x)$, for all $x \in X$ and $h \in H$, is an onto isomorphism.
\end{enumerate}
\end{thm}

\begin{proof}
See \cite[Theorem 3.2.10]{Kalton}.
\end{proof}

\subsection{Separation Theorems}

There are a number of separation theorems in the existing literature. Whilst we do not feel it is necessary to present the Hahn-Banach Theorem as found in almost any introductory textbook on functional analysis, we find it convenient to state two separation theorems. These results are sufficiently general to cover all cases we will need in this thesis.

\begin{thm}[Hahn, Banach] \label{HBSeparationThm}
Let $C$ be a closed convex set in a Banach space $X$. If $x_0 \notin C$ then there is $f \in X^*$ such that $\text{Re}(f(x_0)) > \text{sup}\{ \text{Re}(f(x)) : x \in C\}$.
\end{thm}

\begin{proof}
See \cite[Theorem 2.12]{Fabian}.
\end{proof}

\begin{thm} \label{SeparationThm}
Let $X$ be a topological vector space (TVS) and let $C_1$ and $C_2$ be non-empty convex subsets of $X$ such that $C_2$ has non-empty interior. If $C_1 \cap C_2^{\circ} = \varnothing$ then there is a continuous linear functional $x^* \in X^*$ and a real number $s$ such that
\begin{enumerate}[(1)]
\item $\text{Re } x^*x \geq s$ for each $x \in C_1$;
\item $\text{Re } x^*x \leq s$ for each $x \in C_2$;
\item $\text{Re } x^* x < s$ for each $x \in C_2^{\circ}$.
\end{enumerate}
\end{thm}

\begin{proof}
See \cite[Theorem 2.2.26]{Meg}.
\end{proof}

\subsection{Hereditary Indecomposability}

We have already discussed hereditarily indecomposable (HI) spaces in the introduction. However, since we will want to prove that the spaces constructed in Chapter \ref{MainResult} are HI, we take this opportunity to formally state the HI condition. 

A Banach space $X$ is {\em indecomposable} if there do not exist
infinite-dimensional closed subspaces $Y$ and $Z$ of $X$ such that $X$ can be written as the topological direct sum
$X=Y\oplus Z$; that is to say the bounded operator \[
Y \oplus_{ext} Z \to Y\oplus Z \]
\[
(y,z) \mapsto y+z 
\]
fails to have continuous inverse. $X$ is {\em hereditarily indecomposable} (HI) if
every closed subspace is indecomposable. It follows that $X$ is HI if and only if, whenever $Y$ and $Z$ are infinite dimensional, closed subspaces of $X$, there does not exist $\delta > 0$ such that $\| y + z \| \geq \delta (\|y \| + \|z \|)$ for all $y \in Y, z \in Z$. (We note in particular that were such a $\delta > 0$ to exist, it would certainly have to be the case that $Y \cap Z = \{0 \}$.)  If $X$ has a Schauder basis, it is sufficient to check the condition on block subspaces. More precisely

\begin{prop}
Let $X$ have a Schauder basis. Then $X$ is HI if and only if whenever $Y, Z$ are block subspaces, there does not exist $\delta > 0$ such that $\| y + z \| \geq \delta (\|y \| + \|z \|)$ for all $y \in Y, z \in Z$.
\end{prop}

\begin{proof}
Clearly if $X$ is HI then the condition on block subspaces holds. Conversely, suppose the condition on block subspaces holds and for contradiction that there exist infinite dimensional, closed subspaces $Y$, $Z$, with $Y \cap Z = \{ 0 \}$ and such that $Y\oplus Z$ is a topological direct sum. Consequently, there exists $\delta > 0$ such that $\| y + z \| \geq \delta (\|y \| + \|z \|)$ for all $y \in Y, z \in Z$. (Clearly for this to hold, we must have $\delta \leq 1$.)

We choose $0 < \ve <  1/4$ and such that $\frac{4\ve}{1-4\ve} < \delta$. By Proposition \ref{BasicSequencesTechLemma} we can find norm $1$ sequences $(y_n)_{n=1}^{\infty} \subset Y$, $(z_n)_{n=1}^{\infty} \subset Z$ and block basic sequences $(y'_n)_{n=1}^{\infty}$, $(z'_n)_{n=1}^{\infty}$ with the property that
\[
2K \sum_{n=1}^{\infty} \frac{\|y_n - y'_n \|}{\| y'_n\| } < \ve
\]
and \[
2K \sum_{n=1}^{\infty} \frac{\|z_n - z'_n \|}{\| z'_n\| } < \ve
\] where $K$ is the basis constant of the Schauder basis of $X$. 

Let $Y'$ be the block subspace generated by $(y'_n)_{n=1}^{\infty}$ and $Z'$ the block subspace generated by $(z'_n)_{n=1}^{\infty}$. Note the basis constant of the basic sequence $(y'_n)$ is at most $K$ as it is a block basic sequence. It follows that if $y' := \sum_{n=1}^{\infty} a_n y'_n \in Y'$ then $|a_n| \leq \frac{2K}{\|y'_n\|}\|y'\|$ for every $n$. Consequently, if $y' := \sum_{n=1}^{\infty} a_n y'_n \in Y'$, then $y := \sum_{n=1}^{\infty} a_n y_n$ is a well-defined vector in $Y$ (we recall that $(y_n)$ and $(y'_n)$ are equivalent) and that \[
\| y  - y' \| = \| \sum_{n=1}^{\infty} a_n (y_n - y'_n) \| \leq 2K\|y'\| \sum_{n=1}^{\infty} \frac{\|y_n - y'_n\|}{\|y'_n\|} < \ve \|y'\|.
\] An analogous result holds for $z' \in Z'$; we get a corresponding $z \in Z$ with $\| z - z' \| < \ve \|z'\|$.

We now suppose that $y' \in Y'$ has norm $1$, so that the corresponding vector $y \in Y$ satisfies $\|y\| > 1- \ve$. If $z' \in Z'$ is such that $\| z' \| \leq 1+\delta$ then
\begin{align*}
\| y' - z' \| &= \| (y - z + y' - y) - (z' - z) \| \\
&\geq \| (y-z) - (y-y') \| - \|z' - z\| \\
&\geq \|y - z \| - \| y- y' \| - \| z' - z \| \\
&\geq \delta (\| y\| + \|z \| ) - \ve\|y'\| - \ve\|z'\| \\
&\geq \delta\|y \| - \ve - \ve (1+\delta) \\
&\geq \delta (1-\ve) - 2\ve - \ve\delta \\
&> \frac{\delta}{2}.
\end{align*} where the final inequality follows by choice of $\ve$. On the other hand, if $z' \in Z'$ is such that $\| z' \| > 1 + \delta$, then $\| y' - z' \| \geq \| z' \| - \| y' \| > 1 + \delta - 1 > \frac{\delta}{2}$. 

We have therefore shown that for any $y' \in Y'$ with $\|y' \| = 1$ and $z' \in Z'$, $\| y' - z' \| > \frac{\delta}{2}$. By scaling, this implies that whenever $y' \in Y', z' \in Z'$, $\| y' + z' \| \geq \frac{\delta}{2} \|y ' \|$. By symmetry of the argument, we also have that whenever $y' \in Y', z' \in Z'$, $\| y' + z' \| \geq \frac{\delta}{2} \|z ' \|$. Consequently, for all $y' \in Y', z' \in Z'$, $\| y' + z' \| \geq \frac{\delta}{2} \text{ max}(\| y' \| , \| z' \| ) \geq \frac{\delta}{4}( \| y' \| + \|z ' \| )$; this contradictions the assumed property satisfied by block subspaces. 
\end{proof}

Whilst there are a number of additional results known about HI spaces, the only other result we will need for the purposes of this thesis is the following proposition (which can be found in \cite{AH}).

\begin{prop}\label{equivHIcondition}
Let $X$ be an infinite dimensional Banach space.  Then $X$ is HI
if and only if, for every pair $Y,Z$ of closed, infinite-dimensional
subspaces, and every $\epsilon>0$, there exist $y\in Y$ and $z\in Z$
with $\|y- z\| < \ve \|y+ z\|$.  If $X$ has a
Schauder basis it is enough
that the previous condition should hold for block subspaces.
\end{prop}

\begin{proof}
We claim that whenever $Y$ and $Z$ are infinite dimensional, closed subspaces, $Y+Z$ fails to be a topological direct sum if and only if $\forall \ve > 0, \exists y \in Y, z \in Z$ with $\| y - z \| < \ve \| y+ z \|$. This clearly proves the first part of the proposition. Moreover, once we have proved this claim, we can deduce that if the condition holds for block subspaces, then no two block subspaces form a topological direct sum. The previous proposition then implies that $X$ is HI. Therefore, it only remains to prove the claim made at the beginning of the proof. 

Suppose first that $Y$ and $Z$ are closed infinite dimensional subspaces for which $Y+Z$ fails to be a topological direct sum. We assume for contradiction that there exists $\ve > 0$ such that $\| y - z \| \geq \ve \| y+z \| $ whenever $y \in Y, z\ \in Z$. It follows that $\| y+ z \| \geq \ve \| y - z \| $ for all $y \in Y, z \in Z$. Consequently, $\|y \| = \frac{1}{2} \| y + z + y - z \| \leq \frac12 ( \| y + z \| + \| y - z \| ) \leq \frac12 (\|y +z\| + \frac{1}{\ve} \|y + z\|) = \frac12 (1 + \frac{1}{\ve} ) \| y+z \|$. A similar calculation yields $\| z \| \leq \frac12 (1 + \frac{1}{\ve} ) \| y+z \|$. So \[
\| y + z \| \geq 2( 1 + \frac{1}{\ve})^{-1} \text{ max} (\|y\| , \| z \|) \geq ( 1 + \frac{1}{\ve})^{-1} ( \|y\| + \| z\| )
\]
from which it follows that $Y \oplus Z$ is a direct sum, giving us the required contradiction.

Conversely, assume that $Y$ and $Z$ satisfy the condition that $\forall \ve > 0, \exists y \in Y, z \in Z$ with $\| y - z \| < \ve \| y+ z \|$. We must see that $Y + Z$ fails to be a topological direct sum. We again argue by contradiction, assuming that $Y + Z$ is topological. Consequently, there exists $\delta > 0$ such that $\| y + z \| \geq \delta ( \| y \| + \| z \| )$ whenever $y \in Y, z \in Z$. In particular, we note that this implies that $ \|y + z \| \geq \delta \|z \| $ for all $y \in Y, z \in Z$. It follows that $\| y - z \| = \| y + z - 2z \| \leq \| y + z \| + 2\| z \| \leq (1 + \frac{2}{\delta} ) \| y + z \|$ for all $y \in Y, z \in Z$. On the other hand, taking $\ve = \frac12 (1 + \frac{2}{\delta})^{-1}$, there exist $y \in Y, z \in Z$ with $\| y + z \| < \ve \| y - z \|$ by the hypothesis (with $z$ replaced by $-z$). It follows that for this choice of $y$ and $z$, $\| y - z \| \leq (1 + \frac{2}{\delta} ) \| y + z \| \leq \ve (1 + \frac{2}{\delta} ) \| y - z \| = \frac12 \| y - z \| $. This only fails to be a contradiction if $\| y - z \| = 0$, which, since it is assumed that $Y$ and $Z$ are topological, can only happen if $ y =z =0$. But, recalling that $y, z$ were chosen so as to satisfy $\| y + z \| < \ve \| y - z\|$, we clearly cannot have $y = z = 0$ and we once again have the contradiction we seek.   
\end{proof}

\subsection{Elementary Results from Operator Theory}
We will make repeated use of the following duality results when constructing operators in later chapters of this thesis.
\begin{lem}\label{TquotientiffT*iso}
Let $X$ be a Banach space, $Y$ a normed space, $T\colon X \to Y$ a bounded linear operator. Then $T$ is a quotient operator if, and only if, $T^*\colon Y^* \to X^*$ is an isomorphic embedding.
\end{lem}
\begin{proof}
Suppose first that $T$ is a quotient operator. Then there is some $M > 0$ such that $B_{Y}(0 ; 1) \subseteq M T\big( B_{X}(0;1)\big)$. It follows that \[
M\|T^*y^* \| = \sup_{x\in B_X(0 ; 1)} |y^*(MTx)| \geq \sup_{y\in B_{Y}(0 ; 1)} |y^*y| = \|y^*\|
\]
i.e. that $T^*$ is an isomorphic embedding.
Suppose conversely that $T^*$ is an isomorphic embedding and let $M>0$ be such that $M\|T^*y^*\| \geq \|y^*\|$. We claim that $M \overline{T\big( B_{X}(0 ; 1) \big)} \supseteq B_{Y} (0 ; 1)$. Since $X$ is a Banach space, this is sufficient to deduce that $T\colon X \to Y$ is a quotient operator. Let's suppose by contradiction that \[
M \overline{T\big( B_{X}(0 ; 1) \big)} = \bigcap_{\ve > 0} MT\big(B_{X}(0;1)\big) + \ve B_{Y}^{\circ}(0 ; 1) \nsupseteq B_{Y}^{\circ}(0;1)
\]
So there is some $y_0 \in B_{Y}^{\circ}$ and $\ve > 0$ such that $y_0 \notin MT\big( B_{X}(0 ; 1) + \ve B_{Y}^{\circ}$. Applying Separation Theorem \ref{SeparationThm}, we can find a non-zero $y^* \in Y^*$ with
\[
M\|T^*y^*\| + \ve\|y^*\| = M\sup_{x\in B_{X}(0;1)} \text{Re } y^*Tx + \ve\sup_{y\in B_{Y}^{\circ}} \text{Re } y^*y \leq \text{Re } y^*(y_0) \leq \|y^*\|
\]
But this implies that $(1+\ve)\|y^*\| \leq \| y^* \|$ which is clearly a contradiction.

\end{proof}

\begin{lem} \label{w*tow*impliesdual}
Let $X$ and $Y$ be Banach spaces and suppose $T \colon X^* \to Y^*$ is a bounded linear operator which is also weak$^*$ to weak$^*$ continuous. Then there is a bounded linear operator $S \colon Y \to X$ such that $S^* = T$. 

Furthermore, if for each $n \in \N$, $T_n \colon X^* \to X^*$ is a bounded linear operators which is weak* to weak* continuous and $T \colon X^* \to X^*$ is a bounded linear operator with $\|T_n - T \| \to 0$ (i.e. $T_n \to T$ in the operator norm topology), then $T$ is also weak* to weak* continuous.
\end{lem}

\begin{proof}
Let $J_{X} \colon X \to X^{**}$, $J_{Y} \colon Y \to Y^{**}$ denote the canonical embeddings. We first see that $T^*J_Y$ maps $Y$ into $J_X(X)$. It is enough to see that for every $y \in Y$, the map $T^*J_Yy \colon X^* \to \mathbb{K}$ is weak$^*$ continuous. To this end, we suppose $(x_{\alpha}^*)$ is a net converging in the weak$^*$ topology to some $x^*$ (we write $x_{\alpha}^* \wsto x^*$).  By weak$^*$-weak$^*$ continuity of $T$, we have $Tx_{\alpha}^* \wsto Tx^*$ so that \[
(T^*J_Yy)x_{\alpha}^* = J_Yy(Tx_{\alpha}^*) = Tx_{\alpha}^* (y) \to Tx^*(y) = (T^*J_Yy)x^* \]
as required. We can thus define a bounded linear map $S \colon Y \to X$ by $S := J_X^{-1}T^*J_Y$. We note that for all $x^* \in X^*, y \in Y$ \[
(S^*x^*)y = x^*(Sy) = x^*(J_X^{-1}T^*J_Yy) = T^*J_Yy(x^*) = (Tx^*)y
\]
so that $S^* = T$ as required.

To prove the second part of the lemma, note that by the first part of the proof, each $T_n$ is the dual operator of an operator $S_n \colon X \to X$. Precisely, $S_n = J_{X}^{-1}T_n^*J_X$. Since $T_n \to T$ with respect to the operator norm, $T_n^* \to T^*$ in operator norm. It follows that $(S_n)_{n=1}^{\infty}$ is a Cauchy sequence in $\mathcal{L}(X)$, so there is some bounded linear operator $S$ such that $S_n \to S$. It follows that $S_n^* \to S^*$ in $\mathcal{L}(X^*)$. However, $S_n^* = T_n$ for all $n$, and since $T_n \to T$, uniqueness of limits gives that $T = S^*$. Since $T$ is the dual of some operator, it is certainly weak* continuous.  
\end{proof}

\begin{lem}\label{lemtoproveT*closedRangeImpliesTClosedRange}
Let $X$ and $Y$ be Banach spaces, $T\colon X \to Y$ a bounded linear operator. If $T^*$ is one to one and has closed range then $\im T = Y$.
\end{lem}
\begin{proof}
Since $T^*$ is one to one and has closed range in the Banach space $Y$, $T^*$ is an isomorphism onto its image and thus $T^*\colon Y^* \to X^*$ is an isomorphic embedding. It follows from Lemma \ref{TquotientiffT*iso} that $T$ is a quotient operator. In particular $T$ is onto.
\end{proof}
\begin{thm}[Closed Range Theorem] \label{ClosedRangeThm}
Let $X$ and $Y$ be Banach spaces, $T\colon X \to Y$ a bounded linear operator. Then $T$ has closed range if and only if $T^*$ has closed range.
\end{thm}
\begin{proof}
The proof is essentially taken from \cite{DS}.

Let us suppose first that $T$ has closed range. Since $Y$ is a Banach space and $\im T$ is closed, $T_1: X \to T[X]$ defined by $T_1x = Tx$ is a quotient operator. It follows by Lemma \ref{TquotientiffT*iso} that $T_1^*: T[X]^* \to X^*$ is an isomorphic embedding and consequently that $\im T_1^*$ is closed. We claim that $\im T_1^* = \im T^*$ and thus the image of $T^*$ is closed as required. Indeed suppose $z^* \in T[X]^*$ and let $y^* \in Y^*$ be an extension of $z^*$ to $Y$ (the existence of which is of course guaranteed by the Hahn Banach Theorem). Then for $x \in X$, \[
(T_1^*z^*)x = z^*T_1 x = y^*Tx = (T^*y^*)x
\]
So $T_1^*z^* = T^*y^*$ and $\im T_1^* \subseteq \im T^*$. Conversely, if $y^* \in Y^*$, we let $z^* \in T[X]^*$ be the restriction of $y^*$ to $T[X]$. It is easy to see that $T^*y^* = T_1^*z^*$ and so $\im T^* \subseteq \im T_1^*$. It follows that $\im T_1^* = \im T^*$ as required.

Conversely, suppose $T^*$ has closed range. We let $Z \subseteq Y$ be the Banach space $\overline{T[X]}$ and consider the map $T_1 : X \to Z$ defined by $T_{1}x = Tx$. Since $T_1$ has dense range, it follows that $T_1^*$ is one-to-one. If $x^* \in X^*$ is in the closure of $T_1^*(Z^*)$ then $x^* = \lim_{n\to\infty} T_1^*z_n^*$ where $z_n^*$ are in $Z^*$. We apply the Hahn Banach Theorem obtaining extensions $y_n^* \in Y^*$ of the $z_n^*$. It follows that $x^* = \lim_{n\to\infty}T^*y_n^*$ and since $T^*$ has closed range, $x^* = T^*y^*$ for some $y^* \in Y^*$. If $z^*$ is the restriction of $y^*$ to $Z$, then $x^* = T_1^*z^*$. Hence $T_1^*$ is one-to-one and has closed range. It follows by Lemma \ref{lemtoproveT*closedRangeImpliesTClosedRange} that $\im T_1 = \im T = Z$ and so $T$ has closed range as required.
\end{proof}

As well as the above duality results, we will make use of the following basic sequence technique which provides a sufficient condition for an operator to be compact.
\begin{prop} \label{CompactLemma}
Let $X$ be a Banach space with a Schauder basis and $T \colon X \to X$ a bounded linear operator. If $\,Tx_k \to 0$ for all bounded block basic sequences $(x_k)_{k=1}^{\infty}$, then $T$ is compact. If we demand that the basis of $X$ is shrinking, then the converse is also true.
\end{prop}
\begin{proof}
Denote by $(e_n)_{n=1}^{\infty}$ the Schauder basis of $X$. We noted in Section \ref{BasicSequenceTechniques} that the projections $P_n \colon X \to [ e_j : j \leq n ]$ defined by $\sum_{j=1}^{\infty} a_j e_j \mapsto \sum_{j=1}^n a_j e_j$ are bounded. We will show that $\| T - TP_n \| \to 0$ (as $n \to \infty$); this shows that $T$ is a uniform limit of finite rank operators and consequently that $T$ is compact. 

Suppose for contradiction that $\| T - TP_n \| \arrownot\to 0$. It follows that we can find $\delta > 0$, a strictly increasing sequence $(N_j)_{j=1}^{\infty}$ of natural numbers, and a sequence of norm 1 vectors $(x_j)_{j=1}^{\infty}$ such that \[
\| (T - TP_{N_j}) x_j \| > \delta
\]
Since $(e_n)_{n=1}^{\infty}$ is a Schauder basis, we can find $M_1 > N_1$ such that $\| x_1 - P_{M_1} x_1 \| < \delta / 2 \|T\|$. We set $y_1 = (P_{M_1} - P_{N_1}) x_1$ noting that $\|y_1 \| \leq 2K$ (where $K$ is the basis constant of the basis $(e_n)_{n=1}^{\infty}$). Moreover,
\begin{align*}
\| Ty_1 \| &= \| Tx_1 - TP_{N_1} x_1 - (Tx_1 - TP_{M_1} x_1) \| \\
&\geq \| (T - TP_{N_1})x_1 \| - \| T \circ (I - P_{M_1})x_1 \| > \frac{\delta}{2}
\end{align*}

Now, we can find $N_{j_2} > M_1$ and $M_2 > N_{j_2}$ such that $\| x_{j_2} - P_{M_2} x_{j_2} \| < \delta / 2 \|T\|$. We set $y_2 = (P_{M_2} - P_{N_{j_2}}) x_{j_2}$. Estimating as before yields $\| y_2 \| \leq 2K$ and $\| T y_2 \| > \delta / 2 $. Continuing in this way, we obtain a bounded block basic sequence $(y_n)_{n=1}^{\infty}$ with $\| T y_n \| > \delta / 2 $ for all $n$. This contradicts the hypothesis that $Tx_n \to 0$ whenever $(x_n)_{n=1}^{\infty}$ is a bounded block basic sequence, completing the first part of the proof.

The converse of the theorem does not hold in general. However, in our statement of a partial converse, we demand that the basis is shrinking. Let $(x_n)_{n=1}^{\infty}$ be a bounded block basic sequence. We are required to show that if $T$ is compact, $Tx_n \to 0$. Arguing by contradiction we assume that, after choosing a subsequence and relabelling if necessary, there is $\delta > 0$ such that $\|T x_n \| \geq \delta $ for all $n$. We note that $(x_n)_{n=1}^{\infty}$ is weakly null by Proposition \ref{Kalton3.2.7}. Consequently $(Tx_n)_{n=1}^{\infty}$ converges weakly to $0$. 

Since $T$ is compact, some subsequence $(Tx_{n_j})_{j=1}^{\infty}$ converges in norm to $x \in X$. It follows that $Tx_{n_j} \to x$ in the weak topology. By the above argument and  uniqueness of weak limits we have that $x = 0$.  Since $(Tx_{n_j})_{j=1}^{\infty}$ converges in norm to $x$, it follows that $\|Tx_{n_j} \| \to 0$, contradicting the fact that $\| Tx_{n_j}  \| \geq \delta$ for all $j$. This completes the proof.
\end{proof}
\subsection{Complexification} \label{complexification}

We will require results from spectral theory. Of course, many of these results assume that the Banach algebras in question are over the field of complex scalars. In order to pass to the real case, we recall some basic complexification arguments. If $X$ is a real Banach space, the complexified space is defined as $X_{\C} := X \oplus X$ with (complex) scalar multiplication and vector addition defined by 
\begin{align*}
&(x, y) + (u,v) = (x+u, y+v) \quad \quad \quad \, \forall x, y, u, v \in X \\
&(\alpha + i \beta)(x,y) = (\alpha x - \beta y, \beta x + \alpha y ) \,\,\,\, \forall \alpha , \beta \in \R, x,y  \in X.
\end{align*}

Obviously we can identify $X$ as a (real) subspace of the complexification under the linear injection given by $j \colon X \to X_{\C}$, $x \mapsto (x, 0)$. Noting that $i(y,0) = (0,y)$, we can write the vector $(x,y) \in X_{\C}$ as $j(x) + i j(y)$ and it is obvious that this representation is unique. We will often find it convenient to suppress the use of the embedding $j$ and simply write $z \in X_{\C}$ by $z = x + iy$, where $x, y \in X \subseteq X_{\C}$. Consequently, we write $X_{\C} = X\oplus i X$.

There are many ways to define a complex norm on $X_{\C}$. For the purposes of this thesis, we shall work with the norm defined by $\| x + iy \|_{X_{\C}} : = \sup_{t\in [0, 2\pi] } \| x \cos t - y \sin t \|_{X}$. It is a trivial exercise to check that this defines a norm and we only give the proof of the homogeneity. Note that if $\cos \theta + i \sin \theta \in \C$,
\begin{align*}
\| (\cos \theta + i \sin \theta )(x + iy ) \|_{X_{\C}} &= \sup_{t \in [0, 2 \pi] } \| (x \cos \theta  - y \sin \theta ) \cos t - (x \sin \theta  + y \cos \theta ) \sin t \|_X \\
&= \sup_{t \in [0, 2\pi ] } \| (\cos \theta \cos t - \sin \theta \sin t )x - (\sin \theta \cos t + \cos \theta \sin t ) y \|_X \\
&= \sup_{t \in [0, 2 \pi ] } \| \cos (\theta + t) x - \sin (\theta + t) y \|_X \\
&= \| x+iy\|_{X_{\C}}
\end{align*}
so that $| \lambda | \| z \|_{X_{\C}} = \| \lambda z \|_{X_{\C}}$ for all $\lambda \in \C, z \in X_{\C}$.

We remark that whenever $x, y \in X$, we have $\| j(x) \|_{X_{\C}} = \|x \|_{X}$ and $\| x + iy \|_{X_{\C}} = \| x - iy \|_{X_{\C}}$. Using these observations we obtain the following lemma.
\begin{lem} \label{NormofRePart}
Let $x, y \in X$. Then $\|x \|_X = \| x \|_{X_{\C}} \leq  \| x+ iy \|_{X_{\C}}$ and $\|y \|_X = \| y \|_{X_{\C}} \leq  \| x+ iy \|_{X_{\C}}$
\end{lem}

\begin{proof}
Just note that $2\|x\|_{X_{\C}} \leq \| x+ iy \|_{X_{\C}} + \| x - iy \|_{X_{\C}} = 2 \|x + iy \|_{X_{\C}}$. The second part of the lemma is similar. 
\end{proof}

If $T$ is a (real) operator on $X$, then we can uniquely extend it to a (complex) operator $T_{\C}$ on $X_{\C}$ by $T_{\C} (x + iy )  = Tx + i Ty$. We have the following lemma:
\begin{lem} \label{complexifiedops}
If $T \in \mathcal{L}(X)$ then $T_{\C} \in \mathcal{L}(X_{\C})$ and $\|T\|_{X \to X} = \| T_{\C} \|_{X_{\C} \to X_{\C}}$. Moreover, if $T$ is compact, then so is $T_{\C}$. 
\end{lem}

\begin{proof}
It is easy to see that $T_{\C}$ is a linear operator. Moreover, it is clear that $\|T_{\C} \| \geq \| T \|$ since $T_{\C}$ extends $T$ (and $\| j(x) \|_{X_{\C}} = \| x \|_X)$. Note,
\begin{align*}
\|T_{\C} (x + iy) \|_{X_{\C}} &= \sup_{t \in [0, 2 \pi ] } \| (Tx) \cos t - (Ty) \sin t \|_X =   \sup_{t \in [0, 2 \pi ] } \| T ( x \cos t -  y \sin t ) \|_X \\
&\leq \|T\|  \sup_{t \in [0, 2 \pi ] } \| x \cos t - y \sin t \|_X = \|T\| \| x + iy \|_{X_{\C}}
\end{align*}
so that $\|T_{\C}\| \leq \| T \|$. 

To see that $T_{\C}$ is compact when $T$ is, note that if $(x_n + i y_n)_{n=1}^{\infty}$ is a bounded sequence in $X_{\C}$ then the sequences $(x_n)_{n=1}^{\infty}$ and $(y_n)_{n=1}^{\infty}$ are bounded sequences in $X$ by Lemma \ref{NormofRePart}. By compactness of $T$, we can choose subsequences $(x_{n_k})_{k=1}^{\infty}$ and $(y_{n_k})_{k=1}^{\infty}$ such that both $(Tx_{n_k})$ and $(Ty_{n_k})$ converge. It follows easily that $T_{\C}(x_{n_k} + i y_{n_k} )$ converges in $X_{\C}$, so $T_{\C}$ is compact as required.
\end{proof} 

\subsection{Strictly Singular Operators} \label{SSoperators}
The majority of this thesis is concerned with the interplay between the compact and strictly singular operators. We recall in this section the definition and some some of the elementary results on strictly singular operators which will be used throughout the thesis.

We recall the following definition
\begin{defn}
Let $X, Y$ be Banach spaces and $T\colon X \to Y$ a bounded linear operator. $T$ is {\em strictly singular} if whenever $Z \subseteq X$ is a subspace such that there exists some $\delta >0$ with $\|Tz \| \geq \delta \|z \|$ for all $z \in Z$, then $Z$ is finite dimensional.
\end{defn}
In other words, a strictly singular operator is one which is not an isomorphic embedding on any infinite dimensional (closed) subspace. The following characterisation of strictly singular operators on a Banach space with a Schauder basis will be particularly useful to us later on in this thesis. 
\begin{prop} \label{SSiffSSonblocks}
Let $X$ be a Banach space with a basis $(e_n)_{n=1}^{\infty}$ and suppose $T\colon X \to X$ is a bounded linear operator on $X$. Then $T$ is strictly singular $\iff$ whenever $[y_n : n \in \N]$ is a block subspace of $X$, the restriction, $T|_{[y_n]} \colon [y_n] \to X$ is not an isomorphic embedding.
\end{prop}
\begin{proof} 
Clearly when $T$ is strictly singular, we have the condition about block subspaces. To complete the proof, we show that if $T$ is an isomorphism on some closed infinite dimensional subspace $Y$ of $X$, then $T$ must be an isomorphism on some block subspace of $X$. 

Indeed, suppose for some infinite dimensional, closed subspace $Y$ of $X$, there is $\delta > 0$ such that for every $y \in Y, \|Ty\| \geq \delta\|y\|$. Choose any  $0<\theta < \half$ such that $\|T\|\theta (\half - \theta)^{-1} \leq \delta$. By Proposition \ref{BasicSequencesTechLemma}, there exists a sequence $(y_k)_{k=1}^{\infty} \subseteq Y$ with $\|y_k \| = 1$ for every $k$ and a block basic sequence $(x_k)$ (with respect to the basis $(e_n)$) satisfying \[
2K \sum_{k=1}^{\infty} \frac{ \|y_{k} - x_k \| }{\|x_k\|} < \theta
\] where $K$ is the basis constant of the basis $(e_n)_{n=1}^{\infty}$.

We recall that the sequences $(x_k)$ and $(y_{k})$ are equivalent basic sequences. Moreover, the sequence $(x_k)$ as a block sequence of the $(e_n)$ has basis constant at most $K$. We claim that $T$ is an isomorphism on the block subspace $[x_k]$. To see this, let $\sum_k a_k x_k$ be a norm 1 vector in $[x_k]$, and observe that this implies in particular that $|a_k | \leq \frac{2K}{\|x_k\|}$ since $(x_k)$ is a basic sequence with basis constant at most $K$. Then
\begin{align*}
\|T\big(\sum_k a_k x_k \big) \| &= \big\| T\big(\sum_k a_k y_{k} + \sum_k a_k (x_k - y_{k}) \big) \big\| \\
&\geq \big\|T\big(\sum_k a_k y_{k} \big) \big\| - \big\| T\big( \sum_k a_k(x_k - y_{k}) \big) \big\| \\
&\geq \delta \|\sum_k a_k y_{k} \| - \|T\| 2K \sum_k \frac{ \| x_k - y_{k} \|}{\|x_k\|} \\
&\geq \delta \|\sum_k a_k y_{k} \| - \|T\|\theta.
\end{align*}
We can estimate $\| \sum_k a_k y_{k} \|$ as follows:
\[
\| \sum_k a_k y_{k} \| \geq \| \sum_k a_k x_k \| - \| \sum_k a_k (y_{k} - x_k) \| \geq 1 - 2K \sum_k \frac{ \| y_{k} - x_k \|}{\|x_k\|} \geq 1 - \theta.
\]
So, \[
\|T\big(\sum_k a_k x_k \big) \| \geq \delta \|\sum_k a_k y_{k} \| - \|T\|\theta \geq \delta (1 - \theta) - \|T\| \theta \geq \frac{\delta}{2}
\]
with the final inequality resulting in our choice of $\theta$. We have thus seen that any norm 1 vector $x \in [x_k]$ has $\| Tx\| \geq \frac{\delta}{2} > 0$. Thus $T$ is an isomorphism on $[x_k]$ as claimed.
\end{proof}

There are two further observations about strictly singular operators that are important in relation to the work of this thesis. The first is that the strictly singular operators on a Banach space $X$ form a closed ideal of the operator algebra $\mathcal{L}(X)$. The second is that the the essential spectrum of a strictly singular operator on a complex Banach space is just $\{ 0 \}$. (We will actually give a more formal statement of this fact that also holds for real Banach spaces.) In the remainder of this section, we provide proofs of these facts. We find it convenient to introduce another class of operators.

\begin{defn}
Let $T \colon X \to X$ be a bounded linear operator on a Banach space $X$. We say $T$ is {\em upper semi-Fredholm} if the kernel of $T$ is finite dimensional and $T$ has closed range.
\end{defn}

\begin{rem} \label{finitelysingularopsrem}
\hspace{0.1cm}
\begin{enumerate}
\item We recall that if $T \colon X \to X$ has finite dimensional kernel and finite co-dimensional image, then $T$ is said to be a Fredholm operator. The range of $T$ is automatically closed in this case, thus all Fredholm operators are, in particular, upper semi-Fredholm. Moreover, for Fredholm operators, we can define the index of $T$ by \[
\text{ind }T := \text{dim$\, \Ker T$} - \text{codim}\,\text{im}T.
\]  It is well known (see, e.g. \cite{C*algebras}) that $ \text{ind}:  \mathcal{F} \to \Z$ is a continuous map (here $\mathcal{F} \subseteq \mathcal{L}(X)$ denotes the subset of Fredholm operators). For upper semi-Fredholm operators $T$, we can define the generalised index by the same formula as above; this will give a finite integer if $T$ is in fact Fredholm. Otherwise it must be the case that the image of $T$ has infinite co-dimension and the generalised index is set equal to $-\infty$. 

\item Suppose $T \colon X \to X$ is upper semi-Fredholm. Since finite dimensional subspaces are always complemented, we can write $X$ as the topological direct sum $X  = \Ker T \oplus X' $ for some closed, finite co-dimensional subspace $X' \subset X$. Since also $\text{ Im $T$} = T(X')$ is closed, it is easy to see that $T| \colon X' \to X$ is an isomorphic embedding. In other words, we have seen that an upper semi-Fredholm operator is an isomorphism on some finite co-dimensional, closed subspace.
Conversely, if $T$ is an isomorphism on some closed, finite co-dimensional subspace of $X$, then $T$ is upper semi-Fredholm. Indeed, Let $X_1$ be a finite co-dimensional subspace on which $T$ is an isomorphism. It follows that $\Ker T \cap X_1 = \{ 0 \}$ and so we can write $X$ as the (algebraic) direct sum $X = X_1 \oplus \Ker T \oplus F$ for some (necessarily finite dimensional) subspace $F$. Since $X_1$ is finite co-dimensional, $\text{dim } \Ker T \leq \text{dim } (\Ker T \oplus F ) = \text{codim } X_1 < \infty$. The image of $T$ is closed since $T(X) = T(X_1) \oplus T(F)$ and the sum is easily seen to be topological ($T(X_1)$ is closed since $T$ is an isomorphism on $X_1$ and $T(F)$ is finite dimensional).

We have therefore proved that an operator is upper semi-Fredholm if and only if it is an isomorphism on some finite co-dimensional subspace. In some literature (e.g. \cite{GowMau}, \cite{LT}) one defines an operator to be {\em finitely singular} if it is an isomorphism on some finite co-dimensional subspace. We have therefore proved that the class of finitely singular operators is exactly the same as class of upper semi-Fredholm operators. 

\item We also note that every strictly singular operator on an infinite dimensional Banach space fails to be upper semi-Fredholm.
\end{enumerate}
\end{rem}

Later we will require the following proposition which extends the well known result that the index of a Fredholm operator defines a continuous map into $\Z$, to the case where the operators are upper semi-Fredholm.

\begin{prop} \label{GeneralisedIndex}
If $T \colon X \to Y$ is an upper semi-Fredholm operator, there exists a neighbourhood $\mathcal{U}$ of $0 \in \mathcal{L}(X,Y)$ such that $T + U$ is upper semi-Fredholm for all $U \in \mathcal{U}$ and moreover, the (generalised) index of $T+U$ is equal to the (generalised) index of $T$. 
\end{prop}

\begin{proof}
When $T$ is a Fredholm operator, the result follows from the well known result that the index map, defined on Fredholm operators, is continuous. In the case where $T$ is upper semi-Fredholm but not Fredholm, the result follows easily from the following proposition.
\end{proof}

(We remark that the following proposition can be found in \cite[Proposition 3.1]{MaureyOnline}.)

\begin{prop} \label{CodimClosedUnderPerturbation}
Let $T \colon X \to Y$ be an isomorphic embedding. Then there exists a neighbourhood of $0 \in \mathcal{L}(X,Y)$ such that for every $S$ is this neighbourhood, $T+S$ is an isomorphic embedding and $\Codim (T+S)X = \Codim TX$ (finite or $+\infty$).
\end{prop}

To prove the proposition, we need:

\begin{lem}
Let $T \in \mathcal{L}(X,Y)$ be an isomorphic embedding and $k \geq 0$ an integer. 
\begin{enumerate}
\item If $\Codim TX \geq k$, there exists $c > 0$ such that $\Codim (T+S)X \geq k$ whenever $\|S\| < c$.   
\item If $\Codim TX = k$, there exists $c > 0$ such that $\Codim (T+S)X = k$ whenever $\|S\| < c$. 
\end{enumerate}
\end{lem}

\begin{proof}
Since $T$ is an isomorphic embedding, there exists a $\delta > 0$ such that $\|Tx\| \geq \delta \|x\|$ for all $x \in X$. 

Now, if $\Codim TX \geq k$, there exists a subspace $F \subset Y$ such that $\dim F = k$ and $TX \cap F = \{ 0 \}$. Let $\pi_F : Y \to Y/F$ denote the usual quotient map. We claim that $\pi_F \circ T$ is an isomorphic embedding from $X$ into $Y / F$. Indeed, if not there exists a sequence $(x_n)_{n=1}^{\infty}$ of vectors in $X$ with $\| x_n \| = 1$ for all $n$ and $\|Tx_n + F \| \to 0$. It follows that there exists a sequence of vectors $(f_n)_{n=1}^{\infty}$ in $F$ such that $\| Tx_n - f_n \| \to 0$. It is clear that the sequence $(f_n)_{n=1}^{\infty}$ is bounded and since $F$ is finite dimensional, an easy compactness argument yields that (without loss of generality, after passing to a subsequence and relabelling if necessary) $f_n \to f \in F$. It follows that $Tx_n \to f \in F$. Since $T$ is an isomorphic embedding, $T(X)$ is closed, so $\lim_n Tx_n = f \in TX \cap F = \{ 0 \}$. On the other hand, $\|f \| = \lim_n \|Tx_n \| \geq \delta$. So $\pi_F \circ T$ is an isomorphic embedding as claimed. 

It is easy to see that the property of being an isomorphic embedding is closed under small norm perturbation. Since we know $\pi_F \circ T$ is an isomorphic embedding, it follows that if we take $c >0$ sufficiently small, $\pi_F \circ (T + S) $ is also an isomorphic embedding if $\|S \| < c$. This implies that $F \cap (T+S)X = \{0 \}$, hence $\Codim (T+S)X \geq k$, proving the first claim.

The proof of the second case is similar; in this case we can select the subspace $F$ such that $Y = TX \oplus F$ and $\dim F = k$. It follows (using the same argument as before) that $\pi_F \circ T$ is an onto isomorphism.  Consequently, there exists $c > 0$ such that if $\| S\| < c $ then $\pi_F \circ (T+S)$ is an onto isomorphism. In particular, $F \cap (T+S)X = \{ 0 \}$. Moreover, since $\pi_F \circ (T+S)$ is an onto isomorphism, for every $y \in Y$ there exists $x \in X$ such that $\pi_F (y) = \pi_F ( (T+S)x)$. Consequently, $y - (T+S)x \in F$, showing that $Y = F + (T+S)X$. So $Y = F \oplus (T+S)X$ and $\Codim (T+S)X = k$ as required.
\end{proof}

\begin{proof}[Proof of Proposition \ref{CodimClosedUnderPerturbation}]
Let $c > 0$ be such that $T+S$ is an isomorphic embedding whenever \[
S \in B_{c} : = \{ U \in \mathcal{L}(X,Y) : \|U \| < c \}
\]
For each positive integer $k \geq 0$, the set $D_k = \{ U \in B_c : \Codim (T+U)X = k \}$ is open in $B_c$ by the previous lemma. Moreover, \[
D_k =  \{ U \in B_c : \Codim (T+U)X \geq k+1\}^{\text{C}} \bigcap \cap_{j=0}^{k-1} \{ U \in B_c : \Codim(T+U)X = j \}^{\text{C}},
\]
thus $D_k$ is closed in $B_c$ as an intersection of closed sets (each set in the intersection is a complement of an open set by the previous lemma). 
Since $B_c$ is connected, each $D_k$ is either empty or equal to $B_c$. The result now follows in the case where $\Codim TX = k$ for some positive integer $k \geq 0$. In the case where $\Codim TX = +\infty$, we note that by the preceding argument, $D_k$ has to be empty for every $k \geq 0$. Consequently, $\Codim (T+U)X = +\infty$ for every $U \in B_c$ and the result is again proved.  
\end{proof}

\begin{lem} \label{C1infinitelysingularops}
Let $X$ be a Banach space (real or complex) and suppose $T \colon X \to X$ is a bounded linear operator which fails to be upper semi-Fredholm. Then, for all $\ve > 0$, there exists an infinite dimensional (closed) subspace $Y_{\ve} \subseteq X$ with $\|T|_{Y_{\ve}} \| < \ve$. The restriction of $T$ to $Y_{\ve}$ is a compact operator.
\end{lem}

\begin{proof}
The proof given here follows closely that given in \cite{LT}.  By Remark \ref{finitelysingularopsrem}, we note that $T$ is not an isomorphism on any finite co-dimensional subspace of $X$. 

We begin the proof by fixing a sequence of real numbers $a_n > 0$ such that $\prod_{n=1}^{\infty} (1+a_n) \leq 2$. We will construct inductively a normalised, basic sequence $(x_n)_{n=1}^{\infty}$ with the following properties 
\begin{enumerate}
\item If $m < n$ are natural numbers, $(\lambda_i)_{i=1}^{n}$ are scalars, then \[
\| \sum_{i=1}^m \lambda_i x_i \| \leq \left( \prod_{i=m+1}^{n} (1+a_i) \right) \| \sum_{i=1}^{n} \lambda_i x_i \| 
\]
\item $\|Tx_n \| \leq \ve / 2^{n+2}$ for all $n$.
\end{enumerate}
Note that as a consequence of the first property, $(x_n)$ is a basic sequence with basis constant at most $\prod_{i=1}^{\infty} (1 + a_i) \leq 2$. 
  
To begin the induction, since in particular $T$ is not an isomorphism, we can choose $x_1 \in X$ with $\| x_1 \| = 1$ and $\|Tx_1\| \leq \ve / 8$. 

Inductively, suppose we have constructed $x_1, \dots x_n \in X$ with $\| x_i \| = 1$ and satisfying the two properties above. It is enough to show that we can find $x_{n+1} \in X$ with $\|x_{n+1}\| = 1$, $\|Tx_{n+1} \| \leq \ve /2^{n+3}$ and having the property that whenever $y \in E:= \lin \{ x_1, \dots x_n \}$, $\lambda$ is a scalar, $(1+a_{n+1} ) \| y + \lambda x_{n+1} \|  \geq \|y \|$. 
 
To this end, let $y_1, y_2, \dots y_m$ be a (finite) $\delta$-net for the unit sphere $S_{E}$ of the finite dimensional subspace $E$, where $0 < \delta < 1$ is such that $1 + a_{n+1} \geq 1 / (1-\delta)$. By the Hahn-Banach Theorem, we can choose norm 1 linear functionals $y_1^*, \dots y_m^* \in S_{X^*}$ with $y_i^* (y_i) = 1$ for all $i = 1, \dots m$. By hypothesis of the lemma, $T$ is not an isomorphism when restricted to the finite co--dimensional subspace $Z := \cap_{i=1}^m \Ker y_i^*$, so we may choose a norm 1 vector, $x_{n+1} \in Z$ with $\| Tx_{n+1} \| \leq \ve / 2^{n+3} $. If $y \in S_{E}$, $\lambda$ a scalar, there exists some $i \in \{ 1, \dots m \}$ with $\| y - y_i \| \leq \delta$, and 
\begin{align*}
(1+a_{n+1} ) \| y + \lambda x_{n+1} \| &= (1+a_{n+1}) \| y_i + \lambda x_{n+1} + y - y_i \| \\
&\geq(1+a_{n+1} ) \,  y_i^*(y_i + \lambda x_{n+1} + y - y_i )\\
& \geq (1+a_{n+1}) \left( 1- |y_i^*(y - y_i) |\right) \\
&\geq (1 +a_{n+1} ) \left( 1 - \|y - y_i\| \right) \\
&\geq (1 + a_{n+1} )(1-\delta) \geq 1
\end{align*}
from which we easily conclude that $ \| y + \lambda x_{n+1} \|  \geq \|y \|$ for any $y \in E$ and scalar $\lambda$ as required. 

We set $Y_{\ve} := [ x_n]_{n=1}^{\infty}$, an infinite dimensional subspace of $X$ with Schauder basis $(x_n)_{n=1}^{\infty}$. If $y  = \sum_{n=1}^{\infty} \alpha_n x_n \in Y_{\ve}$ with $\| y \| = 1$, it follows that $|\alpha_n | \leq 4$ for all $n$ and $\|Ty \| =  \| \sum_{n=1}^{\infty} \alpha_n Tx_n \| \leq 4 \ve \sum_{n=1}^{\infty} 2^{-(n+2)} = \ve$, giving the desired estimate for $\|T |_{Y_{\ve}} \|$. It is easy to see (by a similar estimation) that $T$ is the uniform limit of the finite rank operators $T_N : Y_{\ve} \to X$, defined by $\sum_{n=1}^{\infty} \alpha_n x_n \mapsto \sum_{n=1}^N \alpha_n Tx_n$, and thus $T| \colon Y_{\ve} \to X $ is compact as claimed.

\end{proof}

\begin{cor} \label{AlternateCharacterisationOfSSOps}
An operator $T\colon X \to Y$ is strictly singular $\iff$ whenever $Z \subset X$ is an infinite dimensional subspace of $X$, there exists a further infinite dimensional subspace $Z' \subset Z$ such that the restriction of $T$ to $Z'$ is compact. 
\end{cor}

\begin{proof}
Suppose first $T$ is strictly singular and let $Z$ be an infinite dimensional subspace of $X$. We consider the operator $T_1  \colon Z \to X$ where $T_1$ is simply the restriction of $T$ to $Z$. It is easily verified that $T_1$ is strictly singular because $T$ is, and therefore $T_1$ fails to be upper semi-Fredholm. By the previous lemma, we can find an infinite dimensional subspace $Z' \subset Z$ such that the restriction of $T_1$ to $Z'$ is compact. Obviously by definition of $T_1$, this implies that the restriction of $T$ to $Z'$ is compact. 

Conversely, suppose for contradiction that $Z$ is an infinite dimensional subspace of $X$ on which $T$ is an isomorphism. By the assumed property, there exists an infinite dimensional subspace $Z'$ of $Z$ on which $T$ is both a compact operator and an isomorphism. However, these two facts imply that $Z'$ has a compact unit ball, contradicting the fact that $Z'$ is infinite dimensional. It follows that $T$ fails to be an isomorphism on every infinite dimensional subspace, i.e. $T$ is strictly singular. 
\end{proof}

\begin{cor} \label{SSareIdeal}
The set of strictly singular operators on a Banach space $X$ forms a closed ideal of $\mathcal{L}(X)$.
\end{cor}

\begin{proof}
Suppose first that $T, S \colon X \to X$ are strictly singular operators on $X$. We will use Corollary \ref{AlternateCharacterisationOfSSOps} to show $T+S$ is strictly singular. To this end, let $Z$ be an infinite dimensional subspace of $X$. Since $T$ is strictly singular, there exists an infinite dimensional subspace $Z' \subset Z$ on which $T$ is compact. We now apply Corollary \ref{AlternateCharacterisationOfSSOps} to $S$ and the subspace $Z'$, obtaining an infinite dimensional subspace $Z'' \subset Z'$ on which both $S$ and $T$ are compact. It follows that $T+S$ is compact when restricted to $Z''$, and consequently $T+S$ is strictly singular by the corollary.

To prove that the strictly singular operators are an ideal, it remains to prove that if $P$ is strictly singular and $R, Q \in \mathcal{L}(X)$ then $RPQ$ is a strictly singular operator. Suppose $Z$ is an infinite dimensional subspace of $X$ and fix $\ve > 0$. We will exhibit a norm 1 vector, $z \in Z$ with $\|RPQz\| < \ve$. Note that if $Q(Z)$ is finite dimensional, then since $Z$ is infinite dimensional, we must have $\ker Q \cap Z \neq \{ 0 \}$. In this case, we just take $z$ to be any norm 1 vector in $\ker Q \cap Z $ and the proof is trivial. So we assume that $Q(Z)$ is infinite dimensional. Since $P$ is assumed strictly singular, $P$ is not an isomorphism on the infinite dimensional subspace $Q(Z)$. Consequently, we can find a norm 1 vector $x = Q(z_0)$ (for some $z_0 \in Z$) with $\| Px \| = \| PQz_0\| < \ve / \|Q\|\|R\|$. We take $z = z_0 / \| z_0 \|$ noting that $1/\|z_0 \| \leq \|Q \|$ since $1 = \|x \| = \| Qz_0 \|$, so that $\| PQ z \| < \|Q\| \ve / \|Q\| \|R\|$. Therefore, $\|RPQz\| \leq \|R \| \|PQz\| < \ve$ as required. 

Finally, note if $T_n$ is a sequence of strictly singular operators on $X$ with $\| T_n - T \| \to 0$ then $T$ is strictly singular, so that the ideal of strictly singular operators is closed. Indeed, if $Z$ is an infinite dimensional subspace of $X$, $\ve > 0$, then we find $N$ such that $\| T_N - T \| < \ve/2$. Since $T_N$ is strictly singular, there exists $z \in Z$ with $\|z\| = 1$ and $\|T_N z \| < \ve / 2$. It follows that $\| T z \| \leq \|T_N z \| + \| (T - T_N)z \| < \ve$.

\end{proof} 

We conclude this section by stating and proving a result about strictly singular operators that we will need in Chapter \ref{l1Calk}. 

\begin{thm} \label{SSimpliesspecradius0}
Let $X$ be a Banach space (real or complex). If $T \in \mathcal{L}(X)$ is strictly singular then $r(T) := \lim_{n\to \infty} \| T^n + \mathcal{K}(X) \|^{\frac1n} = 0$. 
\end{thm}

In the case where $X$ is a complex Banach space, the above result just says that the essential spectrum of a strictly singular operator on a complex Banach space is $\{0 \}$. To prove the result in the real case, we will make use of a complexification argument and will consequently need the following technical lemma; it is essentially a generalisation of Lemma \ref{C1infinitelysingularops}, with the only real difficulty being to make sure we extract a real subspace of the complexified space. 

\begin{lem} \label{infinitelysingularopscomplexified}
Let $T \colon X \to X$ be a bounded linear operator on a real Banach space $X$, and denote by $X_{\C} = X \oplus i X$ the complexification of $X$, $T_{\C}$ the complex extension of $T$ to $X_{\C}$ (as in Section \ref{complexification}). Suppose there exists $\lambda \in \C$ such that $T_{\C} - \lambda \text{Id}_{X_{\C}}$ fails to be upper semi-Fredholm. Then the operator $T_{\lambda}$ on $X$ defined by $T_{\lambda} = T^2 - 2\text{Re}\lambda T + |\lambda^2| \text{Id}_X$ has the property that for all $\ve > 0$, there exists a real infinite dimensional $Y_{\ve} \subseteq X$ with $\|T_{\lambda}|_{Y_{\ve}}\| < \ve$.   
\end{lem}

\begin{proof}
The proof is very similar to that of Lemma \ref{C1infinitelysingularops} and we continue the notation from that proof. We will construct inductively a normalised sequence $(x_n)_{n=1}^{\infty} \subseteq X$ with  the following properties:
\begin{enumerate}
\item If $m < n$ are natural numbers, $(\lambda_i)_{i=1}^{n}$ are real scalars, then \[
\| \sum_{i=1}^m \lambda_i x_i \| \leq \left( \prod_{i=m+1}^{n} (1+a_i) \right) \| \sum_{i=1}^{n} \lambda_i x_i \| 
\]
\item $\|T_{\lambda} x_n \| \leq \ve / 2^{n+2}$ for all $n$.
\end{enumerate}

%Note that we interchange freely between thinking of $T_{\lambda}$ as an operator on $X$, and as an operator on the subspace of $X_{\C}$ isometric to  $X$.

As in the proof of Lemma \ref{C1infinitelysingularops} , it is a consequence of the first property that $(x_n)$ is a real basic sequence with basis constant at most $\prod_{i=1}^{\infty} (1 + a_i) \leq 2$. Given such a basic sequence, we take $Y_{\ve} := [x_n]_{n=1}^{\infty} \subseteq X$ to complete the proof (which is then identical to the proof of the previous lemma). 

Observe that whenever $x, y \in X$,  $\left( T_{\C} - \overline{\lambda} \text{Id}_{X_{\C}} \right) \left( T_{\C} - \lambda \text{Id}_{X_{\C}} \right)(x + iy) = T_{\lambda} x + i T_{\lambda} y = (T_{\lambda})_{\C}(x+iy)$. This, combined with the fact that $T_{\C} - \lambda \text{Id}_{\C} $ is not an isomorphism on any finite co-dimensional subspace of $X_{\C}$ (see Remark \ref{finitelysingularopsrem}), implies that $(T_{\lambda})_{\C}$ is not an isomorphism on any finite co-dimensional subspace of $X_{\C}$. In particular, we can choose $z_1 \in X_{\C}$ with $\| z_1 \| = 1$ and $\|(T_{\lambda})_{\C}\,  z_1\| \leq \ve / 16$. Since $\|z_1\| = 1$, writing $z_1$ uniquely as $z_1 = w_1 + i y_1$ with $w_1, y_1 \in X$, we must have either $\|w_1 \| \geq \frac12$, in which case we define $x_1 := w_1/ \|w_1 \|$, or $\| y_1 \| \geq \frac12$ (in which case we define $x_1 = y_1 / \| y_1 \|$). In either case, $\|T_{\lambda} x_1 \| \leq \ve / 8$. Indeed, considering the case where $x_1 = w_1/ \|w_1 \|$, we have \[
\| T_{\lambda} x_1 \| \leq \frac{1}{\|w_1\|} \| T_{\lambda} w_1 + i T_{\lambda} y_1 \| = \frac{1}{\|w_1\|} \|(T_{\lambda})_{\C}\,  z_1\| \leq \frac{\ve}{8}.
\]
(The first inequality above follows from Lemma \ref{NormofRePart}, and the final inequality from the fact that $\|w_1\| \geq \frac12$.) The other case is dealt with similarly.

Inductively, suppose we have constructed $x_1, \dots x_n \in X$ with $\| x_i \| = 1$ and satisfying the two properties above. As in the proof of Lemma \ref{C1infinitelysingularops}, it is enough to show that we can find $x_{n+1} \in X$ with $\|x_{n+1}\| = 1$, $\|T_{\lambda} x_{n+1} \| \leq \ve /2^{n+3}$ and having the property that whenever $y \in E:= \lin_{\R} \{ x_1, \dots x_n \}$, $\lambda \in \R$, $(1+a_{n+1} ) \| y + \lambda x_{n+1} \|  \geq \|y \|$. 
 
To this end, let $y_1, y_2, \dots y_m$ be a (finite) $\delta$-net for the unit sphere $S_{E}$ of the finite dimensional subspace $E$, where $0 < \delta < 1$ is such that $1 + a_{n+1} \geq 1 / (1-\delta)$. By the Hahn-Banach Theorem, we can choose norm 1 linear functionals $y_1^*, \dots y_m^* \in S_{X^*}$ with $y_i^* (y_i) = 1$ for all $i = 1, \dots m$. Then $(y_1^*)_{\C}, \dots (y_m^*)_{\C}$ are norm 1 linear functionals on $X_{\C}$ with $(y_i^*)_{\C}(y_i) = 1$ for all $i$. By our earlier observation, $(T_{\lambda})_{\C}$ is not an isomorphism when restricted to the finite co--dimensional subspace $Z := \cap_{i=1}^m \Ker (y_i^*)_{\C} \subseteq X_{\C}$. It follows that we can choose a norm 1 vector, $z_{n+1} \in Z$ with $\| (T_{\lambda})_{\C} \, z_{n+1} \| \leq \ve / 2^{n+4}$. Writing $z_{n+1} = w_{n+1} + i y_{n+1}$ with $w_{n+1}, y_{n+1} \in X$, we must have (for the same reason as earlier) either $\|w_{n+1} \| \geq \frac12$ or $\| y_{n+1} \| \geq \frac12$. As before, we set $x_{n+1}$ to be either  $w_{n+1}/  \| w_{n+1} \|$ or $y_{n+1} / \|y_{n+1} \|$ (depending on which of $w_{n+1}, y_{n+1}$ has norm at least $\frac12)$. The same argument as before yields $\| T_{\lambda} x_{n+1} \| \leq \ve / 2^{n+3}$. Note also that for $j \in \{1, \dots m \}$, $0 = (y_j^*)_{\C} \, z_{n+1} = y_j^*(w_{n+1}) + iy_j^*(y_{n+1})$ and thus $y_j^*(w_{n+1}) = 0 =  y_j^*(y_{n+1})$. It follows that $y_j^* (x_{n+1}) = 0$ for all $j$.  Finally, if $y \in S_{E}$, $\lambda \in \R$ there is some $i \in \{ 1, \dots m \}$ with $\| y - y_i \| \leq \delta$, and

\begin{align*}
(1+a_{n+1} ) \| y + \lambda x_{n+1} \| &= (1+a_{n+1}) \| y_i + \lambda x_{n+1} + y - y_i \| \\
&\geq(1+a_{n+1} ) \,  y_i^*(y_i + \lambda x_{n+1} + y - y_i )\\
& \geq (1+a_{n+1}) \left( 1- | y_i^*(y - y_i) | \right) \\
&\geq (1 +a_{n+1} ) \left( 1 - \|y - y_i\| \right) \\
&\geq (1 + a_{n+1} )(1-\delta) \geq 1
\end{align*}
from which we easily conclude that $ \| y + \lambda x_{n+1} \|  \geq \|y \|$ for any $y \in E, \lambda \in \R$ as required. 

\end{proof}

\begin{proof}[Proof of Theorem \ref{SSimpliesspecradius0}]
Note that the limit certainly exists (in fact, it is a classical result that $\inf_n \| T^n + \mathcal{K}(X) \|^{\frac1n} = \lim_{n\to \infty} \| T^n + \mathcal{K}(X) \|^{\frac1n} $). We consider first the case where $X$ is a complex Banach space and argue by contradiction, assuming $T$ is a strictly singular operator with $r(T) > 0$.  Note it is enough to show that there exists some non-zero $\lambda \in \C$, $0 <  \ve  < |\lambda| $ and an infinite dimensional subspace $Y_{\ve} \subseteq X$ with $\|Ty - \lambda y \| \leq \ve \|y \|$ whenever $ y \in Y_{\ve}$, since this contradicts the strict singularity of $T$.  The argument is identical to that used by Gowers and Maurey to show that every operator on a complex HI space is a strictly singular perturbation of the identity. For completeness we include it here.

Denote by $\Omega_T$ the set of all $\mu \in \C$ such that $T-\mu I$ is upper semi-Fredholm; we recall that the generalised index of such an operator is equal to some finite integer or $-\infty$.

We show that if $r(T) > 0$, there is a non-zero $\lambda \in \C \setminus \Omega_T$. It is a well known theorem of Atkinson that an operator $U \in \mathcal{L}(X)$ is Fredholm if and only if the class $[U]$ of $U$ in the Calkin algebra is invertible. It follows that the essential spectrum of T, $\sigma_{\text{ess}}(T) := \{ \lambda \in \C : T- \lambda I \text{ is not Fredholm} \} $ is precisely the spectrum, $\sigma ([T])$,  of the element $[T] \in \mathcal{L}(X) / \mathcal{K}(X)$.  It is well known that $\sigma ([T]) \neq \varnothing$ (since we are working over complex scalars) and moreover, it follows from well known results about the spectral radius that there exists $\lambda \in \sigma ([T])$ with $|\lambda| = \max_{\mu \in \sigma ([T])} |\mu| = r(T) > 0$.  We claim that $\lambda \notin \Omega_T$. Note that if $\mu \in \C$ with $|\mu | > |\lambda |$ then $\mu \notin \sigma ( [T] )$ and consequently $T-\mu I$ is Fredholm. Moreover, by continuity of the index, $T - \mu I$ is Fredholm with index 0 for all $|\mu | > |\lambda |$ (because this is true when $\mu$ is large enough to make $T - \mu I$ invertible, i.e. when $|\mu| > \| T \|$). 

Now if $\lambda \in \Omega_T$, we know that $T - \lambda I$ must be upper semi-Fredholm with generalised index $-\infty$ (since $T - \lambda I$ is not Fredholm by choice of $\lambda$). It follows from Proposition \ref{GeneralisedIndex} that the generalised index of $(T-\mu I) = -\infty$ for all $\mu$ close to $\lambda$, but this contradicts the fact that $T- \mu I $ is Fredholm with index 0 for all $|\mu| > |\lambda |$. Thus $0 \neq \lambda \in \C \setminus \Omega_T$ as required.

Since $\lambda \notin \Omega_T$, it follows by Lemma \ref{C1infinitelysingularops} that, given $\ve > 0$, we can find an infinite dimensional subspace $Y_{\ve}$ of $X$ with $\| Ty - \lambda y \| \leq \ve \| y \|$ for every $y \in Y_{\ve}$. Taking any $\ve < |\lambda|$ completes the proof of the lemma in the complex case.

To prove the real case, we use a complexification argument that appeared in \cite{GowMau}, though is attributed to Haydon (see \cite[Lemma 20]{GowMau}). As in Section \ref{complexification}, we let $X_{\C} = X \oplus i X$ denote the complexification of $X$ with the norm $\| x + i y \|_{X_{\C}} := \sup_{t\in [0, 2\pi]}  \| x \cos t - y \sin t \|_{X}$.  We recall that when $L \in \mathcal{L}(X)$, there is a unique complex linear extension of $L$, $L_{\C} \in \mathcal{L}(X_{\C})$, defined by $L_{\C}(x + iy) = Lx + iLy$, Moreover, $\| L_{\C} \| = \| L \|$ and if $K \in \mathcal{K}(X)$, then the complex linear extension $K_{\C} \in \mathcal{K}(X_{\C})$. We refer the reader back to Lemma \ref{complexifiedops} for the details. From these observations, it easily follows that $\| L + \mathcal{K}(X) \| = \| L_{\C} + \mathcal{K}(X_{\C}) \|$. 

Now assume as before that $T$ is a strictly singular operator with $r(T) > 0$. It follows by the preceding remarks that $\lim_{n\to \infty} \| (T_{\C})^n + \mathcal{K}(X_{\C}) \|^{\frac1n} = \lim_{n\to \infty} \| T^n + \mathcal{K}(X) \|^{\frac1n} > 0$.  So by the proof just given for the complex case, there exists a non-zero $\lambda \in \C$ such that the operator $T_{\C} - \lambda I_{X_\C} \colon X_{\C} \to X_{\C}$ fails to be upper semi-Fredholm. It follows by Lemma \ref{infinitelysingularopscomplexified} that for all $\ve > 0$, there is an infinite dimensional subspace $Y_{\ve} \subset X$ such that $\| T_{\lambda} |_{Y_{\ve}} \| < \ve $ where we recall that $T_{\lambda}$ is the operator $T^2 - 2\text{Re}\lambda T + |\lambda|^2I \in \mathcal{L}(X)$. If we choose $\ve < |\lambda ^2| $, then it follows that $T^2 - 2Re \lambda T$ is an isomorphism on $Y_{\ve}$. But, if $T$ is strictly singular then so is $T^2 - 2Re \lambda T$ by Corollary  \ref{SSareIdeal}, so we once again have a contradiction.
\end{proof}

%% file: BDExposition.tex
\chapter{The Bourgain-Delbaen Construction} \label{ChBDConst}

\section{Introduction}

In 1980, Bourgain and Delbaen \cite{BD80, B81} introduced two classes of separable $\mathcal{L}_{\infty}$-spaces, $\mathcal{X}$ and $\mathcal{Y}$. We shall look at these classes in detail and give a precise definition of an $\mathcal{L}_{\infty}$ space a little later in this chapter. For now, we simply remark that these spaces were shown to have a number of interesting properties, providing counterexamples for many previously unanswered conjectures in Banach space theory. For instance, none of these spaces have a subspace isomorphic to $c_0$, giving the first set of examples of $\mathscr{L}_{\infty}$ spaces with no $c_0$ isomorph. In fact, it was shown in \cite{BD80} that if $Y \in \mathcal{Y}$, then every infinite dimensional subspace of $Y$ contains a further infinite dimensional subspace which is reflexive. 

More recently, Argyros and Haydon \cite{AH} have managed to modify the original construction to exhibit a space which solves the scalar-plus-compact problem. Although it is a question remaining from Argryos' and Haydon's solution to this problem that really motivates our interest in the Bourgain-Delbaen spaces, we concern ourselves in this chapter with investigating the original Bourgain-Delbaen construction and the newer generalisation used by Arygros and Haydon. We are a little relaxed with our terminology, referring to a Banach space constructed using either of the methods just mentioned as a {\em space of Bourgain-Delbaen type}. As is remarked in \cite{AH}, it is interesting (and worth emphasising) that Bourgain-Delbaen constructions are different from the majority of other Banach space constructions that occur in the literature. It is usual to start with the vector space of finitely supported scalar sequences, $c_{00}$, and complete with respect to some exotic norm. We shall shortly see that spaces of Bourgain-Delbaen type are `exotic subspaces' of $\ell_{\infty}$, i.e. the norm is just the usual $\| \cdot \|_{\infty}$ norm. 

We shall begin the chapter by looking at the Argyros-Haydon generalisation in Section \ref{(s1)GBDconstruction}. Here spaces of Bourgain-Delbaen type are subspaces of $\ell_{\infty}$, obtained as the closed linear span of the biorthogonal functionals of a special kind of Schauder basis for $\ell_1$. We will show that we can make a further, minor modification to the Argyros-Haydon construction. The new generalisation that we obtain will be an essential tool in proving the results of later chapters and as such, we choose to look at it in detail.

Following this, we briefly recall the original Bourgain-Delbaen construction in Section \ref{(s1)BDconstruction}, though we present it via the notation used by Haydon in \cite{Haydon2000}. This really is just notational convenience; it is clear that if one enumerates the set $\Gamma$ (defined in Section \ref{(s1)BDconstruction}) in a natural way, we get  an obvious isometry from the subspace of $\ell_{\infty}(\Gamma)$ appearing in Section \ref{(s1)BDconstruction}, and the subspace of $\ell_{\infty}(\N)$ defined in the original work of Bourgain and Delbaen in \cite{BD80}.

Having looked in detail at the two different constructions, we will then,  in some sense, unify them in Section \ref{unifyingBDAH} by showing that the Argyros-Haydon construction is a genuine generalisation of that due to Bourgain and Delbaen. Precisely, we will show the original Bourgain-Delbaen construction can be obtained as a special case of the Argyros-Haydon construction. Of course, this is not a new result. However, it is, to the author's best knowledge, the only documented exposition describing in detail the relationship between the two constructions. 

One of the advantages of the Arygros-Haydon generalisation is that it gives us a potential way to construct interesting operators on a space of Bourgain-Delbaen type, $B$. The idea is to define an operator on $\ell_1$, take the dual operator, and restrict to $B$. Of course, some care needs to be taken to ensure that the restriction also maps {\em into} $B$. 

Having shown how to describe the original spaces of Bourgain and Delbaen in the new framework of the Argyros-Haydon construction, we use the preceding idea to conclude the chapter (Section \ref{nonsepBDalgebras}) by providing a partial answer to a question of Beanland and Mitchell, \cite{KB2010}, namely if $Y \in \mathcal{Y}$, is $\mathcal{L}(Y)$ separable? Precisely, we will show that for any $X \in \mathcal{X} \cup \mathcal{Y}$, there exists a constant $C_X > 0$ and an uncountable collection of isometries on $X$ which are pairwise distance $C_X$ apart with respect to the operator norm.  Clearly this shows $\mathcal{L}(X)$ is non-separable for any $X \in \mathcal{X}\cup\mathcal{Y}$. We will note in Section \ref{nonsepBDalgebras} that there is in fact a shorter argument to see that $\mathcal{L}(X)$ is non-separable for $X \in \mathcal{X}$ and also defer to this section a short discussion of why the author finds this problem interesting.

\section{The generalised Bougain-Delbaen Construction} \label{(s1)GBDconstruction}

Before continuing any further, we recall that what we want to be able to do is construct interesting $\mathscr{L}_{\infty}$ spaces. After all, it was seen in \cite{AH} that the finite-dimensional subspace structure possessed by such spaces is what proves to be so fundamental in showing that the Argyros-Haydon space solves the scalar-plus-compact problem. Moreover, we will exploit similar techniques to prove results about the operator algebras of the new spaces we construct in subsequent chapters of this thesis. It is about time we define precisely what we mean by an $\mathscr{L}_{\infty}$ space. 

\begin{defn} \label{Linftyspace}
A separable Banach space $X$ is an $\mathscr
L_{\infty,\lambda}$-space if there is an increasing sequence
$(F_n)_{n\in \N}$ of finite dimensional subspaces of $X$ such
that the union $\bigcup_{n\in \N}F_n$ is dense in $X$ and, for
each $n$, $F_n$ is $\lambda$-isomorphic to $\ell_\infty^{\dim F_n}$.
We say a Banach space is an $\mathscr{L}_{\infty}$ space if it is a $\mathscr{L}_{\infty, \lambda}$ space for some $\lambda$. 
\end{defn}

The constructions of Bourgain and Delbaen, and the generalisation due to Argyros and Haydon, provide a way of constructing interesting classes of $\mathscr{L}_{\infty}$ spaces. Consequently, we begin by looking in detail at these constructions, starting with the Argyros-Haydon generalisation. The content of this section follows very closely the work in \cite{AH} and we make no claim to originality unless explicitly stated. 

As remarked in the introduction to this chapter, the idea of the Argyros-Haydon construction is to construct a particular kind of Schauder basis for the space $\ell_1$ and to study the subspace $X$ of $\ell_\infty$ spanned by the biorthogonal elements.  It follows by Proposition \ref{biorthogonalsarebasicsequence} that $\ell_1$ naturally embeds into $X^*$. (Moreover, by Theorem \ref{boundedlycompleteshrinkingbasis}, this embedding is onto $X^*$ precisely when the Schauder basis of $\ell_1$ is boundedly complete.) Consequently, we will think of elements of $\ell_1$ as functionals and, as in \cite{AH}, denote them using a star notation, $b^*$, $c^*$ and so on. In accordance with this notation, we denote the canonical basis of $\ell_1(\N)$ by $(e_n^*)_{n=1}^{\infty}$. 

It is perhaps easiest to understand the construction by working first with $\ell_1(\N)$ and this is indeed what is done in \cite{AH}. The special kind of Schauder basis of $\ell_1(\N)$ is obtained by considering a sequence $(d_n^*)_{n=1}^{\infty} \subseteq \ell_1(\N)$ where each $d_n^*$ has the form $d_n^* = e_n^* - c_n^*$ with $c_1^* = 0$ and $\supp c_n^* \subseteq \{ 1, 2, \dots , (n-1) \}$ for $n \geq 2$.  An easy induction argument yields that the vectors $(d_n^*)_{n=1}^{\infty}$ are linearly independent and that moreover, the linear span $[d^*_1,d^*_2,\dots,d^*_n]$ is the same as $[e^*_1,e^*_2,\dots,e^*_n]$, so the closed linear span, $[d^*_n : n \in \N ]$ is the whole of $\ell_1$. 

The clever part of the construction is to choose the $c_n^*$ in such a way that the $d_n^*$ form a Schauder basis for $\ell_1$. In fact, things are a little more subtle than this, in the sense that the sequence $(d_n^*)_{n=1}^{\infty}$ will actually be a basic sequence equivalent to the canonical basis of $\ell_1$ if the $c_n^*$ have sufficiently small norm. This case certainly won't be useful in producing any interesting Banach spaces, though we defer further discussion of this to Lemma \ref{BDNeedLargeNormPerturbation} so as not to detract from the point at hand. 

Argyros and Haydon showed that the $d_n^*$ will always form a Schauder basis for $\ell_1$ if the $c_n^*$ assume a certain form (we refer the reader to \cite{AH} for more details). For this reason, it turns out that it is more convenient to work with the space $\ell_1(\Gamma)$ for some suitably defined countable set $\Gamma$; the elements $\gamma \in \Gamma$ can then be used to code the form of the vector $c_{\gamma}^* \in \ell_1(\Gamma)$. We choose to immediately state (and prove) the main result of this section, working with this more convenient notation, and refer the reader to \cite{AH} for a more comprehensive introduction. 

\begin{thm}\label{BDThm}
Let $(\Delta_q)_{q\in \N}$ be a disjoint sequence of non-empty
finite sets; write $\Gamma_q=\bigcup_{1\le p\le
q}\Delta_p$, $\Gamma=\bigcup_{ p\in \N}\Delta_p$. Assume that
there exists $\theta<\frac12$ and a mapping $\tau$ defined on
$\Gamma\setminus \Delta_1$, assigning to each $\gamma\in
\Delta_{q+1}$ a tuple of one of the following forms:
\begin{enumerate}
\setcounter{enumi}{-1}
 \item $(\ve, \alpha, \xi)$ with $\ve = \pm 1, 0\le \alpha\le 1$ and $\xi\in \Gamma_q$;
 \item $(p,\beta, b^*)$ with $0\le p< q$, $0<\beta\le \theta$ and $b^*\in \ball
 \ell_1\left(\Gamma_q\setminus \Gamma_p\right)$;
 \item $(\ve, \alpha,\xi,p,\beta,b^*)$ with $\ve = \pm 1,  0<\alpha\le 1$, $1\le p <q$, $\xi\in
 \Gamma_p$, $0<\beta\le \theta$  and $b^*\in \ball
 \ell_1\left  (\Gamma_q\setminus \Gamma_p\right)$.
 \end{enumerate}
Then there exist $d_\gamma^* = e^*_\gamma-c^*_\gamma\in
\ell_1(\Gamma)$ and projections $P^*_{(0,q]}$ on $\ell_1(\Gamma)$
uniquely determined by the following properties:
\begin{enumerate}
  \item[(A)]
  $\displaystyle P^*_{(0,q]}d^*_\gamma = \begin{cases}
 d^*_\gamma\qquad\qquad\qquad\qquad\qquad\text{if }\gamma\in \Gamma_q\\
 0\ \qquad\qquad\qquad\qquad\qquad\text{if }\gamma\in \Gamma\setminus
 \Gamma_q
 \end{cases}$
 \
 \item[(B)] $\displaystyle \qquad\,c^*_\gamma = \begin{cases}
  0  \qquad\qquad\qquad\qquad\qquad\ \text{if }\gamma\in \Delta_1\\
 \ve \alpha  e^*_\xi \qquad \qquad\qquad\qquad\quad\text{if } \tau(\gamma) = (\ve, \alpha,
  \xi)\\
  \beta (I-P^*_{(0,p]}) b^*\qquad\quad \quad\text{
  if }\tau(\gamma)=(p, \beta, b^*)\\
        \ve \alpha e^*_\xi + \beta(I-P^*_{(0,p]}) b^*\quad\ \text{
 if }\tau(\gamma)=(\ve, \alpha, \xi, p,\beta, b^*).
 \end{cases}$
 \end{enumerate}
 The family $(d^*_\gamma)_{\gamma\in \Gamma}$ is a basis for
 $\ell_1(\Gamma)$ with basis constant at most $M=
 (1-2\theta)^{-1}$.  The norm of each projection $P^*_{(0,q]}$ is at
 most $M$.  The biorthogonal vectors $d_\gamma$ generate a
 $\mathscr L_{\infty ,M}$-subspace $X(\Gamma,\tau)$ of $\ell_\infty(\Gamma)$; if the basis $(d_{\gamma})_{\gamma \in \Gamma}$ of $X(\Gamma, \tau)$ is shrinking, then $X^*$ is naturally isomorphic to $\ell_1(\Gamma)$.
 For each $q$ and each $u\in \ell_\infty(\Gamma_q)$, there is a
 unique vector in $[d_\gamma:\gamma\in \Gamma_q]$ whose restriction
 to $\Gamma_q$ is $u$; therefore there exists a unique extension operator
 $i_q:\ell_\infty(\Gamma_q)\to X(\Gamma,\tau) \cap [ d_{\gamma} : \gamma \in \Gamma_q ] $ and this operator has norm at most $M$.
  The subspaces $M_q=[d_\gamma:\gamma\in \Delta_q]
  =i_q[\ell_\infty(\Delta_q)]$ form a
finite-dimensional decomposition (FDD) for $X$.
\end{thm}

\begin{remark}
Strictly speaking, a Schauder basis is, in particular, a sequence indexed by the natural numbers, so there is a natural ordering of the vectors in the sequence. Therefore when we talk of $(d^*_\gamma)_{\gamma\in \Gamma}$ being a basis of $\ell_1(\Gamma)$, or $(d_{\gamma})_{\gamma\in\Gamma}$ being a basis of $X(\Gamma, \tau)$, we really need to enumerate $\Gamma$ in some way. The enumeration we have in mind will always be the one that is described in the proof of this theorem, and we will henceforth speak of (for example) $(d_{\gamma})$ being a basis without mentioning the explicit enumeration.
\end{remark}

We observe that the statement here is almost identical to that of Theorem 3.5 in \cite{AH}. We have added an extra degree of freedom to tuples of form (0) and (2) appearing above  via the addition of the parameter $\ve$ which takes values in $\{ \pm 1 \}$. This means, for example, we can have vectors $c^*$ of the form $-\alpha e_{\xi}^* + \beta (I - P^*_{(0, q]})b^*$.  This change will only be useful for when we come to examine the connection between this construction and the original Bourgain-Delbaen construction. In later chapters, we will always set $\ve = 1$.

The more important change we have made is that we do not demand the set $\Delta_1$ have only one element. This extra freedom will allow us to construct some non-trivial operators on the spaces we construct in later chapters. Since this theorem is so essential to this thesis, we include the proof, though all we are really doing is translating the arguments given in \cite{AH} into the notation used throughout this thesis. 

\begin{proof}[Proof of Theorem \ref{BDThm}]
The construction is recursive and to begin with, we work only algebraically, not worrying about the continuity of the projection maps we define. We start by setting $d_{\gamma}^* = e_{\gamma}^*$ for $\gamma \in \Delta_1$. 

Recursively, assume that $c_{\gamma}^*$ and $d_{\gamma}^*$ have been defined for all $\gamma \in \Gamma_n$ in accordance with the theorem. It follows from the recursive construction that the vectors $( d_{\gamma}^* )_{\gamma \in \Gamma_n}$ are linearly independent and that their span is all of $\ell_1(\Gamma_n)$. Consequently, the (restricted) projections $P_{(0,q]}^* |_{\ell_1{(\Gamma_n)}}$ for $q \leq n$ are uniquely determined on $\ell_1(\Gamma_n)$ by the formula given in (A). For $\gamma \in \Gamma_{n+1}$, we can therefore define vectors $c_{\gamma}^*$ as in (B) and set $d_{\gamma}^* = e_{\gamma}^* - c_{\gamma}^*$. Since $c_{\gamma}^*$ has support in $\Gamma_n$ and $\gamma \in \Delta_{n+1}$ so that $\supp e_{\gamma}^* \subseteq \Delta_{n+1}$, we see that the vectors $(d^*_{\gamma ' })_{\gamma ' \in \Gamma_{n+1}}$ are linearly independent and that their span is $\ell_1(\Gamma_{n+1})$. We can thus extend the projections $P_{(0, q]}^*$ for $q \leq n$  to all of $\ell_1(\Gamma_{n+1})$ in accordance with (A) and define $P^*_{(0,n+1]}$ on $\ell_1(\Gamma_{n+1})$ in the same way. This completes the recursive construction; we obtain vectors $c_{\gamma}^*$ and  $d_{\gamma}^* = e_{\gamma}^* - c_{\gamma}^*$ for each $\gamma \in \Gamma$ and projections $P_{(0,q]}^*$ defined on $c_{00}(\Gamma) \subseteq \ell_1(\Gamma)$, satisfying the conclusions of the theorem. 

We now show that $(d_{\gamma}^*)_{\gamma \in \Gamma}$ is a Schauder basis for $\ell_1(\Gamma)$. In the process, we shall see that the projections $P^*_{(0,q]}$ defined on $c_{00}(\Gamma)$ are (uniformly) bounded by $M$, and thus extend to projections on all of $\ell_1(\Gamma)$. Let us be a little more precise; for $p \geq 1$ let $ k_p=\#\Gamma_p$, set $k_0 = 0$ and let $n\mapsto \gamma(n):\N\to \Gamma$ be
a bijection with the property that for each $q\ge1$, $\Delta_q = \{\gamma(n):k_{q-1}<n\le k_q\}$. We will show $(d_{\gamma(n)}^*)_{n=1}^{\infty}$ is a Schauder basis. It is clear from the recursive construction that this sequence has dense linear span, so it is enough to see that the (densely defined) basis projections, $P_m^*$ defined by \[
d_{\gamma(j)}^* \mapsto \begin{cases} d_{\gamma(j)}^* & \text{ if } j \leq m \\ 0 & \text{ otherwise} \end{cases}
\]
are uniformly bounded. We show $\|P_m^*\| \leq M$ for all $m$. Since $P^*_{(0,q]} = P^*_{k_q}$ for all $q$, we immediately get $\| P^*_{(0,q]} \| \leq M$ as claimed in the theorem. We also note that $P_m^* e_{\gamma(j)}^* = e_{\gamma(j)}^*$ for all $j \leq m$ and $0$ otherwise since $\lin \{ d_{\gamma(j)}^* : j \leq m \} = \lin \{ e_{\gamma(j)}^* : j \leq m \}$; this is again obvious from the recursive construction. 

Since we are working on the space $\ell_1(\Gamma)$ it is enough to show that $\|P^*_me^*_{\gamma(n)}\|\le M$ for every
$m$ and $n$ in $\N$. We shall prove by induction on $n$ that $\| P_{m}^* e_{\gamma(j)}^* \| \leq M$ for all $m, j \leq n$. 

Note that for $j \leq k_1$, $d_{\gamma(j)}^* = e_{\gamma(j)}^*$. Consequently, for $j \leq k_1$, $P_m^* e_{\gamma(j)}^*$ is either equal to $e_{\gamma(j)}^*$ or $0$. In particular, there is nothing to prove for our inductive statement in the case when $1 \leq n \leq k_1$. Suppose inductively that $\| P_{m}^* e_{\gamma(j)}^* \| \leq M$ holds for all $m, j \leq n$ (w.l.o.g. $n \geq k_1$). We must see that this holds for all $m, j \leq n+1$. If $j \leq m$, then, as observed in the previous paragraph, $P^*_m e_{\gamma(j)}^* = e_{\gamma(j)}^*$ so there is nothing to prove. We therefore assume $j > m$. If, in addition, $m < j \leq n$ we are done by the inductive hypothesis. So we just need to see that $\|P^*_m e_{\gamma(n+1)}^* \| \leq M$ for all $m \leq n$. Since we assume $n \geq k_1$, there exists a $q \geq 1$ with $\gamma(n+1) \in \Delta_{q+1}$. Note that this implies 
\begin{equation} \label{Gammaqsubset}
\Gamma_q \subseteq \{ \gamma(1), \gamma(2), \dots , \gamma(n) \}. 
\end{equation}
We now make use of the fact that
$$
e^*_{\gamma(n+1)}= d^*_{\gamma(n+1)}+ c_{\gamma(n+1)}^*, $$
where $\supp c_{\gamma(n+1)}^* \subseteq \Gamma_q \subseteq \{ \gamma(1), \gamma(2), \dots , \gamma(n) \}$ and moreover $c_{\gamma(n+1)}^*$ has one of the forms described in the theorem. 
We consider only the case where $$
c_{\gamma(n+1)}^* = \ve \alpha e_{\xi}^* + \beta (I-P^*_{(0, p]})b^*,
$$
where $\ve =\pm 1, 1\le p < q, \xi \in \Gamma_p, b^* \in \text{ball }\ell_1 (\Gamma_q \setminus \Gamma_p)$ and $\alpha, \beta$ are as in the theorem, with $\beta\le \theta$ by our hypothesis. The other cases are similar.

Now, because $n+1>m$ we have $P^*_m d^*_{\gamma(n+1)}=0$. We can also write $\xi = \gamma(j)$ where $j \leq k_p < k_q \leq n$ (this follows from the fact that $\xi \in \Gamma_p$ with $p < q$ and Equation \ref{Gammaqsubset}). As observered earlier $P^*_{(0,p]} = P^*_{k_p}$. Combining these observations, we get
$$
P^*_me^*_{\gamma(n+1)} = \ve \alpha P^*_me^*_{\gamma(j)} + \beta (P^*_m - P^*_{\min \{ m , k_p \} })b^*.
$$

If $k_p\ge m$ the second term vanishes so that
$$
\|P^*_me^*_{\gamma(n+1)}\|= \alpha\|P^*_me^*_{\gamma(j)}\| \le \|P^*_me^*_{\gamma(j)}\|,
$$
which is at most $M$ by our inductive hypothesis.

If, on the other hand, $k_p<m$, we certainly have $j<m$ (since recall $j \leq k_p$) so that
$P_m^*e^*_{\gamma(j)}=e^*_{\gamma(j)}$, leading to the estimate
$$
\|P^*_me^*_{\gamma(n+1)}\|\le  \alpha\|e^*_{\gamma(j)}\| + \beta\|P^*_mb^*\| +
\beta\|P^*_{k_p}b^*\|.
$$
Now $b^*$ is a convex combination of functionals $\pm e^*_{\gamma(l)}$ with
$l\le k_q \leq n$, and our inductive hypothesis is applicable to all of
these.  We thus obtain
$$
\|P^*_me^*_{\gamma(n+1)}\|\le  \alpha + 2M\beta\le 1+2M\theta= M,
$$
by the definition of $M=1/(1-2\theta)$ and the assumption that
$0\le\beta\le \theta$.

We conclude that $(d_{\gamma(n)}^*)_{n=1}^{\infty}$ is a Schauder basis for $\ell_1(\Gamma)$ as claimed. It follows from Proposition \ref{biorthogonalsarebasicsequence} that the biorthogonal functionals $(d_{\gamma})_{\gamma \in \Gamma}$ form a Schauder basis for the Banach space $X:= X(\Gamma, \tau) = [d_{\gamma} : \gamma \in \Gamma ] \subseteq \ell_{\infty}(\Gamma)$, or more precisely that the sequence $(d_{\gamma(n)})_{n=1}^{\infty}$ is a Schauder basis for $X$. We easily deduce from this that the finite dimensional subspaces $M_q := [d_\gamma : \gamma \in \Delta_q]$ form a finite dimensional decomposition of $X$. Moreover, standard results in Banach space theory yield that $X^*$ is naturally isomorphic to $\ell_1$ in the case that the basis $(d_{\gamma})_{\gamma\in\Gamma}$ is shrinking (see Theorem \ref{boundedlycompleteshrinkingbasis}).

We must also show that $X$ is a $\mathscr{L}_{\infty , M}$ space. To see this, we consider the projections $P^*_{(0, n]}$ defined in the theorem. We have just shown these are well-defined and moreover $\|P^*_{(0,n]}\| \leq M$.  If we modify $P^*_{(0,n]}$ by
taking the codomain to be the image $\text{im}\,P^*_{(0,n]}=\ell_1(\Gamma_n)$,
rather than the whole of $\ell_1(\Gamma)$, what we have is a quotient 
operator, which we shall denote $q_n$, of norm at most $M$. The
dual of this quotient operator is an isomorphic embedding
$i_n:\ell_\infty(\Gamma_n)\to \ell_\infty(\Gamma)$, also of norm at most $M$. Of course, this is immediate by Lemma \ref{TquotientiffT*iso}. However, we can also show it explicity; if $u\in \ell_\infty(\Gamma_n), \gamma \in \Gamma_n$, we have
$$
(i_nu)(\gamma) = \langle i_nu, e_{\gamma}^* \rangle= \langle u,  q_ne^*_{\gamma} \rangle =
\langle u , e^*_{\gamma} \rangle = u(\gamma).
$$
So $i_n$ is an {\em extension operator} $\ell_\infty(\Gamma_n) \to
\ell_\infty(\Gamma)$ and we have
$$
\|u\|_\infty \le \|i_nu\|_\infty \le M\|u\|_\infty
$$
for all $u\in \ell_\infty(\Gamma_n)$.  We claim that the image of $i_n$ is precisely $[d_{\gamma} \colon \gamma \in \Gamma_n]$ and that for each $\gamma \in \Delta_n$, $i_ne_{\gamma} = d_{\gamma}$ so that $i_n(\ell_{\infty}(\Delta_n)) = M_n$ as claimed. We prove these facts in the following lemma (Lemma \ref{i_nlem}) as we shall use them again later. 

Now, since the image of $i_n$ is $M$-isomorphic to $\ell_{\infty}(\Gamma_n) \equiv \ell_{\infty}^{k_n}$ and $\cup_{n\in \N} [d_{\gamma} : \gamma \in \Gamma_n]$ is dense in $X$, it follows that $X$ is a $\mathscr L_{\infty, M}$-space. 

To complete the proof, it remains to prove that for each $n \in \N$ and $u \in \ell_{\infty}(\Gamma_n)$ there exists a unique vector in $[d_{\gamma} : \gamma \in \Gamma_n]$ which when restricted to $\Gamma_n$ is precisely the vector $u$. In fact, we have already shown the existence of this vector; $i_n(u)$ satisfies this property by the above argument. So we only need to prove uniqueness. Suppose $v(u) = \sum_{\gamma \in \Gamma_n} \alpha_{\gamma} d_{\gamma}$ is such a vector. It follows that $\alpha_{\gamma} = v(u)(d_{\gamma}^*) = u(d_{\gamma}^*)$; the final equality must be true because $\supp d_{\gamma}^* \subset \Gamma_n$. Thus the $\alpha_{\gamma}$ are uniquely determined by $u$ and so $v(u)$ is unique.
\end{proof}

\begin{lem} \label{i_nlem}
The image of the map $i_n \colon \ell_{\infty}(\Gamma_n) \to \ell_{\infty}(\Gamma)$ defined in the proof of the above theorem is precisely $[d_{\gamma} \colon \gamma \in \Gamma_n]$. Moreover, for each $\gamma \in \Delta_n$, $i_n e_{\gamma} = d_{\gamma}$. 
\end{lem}

\begin{proof}
We saw that the operator $i_n$ is an isomorphic embedding, so its image is $k_n$-dimensional, where as in the previous proof, $k_n = \#\Gamma_n$. Since the vectors $(d_\gamma)_{\gamma \in \Gamma_n} \subseteq \ell_{\infty}(\Gamma)$ are linearly independent, it is enough to show that for each $\gamma \in \Gamma_n$ there is a $u \in \ell_{\infty}(\Gamma_n)$ such that $i_nu = d_{\gamma}$. Since $[d_{\theta}^* : \theta \in \Gamma] = \ell_1(\Gamma)$, it would be enough to see there is a $u \in \ell_{\infty}(\Gamma_n)$ such that $\langle i_nu , d_{\theta}^* \rangle = \langle d_{\gamma} , d_{\theta}^* \rangle = \delta_{\theta, \gamma}$. This ends up being an easy exercise in linear algebra.

We choose to give an alternative proof. Fix $\gamma \in \Gamma_n$. We will find $u \in \ell_{\infty}(\Gamma_n)$ such that $i_n u(\theta) = d_{\gamma}(\theta)$ for all $\theta \in \Gamma$. Since $i_n$ is an extension operator, if such a $u$ exists, we must have \[
u(\theta) = i_nu (\theta) = d_\gamma (\theta)
\]
whenever $\theta \in \Gamma_n$. Consequently we set $u = \left(d_{\gamma}(\theta) \right)_{\theta \in \Gamma_n}  \in \ell_{\infty}(\Gamma_n)$. All that remains to be seen is that $i_nu(\theta) = d_\gamma(\theta)$ for all $\theta \in \Gamma \setminus \Gamma_n$. To this end we note that 
\begin{align*}
i_nu(\theta) = \langle i_n u , e_{\theta}^* \rangle = \langle u , P_{(0,n]}^* e_{\theta}^* \rangle &= \sum_{\xi \in \Gamma_n} d_{\gamma}(\xi)[P_{(0,n]}^*e_{\theta}^*](\xi) \\
&= \langle d_{\gamma} , P_{(0, n]}^* e_{\theta}^* \rangle \\
&= \langle d_{\gamma} , e_{\theta}^* \rangle \\
&= d_{\gamma} (\theta).
\end{align*}
The penultimate equality follows from the fact that $d_{\gamma} (x^*) = d_{\gamma} (P_{(0, n]}^* x^* )$ for all $\gamma \in \Gamma_n$ and $x^* \in \ell_1(\Gamma)$. This is a consequence of the biorthogonalilty of the sequences $(d_{\gamma}^*) \subseteq \ell_{1}(\Gamma)$ and $(d_{\gamma}) \subseteq \ell_{\infty}(\Gamma)$.

Finally, note that if $\gamma \in \Delta_n$ and $\theta \in \Gamma_n$ then \[
d_{\gamma}(\theta) = \langle d_{\gamma}, e_{\theta}^* \rangle = \langle d_{\gamma}, d_{\theta}^* \rangle + \langle d_{\gamma}, c_{\theta}^* \rangle =  \langle d_{\gamma}, d_{\theta}^* \rangle
\]
since $c_{\theta}^* \in \ell_1(\Gamma_{n-1})$ so $\langle d_{\gamma}, c_{\theta}^* \rangle = 0$. Therefore $d_{\gamma}(\theta) \neq 0$ only when $\gamma = \theta$, in which case $d_{\gamma}(\theta) = 1$. In other words, when $\gamma \in \Delta_n$, $u = e_{\gamma}$, proving that $i_n e_{\gamma} = d_{\gamma}$ as claimed.
\end{proof}

Having proved the main theorem of this section, we return to a remark that we noted earlier. One can think of the basis $(d_{\gamma}^*)_{\gamma \in \Gamma}$ of $\ell_1(\Gamma)$ obtained in the Bourgain-Delbaen construction as a perturbation (by $c_{\gamma}^*$) of the canonical basis $(e_{\gamma}^*)_{\gamma \in \Gamma}$. To obtain any interesting Banach spaces, we note that the norms of the perturbing vectors, i.e. the $c_{\gamma}^*$'s has to be `large'. If not, then the basis $(d_{\gamma}^*)_{\gamma \in \Gamma}$ will be equivalent to the canonical basis of $\ell_1$ and the corresponding Bourgain-Delbaen space will simply be an isomorph of $c_0$. More precisely, we have the following lemma. 

\begin{lem} \label{BDNeedLargeNormPerturbation}
Let $(d_n^*)_{n=1}^{\infty}$ be a sequence of vectors in $\ell_1(\N)$ with the property that for each $n \in \N$ there exists a vector $c_n^* \in \ell_1(\N)$ with $d_n^* = e_n^* - c_n^*$. Suppose further that there exists $\theta < 1$ such that $\|c_n^* \|_{\ell_1} \leq \theta$ for all $n$. Then $(d_n^*)_{n=1}^{\infty}$ is equivalent to the canonical basis $(e_n^*)_{n=1}^{\infty}$ of $\ell_1$ and the closed subspace of $\ell_{\infty}$ generated by the biorthogonal vectors, $[d_n : n \in \N]$, is isomorphic to $c_0$. 
\end{lem} 

\begin{proof}
It is easily checked that $T \colon \ell_1 \to \ell_1$, $T (\sum_{j=1}^{\infty} a_j e_j^* ) = \sum_{j=1}^{\infty} a_j d_j^*$ defines a bounded operator on $\ell_1$ (with norm at most $1 + \theta$). To see the sequence $(d_n^*)$ is equivalent to the canonical basis, it is enough to see that $T$ has continuous inverse. By standard results, it is enough to show that $\| I - T \| < 1$. This is easy, since $\| (I - T) \sum_{j=1}^{\infty} a_j e_j^* \|_{\ell_1} = \| \sum_{j=1}^{\infty} a_j (e_j^* - d_j^*) \| = \| \sum_{j=1}^{\infty} a_j c_j^* \| \leq \theta \| (a_j) \|_{\ell_1}$ by the assumed norm condition on the $c_n^*$ vectors. It follows that $\| I - T \| \leq \theta < 1$.

The fact that $X:= [d_n : n \in \N]$ is isomorphic to $c_0$ follows by taking the dual of $T$ and restricting to $X$. It is straightforward to check that $T^*$ maps the vector $d_n \in \ell_{\infty}$ to the vector $e_n \in \ell_{\infty}$. Moreover, since $T$ is an isomorphism, so is $T^*$, and consequently $T^*(X)$ is closed. The previous two observations imply that $T^*$ restricted to $X$ is an isomorphism onto $T^*(X) = [e_n : n \in \N] = c_0$. 
\end{proof}

We conclude this section by proving that the basis, $(d_{\gamma})_{\gamma \in \Gamma}$, of any Bourgain-Delbaen space, $X(\Gamma, \tau)$, is normalised. 

\begin{lem} \label{BasisIsNormalized}
Continuing with the notation as above, if $X(\Gamma, \tau)$ is a Bourgain-Delbaen space as in Theorem \ref{BDThm}, then the basis $(d_{\gamma})_{\gamma \in \Gamma}$ is normalised. 
\end{lem}

\begin{proof}
Fix $\nu \in \Gamma$. We prove by induction on $n$ that if $\gamma \in \Delta_n$ then $|d_{\nu} (e_{\gamma}^*) | \leq 1$. When $\gamma \in \Delta_1$, $e_{\gamma}^* = d_{\gamma}^*$ and the proof is obvious. So assume the estimate holds for all $\gamma ' \in \Gamma_k$ for some $k$ and choose $\gamma \in \Delta_{k+1}$. We make use of the decomposition $e_{\gamma}^* = d_{\gamma}^* + c_{\gamma}^*$ and assume that $c_{\gamma}^* = \ve \alpha e_{\xi}^* + \beta P_{(l,k]}^* b^*$ for some $l < k$, $\alpha \leq 1$, $\beta \leq \theta < \frac12$, $\xi \in \Gamma_l$ and $b^* \in \text{ball }\ell_1(\Gamma_k \setminus \Gamma_l)$; the other possible forms of $c_{\gamma}^*$ can be treated similarly and are in any case easier. Now,
\[
 d_{\nu} (e_{\gamma}^*) = \langle d_{\nu} , d_{\gamma}^* +  \ve \alpha e_{\xi}^* + \beta P_{(l,k]}^* b^* \rangle
 \]
 We consider two possible cases. If $d_{\nu}(d_{\gamma}^*) \neq 0$, then we must have $\nu = \gamma$. In particular $\nu \in \Delta_{k+1}$ and $d_{\nu}(e_{\gamma}^*) = d_{\nu}(d_{\gamma}^*) = 1$ since $\supp c_{\gamma}^* \subseteq \ell_1(\Gamma_k)$. Consequently the inequality holds.
 
 The remaining case is when $d_{\nu}(d_{\gamma}^*) = 0$. We divide this case into two further possibilities. If, in addition, $\nu \in \Gamma_l$ then it follows that $d_{\nu}(e_{\gamma}^*) = d_{\nu} (\ve \alpha e_{\xi}^* )$ and consequently, $|d_{\nu}(e_{\gamma}^*)| \leq \alpha |d_{\nu}(e_{\xi}^*)| \leq 1$ by the inductive hypothesis and the fact that $\alpha \leq 1$. 
 
Finally, if $d_{\nu}(d_{\gamma}^*) = 0$ and $\nu \in \Gamma \setminus \Gamma_l$, $d_{\nu}(e_{\gamma}^*) = \langle d_{\nu} , \beta P_{(l,k]}^* b^* \rangle = \beta \langle P_{(l,k]}d_{\nu} , b^* \rangle$. This last expression is either 0 (if $\nu \in \Gamma \setminus \Gamma_k$) or $\beta \langle d_{\nu} , b^* \rangle$ otherwise. In either case, it is again easily seen using the inductive hypothesis (and that $b^* \in \text{ball }\ell_1(\Gamma_k)$) that $|d_{\nu}(e_{\gamma}^*)| \leq 1$ as required.

It follows that $\|d_{\nu}\| \leq 1$. On the other hand, it is easily seen that $d_{\nu} (e_{\nu}^*) = \langle d_{\nu} , d_{\nu}^* + c_{\nu}^* \rangle = 1$ so that $\| d_{\nu} \| \geq 1$. It follows that $\| d_{\nu} \| = 1$ as required. 
\end{proof}

\section{The Bourgain-Delbaen Construction} \label{(s1)BDconstruction}
We will return to the construction just described shortly.  For now, we move on to look at a few details of the original Bourgain-Delbaen construction so that we can explain the relationship between the two constructions in the following section. The notation that follows is that used by Haydon in \cite{Haydon2000}. 

We begin by setting $\Delta_{1}$ to be the set consisting of just one element. Inductively, if we have sets $\Delta_k$ defined for $k = 1, 2, \dots, n$, we set $\Gamma_n = \cup_{k=1}^{n}\Delta_k$ and define \[
\Delta_{n+1} = \{n+1 \} \times \bigcup_{1\leq k < n} \{ k \} \times \Gamma_k \times \Gamma_n \times \{ \pm 1 \} \times \{ \pm 1 \}.
\]
We let $\Gamma = \cup_{n \geq 1} \Gamma_n$. We make two observations. First, $\Delta_2 = \varnothing$ so that $\Gamma_1 = \Gamma_2$. Second, an element $\gamma \in \Delta_{n+1}$ is a 6-tuple of the form $\gamma = (n+1, k, \xi, \eta, \ve, \ve ')$. We shall use the terminology introduced by Haydon and say that such a $\gamma$ has rank $n+1$. 

For fixed real numbers $a, b > 0$, we define inductively, linear maps $u_n \colon \ell_{\infty}(\Gamma_n) \to \ell_{\infty}(\Delta_{n+1})$ (for $n \geq 2$) and $i_{m,n} \colon \ell_{\infty}(\Gamma_m) \to \ell_{\infty}(\Gamma_n)$ (for all $n > m$, $n$, $m \in \N$). We denote by $i_n$ the map $i_{n, n+1}$.  We set $i_{1} = i_{1,2} := id_{\ell_{\infty}(\Gamma_1)}$. Suppose, inductively, that $i_k$ has been defined for all $k < n$. Let $i_{k,n} = i_{n-1} \circ \dots \circ i_k \colon \ell_{\infty}(\Gamma_k) \to \ell_{\infty}(\Gamma_n)$. For $x \in \ell_{\infty}(\Gamma)$ we denote by $\pi_k x$ the obvious restriction of $x$ onto coordinates in $\Gamma_k$. We define $u_n \colon \ell_{\infty}(\Gamma_n) \to \ell_{\infty}(\Delta_{n+1})$ by \[
(u_n x)(n+1, k, \xi, \eta, \ve, \ve ') = \ve a x( \xi ) + \ve ' b [x(\eta) - (i_{k,n}\pi_k x)(\eta)]
\]
and $i_n \colon \ell_{\infty}(\Gamma_n) \to \ell_{\infty}(\Gamma_{n+1})$ by \[
i_{n}x(\gamma)  = \begin{cases} x( \gamma ) & \gamma \in \Gamma_n \\ (u_{n}x)(\gamma) & \gamma \in \Delta_{n+1}. \end{cases}
\]
It is easily seen that for $m < n < p$ and $x \in \ell_{\infty}(\Gamma_m)$ we have \[
(i_{m,p}x)|_{\Gamma_n} = i_{m,n}x \]
and so we can well-define a linear mapping $j_m \colon \ell_{\infty}(\Gamma_m) \to \R ^{\Gamma}$ by setting \[
(j_m x)(\delta) = (i_{m,n}x)(\delta) ~~\text{whenever}~~ \delta \in \Gamma_n. \]

Bourgain and Delbaen supposed the numbers $a$ and $b$ are chosen subject to one of the following conditions:
\begin{enumerate}[(1)]
\item $0< b < \frac12 < a < 1$ and $a+b >1$ (in which case we define $\lambda := \frac{a}{1-2b}$)
\item $a =1$ and $ 0< b < \frac12 $ such that $1 + 2b\lambda \leq \lambda$ for some $\lambda > 1$.
\end{enumerate}

In either of these cases, we obtain the following lemma:

\begin{lem} \label{boundednormofimn}
The maps $j_m$ take values in $\ell_{\infty}(\Gamma)$. Moreover, for each $m$, $j_m$ is a linear isomorphic embedding of $\ell_{\infty}(\Gamma_m)$ into $\ell_{\infty}(\Gamma)$ satisfying $\| x \| \leq \| j_m (x) \| \leq \lambda \| x \|$ for all $x \in \ell_{\infty}(\Gamma_m)$ (where $\lambda$ is as above). If $m < n$ then $\im j_m \subseteq \im j_n$.
\end{lem}

\begin{proof}
It is clear from the definition that $j_m$ is an extension operator from $\ell_{\infty}(\Gamma_m)$ to $\ell_{\infty}(\Gamma)$ and consequently $\| x \| \leq \| j_m (x) \|$ for all $x \in \ell_{\infty}(\Gamma_m)$. To complete the first part of the lemma, it is clearly enough to see that $\| i_{m,n}\| \leq \lambda$ for all $m < n$. The proof was given in \cite{BD80} and is an easy induction argument. Indeed, since $i_{1,2} = id_{\ell_{\infty}(\Gamma_1)}$, we have $\|i_{1,2}\| = 1 \leq \lambda$. Suppose inductively $\| i_{k,m} \| \leq \lambda$ for all $k< m \leq n$. We consider an $x \in \ell_{\infty}(\Gamma_n)$ and show first that $\|i_{n, n+1} \| \leq \lambda$. It is enough to show that $|u_n(x)(\gamma)| \leq \lambda \| x \| $ for $\gamma \in \Delta_{n+1}$. To this end, write $\gamma = (n+1, k, \xi, \eta, \ve, \ve')$. We have 
\begin{align*}
|u_nx (\gamma)| &= | a \ve x(\xi)  + b \ve' [x(\eta) - i_{k, n} \pi_k x (\eta) ] | \\
&\leq a \| x \| + b( \| x \| + \lambda \| x\| ) \\
&\leq (a + 2b\lambda) \|x\|
\end{align*}
where we have again made use of the fact that $\lambda \geq 1$ in the final inequality. We are now done, since in either of the cases we have $a + 2b\lambda \leq \lambda$. It remains to consider $\|i_{m, n+1} \|$ when $m < n$. Let $x \in \ell_{\infty}(\Gamma_m)$. If $\gamma \in \Gamma_n$, we have $|i_{m,n+1}x(\gamma)| = |i_{n, n+1} \circ i_{m, n} x (\gamma)| = |i_{m,n}x(\gamma)| \leq \lambda \|x \|$ by the inductive hypothesis. So it once again remains to consider those $\gamma = (n+1, k, \xi, \eta, \ve, \ve')$ in $\Delta_{n+1}$. By a similar argument we get \[
|i_{m,n+1} x (\gamma)| =  | a \ve i_{m,n}x(\xi)  + b \ve' [i_{m,n} x(\eta) - i_{k, n} \pi_k i_{m,n} x (\eta) ] |. 
\]
We consider two separate cases. Firstly, if $k \leq m$, since $\xi \in \Gamma_k$, we have $i_{m,n}x(\xi) = x(\xi)$ and also, from construction of the map $i_{m,n}$ we have $\pi_k i_{m,n} x = \pi_k x$. We therefore have  $|i_{m,n+1} x (\gamma)| = | a \ve x(\xi) + b \ve' [ i_{m,n} x(\eta) - i_{k,n} \pi_k x (\eta) ]|$ and estimating as before yields this is at most $(a + 2b\lambda )\|x \| \leq \lambda \|x \|$. 

Otherwise $n > k > m$ and it is easily verified that $\pi_k i_{m, n}x = i_{m, k} x$. So $i_{k, n} \pi_k i_{m,n} x = i_{k,n} \circ i_{m, k} x = i_{m,n} x$. Consequently $| i_{m, n+1}x (\gamma) | = |a \ve i_{m,n} x (\xi) | \leq  \|i_{m,n} x\| \leq \lambda \|x \|$. Here we made use of the fact that $a \leq 1$ and the inductive hypothesis. This completes the induction and the first part of the proof.

It remains to see that if $m < n$, $\im j_m \subseteq \im j_n$. We will show that for $x \in \ell_{\infty}(\Gamma_m)$, $j_m x = j_n (\pi_nj_m x)$. It is easy to see that $\pi_n j_m x = i_{m,n} x$ so it suffices to show that $j_m x = j_n (i_{m,n} x)$.  To see this, note that if $\delta \in \Gamma$, we can choose $p > n > m$ such that $\delta \in \Gamma_p$. We then have \[
j_n (i_{m,n} x) (\delta) = i_{n,p} (i_{m,n} x) (\delta) = i_{m,p} x(\delta) = j_m x (\delta)
\]
as required.
\end{proof}

It follows from the lemma that \[
X = X_{a,b} = \overline{ \bigcup_{n \geq 1} \text{im}(j_n) } \]
is a closed subspace of $\ell_{\infty}(\Gamma)$. Moreover, the lemma shows that the image, $\text{im}(j_n)$ is $\lambda$-isomorphic to $\ell_{\infty}(\Gamma_n)$ for every $n$ so that $X$ is a $\mathcal{L}_{\infty , \lambda}$ space.  In Bourgain's notation, the spaces of class $\mathcal{X}$ are the spaces $X_{a,b}$ for which $a =1$ and $0<b< \frac12$, whilst those of class $\mathcal{Y}$ are the spaces $X_{a,b}$ with $0< b < \frac12 < a < 1$ and $a+b >1$.

\begin{rem} \label{OriginalBDRem}
We conclude this section with a remark. It is shown in \cite{Haydon2000} that each of the spaces, $X_{a,b}$, described above has a natural Schauder basis, and finite dimensional decomposition. More precisely, for $\gamma \in \Gamma$, let $n= \rank(\gamma)$ and $e_{\gamma}$ be the usual unit vector (i.e. $e_{\gamma}(\delta) = 1$ if $\delta = \gamma$ and $0$ otherwise). Define $d_{\gamma} := j_n(e_{\gamma})$. The sequence $(d_{\gamma})_{\gamma \in \Gamma}$ (under a suitable enumeration) is a Schauder basis for $X$. In fact, it is observed in \cite{Haydon2000} that this basis is normalized, i.e. $\norm{d_{\gamma}} = 1$ for all $\gamma$. Moreover, if we let $M_n := \lin \{d_{\gamma} : \gamma \in \Delta_n \}$, then $(M_n)_{n=1}^{\infty}$ is a FDD for $X$. We will see in the next section that once we have shown that we can describe the original Bourgain-Delbaen construction in the Argyros-Haydon framework, these facts are immediate. 
\end{rem}

\section{Connecting the two constructions} \label{unifyingBDAH}
As explained earlier, we wish to formulate the original Bourgain-Delbaen construction just discussed within the framework of the generalised construction discussed in Section \ref{(s1)GBDconstruction}.  We continue with the same notation from the previous section; in particular, we work in this section with the set $\Gamma$ just described in the previous section. We continue to denote vectors of $\ell_1$ with a star notation. This notation allows us to distinguish between $e_{\gamma}$ (the standard unit vector in $\ell_{\infty}(\Gamma)$) and $e_{\gamma}^{*}$ (the standard unit vector in $\ell_{1}(\Gamma)$) for example. Our aim is to construct a suitable $\tau$ mapping on $\Gamma$ as in Theorem \ref{BDThm} so that we can apply this theorem, obtaining a space of Bourgain-Delbaen type, which is precisely the same space obtained in the previous section. 

We note that the extension operator $j_n \colon \ell_{\infty}(\Gamma_n) \to \ell_{\infty}(\Gamma)$ defined in the previous section could really be thought of as an isomorphic embedding from $\ell_1(\Gamma_n)^*$ to $\ell_1(\Gamma)^*$ under the canonical identifications, $\ell_{1}(\Gamma_n)^{*} \equiv \ell_{\infty}(\Gamma_n)$ and $\ell_{1}(\Gamma)^{*} \equiv \ell_{\infty}(\Gamma)$. Thought of as a map on these dual spaces, it is easy to see that $j_n$ is weak*-norm, and thus weak*-weak* continuous. Indeed, this follows since the weak* and norm topologies coincide on finite dimensional spaces, and we saw in the previous section that $j_n$ is norm-norm continuos. It follows by Lemma \ref{TquotientiffT*iso} that $j_n$ can be obtained as the dual of some operator from $\ell_1(\Gamma)$ to $\ell_1(\Gamma_n)$. We will show that under the correct choice of $\tau$ mapping on  $\Gamma$, this dual operator is exactly the operator $P_{(0,n]}^*$ appearing in Theorem \ref{BDThm}.

 If we are attempting to find the operator which has dual $j_n$, it is sensible to consider the dual of $j_n$ since, roughly speaking, $j_n$ and $j_n^{**}$ can be considered `the same' operator. More precisely, it is well known that for any operator $T \colon X \to Y$, the operator $T^{**}J_X$ maps into $J_Y(Y)$ and moreover, $T = J_{Y}^{-1} T^{**} J_X$. Thinking about how the operators $j_n$ are defined, to compute $j_n^*$, we must first compute the dual of the operators $i_{m,n}$. 

We begin by computing the dual operator of $i_n$, where we assume $n \geq 2$, since $i_1 = \text{Id}_{\ell_{\infty}(\Gamma_1)}$. Throughout this section, we will make the usual identifications of $\ell_1(\Gamma_n)^*$ with $\ell_{\infty}(\Gamma_n)$ and $\ell_{\infty}(\Gamma_n)^*$ with $\ell_1(\Gamma_n)$. Under these identifications, we compute $i_n^*$, obtaining $i_{n}^{*} \colon \ell_{1}(\Gamma_{n+1}) \to \ell_{1}(\Gamma_n)$. For $x^* \in \ell_{1}(\Gamma_{n+1}), \gamma \in \Gamma_n$, 
\begin{align*}
[i_{n}^{*} (x^{*})](\gamma) &= x^{*}(i_{n}(e_{\gamma})) \\
&= x^{*}(\gamma) + \sum_{\xi \in \Delta_{n+1}} x^*(\xi) [i_n(e_{\gamma})](\xi).
\end{align*}

We observe that if $\supp x^* \subseteq \Gamma_n$ then all the terms appearing in the above sum are zero except the $x^*(\gamma)$ term. Thus, for such vectors, $i_n^* x^* = x^*$ and we see that $i_n^* \colon \ell_{1}(\Gamma_{n+1}) \to \ell_{1}(\Gamma_n)$ is a projection {\em onto} $ \ell_1(\Gamma_n)$. Furthermore, we note that this is not the usual projection on $\ell_{1}(\Gamma_n)$ (which would send $e_{\theta}^*$ to itself when $\theta \in \Gamma_n$ and $0$ otherwise). To see this, we exhibit a $\tilde{\gamma} \in \Gamma$ with $\rank \tilde{\gamma} = n+1$ but $i_n^* e_{\tilde{\gamma}}^* \neq 0$. We take $\tilde{\gamma} = (n+1, k, \xi, \eta, 1, 1)$ with $k < n$ and both $\xi$ and  $\eta$ belonging to $\Gamma_k$. Setting $x^* = e_{\tilde{\gamma}}^*$ in the previous formula we see that $i_n^* (e_{\tilde{\gamma}}^*)(\xi) = i_n(e_{\xi})(\tilde{\gamma}) =  ae_{\xi}(\xi) + b [e_{\xi}(\eta) - (i_{k,n}\pi_k e_{\xi})(\eta)] =  a \neq 0$. (Here, we made use of the fact that $i_{k,n} \pi_k e_{\xi} (\eta) = e_{\xi} (\eta) $ since $\xi$ and $\eta$ are in $\Gamma_k$.)

Given a $\tilde{\gamma}$ with $\rank \tilde{\gamma} = n+1$, we can obviously decompose $e_{\tilde{\gamma}}^*$ into a sum of the projection, $i_n^*e_{\tilde{\gamma}}^*$, of $e_{\tilde{\gamma}}^*$, and a vector in $\ker i_n^*$.  Precisely, we have \[
e_{\tilde{\gamma}}^* = i_n^*(e_{\tilde{\gamma}}^*) + [e_{\tilde{\gamma}}^* - i_n^*(e_{\tilde{\gamma}}^*)]. 
\]
Recalling that we are attempting to obtain the original Bourgain-Delbaen spaces via the generalised construction of Argyros-Haydon and comparing this formula with Theorem \ref{BDThm},  we define $c_{\tilde{\gamma}}^* := i_n^*(e_{\tilde{\gamma}}^*)$ and $d_{\tilde{\gamma}}^* = e_{\tilde{\gamma}}^* - c_{\tilde{\gamma}}^*$ (note we are still assuming $n\geq 2$). For $\gamma \in \Gamma_1$ define $c_{\gamma}^* := 0$. 

\begin{rem} \label{remconfusingnotation}
Note the vectors $d_{\gamma}^*$ and $c_{\gamma}^*$ just defined should not be confused with the vectors appearing in Thorem $\ref{BDThm}$; in any case, we haven't even defined a mapping $\tau$ on $\Gamma$ yet for the theorem to be applied in any sensible way. The vectors are defined by the formulae just stated and we work with these vectors throughout this section unless explicitly stated otherwise. The reason for this (perhaps slightly confusing) notation is the following. We will show that for each $\gamma \in \Gamma$, it is possible to define a tuple, $\tau({\gamma})$, having one of the forms appearing in Theorem $\ref{BDThm}$. We can thus apply the theorem, and obtain vectors $\widehat{c_{\gamma}^*}$ and $\widehat{d_{\gamma}^*} = e_{\gamma}^* - \widehat{c_{\gamma}^*}$. We will eventually show that $\widehat{c_{\gamma}^*} = c_{\gamma}^*$ and consequently $\widehat{d_{\gamma}^*} = d_{\gamma}^*$. Moreover, it will turn out that the space of Bourgain-Delbaen type constructed via Theorem \ref{BDThm} under this $\tau$ mapping is precisely the original Bourgain-Delbaen space of the previous section (for a given pair of real numbers $a,b$ satisfying the conditions already discussed). 
\end{rem}

It remains to see how to construct such a $\tau$ map. To this end, we write $\tilde{\gamma} = (n+1, k, \vphi, \eta, \ve, \ve ')$ and for $\gamma \in \Gamma_n$, we look at
\begin{align*}
[i_n^{*}(e_{\tilde{\gamma}}^{*})](\gamma) &= e_{\tilde{\gamma}}^{*}(\gamma) + \sum_{\xi \in \Delta_{n+1}}e_{\tilde{\gamma}}^{*}(\xi)[i_n(e_{\gamma})](\xi) \\
&= i_{n}e_{\gamma}(\tilde{\gamma}) \\
&= u_{n}e_{\gamma}(\tilde{\gamma}) \\
&= \ve a e_{\gamma}(\vphi) + \ve ' b[e_{\gamma} - i_{k,n}\pi_k e_{\gamma}](\eta).
\end{align*}

The second term appearing in the above sum is clearly the harder term to deal with so we look at this first. There is one situation in which this term is easy to compute, namely if $\eta \in \Gamma_k$. In this case $i_{k,n}\pi_k e_{\gamma}(\eta) = \pi_k e_{\gamma}(\eta)$, so $[e_{\gamma} - i_{k,n}\pi_k e_{\gamma}](\eta) = [e_{\gamma} - \pi_k e_{\gamma}](\eta)$. We then consider cases: 
\begin{enumerate}[(1)]
\item $\gamma \neq \eta$. In this case, we clearly have that $e_{\gamma}(\eta) = 0$. Since the only possible vectors that $\pi_k e_{\gamma}$ can be are $0$ and $e_{\gamma}$, we see that $\pi_k e_{\gamma} (\eta) = 0$ also. Thus $[e_{\gamma} - i_{k,n}\pi_k e_{\gamma}](\eta) = 0$.
\item $\gamma = \eta \in \Gamma_k$. In this case, $\pi_k e_{\gamma} = e_{\gamma}$ and we again find $[e_{\gamma} - i_{k,n}\pi_k e_{\gamma}](\eta) = 0$.
\end{enumerate}

To obtain a useful expression for $[e_{\gamma} - i_{k,n}\pi_k e_{\gamma}](\eta)$ when $\eta \in \Gamma_n \setminus \Gamma_k$, it turns out to be useful to look at the dual operator of $i_{m,n} \colon \ell_{\infty}(\Gamma_m) \to \ell_{\infty}(\Gamma_n)$, giving $i_{m,n}^* \colon \ell_{1}(\Gamma_n) \to \ell_{1}(\Gamma_m)$. For $x^* \in \ell_{1}(\Gamma_n), \gamma \in \Gamma_m$ 
\begin{align*}
[i_{m,n}^{*}(x^{*})](\gamma) &= x^{*}(i_{m,n}e_{\gamma}) \\
&= x^{*}(\gamma) + \sum_{\xi \in \Gamma_n \setminus \Gamma_m} x^{*}(\xi)[i_{m,n}e_{\gamma}](\xi).
\end{align*}

We take this opportunity to document a special case of this formula which will be particularly useful to us later. We note that for $\eta \in \Gamma_n \setminus \Gamma_m , \gamma \in \Gamma_m$, we have 
\begin{equation} \label{(s1)2ndtermcomputation}
[i_{m,n}^*(e_{\eta}^*)](\gamma) = i_{m,n}e_{\gamma}(\eta)
\end{equation}

As before, it is easy to see that each $i_{m,n}^*$ is a projection onto $\ell_1(\Gamma_m)$. Since $\cup_n \ell_1(\Gamma_n) = c_{00}(\Gamma)$ is a dense subspace of $\ell_1(\Gamma)$, it is natural to attempt to define projections $P_{(0,k]}^{*} \colon \ell_{1}(\Gamma) \to \ell_{1}(\Gamma_k)$ by \[
P_{(0,k]}^{*}(x) = i_{k,n}^*(x) ~~~\text{whenever} ~~~ x \in \ell_{1}(\Gamma_n), (n > k). \]

Assuming this is well-defined, the formula defines (linear) projections $P_{(0,k]}^{*}$ from the dense subspace $c_{00}(\Gamma) \subseteq \ell_1(\Gamma)$, onto $\ell_1(\Gamma_k)$. Moreover, it follows from the proof of Lemma \ref{boundednormofimn} (and standard facts about dual operators) that $\|i_{k,n}^*\| \leq \lambda$ for all $k < n$. Consequently $P_{(0,k]}^{*}$ is bounded (with norm at most $\lambda$) and thus extends to a bounded projection defined on $\ell_{1}(\Gamma)$ as required. So it remains to check that we can well define $P_{(0,k]}^*$ by the above formula. Suppose $x^* \in \ell_{1}(\Gamma_n) \subset \ell_{1}(\Gamma_{n'})$ ($n' > n > k$). We saw earlier that for $\gamma \in \Gamma_k$,
\begin{align*}
[i_{k,n'}^*(x^*)](\gamma) &= x^*(\gamma) + \sum_{\xi \in \Gamma_{n'} \setminus \Gamma_{k}}x^*(\xi) [i_{k,n'}e_{\gamma}](\xi) \\
&= x^*(\gamma) + \sum_{\xi \in \Gamma_{n} \setminus \Gamma_{k}}x^*(\xi) [i_{k,n'}e_{\gamma}](\xi)
\end{align*}
where the final equality is because $x^* \in \ell_{1}(\Gamma_n)$ so $x^*(\xi) = 0$ for $\xi \in \Gamma_{n'} \setminus \Gamma_{n}$. It is immediate from the definition of the operators $i_{m,n}$ that $i_{k,n'} = i_{n,n'} \circ i_{k,n}$. So for $\xi \in \Gamma_n$, $i_{k,n'}e_{\gamma}(\xi) = i_{n,n'} \circ i_{k,n}e_{\gamma}(\xi) = i_{k,n}e_{\gamma}(\xi)$. Combining these observations, we find that 
\begin{align*}
[i_{k,n'}^*(x^*)](\gamma) &= x^*(\gamma) + \sum_{\xi \in \Gamma_{n} \setminus \Gamma_{k}}x^*(\xi) [i_{k,n'}e_{\gamma}](\xi) \\
&= x^*(\gamma) + \sum_{\xi \in \Gamma_{n} \setminus \Gamma_{k}}x^*(\xi) [i_{k,n}e_{\gamma}](\xi) \\
&= [i_{k,n}^*(x^*)](\gamma)
\end{align*}
so that $P_{(0,k]}^*$ is well defined. 

We recall what we set out to do. For $\tilde{\gamma} = (n+1, k, \vphi, \eta, \ve, \ve ') \in \Delta_{n+1}$ and $\gamma \in \Gamma_n$, we were attempting to find an expression for $c_{\tilde{\gamma}}^*(\gamma) := [i_n^{*}(e_{\tilde{\gamma}}^{*})](\gamma)$. We have seen already that 

$c_{\tilde{\gamma}}^*(\gamma)  = \ve a e_{\gamma}(\vphi) + \ve ' b[e_{\gamma} - i_{k,n}\pi_k e_{\gamma}](\eta)$.

We consider 2 separate cases:
\begin{enumerate}[(1)]
\item $\eta \in \Gamma_k$. We have seen that the second term in the above expression is just $0$. We thus see that $c_{\tilde{\gamma}}^*(\gamma)  = \ve a e_{\gamma}(\vphi) = \ve a e_{\vphi}^*(\gamma)$. In other words, $c_{\tilde{\gamma}}^* = \ve a e_{\vphi}^*$.
\item $\eta \in \Gamma_n \setminus \Gamma_k$. If in addition, $\rank(\gamma) \leq k$, it follows that $\pi_k e_{\gamma} = e_{\gamma}$. Combining this observation with equation (\ref{(s1)2ndtermcomputation}) and the definition of $P_{(0,k]}^*$, we get
 \[ 
 [P_{(0,k]}^*(e_{\eta}^*)](\gamma) = [i_{k,n}^{*}(e_{\eta}^*)](\gamma) = i_{k,n}e_{\gamma}(\eta) = i_{k,n}\pi_{k}e_{\gamma}(\eta).
\]
Thus
\begin{align*}
c_{\tilde{\gamma}}^*(\gamma)  &= \ve a e_{\gamma} (\vphi ) + \ve ' b[e_{\gamma} - i_{k,n}\pi_k e_{\gamma}](\eta) \\
&= \ve a e_{\vphi}^*(\gamma) + \ve ' b [e_{\eta}^* - [P_{(0,k]}^*(e_{\eta}^*)](\gamma) \\
&= \ve a e_{\vphi}^*(\gamma) + \ve ' b \left\{ [ (I - P_{(0,k]}^{*}) (e_{\eta}^{*})] ( \gamma ) \right\}.
\end{align*}
If on the other hand $\rank (\gamma) > k$, $\pi_k e_{\gamma} = 0$ so 
\begin{align*}
c_{\tilde{\gamma}}^*(\gamma)  &= \ve a e_{\gamma}\vphi + \ve ' b[e_{\gamma} - i_{k,n}\pi_k e_{\gamma}](\eta) \\
&= \ve a e_{\vphi}^*(\gamma) + \ve ' b e_{\eta}^*(\gamma) \\
&= \ve a e_{\vphi}^*(\gamma) + \ve ' b \left\{ [ (I - P_{(0,k]}^{*}) (e_{\eta}^{*})] ( \gamma ) \right\}
\end{align*}
where the final equality follows since $\rank(\gamma) > k \implies [P_{(0,k]}^*(x^*)](\gamma) = 0$ for any $x^* \in \ell_{1}(\Gamma)$. We thus conclude that \[
c_{\tilde{\gamma}}^* = \ve a e_{\vphi}^* + \ve ' b (I - P_{(0,k]}^*)(e_{\eta}^*).
\]
\end{enumerate}

We make one more observation before coming to the main result of this section.
\begin{rem} \label{(s1)remP*} 
An easy induction argument yields that $\lin \{ d_{\gamma}^* : \rank \gamma \leq n \} = \lin \{ e_{\gamma}^* : \rank \gamma \leq n \}$ for every $n$. Indeed, the base case is immediate since by definition, $d_{\gamma}^* = e_{\gamma}^*$ for $\rank \gamma \leq 2$ (recalling that $\Delta_2 = \varnothing$ so $\Gamma_1 = \Gamma_2$). The inductive step follows from the relation $e_{\gamma}^* = d_{\gamma}^* + c_{\gamma}^*$ for all $\gamma$ and $\supp c_{\gamma}^* \subseteq \Gamma_n$ when $\rank \gamma = n+1$. 

If follows that $[d_{\gamma}^* : \gamma \in \Gamma] = \ell_{1}(\Gamma)$.  Consequently the operators $P_{(0,k]}^*$ previously defined are completely determined by their action on the $d_{\gamma}^*$. In fact \[
P_{(0,k]}^*(d_{\gamma}^*) = \begin{cases} d_{\gamma}^* & \rank(\gamma) \leq k \\ 0 & \text{otherwise.} \end{cases}
\]
\end{rem}
To see this, we note the formula certainly holds when $\rank(\gamma) \leq k$ since we have already observed that the maps $P_{(0,k]}^*$ are projections onto $\ell_1(\Gamma_k)$. 
It remains to see  that $P_{(0,k]}^*(d_{\gamma}^*) = 0$ when $\rank(\gamma) > k$. We consider the case where $k \geq 2$. Suppose first $\rank(\gamma) = k+1$. Then $P_{(0,k]}^*(d_{\gamma}^*) = i_k^*(d_{\gamma}^*) = i_k^*(e_{\gamma}^* - c_{\gamma}^*) = i_k^*(e_{\gamma}^*) - c_{\gamma}^* = c_{\gamma}^* - c_{\gamma}^* = 0$. (We have of course made use of the definition of the vector $c_{\gamma}^*$ as being equal to $i_k^* e_{\gamma}^*$ since $\rank \gamma = k+1$). If $\rank(\gamma) = n > k+1$, $P_{(0,k]}^*(d_{\gamma}^*) = i_{k,n}^*(d_{\gamma}^*) = i_{k,n-1}^* \circ i_{n-1, n}^* (d_{\gamma}^*) = i_{k,n-1}^*(0) = 0$ where the penultimate equality follows from the previous argument with $k$ replaced by $n-1$. Note we also made use of the following easily verified fact: if $k < m < n$ then $i_{k,n}^* = i_{k,m}^* \circ i_{m,n}^*$. The proof for the case $k=1$ is similar. 

We are finally ready to present the main result of this section. 

\begin{prop} \label{(s1)mainpropn2}

Let $(\Delta_q)_{q=1}^{\infty}$, $\Gamma_n$ and $\Gamma$ be the sets defined as in the original Bourgain-Delbaen construction and let $a, b$ be real numbers satisfying either of the two conditions assumed by Bourgain and Delbaen (see the previous section). We define $\tau$ on $\Gamma \setminus \Gamma_1$ as follows. Given $\gamma \in \Delta_{n+1}$, write $\gamma = (n+1, l, \xi, \eta, \ve, \ve')$ and set \[
\tau(\gamma) = \begin{cases} (\ve, a, \xi) & \text{ if $\eta \in \Gamma_l$} \\ (\ve, a, \xi, l, b, \ve ' e_{\eta}^*) & \text{ if $\eta \in \Gamma_n \setminus \Gamma_l$}. \end{cases}\]
(We clarify that the slightly clumsy notation $\ve ' e_{\eta}^*$ appearing in the final tuple above is to be interpreted as the scalar $\ve '$ multiplied by the vector $e_{\eta}^*$).
Noting that in either of the original BD setups, $a \leq 1$ and $b < \frac12$, Theorem \ref{BDThm} yields that there exist $\widehat{d_{\gamma}^*} = e_{\gamma}^* - \widehat{c_{\gamma}^*} \in \ell_{1}(\Gamma)$ and projections $\widehat{P_{(0,q]}^*}$ on $\ell_{1}(\Gamma)$ uniquely determined by the following properties:
\begin{align*}
\widehat{ P_{(0,q]}^{*}} \widehat{d_{\gamma}^{*}} &= \begin{cases} \widehat{d_{\gamma}^{*}} & \gamma \in \Gamma_{q} \\ 0 & \gamma \in \Gamma \setminus \Gamma_q \end{cases} \\
 \widehat{c_{\gamma}^*} &= \begin{cases} 0 & \gamma \in \Delta_1 \\ \ve a e_{\xi}^{*} & \gamma = (n+1, l, \xi, \eta, \ve , \ve ') ~\text{ and } \eta \in \Gamma_{l} \\ \ve a e_{\xi}^{*} + \ve ' b(I - \widehat{P_{(0,l]}^{*}})e_{\eta}^{*} & \gamma = (n+1, l, \xi, \eta, \ve , \ve ') ~\text{ and } \eta \in \Gamma_{n} \setminus \Gamma_{l}. \end{cases}
\end{align*}
The family $\left( \widehat{d_{\gamma}^*} \right)_{\gamma \in \Gamma}$ is a basis for $\ell_{1}(\Gamma)$ with basis constant at most $\lambda$ (as defined in the previous section). Moreover, the norm of each projection $\widehat{P_{(0,q]}^{*}}$ is at most $\lambda$ and the biorthogonal elements $d_{\gamma}$ of the $\widehat{d_{\gamma}^*}$ generate a $\mathcal{L}_{\infty , \lambda}$ subpace of $\ell_{\infty}(\Gamma)$, $X(\Gamma) := X(\Gamma, \tau)$.

With the notation that has been used throughout this section, we have $\widehat{c_{\gamma}^{*}}$ (and hence the $\widehat{d_{\gamma}^*}$) coincide with the $c_{\gamma}^*$ (respectively $d_{\gamma}^*$) and $\widehat{P_{(0,k]}^*} = P_{(0,k]}^*$. Moreover, the mappings $\left( P_{(0,k]}^* \right)^* \colon \ell_{\infty}(\Gamma_k) \to \ell_{\infty}(\Gamma)$ coincide with the mappings $j_k$ of the original Bourgain-Delbaen construction.
\end{prop}

Before giving the proof we make a few observations that follow from the proposition. Note we have the following important corollary. 
\begin{cor}
The subspace $X(\Gamma)$ of the above proposition is precisely the original Bourgain-Delbaen space $X_{a,b}$. 
\end{cor}
\begin{proof}
It follows from the proof of Theorem \ref{BDThm} and Lemma \ref{i_nlem} that \[
X(\Gamma) = \overline{ \bigcup_{n \in \N}\text{im}\left( \widehat{ P^*_{(0,n] }} \right)^* }.
\]
So, by the proposition,
\begin{align*}
X(\Gamma) &= \overline{ \bigcup_{n \in \N}\text{im}\left( P_{(0,n] }^* \right)^* } \\
&=\overline{ \bigcup_{n \in \N} \text{im } j_n} \\
&= X_{a,b}.
\end{align*}

\end{proof}

Thus we can think of the original Bourgain-Delbaen construction in terms of the Argyros-Haydon construction as we wanted. Moreover, we see from the generalised construction, i.e. by Theorem \ref{BDThm}, that the biorthogonal elements $d_{\gamma}$ of the $d_{\gamma}^* = \widehat{d_{\gamma}^*}$ are a basis of $X(\Gamma) = X_{a,b}$. It follows by Lemma \ref{i_nlem} and the proposition that for $\gamma \in \Delta_n$, $d_{\gamma} = j_n e_{\gamma}$. The fact this basis is normalised follows from Lemma \ref{BasisIsNormalized}. Thus we obtain all the results stated in Remark \ref{OriginalBDRem} as immediate corollaries. 

It remains for us to prove the proposition.

\begin{proof}[Proof of Proposition ~\ref{(s1)mainpropn2}]
We prove by induction that the statement {\em `` $\widehat{ d_{\gamma}^*} = d_{\gamma}^*$ for all $\gamma \in \Gamma_n$ and for all $k \in \N$, $P_{(0,k]}^* |_{\ell_1(\Gamma_n)} = \widehat{P_{(0,k]}^*} |_{\ell_1(\Gamma_n)}$''  } holds for every $n \in \N$.

We recall that the vectors $\widehat{c_{\gamma}^*}$, $\widehat{d_{\gamma}^*}$ and operators $\widehat{P_{(0,k]}^*}$ are defined inductively. It is immediate from the definitions that when $\gamma \in \Gamma_{1}$, $\widehat{d_{\gamma}^*} = d_{\gamma}^* = e_{\gamma}^*$. Thus by Remark \ref{(s1)remP*} and the way in which the $\widehat{P^*_{(0,k]}}$ are constructed, we have that for $\gamma \in \Gamma_{1}$, $P^*_{(0,k]}d_{\gamma}^* = \widehat{P_{(0,k]}^*} \widehat{d_{\gamma}^*}$. So $\forall k \in \N$ \[
P^{*}_{(0,k]}|_{\ell_{1}(\Gamma_{1})} = \widehat{P_{(0,k]}^*}|_{\ell_{1}(\Gamma_{1})}
\]
and the statement is true for $n = 1$. Suppose inductively that for some $n \geq 1$, $\widehat{d_{\gamma}^*} = d_{\gamma}^*$ for all $\gamma \in \Gamma_n$ and $P^{*}_{(0,k]}|_{\ell_{1}(\Gamma_{n})} = \widehat{P_{(0,k]}^*}|_{\ell_{1}(\Gamma_{n})}$. Then for $\gamma = (n+1, l, \xi, \eta, \ve , \ve ') \in \Delta_{n+1}$, 
\begin{align*}
\widehat{c_{\gamma}^*} &= \begin{cases} \ve a e_{\xi}^{*} & \eta \in \Gamma_{l} \\ \ve a e_{\xi}^{*} + \ve ' b(I - \widehat{P_{(0,l]}^{*}})e_{\eta}^{*} &  \eta \in \Gamma_{n} \setminus \Gamma_{l} \end{cases} \\
&= \begin{cases} \ve a e_{\xi}^{*} & \eta \in \Gamma_{l} \\ \ve a e_{\xi}^{*} + \ve ' b(I - P_{(0,l]}^{*})e_{\eta}^{*} &  \eta \in \Gamma_{n} \setminus \Gamma_{l} \end{cases} \\
&= c_{\gamma}^*.
\end{align*}
The first of the above equalities holds due to the inductive construction of the $\widehat{c_{\gamma}^*}$, the second by the inductive hypothesis and the final equality is a consequence of the very definition of the $P_{(0,k]}^*$ and all our previous calculations. It follows that $\widehat{d_{\gamma}^*} = e_{\gamma}^* - \widehat{c_{\gamma}^*} = d_{\gamma}^*$. We again get that $P^{*}_{(0,k]}|_{\ell_{1}(\Gamma_{n+1})} = \widehat{P_{(0,k]}^*}|_{\ell_{1}(\Gamma_{n+1})}$ for all $k \in \N$ as a consequence of Remark \ref{(s1)remP*} and the way in which the $\widehat{P^*_{(0,k]}}$ are constructed. The first part of the proof is thus complete by induction and an elementary density argument. As a consequence we now drop the `hat notation'. 

It remains to prove $\left( P_{(0,k]}^* \right)^* = j_k \colon \ell_{\infty}(\Gamma_k) \to \ell_{\infty}(\Gamma)$. For $x \in \ell_{\infty}(\Gamma_k)$, $\gamma \in \Gamma$ \[
\left[ \left( P_{(0,k]}^* \right)^* (x) \right](\gamma) = \langle x , P_{(0,k]}^*(e_{\gamma}^*) \rangle \]
where we are, of course, thinking of $x$ as an element of $\ell_{1}(\Gamma_k)^*$ (under the canonical identification) in the RHS of the equality. If $\rank(\gamma) \leq k$, $P_{(0,k]}^*e_{\gamma}^* = e_{\gamma}^*$ from which it follows that $\left[ \left( P_{(0,k]}^* \right)^* (x) \right](\gamma) = x(\gamma) = [j_k(x)](\gamma)$ as required. Obviously we must see that this equality also holds when $\rank \gamma > k$. We will need the following lemma. 

\begin{lem} \label{(s1)lem_sec2}
If $x \in \ell_{\infty}(\Gamma_{n})$ (which we identify in the canonical way with $\ell_{1}(\Gamma_n)^*$) and $\gamma \in \Gamma$ with $\rank(\gamma) = n+1$ then \[
x(c_{\gamma}^*) = [j_{n}x](\gamma).
\]
\end{lem}

In order not to detract from the proof at hand, we defer the proof of this lemma to the end.

Suppose now that $\rank(\gamma) = k+1$. Then $P_{(0,k]}^*e_{\gamma}^{*} = i_{k}^*e_{\gamma}^* = c_{\gamma}^*$. So \[
\left[ \left( P_{(0,k]}^* \right)^* (x) \right](\gamma) = x(c_{\gamma}^*) = [j_{k}x](\gamma)
\]
where we recall $x \in \ell_{\infty}(\Gamma_k)$ and appeal to Lemma \ref{(s1)lem_sec2} to obtain the final equality.

Finally, we assume $\rank(\gamma) = n > k+1$. We have 
\begin{align*}
\left[ \left( P_{(0,k]}^* \right)^* (x) \right](\gamma) &= \langle x , P_{(0,k]}^*(e_{\gamma}^*) \rangle \\
&= \langle x ,  i_{k,n}^*(e_{\gamma}^*)\rangle \\
&= \langle x ,  i_{k,n-1}^* i_{n-1,n}^*(e_{\gamma}^*) \rangle \\
&= \langle x \circ i_{k,n-1}^* , i^*_{n-1, n} e_{\gamma}^* \rangle \\
&= \langle x \circ i_{k, n-1}^*,  c_{\gamma}^* \rangle.
\end{align*}

We now think of $x \circ i_{k,n-1}^* \in \ell_{1}(\Gamma_{n-1})^*$ as an element of $\ell_{\infty}(\Gamma_{n-1})$. Appealing once again to Lemma \ref{(s1)lem_sec2} we see \[
\left[ \left( P_{(0,k]}^* \right)^* (x) \right](\gamma) = x \circ i^*_{k,n-1} (c_{\gamma}^*) = [j_{n-1}(x \circ i_{k,n-1}^*)](\gamma).
\]

So to complete the proof it is enough to see that $[j_{n-1}(x \circ i_{k,n-1}^*)](\gamma) = [j_k(x)](\gamma)$. But we recall from the original B.D construction that for any $x \in \ell_{\infty}(\Gamma_k)$, $j_k(x) = j_{n-1}(i_{k,n-1}x)$. (We saw this in the proof of Lemma \ref{boundednormofimn}.) So in fact, all we need to see is that when $i_{k,n-1}x$ is thought of as an element of $\ell_{1}(\Gamma_{n-1})^*$ under the canonical identification, it is precisely the element $x \circ i_{k,n-1}^* \in \ell_{1}(\Gamma_{n-1})^*$. To this end, let $y^*$ be an element of $\ell_{1}(\Gamma_{n-1})$. So 
\begin{equation} \label{(s1)eqn:calculation1}
x \circ i_{k,n-1}^*(y^*) = \sum_{\theta \in \Gamma_{n-1}} x \circ i_{k,n-1}^*(e_{\theta}^*)y^*(\theta) 
\end{equation}
where we are, of course, still thinking of $x$ as an element of $\ell_{1}(\Gamma_k)^*$. Really $x \in \ell_{\infty}(\Gamma_k)$, so we can compute 
\begin{equation} \label{(s1)eqn:calculation2}
x \circ i_{k,n-1}^*(e_{\theta}^*) = \sum_{\nu \in \Gamma_k} x(\nu)[i_{k,n-1}^*(e_{\theta}^*)](\nu).
\end{equation}
Fortunately, we computed earlier that for $\nu \in \Gamma_k$,  \[
[i_{k,n-1}^*(e_{\theta}^*)](\nu) = e_{\theta}^*(\nu) + \sum_{\ve \in \Gamma_{n-1} \setminus \Gamma_{k}} e_{\theta}^*(\ve)[i_{k,n-1}e_{\nu}](\ve). 
\]
If $\theta \in \Gamma_{n-1} \setminus \Gamma_k$ then this is just $[i_{k,n-1}e_{\nu}](\theta)$, whilst for $\theta \in \Gamma_k$, we obtain $e_{\theta}^*(\nu) = e_{\nu}(\theta)$. So by Equation (\ref{(s1)eqn:calculation2}) we find \[
x \circ i_{k,n-1}^*(e_{\theta}^*) = \begin{cases} \sum_{\nu \in \Gamma_k}x(\nu)e_{\nu}(\theta) = x(\theta) & \text{if } \theta \in \Gamma_k \\ \sum_{\nu \in \Gamma_k}x(\nu)[i_{k,n-1}e_{\nu}](\theta) = i_{k,n-1}x(\theta) & \text{if } \theta \in \Gamma_{n-1} \setminus \Gamma_k. \end{cases}
\]

Substituting this back into Equation (\ref{(s1)eqn:calculation1}) we get
\begin{align*}
x\circ i_{k,n-1}^*(y^*) &= \sum_{\theta \in \Gamma_k} x(\theta)y^*(\theta) + \sum_{\theta \in \Gamma_{n-1} \setminus \Gamma_k}[i_{k,n-1}x](\theta)y^*(\theta) \\
&= \sum_{\theta \in \Gamma_{n-1}}[i_{k,n-1}x](\theta)y^*(\theta) \\
&= [i_{k,n-1}x](y^*)
\end{align*}
where we think of $i_{k,n-1}x$ as an element of $\ell_{1}(\Gamma_{n-1})^*$. The second equality above comes from observing that when $\theta \in \Gamma_k$, $[i_{k,n-1}x](\theta) = x(\theta)$. 
\end{proof}

\begin{proof}[Proof of Lemma ~\ref{(s1)lem_sec2}]
We use the same notation as that used for the statement of the lemma. We consider only the more complex case where $\gamma = (n+1, l, \xi, \eta, \ve, \ve ')$ and $\eta \in \Gamma_n \setminus \Gamma_l$ since the other case is similar (and easier). So,
\begin{align*}
x(c_{\gamma}^*) &= \sum_{\theta \in \Gamma_n}x(\theta)c_{\gamma}^*(\theta) \\
&= \sum_{\theta \in \Gamma_n} x(\theta)\ve ae_{\xi}^*(\theta) + \sum_{\theta \in \Gamma_n}x(\theta)\ve ' b[(I - P_{(0,l]}^*)e_{\eta}^*](\theta) \\
&= \ve a x(\xi) + \sum_{\theta \in \Gamma_n}x(\theta)\ve ' b[(I - P_{(0,l]}^*)e_{\eta}^*](\theta)
\end{align*}
since $e_{\xi}^*(\theta)$ can only be non-zero when $\theta = \xi$ and since $\xi \in \Gamma_l$, $l < n$, this happens precisely once. If $\rank(\theta) \leq l$, i.e. $\theta \in \Gamma_l$, since we also assume $\eta \in \Gamma_n \setminus \Gamma_l$, the definition of $P^*_{(0, l]}$ given earlier and Equation (\ref{(s1)2ndtermcomputation}) yields
\[
[P_{(0,l]}^*e_{\eta}^*](\theta) = [i_{l,n}^*(e_{\eta}^*)](\theta) = [i_{l,n}e_{\theta}](\eta) = [i_{l,n}\pi_l e_{\theta}](\eta) 
\]
where the final equality holds since it is assumed $\rank(\theta) \leq l$. It is then easily seen that \[
[(I-P_{(0,l]}^*)e_{\eta}^*](\theta) = \begin{cases} e_{\eta}^*(\theta) = e_{\theta}(\eta) & \text{ if } \rank(\theta) > l \\ [e_{\theta} - i_{l,n}\pi_l e_{\theta}](\eta) & \text{ if } \rank(\theta) \leq l \end{cases}
\]
However, it is clear that $[e_{\theta} - i_{l,n}\pi_l e_{\theta}](\eta) = e_{\theta}(\eta)$ if $\rank(\theta) > l$. So in fact, we have shown $[(I-P_{(0,l]}^*)e_{\eta}^*](\theta) = [e_{\theta} - i_{l,n}\pi_l e_{\theta}](\eta)$. Consequently, 
\begin{align*}
\ve ' b \sum_{\theta \in \Gamma_n} x(\theta) [(I-P_{(0,l]}^*)e_{\eta}^*](\theta) &= \ve ' b \sum_{\theta \in \Gamma_n} x(\theta) [e_{\theta} - i_{l,n}\pi_l e_{\theta}](\eta) \\
&= \ve ' b \{ (x-i_{l,n}\pi_l x)(\eta) \} 
\end{align*}
since clearly $x = \sum_{\theta \in \Gamma_n} x(\theta) e_{\theta}$ and all the maps are linear.  So, we finally have \[
x(c_{\gamma}^*) = \ve a x(\xi) + \ve ' b \{ (x-i_{l,n}\pi_l x)(\eta) \}  = [j_n(x)](\gamma)
\] 
as claimed.
\end{proof}

\section{The operator algebras for the original Bourgain-Delbaen spaces} \label{nonsepBDalgebras}

We conclude this chapter by exhibiting an uncountable set of isometries on the original Bourgain-Delbaen spaces which are pairwise distance $C$ apart (with respect to the usual operator norm) for some constant $C>0$. In particular, this establishes that the operator algebra $\mathcal{L}(X)$ is non-separable for any of the original Bourgain-Delbaen spaces $X$.

We take this opportunity to note that it is in fact easy to show $\mathcal{K}(X)$ is non-separable when $X$ is a Bourgain-Delbaen space of class $\mathcal{X}$, so certainly $\mathcal{L}(X)$ is non-separable. Indeed, it was shown in \cite[Theorem 4.4]{BD80} that in this case, every infinite dimensional subspace of $X$ has a subspace isomorphic to $\ell_1$. In particular, $\ell_1$ embeds into $X$. Since the dual operator of an isomorphic embedding is a quotient operator, it follows that there is a quotient operator $Q: X^* \to \ell_{\infty}$. Consequently, $X^*$ cannot be separable and therefore neither is $\mathcal{K}(X)$ since $X^*$ always embeds isometrically into $\mathcal{K}(X)$.

The non-separability is interesting because, as commented by Beanland and Mitchell in \cite{KB2010}, Emmanule showed that $\ell_{\infty}$ does not embed into $\mathcal{L}(X)$ when $X \in \mathcal{X}$, thus giving an example of a non-separable Banach space containing no isomorphic copy of $\ell_{\infty}$. Similar questions about the separability of $\mathcal{L}(Y)$ when $Y \in \mathcal{Y}$ remained open however. This section addresses this problem; we remark that at the time of writing, it remains unknown if $\ell_{\infty}$ embeds into $\mathcal{L}(Y)$.

\subsection{Construction of the basic operator} \label{(s1)basic_operator_construction}

For the remainder of this chapter, we fix real numbers $a,b$ satisfying either of the conditions of Bourgain and Delbaen (see Section \ref{(s1)BDconstruction}) and work with the Bourgain-Delbaen space $X = X_{a,b} = X(\Gamma, \tau)$ (for some suitable $\tau$ mapping, as shown in the previous section). We begin by inductively constructing an isometry $S^* : \ell_{1}(\Gamma) \to \ell_{1}(\Gamma)$. We will show that the dual of this operator restricts to give an isometry $S: X \to X$. In the following section, we will then take appropriate compositions of operators of this type to obtain the uncountable set of isometries we seek. We continue with the notation used in the preceding section.

To construct $S^*$, we fix a $\gamma = (n+1, k, \xi, \eta, \ve, \ve ') \in \Delta_{n+1}$ (where $n\geq 2$) and aim to construct (inductively) a bijective mapping $F \colon \Gamma \to \Gamma$ with the following properties
\begin{enumerate}[(1)]
\item $F(\gamma) = (n+1, k, \xi, \eta, -\ve, -\ve ') =: \tilde{\gamma}$
\item $F(\tilde{\gamma}) = \gamma$
\item $F(\tau) = \tau$ if $\rank(\tau) \leq \rank(\gamma)$ and $\tau \notin \{ \gamma , \tilde{\gamma} \}$
\item $\rank(F(\tau)) = \rank(\tau) ~~ \forall ~~ \tau \in \Gamma$. We shall say $F$ is `rank-preserving'.
\end{enumerate}
Note that we must assume $n \geq 2$ in the above since we recall that $\Delta_2 = \varnothing$ and $\Gamma_1 = \Gamma_2$ is simply a singleton set. 

The operator $S^* \colon \ell_1(\Gamma) \to \ell_{1}(\Gamma)$ will map $d_{\tau}^*$ to $d_{F(\tau)}^*$ for all $\tau \notin \{ \gamma , \tilde{\gamma} \}$ and $d_{\tau}^*$ to $-d_{F(\tau)}^*$ for $\tau \in \{ \gamma , \tilde{\gamma} \}$. (The reason for this sign change in the special case where $\tau \in \{ \gamma , \tilde{\gamma} \}$ will shortly become apparent.) The important point here is that we know how $S^*$ acts on the $d_{\tau}^*$ vectors; it is this property that will ensure the dual map restricts to give a map from $X$ into $X$. Of course, the property just described only defines a linear map on the dense subspace $\text{sp} \{ d_{\tau}^* : \tau \in \Gamma \}$ of $\ell_1(\Gamma)$. We also need to show that this map is bounded, so that it can be extended to a bounded operator on $\ell_1(\Gamma)$.  Note that since we are working with the $\ell_1$ norm, in order to prove $S^*$ is bounded, we only need to be able to control $\norm{S^*(e_{\tau}^*)}$. More precisely, if there is some $M \geq 0$ such that $\norm{S^*(e_{\tau}^*)} \leq M$ for every $\tau \in \Gamma$, then it is elementary to check that $S^*$ is bounded with norm at most $M$.

To begin the inductive constructions of $S^*$ and $F$ we define $F \colon \Gamma_{n+1} \to \Gamma_{n+1}$ and $S^* \colon \ell_{1} (\Gamma_{n+1}) \to \ell_{1}(\Gamma_{n+1})$ by 
\begin{align*}
F(\tau) &= \begin{cases} \tilde{\gamma} & \text{ if } \tau = \gamma \\ \gamma & \text{ if } \tau = \tilde{\gamma} \\ \tau & \tau \notin \{ \gamma , \tilde{\gamma} \} \end{cases} \\
S^*(d_{\tau}^*) &= \begin{cases} d_{\tau}^* & \text{ if } \tau \notin \{ \gamma , \tilde{\gamma} \} \\ -d_{F(\tau)}^* & \text{ if } \tau \in \{ \gamma , \tilde{\gamma} \}. \end{cases}
\end{align*}
Recall that if $\rank \tau \leq n+1, c_{\tau}^* \in \text{sp}\{ e_{\theta}^* : \theta \in \Gamma_n \} = \text{sp} \{ d_{\theta}^* : \theta \in \Gamma_n \}$. Consequently, using the linearity of $S^*$, it is easily seen that for all $\tau \in \Gamma$ with $\rank \tau \leq n+1$, $S^*c_{\tau}^* = c_{\tau}^*$. Therefore, $S^*e_{\tau}^* =  S^*(c_{\tau}^*) + S^*(d_{\tau}^*) = c_{\tau}^* + d_{\tau}^* = e_{\tau}^*$ for all $\tau \in \Gamma_{n+1} \setminus \{ \gamma , \tilde{\gamma} \}$. We also have 
\begin{align*}
S^*(e_{\gamma}^*) = S^*(c_{\gamma}^* + d_{\gamma}^* ) &= c_{\gamma}^* - d_{\tilde{\gamma}}^* \\
&= -c_{\tilde{\gamma}}^* - d_{\tilde{\gamma}}^* + (c_{\gamma}^* + c_{\tilde{\gamma}}^* ) \\
&= -e_{\tilde{\gamma}}^*
\end{align*}
since if we look at the possible values of $c_{\gamma}^*$ and $c_{\tilde{\gamma}}^*$ obtained in Section \ref{unifyingBDAH} we see $c_{\gamma}^* + c_{\tilde{\gamma}}^* = 0$. (We remark that this computation works due to the `sign change', $S^*d_{\gamma}^* = -d_{\tilde{\gamma}}^*$, noted earlier.) In a similar way we obtain $S^*(e_{\tilde{\gamma}}^*) = -e_{\gamma}^*$. The above discussion can be succinctly summarised by the following formulae which are valid for all $\tau \in \Gamma_{n+1}$.  \[
S^*(e_{\tau}^*) = \begin{cases} e_{\tau}^* = e_{F(\tau)}^* & \text{ if } \tau \notin \{ \gamma , \tilde{\gamma} \} \\ -e_{F(\tau)}^* & \text{ if } \tau \in \{ \gamma , \tilde{\gamma} \}. \end{cases}
\]
Since also $c_{\tau}^* = e_{\tau}^* - d_{\tau}^*$ for any $\tau \in \Gamma$ it is easily checked that 
\[
S^*(c_{\tau}^*) = \begin{cases} c_{\tau}^* = c_{F(\tau)}^* & \text{ if } \tau \notin \{ \gamma , \tilde{\gamma} \} \\ c_{\tau}^* = -c_{F(\tau)}^* & \text{ if } \tau \in \{ \gamma , \tilde{\gamma} \}. \end{cases}
\]
(We once again made use of the fact that $c_{\gamma}^* + c_{\tilde{\gamma}}^* = 0$ in the above formula). 

It is now clear how the inductive construction proceeds. We assume inductively that for some $k \in \N$ we have defined $F \colon \Gamma_{n+k} \to \Gamma_{n+k}$ and a linear map $S^* \colon \ell_1(\Gamma_{n+k}) \to \ell_1(\Gamma_{n+k})$ satisfying:

\begin{enumerate}[(1)]
\item $F \colon \Gamma_{n+k} \to \Gamma_{n+k}$ is a bijection, satisfying the four desired properties of $F$ given at the beginning of this section.
\item For each $\tau \in \Gamma_{n+k}$, 
\begin{enumerate}[(i)]
\item \[
S^*(e_{\tau}^*) = \begin{cases} e_{F(\tau)}^* & \text{ if } \tau \notin \{ \gamma , \tilde{\gamma} \} \\ -e_{F(\tau)}^* & \text{ if } \tau \in \{ \gamma , \tilde{\gamma} \} \end{cases} \]
\item \[
S^*(d_{\tau}^*) = \begin{cases} d_{F(\tau)}^* & \text{ if } \tau \notin \{ \gamma , \tilde{\gamma} \} \\ -d_{F(\tau)}^* & \text{ if } \tau \in \{ \gamma , \tilde{\gamma} \} \end{cases} \]
\item \[
S^*(c_{\tau}^*) = \begin{cases} c_{F(\tau)}^* & \text{ if } \tau \notin \{ \gamma , \tilde{\gamma} \} \\ -c_{F(\tau)}^* & \text{ if } \tau \in \{ \gamma , \tilde{\gamma} \}. \end{cases} \]
\end{enumerate}
\end{enumerate}
noting that the previous discussion gives us the case $k=1$. We will show that we can extend the maps $F$ and $S^*$ to $F \colon \Gamma_{n+k+1} \to \Gamma_{n+k+1}$ and $S^* \colon \ell_{1}(\Gamma_{n+k+1}) \to \ell_{1}(\Gamma_{n+k+1})$ such that the above properties are still satisfied. Let $\tau = (n+k+1, l, \theta, \vphi, \epsilon, \epsilon ')$. Recalling that $c_{\tau}^* \in \ell_1(\Gamma_{n+k})$, $S^*(c_{\tau}^*)$ is already defined by the inductive hypothesis.  In fact, by Proposition \ref{(s1)mainpropn2},
\[
S^*(c_{\tau}^*) = \begin{cases} \epsilon a S^*(e_{\theta}^*) & \text{ if } \vphi \in \Gamma_l \\ \epsilon a S^*(e_{\theta}^*) + \epsilon ' b S^*\left( (I-P_{(0,l]}^*)e_{\vphi}^* \right) & \text{ if } \vphi \in \Gamma_{n+k} \setminus \Gamma_l. \end{cases}
\]
$S^*(e_{\theta}^*)$ and $S^*(e_{\vphi}^*)$  are given by the inductive hypothesis and property (2i) of $S^*$. The slightly harder term to deal with is $S^*(P^*_{(0,l]}e_{\vphi}^*)$ where $\vphi \in \Gamma_{n+k} \setminus \Gamma_{l}$. We observe that since $e_{\vphi}^* \in \ell_{1}(\Gamma_{n+k})$ we can write
\[
e_{\vphi}^* = \sum_{\omega \in \Gamma_{n+k} \setminus \{ \gamma , \tilde{\gamma} \} } \alpha_{\omega}d_{\omega}^* + \alpha_{\gamma}d_{\gamma}^* + \alpha_{\tilde{\gamma}}d_{\tilde{\gamma}}^* 
\]
By the inductive hypothesis, we get 
\[
S^*(e_{\vphi}^*) = \sum_{\omega \in \Gamma_{n+k} \setminus \{ \gamma , \tilde{\gamma} \} } \alpha_{\omega}d_{F( \omega )}^* - \alpha_{\gamma}d_{\tilde{\gamma}}^* - \alpha_{\tilde{\gamma}}d_{\gamma}^* 
\]
and, using the fact that $F$ is rank preserving, we find that 
\[
P^*_{(0,l]} ( S^*(e_{\vphi}^*) ) = \begin{cases} \sum_{\omega \in \Gamma_{l}} \alpha_{\omega}d_{F(\omega )}^* & \text{ if } l < n+1 \\ \sum_{\omega \in \Gamma_{l} \setminus \{ \gamma , \tilde{\gamma} \} } \alpha_{ \omega }d_{F( \omega )}^* - \alpha_{\gamma}d_{\tilde{\gamma}}^* - \alpha_{\tilde{\gamma}}d_{\gamma}^* & \text{ if } n+1 \leq l < n+k. \end{cases}
\]
In either case, it is now easily seen by a similar computation that $P^*_{(0,l]} ( S^*(e_{\vphi}^*) ) = S^*(P_{(0,l]}^*e_{\vphi}^*)$. From this computation and property 2(i) of the inductive hypothesis, we find that
\begin{align*}
S^*P_{(0,l]}^* e_{\vphi}^* = \begin{cases} P^*_{(0,l]}(e_{F(\vphi)}^*) & \text{ if } \vphi \notin \{ \gamma , \tilde{\gamma} \} \\ P^*_{(0,l]}(-e_{F(\vphi)}^*) & \text{ if } \vphi \in \{ \gamma , \tilde{\gamma} \}. \end{cases}
\end{align*}
Note that since $\theta \in \Gamma_l$, it not possible to have $\vphi \in \Gamma_{n+k} \setminus \Gamma_l$ and both $\theta$ and $\vphi$ in $\{ \gamma , \tilde{\gamma} \}$. So when $\vphi \in \Gamma_{n+k} \setminus \Gamma_l$, it is apparent from our computations that there are precisely 3 possibilities to consider when computing $S^*(c_{\tau}^*)$. We obtain
\begin{align*}
S^*(c_{\tau}^*) &= S^* \left( \epsilon a e_{\theta}^* + \epsilon ' b (I-P_{(0,l]}^*)e_{\vphi}^* \right) \\
&= \begin{cases} \epsilon a e_{F(\theta)}^* + \epsilon ' b \left\{ (I - P_{(0,l]}^*)(e_{F(\vphi)}^*) \right\} & \theta , \vphi \notin \{ \gamma , \tilde{\gamma} \} \\ -\epsilon a e_{F(\theta)}^* + \epsilon ' b \left\{ (I - P_{(0,l]}^*)(e_{F(\vphi)}^*) \right\} & \theta \in \{ \gamma , \tilde{\gamma} \} , \vphi \notin \{ \gamma , \tilde{\gamma} \} \\ \epsilon a e_{F(\theta)}^* + \epsilon ' b \left\{ (I - P_{(0,l]}^*)(-e_{F(\vphi)}^*) \right\} & \theta \notin \{ \gamma , \tilde{\gamma} \} , \phi \in \{ \gamma , \tilde{\gamma} \}. \end{cases}
\end{align*}
In the easier case where $\vphi \in \Gamma_l$, we have \[
S^*(c_{\tau}^*) = \epsilon a S^*(e_{\theta}^*) = \begin{cases} \epsilon a e_{F(\theta)}^* & \text{ if } \theta \notin \{ \gamma , \tilde{\gamma} \} \\ -\epsilon a e_{F(\theta)}^* & \text{ if } \theta \notin \{ \gamma , \tilde{\gamma} \}. \end{cases}
\]
Finally, we note that since $F$ is rank-preserving, all of the above formulae are equal to $c_{\Xi}^*$ for some suitably chosen $\Xi \in \Delta_{n+k+1}$. We can therefore define $F$ on $\Delta_{n+k+1}$ so that $S^*(c_{\tau}^*) = c_{F(\tau)}^*$. Our above calculations show that in order to achieve this, the correct definition of $F$ is $F\left( (n+k+1, l, \theta, \vphi , \epsilon , \epsilon ') \right) =$
\[
 \begin{cases} (n+k+1, l, F(\theta), \vphi , \epsilon , \epsilon ') & \text{ if } \vphi \in \Gamma_l \text{ and } \theta \notin \{ \gamma , \tilde{\gamma} \} \\ (n+k+1, l, F(\theta), \vphi , -\epsilon , \epsilon ') & \text{ if } \vphi \in \Gamma_l \text{ and } \theta \in \{ \gamma , \tilde{\gamma} \} \\ (n+k+1, l, F(\theta), F(\vphi) , \epsilon , \epsilon ') & \text{ if } \vphi \in \Gamma_{n+k} \setminus \Gamma_l \text{ and } \theta , \vphi \notin \{ \gamma , \tilde{\gamma} \} \\ (n+k+1, l, F(\theta), F(\vphi) , -\epsilon , \epsilon ') & \text{ if } \vphi \in \Gamma_{n+k} \setminus \Gamma_l \text{ and } \theta  \in \{ \gamma , \tilde{\gamma} \} , \vphi \notin \{ \gamma , \tilde{\gamma} \} \\ (n+k+1, l, F(\theta), F(\vphi) , \epsilon , -\epsilon ') & \text{ if } \vphi \in \Gamma_{n+k} \setminus \Gamma_l \text{ and } \theta \notin \{ \gamma , \tilde{\gamma} \} , \vphi \in \{ \gamma , \tilde{\gamma} \}. \end{cases}
\]
$F \colon \Gamma_{n+k + 1} \to \Gamma_{n+k + 1}$ is easily seen to be a bijection. We now extend the definition of $S^*$ by defining $S^*(d_{\tau}^*) := d_{F(\tau)}^*$ for $\tau \in \Delta_{n+k+1}$. Since $e_{\tau}^* = c_{\tau}^* + d_{\tau}^*$ it is clear that $S^*(e_{\tau}^*) = e_{F(\tau)}^*$ for any $\tau \in \Delta_{n+k+1}$. Thus we have extended $F$ and $S^*$ and both maps have all the desired properties that were stated earlier. 

By induction, we have thus succeeded in defining a linear map $S^* \colon c_{00}(\Gamma) \to \ell_1(\Gamma)$. It is bounded since, by construction, $S^*(e_{\tau}^*) = \pm e_{F(\tau)}^*$ for all $\tau \in \Gamma$ (with a - sign only when $\tau \in \{ \gamma , \tilde{\gamma} \}$). So $\norm{S^*(e_{\tau}^*)} = 1$ for all $\tau \in \Gamma$ and therefore $S^*$ is bounded by our earlier remark. It follows that $S^*$ extends to a bounded linear map on $\ell_1(\Gamma)$. In fact, $S^*$ is an isometry onto $\ell_{1}(\Gamma)$. Indeed, since $F$ is a bijection, and $S^*(e_{\tau}^*) = \pm e_{F(\tau)}^*$ for all $\tau \in \Gamma$, $S^*$ is certainly onto. Moreover, suppose $\sum_{\tau \in \Gamma} a_{\tau}e_{\tau}^* \in \ell_{1}(\Gamma)$, then 
\begin{align*}
\norm{ S^* \left( \sum_{\tau \in \Gamma} a_{\tau} e_{\tau}^* \right) } &= \norm{ \sum_{\tau \in \Gamma \setminus \{ \gamma , \tilde{\gamma} \} } a_{\tau} e_{F(\tau)}^* - a_{\gamma} e_{\tilde{\gamma}}^* - a_{\tilde{\gamma}} e_{\gamma}^*} \\
&= \sum_{ \tau \in \Gamma \setminus \{ \gamma , \tilde{\gamma} \} } |a_{\tau}| + |-a_{\gamma}| + |-a_{\tilde{\gamma}}| \\
&= \sum_{\tau \in \Gamma} |a_{\tau}| \\
&= \norm{\sum_{\tau \in \Gamma} a_{\tau}e_{\tau}^*}
\end{align*}
where we have repeatedly made use of the fact that $F: \Gamma \to \Gamma$ is a bijection, so $e_{F(\tau)}^* \neq e_{F(\tau ')}^*$ when $\tau \neq \tau '$. It follows that the dual operator $(S^*)^* \colon \ell_{1}(\Gamma)^* \to \ell_{1}(\Gamma)^*$ is also an isometry. For convenience, we drop the brackets and just write $S^{**}$ for $(S^*)^*$. 

Recall that $d_{\tau} \in \ell_{1}(\Gamma)^*$ denotes the usual biorthogonal vector (as in Section \ref{(s1)BDconstruction}) of the basis $(d_{\theta}^*)_{\theta \in \Gamma}$ of $\ell_1(\Gamma)$. Using the fact that 
\[
S^*(d_{\tau}^*) = \begin{cases} d_{F(\tau)}^* &\text{ if } \tau \notin \{ \gamma , \tilde{\gamma} \} \\ -d_{F(\tau)}^* & \text{ if } \tau \in \{ \gamma , \tilde{\gamma} \} \end{cases}
\]
it is elementary to check that
\[
S^{**}(d_{\omega}) = \begin{cases} d_{F^{-1}(\omega)} &\text{ if } \omega \notin \{ \gamma , \tilde{\gamma} \} \\ -d_{F^{-1}(\omega)}^* & \text{ if } \omega \in \{ \gamma , \tilde{\gamma} \} \end{cases}
\]

Recalling from Section \ref{(s1)BDconstruction} that the Bourgain-Delbaen space is precisely $X = [d_{\gamma} : \gamma \in \Gamma]$, we see that $S^{**}$ restricts to give an onto isometry $S := S^{**}| \colon X \to X$. We shall call $S$ a {\em basic operator}. By the previous observation, we get
\begin{lem} \label{(s1)formula_for_basic_operator}
Let $\sum_{\omega} \alpha_{\omega} d_{\omega} \in X$. Then 
\[
S\left( \sum_{\omega} \alpha_{\omega} d_{\omega} \right) = \sum_{\omega} \tilde{\alpha}_{\omega}d_{F^{-1}(\omega)} 
\]
where
 \[
\tilde{\alpha}_{\omega} = \begin{cases} \alpha_{\omega} & \text{ if } \omega \notin \{ \gamma , \tilde{\gamma} \} \\ -\alpha_{\omega} & \text{ if } \omega \in \{ \gamma , \tilde{\gamma} \} \end{cases}
\]
\end{lem}

Of course, the definition of $S^*$, and consequently $S$, depends on the initial choice of $\gamma \in \Delta_{n+1}$. Moreover, the construction of $S^*$ is very much dependent on the bijection $F \colon \Gamma \to \Gamma$, which also depends on our initial choice of $\gamma \in \Delta_{n+1}$. We define some terminology that will be useful later.
\begin{defn} \label{(s1)defn:terminology_of_basic_operators}
Suppose $S$ is constructed as above by fixing $\gamma \in \Delta_{n+1}$. We shall say that $S$ has \emph{rank} $n$ and that $\gamma$ is the \emph{base} of $S$. We write $\base(S) = \gamma$ and $\rank(S) = n$. We shall say the bijection $F \colon \Gamma \to \Gamma$, used in the construction and definition of $S^*$, is the \emph{underlying function on $\Gamma$}.

We will usually denote a rank $n$ basic operator by $S_n$ and its underlying function on $\Gamma$ by $F_n$. 
\end{defn}

\begin{lem} \label{(s1)basic_operators_of_different_ranks_are_separated}
Let $m < n$ and suppose $S_m, S_n \colon X \to X$ are basic operators of ranks $m$ and $n$ respectively. If $\lambda$ is the basis constant of the basis $(d_{\theta}^*)_{\theta \in \Gamma}$ of $\ell_1(\Gamma)$ such that $X = [ d_{\theta} : \theta \in \Gamma]$,  then \[
\norm{ S_n - S_m } \geq \frac{1}{2\lambda}
\]
\end{lem}
\begin{proof}
Let $\gamma = \base(S_m)$ (so in particular $\rank(\gamma) = m+1$) and let $F_n \colon \Gamma \to \Gamma$ be the the underlying function on $\Gamma$ for $S_n$. So by Lemma \ref{(s1)formula_for_basic_operator}, $S_n(d_{\gamma}) = d_{F_n^{-1}(\gamma)} = d_{\gamma}$ since $F_n|_{\Gamma_n} = \text{Id}|_{\Gamma_n}$, and $S_m(d_{\gamma}) = -d_{\tilde{\gamma}}$. We recall that $\norm{d_{\tau}} = 1$ for all $\tau \in \Gamma$ (see Lemma \ref{BasisIsNormalized} and Remark \ref{(s1)BDconstruction}). (This was also proved using a different method in \cite{Haydon2000}.) Using this fact, we can estimate as follows.  
\begin{align*}
\norm{ S_n - S_m } &\geq \norm {S_n(d_{\gamma}) - S_m(d_{\gamma}) } \\
&= \norm{d_{\gamma} + d_{\tilde{\gamma}} } \\
&\geq |d_{\gamma}(\| d_{\gamma}^*\|^{-1}d_{\gamma}^*) + d_{\tilde{\gamma}}(\| d_{\gamma}^*\|^{-1}d_{\gamma}^*)| \\
&= \| d_{\gamma}^*\|^{-1} \geq \frac{1}{2\lambda}.
\end{align*}
To obtain the final inequality, we have made use of the fact that for all $\tau \in \Gamma$, $\|d_{\tau}^* \| = \| e_{\tau} ^* - P^*_{(0, \text{rank} \tau)} e_{\tau}^*  \| \leq 1 + \lambda \leq 2\lambda$. 
\end{proof}
We therefore have at least a countably infinite number of operators on $X$, all of which are pairwise distance at least $1$ apart. To establish the non-separability of $\mathcal{L}(X)$ we need to work a bit harder. 

\subsection{Non-separability of $\mathcal{L}(X)$} \label{(s1)result}

In the final section of this chapter, we establish the non-separability of $\mathcal{L}(X)$. Throughout this section, $\lambda$ is the same constant as appearing in the previous lemma. The idea is to construct an uncountable set of operators on $X$, all of which are pairwise distance $1/2\lambda$ apart by taking suitable compositions of basic operators of increasing ranks. For each $n \in \N$, $n \geq 2$, we fix a $\gamma_n \in \Delta_{n+1}$ and let $S_n$ be the basic operator (as constructed in Section \ref{(s1)basic_operator_construction}) with $\base(S_n) = \gamma_n$. We work with these operators for the remainder of the section. The underlying function on $\Gamma$ for $S_n$ will be denoted by $F_n$. We recall that each $S_n$ is obtained by taking the dual of some operator $S_n^* \colon \ell_1(\Gamma) \to \ell_1(\Gamma)$ and restricting it to $X$. 

Before progressing any further, we need
\begin{prop} \label{(s1)propn1:sec4}
Let $(n_j)_{j=1}^{\infty}$ be a strictly increasing sequence of natural numbers with $n_1 \geq 2$. For $p \leq m, (p, m \in \N)$ denote by $T_{p,m}^*$ the composition of basic operators $T_{p,m}^* := S_{n_m}^* \circ \dots \circ S_{n_{p+1}}^* \circ S_{n_p}^*$. For any $x^* \in \ell_{1}(\Gamma)$, the sequence $(T_{1,n}^*x^*)_{n=1}^{\infty}$ converges in $\ell_{1}(\Gamma)$.
\end{prop}
\begin{proof}
It is enough to show $(T_{1,n}^*x^*)_{n=1}^{\infty}$ is a Cauchy sequence. We note that for $k > l$, 
\[
\norm{T_{1,k}^*x^* - T_{1,l}^*x^*} = \norm{T_{l+1, k}^* (T_{1,l}^*x^*) - T_{1,l}^*x^* }
\]
We recall that for any $p \in \N$, $S_{p}^*|_{\ell_{1}(\Gamma_p)} = \text{Id}_{\ell_{1}(\Gamma_p)}$. Therefore, since the sequence $(n_j)_{j=1}^{\infty}$ is strictly increasing, it is easily seen that $T_{l+1, k}^*|_{\ell_{1}(\Gamma_{n_{l+1}})} = \text{Id}|_{\ell_{1}(\Gamma_{n_{l+1}})}$. Since $(d_{\omega}^*)_{\omega \in \Gamma}$ is a (Schauder) basis for $\ell_{1}(\Gamma)$, we can write 
\[
T_{1,l}^*x^* = \sum_{\omega \in \Gamma} \alpha_{\omega}^{(l)}d_{\omega}^*
\]
So
\begin{align*}
\norm{T_{1,k}^*x^* - T_{1,l}^*x^*} &= \Big\|  T_{l+1, k}^* \Big( \sum_{\omega \in \Gamma} \alpha_{\omega}^{(l)}d_{\omega}^* \Big) - \sum_{\omega \in \Gamma} \alpha_{\omega}^{(l)}d_{\omega}^* \Big\| \\
&= \Big\| \sum_{\omega \in \Gamma_{n_{l+1}}} \alpha_{\omega}^{(l)}d_{\omega}^* + T_{l+1,k}^* \Big( \sum_{\omega \in \Gamma \setminus \Gamma_{n_{l+1}}} \alpha_{\omega}^{(l)}d_{\omega}^* \Big) - \sum_{\omega \in \Gamma} \alpha_{\omega}^{(l)}d_{\omega}^* \Big\| \\
&\leq 2\Big\| \sum_{\omega \in \Gamma \setminus \Gamma_{n_{l+1}} } \alpha_{\omega}^{(l)}d_{\omega}^* \Big\|
\end{align*}
where the final inequality is obtained from the triangle inequality and the fact that $T_{l+1,k}^*$ is a composition of isometries, therefore has norm $1$. To prove the proposition, it is enough to show $\| (T_{1,l}^* - P_{(0, n_{l+1}]}^* T_{1,l}^* ) x^*\| \to 0$ as $l \to \infty$. 

To see this, we consider $z^* = \sum_{\omega \in \Gamma} \delta_{\omega} d_{\omega}^*$ to be any vector in $\ell_1(\Gamma)$. For any $j, m \in \N$, since $F_m$ preserves the rank of all elements in $\Gamma$, an easy computation yields that $P_{(0,j]}^* S_m^* z^* = S_m^* P_{(0,j]}^* z^*$. Consequently, as a composition of $S_m^*$ operators, the same equality holds with $S_m^*$ replaced by $T_{l,k}^*$. for any $l,k \in \N$. Using this observation,  it is now easy to see that $\| (T_{1,l}^* - P_{(0, n_{l+1}]}^* T_{1, l}^*) x^*\| \to 0$ as $l \to \infty$. Indeed, \[
\| (T_{1,l}^* - P_{(0, n_{l+1}]}^* T_{1, l}^*) x^*\| = \| T^*_{1,l} (I - P_{(0, n_{l+1}]}^* ) x^* \| = \| (I - P_{(0, n_{l+1}]}^*) x^* \|.
\]
We have made use of the fact that the operators $T_{1,l}^*$ are isometries on $\ell_1(\Gamma)$ to obtain the final equality. The right-hand-side of the above expression clearly converges to $0$ as $l \to \infty$ since the sequence $(n_l)_{l=1}^{\infty} \to \infty$ and the $P^*_{(0, n_{l+1}]}$ operators are basis-projections.
% We now recall that $\sum_{\omega \in \Gamma} \alpha_{\omega}^{(l)}d_{\omega}^* = T_{1,l}^*x^*$ and write $x^* = \sum_{\omega \in \Gamma}\beta_{\omega}d_{\omega}^*$. By the previous observation we have,
%\begin{align*}
%\Big\| \sum_{\omega \in \Gamma \setminus \Gamma_{n_{l+1}} } \alpha_{\omega}^{(l)}d_{\omega}^* \Big\| &= \| T_{1,l}^* x^* - P_{(0, n_{l+1}]}^* T_{1,l}^* x^* \| \\
%&= \|T_{1,l}^* ( x^* - P_{(0, n_{l+1}]}^* x^* ) \| \\
%& = \|  x^* - P_{(0, n_{l+1}]}^* x^* \| \\
%&= \Big\|  \sum_{\omega \in \Gamma \setminus \Gamma_{n_{l+1}} } \beta_{\omega}d_{\omega}^* \Big\|
%\end{align*}
%where, in the penultimate equality, we made use of the fact that $T_{1,l}^*$ is an isometry. The last expression clearly converges to $0$ as $l \to \infty$. This completes the proof. 
%
%  then it is easily seen from the construction of the $S_m^*$'s that 
%\[
%\Big\| \sum_{\omega \in \Gamma \setminus \Gamma_{n_{l+1}} } \alpha_{\omega}^{(l)}d_{\omega}^* \Big\| = \Big\|  \sum_{\omega \in \Gamma \setminus \Gamma_{n_{l+1}} } \beta_{\omega}d_{\omega}^* \Big\|
%\]
%and the right hand side of the above expression certainly does converge to 0 as $l \to \infty$. This completes the proof.
\end{proof}

Using the above observation, it is easy to prove the following corollary:
\begin{cor} \label{(s1)corol:sequence_operator}
Let $(n_j)_{j=1}^{\infty}$ be a strictly increasing sequence of natural numbers with $n_1 \geq 2$, and let $T_{1,n}^*$ ($n \in \N$) be defined as in Proposition \ref{(s1)propn1:sec4}. Then we may well-define a linear operator $T_{(n_j)_{j=1}^{\infty}}^* \colon \ell_{1}(\Gamma) \to \ell_{1}(\Gamma)$ by 
\[
T_{(n_j)_{j=1}^{\infty}}^*(x^*) = \lim_{j \to \infty} T_{1,j}^*(x^*) \hspace{25pt} (x^* \in \ell_{1}(\Gamma))
\]
Moreover, $T_{(n_j)_{j=1}^{\infty}}^* \colon \ell_1(\Gamma) \to \ell_1(\Gamma)$ is an onto isometry.
\end{cor}

We are finally ready to prove the main result of this chapter.

\begin{thm} \label{NonSepThemOfOriginalBDs}
Given a strictly increasing sequence of natural numbers $(n_j)_{j=1}^{\infty}$ with $n_{1} \geq 2$, the dual operator of the operator $T_{(n_j)_{j=1}^{\infty}}^*$ restricts to give an onto isometry, $T_{(n_j)_{j=1}^{\infty}} \colon X \to X$. Moreover, if $(n_j)_{j=1}^{\infty}$ and $(\widetilde{n_j})_{j=1}^{\infty}$ are two different strictly increasing sequences, with $n_k = \widetilde{n_k}$ for $k = 1,2$ and $n_1 \geq 2$ then 
\[
\norm{ T_{(n_j)_{j=1}^{\infty}} - T_{(\widetilde{n_j})_{j=1}^{\infty}} } \geq \frac{1}{2\lambda}
\]
Since the set of all strictly increasing sequences of natural numbers with the first two terms being fixed is uncountable, it follows that $\mathcal{L}(X)$ is non-separable.
\end{thm}

\begin{proof}
The dual operator of $T_{(n_j)_{j=1}^{\infty}}^*$ is an onto isometry of $\ell_{\infty}(\Gamma)$ by standard duality arguments. So for the first part of the proof, we must only see that it maps $X$ {\em into} (and onto) $X$. We fix a $\theta \in \Gamma$ and show $(T_{(n_j)}^*)^*(d_{\theta}) \in [d_{\gamma} : \gamma \in \Gamma]$. For $\tau \in \Gamma$, we have
\begin{align*}
[(T_{(n_j)}^*)^*(d_{\theta})](d_{\tau}^*) &= d_{\theta} \big( T_{(n_j)}^*(d_{\tau}^*) \big) \\
&= \lim_{j \to \infty} d_{\theta} \big( T_{1,j}^* (d_{\tau}^*) \big) \\
&= \lim_{j \to \infty} d_{\theta} \big( S_{n_j}^* \circ \dots \circ S_{n_1}^*(d_{\tau}^*) \big)
\end{align*}

We let $m$ be mininal such that $\tau \in \Gamma_{n_m}$ (i.e. $\tau \notin \Gamma_{n_{m-1}}$, so in particular $\rank(\tau) > n_{m-1}$) and consider 3 possibilities:
\begin{enumerate}[(1)]
\item $m=1$. In this case $S_{n_j}^* \circ \dots \circ S_{n_1}^*(d_{\tau}^*) = d_{\tau}^*$ so it follows that $[(T_{(n_j)}^*)^*(d_{\theta})](d_{\tau}^*) = d_{\theta}(d_{\tau}^*)$.  
\item $m=2$. Now we have that $\forall k \geq 1$, $S_{n_k}^* \circ \dots \circ S_{n_1}^*(d_{\tau}^*) = S_{n_1}^*(d_{\tau}^*)$. It follows that
\begin{align*}
[(T_{(n_j)}^*)^*(d_{\theta})](d_{\tau}^*) &= d_{\theta}(S_{n_1}^*d_{\tau}^*) \\
&= \begin{cases} d_{\theta}(d_{F_{n_1}(\tau)}^*) & \text{ if } \tau \neq \base S_{n_1}^* \text{ or } F_{n_1}(\base S_{n_1}^*) \\ -d_{\theta}(d_{F_{n_1}(\tau)}^*) & \text{ otherwise. } \end{cases}
\end{align*}
\item $m > 2$. It is seen that $\forall k \geq m-1$, $S_{n_k}^* \circ \dots \circ S_{n_1}^*(d_{\tau}^*) = S_{n_{m-1}}^* \circ \dots \circ S_{n_1}^*(d_{\tau}^*)$. So 
\[
[(T_{(n_j)}^*)^*(d_{\theta})](d_{\tau}^*) = d_{\theta} \Big( S_{n_{m-1}}^* \circ \dots \circ S_{n_1}^*(d_{\tau}^*) \Big).
\]
Since $\rank(\tau) > n_{m-1} > n_{m-2} \dots > n_1$ and we are assuming $m > 2$, $S_{n_{m-2}}^* \circ \dots \circ S_{n_1}^*(d_{\tau}^*) = d_{F_{n_{m-2}} \circ \dots \circ F_{n_1}(\tau)}^*$ and so $d_{\theta} \big( S_{n_{m-1}}^* \circ \dots \circ S_{n_1}^* (d_{\tau}^*) \big) = d_{\theta} \big( S_{n_{m-1}}^* (d_{F_{n_{m-2}} \circ \dots \circ F_{n_1}(\tau)}^*)\big)$. We consider two sub-cases:
\begin{enumerate}[(a)]
\item Either $\base ( S_{n_{m-1}}^* ) = F_{n_{m-2}} \circ \dots \circ F_{n_1}(\tau)$ or $F_{n_{m-1}}\big( \base S_{n_{m-1}}^* \big) = F_{n_{m-2}} \circ \dots \circ F_{n_1}(\tau)$. We note that in particular this implies $\rank(\tau) = n_{m-1} + 1$. In this case, $S_{n_{m-1}}^* (d_{F_{n_{m-2}} \circ \dots \circ F_{n_1}(\tau)}^*) = -d_{F_{n_{m-1}} \circ F_{n_{m-2}} \circ \dots \circ F_{n_1}(\tau)}^*$. It follows that $[(T_{(n_j)}^*)^*(d_{\theta})](d_{\tau}^*) = -d_{\theta}(d_{F_{n_{m-1}} \circ F_{n_{m-2}} \circ \dots \circ F_{n_1}(\tau)}^*)$.
\item Both $\base S_{n_{m-1}^*}$ and $F_{n_{m-1}}(\base S_{n_{m-1}}^*)$ are not equal to $F_{n_{m-2}} \circ \dots \circ F_{n_1}(\tau)$. An argument similar to the previous case shows $[(T_{(n_j)}^*)^*(d_{\theta})](d_{\tau}^*) = d_{\theta}(d_{F_{n_{m-1}} \circ F_{n_{m-2}} \circ \dots \circ F_{n_1}(\tau)}^*)$
\end{enumerate}
\end{enumerate}
We note that in all possible cases, $[(T_{(n_j)}^*)^*(d_{\theta})](d_{\tau}^*) \neq 0$ only if $\rank(\theta) = \rank(\tau)$. So given $x^* = \sum_{\omega \in \Gamma} \alpha_{\omega}d_{\omega}^* \in \ell_{1}(\Gamma)$ we have 
\begin{align*}
[(T_{(n_j)}^*)^*(d_{\theta})](x^*) = [(T_{(n_j)}^*)^*(d_{\theta})] \big( \sum_{\omega \in \Gamma} \alpha_{\omega} d_{\omega}^* \big) &= \sum_{\substack{\omega \in \Gamma \\ \rank(\omega) = \rank(\theta) }} \alpha_{\omega} [(T_{(n_j)}^*)^*(d_{\theta})](d_{\omega}^*) \\
&= \sum_{\substack{ \omega \in \Gamma \\ \rank(\omega) = \rank(\theta) }} \lambda_{\omega}d_{\omega}(x^*)
\end{align*} 
where $\lambda_{\omega} := [(T_{(n_j)}^*)^*(d_{\theta})](d_{\omega}^*)$ and we have made use of the obvious fact that $\alpha_{\omega} = d_{\omega}(x^*)$. It follows that $(T_{(n_j)}^*)^*(d_{\theta}) \in X$ as it is a finite linear combination of the $d_{\gamma}$. Thus $(T_{(n_j)}^*)^*$ does indeed restrict to give an isometry from $X$ into $X$. One can do similar computations to verify that $(T_{(n_j)}^*)^*$ maps $X$ onto $X$. It follows that the operator $T_{(n_j)} \colon X \to X$, defined in the theorem, is an onto isometry as required.

We suppose now $(n_j)_{j=1}^{\infty}$ and $(\widetilde{n_j})_{j=1}^{\infty}$ are two strictly increasing sequences which have the same first two terms and such that $n_1 \geq 2$. We choose $k$ minimal s.t. $n_k \neq \widetilde{n_k}$ and w.lo.g. assume $n_k > \widetilde{n_k}$. We note that by the minimality of $k$ we have $n_j = \widetilde{n_j}$ whenever $j < k$ and that by assumptions on the sequences, $k > 2$. We note that for any $\theta , \tau \in \Gamma$ 
\[
\norm{T_{(n_j)} - T_{(\widetilde{n_j})} } \geq \Big| [T_{(n_j)}(d_{\theta})](\|d_{\tau}^*\|^{-1}d_{\tau}^*) - [T_{(\widetilde{n_j})}(d_{\theta})](\|d_{\tau}^*\|^{-1}d_{\tau}^*) \Big|.
\]
We will find specific $\theta, \tau \in \Gamma$ such that the right hand side of the above inequality is equal to $\|d_{\tau}^*\|^{-1}$, which will complete the proof. Our previous calculations show that 
\[
[T_{(n_j)}(d_{\theta})](d_{\tau}^*) = d_{\theta}(\pm d_{F_{n_{m-1}} \circ \dots \circ F_{n_1}(\tau)}^*)
\]
where $m \in \N$ is minimal such that $\tau \in \Gamma_{n_m}$ (and it is assumed that $\tau$ is chosen s.t. $m > 2$, i.e. $\tau$ is chosen in $\Gamma \setminus \Gamma_{n_2}$). Moreover we have a minus sign precisely when $\base S_{n_{m-1}} = F_{n_{m-2}} \circ \dots \circ F_{n_1}(\tau)$ or $F_{n_{m-1}} \big( \base S_{n_{m-1}} \big)= F_{n_{m-2}} \circ \dots \circ F_{n_1}(\tau) $. Similarly, 
\[
[T_{(\widetilde{n_j})}(d_{\theta})](d_{\tau}^*) = d_{\theta}(\pm d_{F_{\widetilde{n_{p-1}}} \circ \dots \circ F_{\widetilde{n_1}}(\tau)}^*)
\]
where $p \in \N$ is minimal such that $\tau \in \Gamma_{\widetilde{n_p}}$ (and it is assumed that $\tau$ is chosen s.t. $p > 2$). Moreover we have a minus sign precisely when $\base S_{\widetilde{n_{p-1}}} = F_{\widetilde{n_{p-2}}} \circ \dots \circ F_{\widetilde{n_1}}(\tau)$ or $F_{\widetilde{n_{p-1}}} \big( \base S_{\widetilde{n_{p-1}}} \big) = F_{\widetilde{n_{p-2}}} \circ \dots \circ F_{\widetilde{n_1}}(\tau) $.

We will in fact choose $\tau$ such that $\rank(\tau) = \widetilde{n_k} + 1$. It is easily seen that with a choice of $\tau$ like this, the `$p$' above must be equal to $k+1$. On the other hand, the choice of $m$ above must in fact be equal to $k$ since $\widetilde{n_k} < n_k \implies n_k \geq \widetilde{n_k} + 1$ and thus $\tau \in \Gamma_{\widetilde{n_k} + 1} \subseteq \Gamma_{n_k}$, whilst $\tau \notin \Gamma_{n_{k-1}} = \Gamma_{\widetilde{n_{k-1}}}$. Since we have observed $k > 2$ our above calculations remain valid, and we find that for such a choice of $\tau$ we have
\begin{align*}
[T_{(\widetilde{n_j})}(d_{\theta})](d_{\tau}^*) &= d_{\theta}(\pm d_{F_{\widetilde{n_{k}}} \circ \dots \circ F_{\widetilde{n_1}}(\tau)}^*) \\
[T_{(n_j)}(d_{\theta})](d_{\tau}^*) &= d_{\theta}(\pm d_{F_{n_{k-1}} \circ \dots \circ F_{n_1}(\tau)}^*)
\end{align*}
where we have a minus sign precisely in the cases described above. Since all the $F_m$'s are rank preserving, a perfectly good choice of $\tau$ fulfulling the condition that $\rank(\tau) = \widetilde{n_k} + 1$ is $\tau = F_{\widetilde{n_1}}^{-1} \circ \dots \circ F_{\widetilde{n_{k-1}}}^{-1} ( \base S_{\widetilde{n_k}})$. If we now choose $\theta = F_{\widetilde{n_k}} (\base S_{\widetilde{n_k}})$ we see that 
\begin{align*}
[T_{(\widetilde{n_j})}(d_{\theta})](d_{\tau}^*) &= d_{\theta}\big( -d_{F_{\widetilde{n_k}}(\base S_{\widetilde{n_k}})}^* \big) = -1 \\
[T_{(n_j)}(d_{\theta})](d_{\tau}^*) &= d_{\theta}(\pm d_{\base S_{\widetilde{n_k}}}^*) = 0
\end{align*}
where to establish the final equalities we made use of the facts that $\widetilde{n_j} = n_j$ for all $j < k$ and $\theta = F_{\widetilde{n_k}}(\base S_{\widetilde{n_k}}) \neq \base S_{\widetilde{n_k}}$ by the very construction of $F_{\widetilde{n_k}}$. So, for these choices of $\theta , \tau$ we finally get
\begin{align*}
\norm{T_{(n_j)} - T_{(\widetilde{n_j})} } &\geq \Big| [T_{(n_j)}(d_{\theta})](\|d_{\tau}^*\|^{-1}d_{\tau}^*) - [T_{(\widetilde{n_j})}(d_{\theta})](\|d_{\tau}^*\|^{-1}d_{\tau}^*) \Big| \\
&= \|d_{\tau}^*\|^{-1} \\
&\geq \frac{1}{2\lambda}
\end{align*}
as required.
\end{proof}
We conclude this chapter by observing that we can use the non-separability result just obtained to provide a short proof that no Bourgain-Delbaen space of class $\mathcal{Y}$ can have the scalar-plus-compact property. 
\begin{cor}
There exists no $X \in \mathcal{Y}$ for which every bounded operator on $X$ is a scalar multiple of the Identity plus a compact operator.
\end{cor}
\begin{proof}
It was shown in \cite{BD80} that for any Bourgain-Delbaen space $X \in \mathcal{Y}$, $X^*$ is isomorphic to $\ell_1$. Moreover, we have seen that all spaces of Bourgain-Delbaen type have a Schauder basis, and consequently have the approximation property. For any Banach space $Y$ with separable dual and approximation property, the space of compact operators on $Y$, $\mathcal{K}(Y)$ is separable, and thus so also is the set of all operators of the form $\lambda I_Y + K$ ($\lambda \in \R, K \in \mathcal{K}(Y)$). Therefore if $X \in \mathcal{Y}$ had the scalar-plus-compact property, its operator algebra would be separable. This contradicts Theorem \ref{NonSepThemOfOriginalBDs}.
\end{proof}

Whilst our non-separability result allows us to give a nice proof that the spaces in $\mathcal{Y}$ do not have the scalar-plus-compact property, we remark that the ideas in this section can be suitably modified to produce much stronger results. Indeed, one can use similar ideas to construct non-trivial projections on any of the Bourgain-Delbaen spaces $X$ in $\mathcal{X}\cup\mathcal{Y}$. (A non-trivial projection is one which has both infinite dimensional kernel and image.)  It follows that $X$ does not have the property that every bounded linear operator on $X$ is a strictly singular perturbation of a scalar operator. Obviously this implies that the previous corollary holds for any of the original Bourgain-Delbaen spaces, not just those of class $\mathcal{Y}$.

%% file: MainResult.tex
\chapter{Spaces with few but not very few operators} \label{MainResult}

\section{The Main Theorem}

Having looked in detail at the Bourgain-Delbaen construction in the previous chapter, we move on to consider a problem posed by Argyros and Haydon in \cite{AH}. We begin by introducing the following definition. 

\begin{defn}
Let $X$ be a Banach space. We will say $X$ has {\em few operators} if every operator from $X$ to itself is a strictly singular perturbation of the identity, that is, expressible as $\lambda I + S$ for some strictly singular operator $S \colon X \to X$. We say $X$ has {\em very few operators} if every operator from $X$ to itself is a compact perturbation of the identity.
\end{defn}

The fact that Banach spaces with few operators exist was first established in 1993 by Gowers and Maurey in \cite{GowMau}. The existence of a Banach space with very few operators was shown much more recently in \cite{AH}. In addition to their obvious intrinsic interest, such Banach spaces exhibit remarkable Banach space structure, as discussed in the introduction to this thesis. The motivation behind this chapter comes from a question that naturally arises from the work of Argyros and Haydon in \cite{AH}.

We recall once again that the Banach space $\XK$, constructed by Argyros and Haydon in \cite{AH}, is done so using the generalised Bourgain-Delbaen construction, as discussed in the previous chapter. Consequently, it has a Schauder basis and is a separable $\mathscr{L}_{\infty}$ space. It was seen in \cite{AH} that the control over the finite dimensional subspace structure provided by the $\mathscr{L}_{\infty}$ property was essential in proving that $\XK$ has very few operators and we will use the same type of argument later in this chapter. Moreover, one can exploit the $\mathcal{L}_{\infty}$ structure to show that the dual space of $\XK$ is $\ell_1$.

In light of the previous discussion, it is natural to conjecture that the ideals of strictly singular and compact operators on a separable $\mathcal{L}_{\infty}$ space with $\ell_1$ dual coincide.  Clearly this must be true for the space $\XK$. Moreover, if one considers the space $c_0$, the most obvious example of a separable $\mathcal{L}_{\infty}$ space with $\ell_1$ dual, it is once again found that the conjecture holds. Indeed, it is well known (see, e.g. \cite[Theorem 2.4.10]{Kalton}) that a bounded linear operator defined on $c_0$ is compact $\iff$ it is weakly compact $\iff$ it is strictly singular. More generally, for spaces $X$ with $\ell_1$ dual, the weakly compact and norm compact operators from $X$ to itself must coincide as a consequence of $\ell_1$ having the Schur property; this provides further hope that the conjecture holds if we demand $\ell_1$ duality. 

%As we shall see shortly, the above conjecture is not true in general. At this point, one might suspect that the conjecture was too optimistic. Banach space with the aforementioned properties coincide? If a Banach space possesses the properties just described and has few operators, must it necessarily have very few operators? More precisely, must a (separable) $\mathscr{L}_{\infty}$ space with $\ell_1$ dual and few operators necessarily have very few operators? Alternatively, we may rephrase the question yet again; must the ideal of strictly singular and compact operators on a Banach space with the aforementioned properties coincide? This question motivates the rest of this chapter; clearly the answer to the question in the case of the space $\XK$ is yes. Moreover, 
%
% we will exhibit a (negative) solution to this problem.  Before doing so, we choose to document two results which are related to this question and provide further motivation. 
%
%Firstly, we observe that 
%
On the other hand, if we instead look at operators defined on $X$, a separable $\mathcal{L}_{\infty}$ space with $\ell_1$ dual, but mapping into a different target space, it is certainly possible for us to construct strictly singular operators that are not weakly compact (and hence not compact). To see this we need the following lemma.
\begin{lem}
Let $X$ be a separable Banach space and suppose $\ell_1$ embeds isomorphically into $X^*$. Then there is a quotient operator $Q: X \to c_0$.
\end{lem}

\begin{proof}
Since $\ell_1$ embeds into $X^*$, we can find a sequence $(y_n^*)$ in the unit ball, $B_{X^*}$, of $X^*$ equivalent to the  canonical basis $(e_n)$ of $\ell_1$. Since $X$ is separable, the weak* topology restricted to the $B_{X^*}$ is metrizable. It follows from this (and weak$^*$ compactness of $B_{X^*}$) that we may assume (passing to a subsequence of the $(y_n^*)$ if necessary) that the sequence $(y_n^*)$ is weak$^*$ convergent. Now the sequence $(x_n^*) \subseteq X^*$ where $x_n^* := y_{2n}^* - y_{2n-1}^*$ is weak$^*$ null and moreover, $(x_n^*) \sim (e_n) \subseteq \ell_1$ (recall every seminormalized block basic sequence of $\ell_1$ is equivalent to the canonical basis of $\ell_1$, see, e.g. \cite[Lemma 2.1.1, Remark 2.1.2]{Kalton}). Let $T \colon c_0^* \equiv \ell_1 \to X^*$ be the isomorphic embedding which maps the basis vector $e_n$ of $\ell_1$ to $x_n^*$. We claim that $T$ is weak$^*$-weak$^*$ continuous. Indeed, suppose $((a_n^{\beta})_{n=1}^{\infty})_{\beta \in I}$ is a weak* convergent net in $c_0^*$ with $\lim_{\beta} (a_n^{\beta}) =  (\alpha_n)$, i.e. whenever $(\xi_n)_{n=1}^{\infty} \in c_0$, $\sum_{n=1}^{\infty} a_n^{\beta} \xi_n \to \sum_{n=1}^{\infty} \alpha_n \xi_n$. Now, for any $x \in X$, the sequence $(x_n^*x) \in c_0$ since the sequence $(x_n^*)$ is weak$^*$ null. It follows that \[
T\big( (a_n^{\beta}) \big) x = \sum_{n=1}^{\infty}a_n^{\beta} x_n^*x \to \sum_{n=1}^{\infty}\alpha_n x_n^*x = T\big( (\alpha_n ) \big) x
\]
i.e. $T\big( (a_n^{\beta}) \big) \wsto T\big( (\alpha_n) \big)$ as required. It follows by Lemma \ref{w*tow*impliesdual} that $T : c_0^* \to X^*$ is the dual of some operator $Q \colon X \to c_0$. Moreover, by Lemma \ref{TquotientiffT*iso}, $Q$ is a quotient operator.
\end{proof}

\begin{cor}
Suppose $X$ is an $\mathscr{L}_{\infty}$ space of Bourgain-Delbaen type, i.e. obtained from one of the constructions discussed in the previous chapter. If $X$ has no subspace isomorphic to $c_0$ then there exists a strictly singular operator $Q\colon X \to c_0$ which fails to be weakly compact.
\end{cor}
\begin{proof}
We have seen in the previous chapter that all the $\mathscr{L}_{\infty}$ spaces $X$ of Bourgain-Delbaen type are separable (they have a Schauder basis) and moreover, that $\ell_1$ embeds into the dual space. Consequently, by the above lemma, there is a quotient operator $Q \colon X \to c_0$. 

Suppose for contradiction that $Q$ is weakly compact. Since $Q$ is a quotient operator, it follows from the Open Mapping Theorem, that there is an $M$ such that $B_{c_0}^{\circ} \subseteq MQ(B_{X}^{\circ}) \subseteq MQ(B_X)$. Therefore, $B_{c_0} = \overline{B_{c_0}^{\circ}}^w \subseteq \overline{M Q(B_X) }^w$ so that $B_{c_0}$ is weakly compact (it is a weakly closed subset of the weakly compact set $\overline{M Q(B_X) }^w$). This contradicts the fact that $c_0$ is not reflexive, so $Q$ cannot be weakly compact.

Since $X$ has no subspace isomorphic to $c_0$ it follows that $Q$ is strictly singular.  To see this, suppose there is an infinite dimensional subspace $Y$ on which $Q$ is an isomorphism. Then, by \cite[Proposition 2.2.1]{Kalton} the subspace $Q(Y)$ contains a subspace $Z$ isomorphic to $c_0$. The image $(Q|_{Y\to Q(Y)})^{-1} (Z) $ is then a subspace of $X$ isomorphic to $c_0$, giving us a contradiction.
\end{proof}

We remark that there do exist $\mathscr{L}_{\infty}$ spaces of Bourgain-Delbaen type which have no subspace isomorphic to $c_0$ and therefore satisfy the hypotheses of the previous corollary. Indeed, it is immediate from the HI property that the Arygros-Haydon space is one such example. In fact, it was shown in \cite{BD80}, that if $X \in \mathcal{X}$, $Y \in \mathcal{Y}$, where $\mathcal{X}, \mathcal{Y}$ are the original classes of Bourgain-Delbaen spaces, then every infinite dimensional subspace of $X$ contains a subspace isomorphic to $\ell_1$ and every infinite dimensional subspace of $Y$ contains an infinite dimensional reflexive subspace. It follows from these facts and elementary results in Banach space theory that the original Bourgain-Delbaen spaces also have no subspace isomorphic to $c_0$.

Of course, the previous corollary does not give us a counterexample to the conjecture posed at the beginning of this chapter; the proof relies on exploiting properties of the space $c_0$ which features as the co-domain of the operator exhibited. We are interested in operators from a space {\em to itself}. The purpose of this chapter is to exhibit a class of Banach spaces which are genuine counterexamples to the previously stated conjecture. In fact, we provide counterexamples to the following, stronger version of the original conjecture; `must a separable $\mathcal{L}_{\infty}$ space with few operators and $\ell_1$ dual necessarily have very few operators?'. This question was originally posed by Argyros and Haydon in \cite[Problem 10.7]{AH}. Our main result is the following:

\begin{thm}
Given any $k \in \N, k \geq 2$, there is a separable $\mathscr{L}_{\infty}$ space, $\X_k$, with the following properties:-
\begin{enumerate}
\item $\X_k$ is hereditarily indecomposable (HI) and $\X_k^* = \ell_1$.
\item There is a non-compact bounded linear operator $S \colon \X_k \to \X_k$. $S$ is nilpotent of degree $k$, i.e. $S^j \neq 0$ for $1 \leq j < k$ and, $S^k = 0$. 
\item Moreover, $S^j$ ($ 0 \leq j \leq k-1$) is not a compact perturbation of any linear combination of the operators $S^l, l \neq j$.
\item The operator $S \colon \X_k \to \X_k$ is strictly singular (and consequently $S^j$ is strictly singular for all $j \geq 1$).
\item Every operator $T$ on $\X_k$ can be uniquely represented as $T = \sum_{i=0}^{k-1} \lambda_i S^i  + K$, where $\lambda_i \in \R, (0 \leq i \leq k-1)$ and $K$ is a compact operator on $\X_k$.
\end{enumerate}
\end{thm}

It is immediate from the above properties (and the fact that the strictly singular operators are a closed ideal in the operator algebra - see Corollary \ref{SSareIdeal}) that the spaces $\X_k$ have few operators but not very few operators. In other words, $\R I + \mathcal{K}(\X_k) \subsetneqq \mathcal{L}(\X_k)  \subseteq \R I + \mathcal{S}\mathcal{S} (\X_k) $ (where $\mathcal{S}\mathcal{S} (\X_k) $ is the space of strictly singular operators on $\X_k$ ). We thus have a negative solution to Problem 10.7 of Argyros and Haydon (\cite{AH}). 

In addition to solving the previously discussed problem, it turns out that there are a number of other interesting consequences of these constructions, specifically concerning the Calkin algebras and the structure of closed ideals in $\mathcal{L}(\X_k)$. Since the details of the proof of the main result are fairly long and technical, we choose to present these corollaries first. 

\section{Corollaries of the Main Theorem}
\subsection{On the structure of the closed ideals in $\mathcal{L}(\X_k)$.}
The structure of norm-closed ideals in the algebra $\mathcal{L}(X)$ of all bounded linear operators on an infinite dimensional Banach space $X$ is generally not well understood. For example, classifying all the norm-closed ideals in $\mathcal{L}(\ell_p \oplus \ell_q)$ with $p \neq q$ remains an open problem (though some progress has been made in \cite{SSJ}). What is known, is that for the $\ell_p$ spaces, $1 \leq p < \infty$, and $c_0$, there is only one non-trivial closed ideal in $\mathcal{L}(X)$, namely the ideal of compact operators. This was proved by Calkin, \cite{Calkin}, for $\ell_2$ and then extended to $\ell_p$ and $c_0$ by Gohberg et al., \cite{Gohberg}. More recently, the complete structure of closed ideals in $\mathcal{L}(X)$ was described in \cite{LLR04} for $X = (\oplus_{n=1}^{\infty} \ell_2^n )_{c_0}$ and in \cite{LSZ06} for $X= (\oplus_{n=1}^{\infty} \ell_2^n )_{\ell_1}$.  In both cases, there are exactly two nested proper closed ideals. Until the space constructed by Argyros and Haydon, \cite{AH}, these were the only known separable, infinite dimensional Banach spaces for which the norm-closed ideal structure of the operator algebra is completely known. 

Clearly the space $\XK$ of Argyros and Haydon provides another example of a separable Banach space for which the ideal of compact operators is the only (proper) closed ideal in the operator algebra. The spaces constructed in this chapter allow us to add to the list of spaces for which the ideal structure of $\mathcal{L}(X)$ is completely known. In fact, we see that we can construct Banach spaces for which the ideals of the operator algebra form a finite, totally ordered lattice of arbitrary length. More precisely, the following is an immediate consequence of our main theorem. 

\begin{thm} \label{NormClosedIdealsXk}
There are exactly $k$ norm-closed, proper ideals in $\mathcal{L}(\X_k).$ The lattice of closed ideals is given by \[
\mathcal{K}(\X_k)  \subsetneq \langle S^{k-1} \rangle \subsetneq \langle S^{k-2} \rangle \dots \langle S \rangle \subsetneq \mathcal{L}(X_k).
\]
Here, if $T$ is an operator on $\X_k$, $\langle T \rangle$ is the smallest norm-closed ideal in $\mathcal{L}(\X_k)$ containing T.
\end{thm}

\subsection{The Calkin algebra $\mathcal{L}(\X_k) / \mathcal{K}(\X_k)$.}
We note that as a consequence of properties (3) - (5) of the main theorem, the Calkin algebra $\mathcal{L}(\X_k)/ \mathcal{K}(\X_k) $ is $k$ dimensional with basis $\{ I + \mathcal{K}(\X_k), S+ \mathcal{K}(\X_k), \dots ,S^{k-1}+ \mathcal{K}(\X_k) \}$. More precisely, it is isomorphic as an algebra to the subalgebra $\mathcal{A}$ of $k\times k$ upper-triangular-Toeplitz matrices, i.e. $\mathcal{A}$ is the subalgebra of $\text{Mat}(k\times k)$  generated by \[
\left\{  \begin{pmatrix} 0 & 1 \\ & 0 & 1 \\  & & \ddots & \ddots \\ & & & \ddots & 1 \\ & & & & 0 \end{pmatrix}^j : 0 \leq j \leq k-1 
\right\}
\]
An explicit isomorphism is given by $\psi \colon \mathcal{L}(\X_k) / \mathcal{K}(\X_k) \to \mathcal{A} \, \cong \R[X]/ \langle x^k \rangle$, \[
\sum_{j=0}^{k-1} \lambda_j S^j + \mathcal{K}(\X_k) \mapsto
 \begin{pmatrix} \lambda_0 & \lambda_1 & \lambda_2 & \cdots & \cdots & \lambda_{k-1} \\

0 & \lambda_0 & \lambda_1 & \lambda_2 & \cdots & \lambda_{k-2} \\

0 & 0 & \lambda_0 & \lambda_1 & \ddots &  \vdots \\

\vdots & \vdots & 0 & \ddots & \ddots& \vdots \\

\vdots & \vdots  & \vdots & & \ddots & \vdots \\

0 & 0 & 0 & \cdots & 0 & \lambda_0

\end{pmatrix}
\]

We remark that the quotient algebra $\mathcal{L}(X) / \mathcal{S}\mathcal{S} (X)$ (where $ \mathcal{S}\mathcal{S} (X)$ denotes the strictly singular operators on $X$) was studied by Ferenczi in \cite{Fer} for X a HI, or more generally, a $\text{HD}_n$ space. Here it was shown that for a complex $\text{HD}_n$ space, $\text{dim}(\mathcal{L}(X) / \mathcal{S}\mathcal{S} (X)) \leq n^2$.  The above remarks show that the Calkin algebra of a HI space can behave somewhat differently however; indeed any finite dimension for the Calkin algebra can be achieved. In the next chapter, we will extend these results further still, and show that it is possible to obtain a Banach space $\X_{\infty}$ which has Calkin algebra isometric as  a Banach algebra to $\ell_1(\N_0)$.

\subsection{Commutators in Banach Algebras}
Rather unexpectedly, the space $\X_2$ also provides a counterexample to a conjecture of Professor William B. Johnson; the author is thankful to Professor Johnson for bringing this to his attention. To motivate the conjecture we recall the following classical theorem of Wintner. 

\begin{thm}[Wintner]
The identity in a unital Banach algebra, $\mathcal{A}$, is not a commutator, that is, there do not exist $x, y \in \mathcal{A}$ such that $1 = [ x, y ] := xy - yx$. 
\end{thm}

\begin{proof}
Suppose for contradiction the theorem is false and choose $x,y \in \mathcal{A}$ such that $1 = xy - yx$. We claim that $xy^n - y^nx = ny^{n-1}$ for all $n \in \N$. The proof is a simple induction; clearly it is true for $n=1$. Suppose that $xy^n - y^nx = ny^{n-1}$ holds for some $n \in \N$. Then 
\[
xy^{n+1} - y^{n+1}x = (xy^n - y^n x) y +y^n (xy - yx) = ny^{n-1}\cdot y + y^n \cdot 1 = (n+1)y^n
\]
completing the induction.

Observe that $y^n \neq 0$ for any $n$, since otherwise there is a smallest value of $n$ with $y^n = 0$. However, we now find that $0 = xy^n - y^n x = ny^{n-1}$, contradicting the minimality of $n$. Finally, we have for any $n$ \[
n \| y^{n-1} \| = \| xy^{n} - y^{n}x \| \leq 2 \|x \| \|y\| \|y^{n-1} \|
\]
and since $\| y^{n-1} \| \neq 0$ for all $n$, we have $n \leq 2 \|x\|\|y\|$ for all $n$. This obvious contradiction completes the proof.
\end{proof}

If $X$ is a Banach space and $\mathcal{M}$ is a proper norm-closed ideal of $\mathcal{L}(X)$, applying the above theorem to the unital Banach algebra $\mathcal{L}(X) / \mathcal{M}$, it is easily seen that no operator of the form $\lambda I + M$ with $\lambda \neq 0$ and $M \in \mathcal{M}$ is a commutator in the algebra $\mathcal{L}(X)$. In fact, as is commented in \cite{JohnsonCommutatorsOnSumlp}, this is in general the only known obstruction for an operator in $\mathcal{L}(X)$ to fail to be a commutator when $X$ is any infinite dimensional Banach space. Consequently the authors of \cite{JohnsonCommutatorsOnSumlp} define a {\em Wintner space} as a Banach space $X$ such that every non commutator in $\mathcal{L}(X)$ is of the form $\lambda I + M$ where $\lambda \neq 0$ and $M$ lies in a proper norm-closed ideal.

Johnson conjectured that every infinite dimensional Banach space is a Wintner space. Until the existence of the spaces constructed in this thesis, all infinite dimensional Banach spaces $X$ for which the commutators in $\mathcal{L}(X)$ could be classified satisfied this conjecture. However, we can easily see that $\X_2$ fails to be a Wintner space. 

\begin{lem}
$\X_2$ is not a Wintner space.
\end{lem}

\begin{proof}
We note that the operator $S \in \mathcal{L}(\X_2)$ fails to be of the form $\lambda I + M$ with $\lambda \neq 0$ and $M$ lying in a proper norm-closed ideal of $\mathcal{L}(\X_2)$. Indeed, if $S$ has such a form, we can write $S = \lambda I + \mu S + K$ where $\lambda \neq 0$, since by Theorem \ref{NormClosedIdealsXk} the only proper norm-closed ideals in $\X_2$ are $\mathcal{K}(\X_2)$ and $\langle S \rangle = \{ \mu S + K : \mu \in \R, K \in \mathcal{K}(\X_2) \}$. However, this contradicts the fact that the vectors $I + \mathcal{K}(\X_2)$ and $S + \mathcal{K}(\X_2)$ are linearly independent in $\mathcal{L}(\X_2) / \mathcal{K}(\X_2)$ (see property (3) of the main theorem).

To complete the proof, it suffices to see that $S$ is not a commutator; this is easy using the operator representation, property (5), of the main theorem. Indeed suppose for contradiction that there exist $T_1, T_2 \in \mathcal{L}(\X_2)$ with $S = [T_1,T_2]$. For $i = 1,2$, we can write $T_i = \lambda_i I + \mu_i S + K_i$ where $\lambda_i, \mu_i$ are scalars and $K_i$ are compact operators. We then have
\begin{align*}
S &= T_1T_2 - T_2T_1 \\
&=  (\lambda_1 I + \mu_1 S + K_1)(\lambda_2 I + \mu_2 S + K_2) - (\lambda_2 I + \mu_2 S + K_2)(\lambda_1 I + \mu_1 S + K_1) \\
&=K'
\end{align*}for some compact operator $K'$. However, we know by property (2) of the main theorem that $S$ is not compact.Therefore $S$ cannot be a commutator in $\mathcal{L}(\X_2)$.
\end{proof}

\subsection{Invariant subspaces.}
Finally, we remark that it was shown in \cite{AH} that all operators on the Argyros-Haydon space admit non-trivial, closed, invariant subspaces. The same is true for all the spaces $\X_k$. Indeed, by a result of Lomonosov (see \cite{Lom}), if an operator $T$ commutes with a non-zero compact operator, then $T$ has a proper, closed, invariant subspace. In particular, if there is some polynomial of $T$ which is compact and non-zero, then certainly $T$ has a proper, closed, invariant subspace.

For an operator $T \colon \X_k \to \X_k$, $T = \sum_{j=0}^{k-1} \lambda_j S^j  + K$, we consider the polynomial of $T$, given by $\mathcal{P}(T):= (T - \lambda_0 I)^{k}$. It follows (by the ring isomorphism of the Calkin algebra with the ring $\R[X] / \langle x^k \rangle$) that $\mathcal{P}(T)$ is a compact operator. So if $\mathcal{P}(T) \neq 0$ then we are done by the result of Lomonosov. Otherwise it is clear that $\lambda_0$ is an eigenvalue of $T$, so it has a one dimensional invariant subspace.

\section{The Basic Construction}

The rest of this chapter is devoted to proving the main theorem. Before continuing any further, we define the notation and terminology that shall be used throughout this chapter. We construct our spaces $\X_k$ by modifying the ideas used by Argyros and Haydon in their construction of the space $\XK$. Consequently, it is convenient for us to work with the same notation and terminology that they introduced in $\cite{AH}$. 

We will be working with two strictly increasing sequences of natural numbers $(m_j)$ and $(n_j)$ which satisfy the following assumptions.
\begin{assump}\label{mnAssump}
We assume that $(m_j,n_j)_{j\in \N}$ satisfy the following:
\begin{enumerate}
\item $m_1\ge 4$; 
\item $m_{j+1} = m_j^2$;
\item $n_1\ge m_1^2$; 
\item $n_{j+1} \ge
(16n_j)^{\log_2m_{j+1}}=m_{j+1}^2(4n_j)^{\log_2m_{j+1}}.$
\end{enumerate}
\end{assump}

We note that these are almost the same assumptions as in \cite{AH}; the only difference is that we demand the stronger condition that $m_{j+1}= m_j^2$ for all $j \in \N$. (In \cite{AH}, it was only assumed that $m_{j+1} \geq m_j^2$ for all $j$.) Consequently, we have no problems drawing upon results from \cite{AH}. In fact, for the purposes of this chapter, asking that $m_{j+1} \geq m_j^2$ for all $j \in \N$ would have been sufficient, but we need the slightly stronger condition in the following chapter. 

Like the Argyros-Haydon space, our spaces $\X_k$ will be obtained from the generalised Bourgain-Delbaen construction, described in the previous chapter (Theorem \ref{BDThm}).  Continuing with the same notation used in Theorem \ref{BDThm}, we recall that the subspaces $M_n = [d_\gamma:\gamma\in \Delta_n]$
form a finite-dimensional decomposition for $X=X(\Gamma, \tau)$. For each interval
$I\subseteq \N$ we define the projection $P_I:X\to \bigoplus
_{n\in I}M_n$ in the natural way; this is consistent with our use of
$P^*_{(0,n]}$ in Theorem~\ref{BDThm}. As in \cite{AH}, many of the arguments will
involve sequences of vectors that are block sequences with respect
to this FDD. It will therefore be useful to make the following definition; for $x \in X$, we define the {\em range}  of $x$, denoted $\ran x$, to be the smallest interval $I\subseteq \N$ such that $x\in
\bigoplus_{n\in I}M_n$.  

We recall from Theorem \ref{BDThm} that there is an `$\ve$ parameter' appearing in the admissible tuples of $\tau(\gamma)$ in the cases (0) and (2). Whilst this extra degree of freedom was required in the previous chapter in order for us to unite the original and generalised Bourgain-Delbaen constructions, we will not require it for the rest of this thesis; $\ve$ shall always be set equal to $1$. Consequently, we choose to simplify our notation and simply omit the $\ve$ appearing in these tuples. This is also consistent with the notation from \cite{AH}. With this new notation, if $\gamma \in \Delta_{n+1}$, we say the corresponding vector $c_{\gamma}^* \in \ell_1(\Gamma)$ is a \emph{Type 0 BD-functional} if $\tau(\gamma) = (\alpha, \xi)$ (in the notation of Theorem \ref{BDThm} this would have been $\tau(\gamma) = (1, \alpha, \xi)$). Similarly, $c_{\gamma}^*$ is called a \emph{Type 1 BD-functional} if $\tau(\gamma) = (p, \beta, b^*)$ or a \emph{Type 2 BD-functional} if $\tau(\gamma) = (\alpha, \xi, p, \beta, b^*)$ (see Theorem \ref{BDThm}). We will call $\gamma$ a Type 0, 1 or 2 element in each of the respective cases.

Our space will be a specific Bourgain-Delbaen space very similar to the space $\XK$ constructed in \cite{AH}.  We adopt the same notation used in \cite{AH}, in which elements $\gamma$ of $\Delta_{n+1}$ automatically code the corresponding BD-functionals. Consequently, we can write $X(\Gamma)$ rather than $X(\Gamma, \tau)$ for the resulting $\mathscr{L}_{\infty}$ space. To be more precise, an element $\gamma$ of $\Delta_{n+1}$ will be a tuple
of one of the following forms:
\begin{enumerate}
\item $\gamma= (n+1, p, \beta,b^*)$,\quad in which case $\tau(\gamma)=(p,\beta,b^*)$;
\item $\gamma = (n+1,\xi,\beta,b^*)$ in which case $\tau(\gamma)=(1,\xi,\rank \xi, \beta,b^*)$.
\end{enumerate}
In each case, the first co-ordinate of $\gamma$ tells us what the
{\em rank} of $\gamma$ is, that is to say to which set
$\Delta_{n+1}$ it belongs, while the remaining co-ordinates specify
the corresponding BD-functional.

We observe that BD-functionals of Type 0 do not arise in this construction. Moreover, in the definition of a Type 2 functional,
the scalar $\alpha$ that occurs
 is always 1 and $p$ is always equal to $\rank \xi$. As in the Argyros-Haydon construction, we shall make
the further restriction that the $\beta$ parameter must be of the form
$m_j^{-1}$, where the sequences $(m_j)$ and $(n_j)$ satisfy
Assumption~\ref{mnAssump}.  We shall say that the element $\gamma$
has {\em weight} $m_j^{-1}$. In the case of
a Type 2 element $\gamma=(n+1, \xi, m^{-1}_j,b^*)$ we shall insist
that $\xi$ be of the same weight, $m_j^{-1}$, as $\gamma$.

To ensure that the sets $\Delta_{n+1}$ are finite we shall admit
into $\Delta_{n+1}$ only elements of weight $m_j$ with $j\le n+1$. A
further restriction involves the recursively defined function called ``age'' (also defined in \cite{AH}). For a Type 1 element, $\gamma=(n+1, p, \beta, b^*)$, we
define $\age\gamma=1$. For a Type 2 element, $\gamma=(n+1, \xi,
m^{-1}_j,b^*)$, we define $\age \gamma= 1 + \age \xi$, and further
restrict the elements of $\Delta_{n+1}$ by insisting that the age of
an element of weight $m_j^{-1}$ may not exceed $n_j$. Finally, we
shall restrict the functionals $b^*$ that occur in an element of
$\Delta_{n+1}$ by requiring them to lie in some finite subset $B_n$
of $\ell_1(\Gamma_n)$. It is convenient to fix an increasing
sequence of natural numbers $(N_n)$ and take $B_{p,n}$ to be the set
of all linear combinations $b^*=\sum_{\eta\in \Gamma_n\setminus
\Gamma_p}a_\eta e^*_\eta$, where $\sum_\eta|a_\eta|\le 1$ and each
$a_\eta$ is a rational number with denominator dividing $N_n!$. We
may suppose the $N_n$ are chosen in such a way that $B_{p,n}$ is a
$2^{-n}$-net in the unit ball of $\ell_1(\Gamma_n\setminus
\Gamma_p)$. The above restrictions may be summarized as follows.

\begin{assump}\label{DeltaUpperAssump}
\begin{align*}
\Delta_{n+1} &\subseteq \bigcup_{j=1}^{n+1} \left\{(n+1, p, m_j^{-1},b^*): 0\leq p < n,\ b^*\in B_{p,n}\right\}\\
&\cup
\bigcup_{p=1}^{n-1}\bigcup_{j=1}^{p}\left\{(n+1,\xi,m_j^{-1},b^*):
\xi\in \Delta_p, \weight \xi = m_j^{-1},\ \age\xi<n_j,\ b^*\in
B_{p,n}\right\}
\end{align*}
\end{assump}

As in \cite{AH} we shall also assume that $\Delta_{n+1}$ contains a rich supply of
elements of ``even weight'', more exactly of weight $m_j^{-1}$ with
$j$ even.

\begin{assump}\label{DeltaLowerAssump}
\begin{align*}
\Delta_{n+1} &\supseteq \bigcup_{j=1}^{\lfloor( n+1)/2\rfloor} \left\{(n+1,p, m_{2j}^{-1},b^*): 0\leq p < n,\ b^*\in B_{p,n}\right\}\\
&\cup\bigcup_{p=1}^{n-1} \bigcup_{j=1}^{\lfloor p/2\rfloor}
\left\{(n+1,\xi,m_{2j}^{-1},b^*): \xi\in \Delta_p, \weight \xi =
m_{2j}^{-1},\ \age\xi<n_{2j},\ b^*\in B_{p,n}\right\}
\end{align*}
\end{assump}

For the main construction, there will be additional restrictions on the elements with ``odd weight'' $m_{2j-1}^{-1}$, though we will come to these later. To begin with, we shall work with a different class of Bourgain-Delbaen spaces, and therefore choose to slightly change our notation from the above, reserving the above notation for the main construction. In what follows, $\Upsilon$ will take the role of $\Gamma$ above, and $\Lambda_n$ the role of the $\Delta_n$'s. In particular, $\Upsilon = \cup_{n=1}^{\infty} \Lambda_n$ and $\Upsilon_n = \cup_{j=1}^n \Lambda_j$.

Now, we fix a $k \in \N, k \geq 2$ and define the space $X(\Upsilon)$ to be the Bourgain-Delbaen space given as in Theorem \ref{BDThm}, where $\Upsilon = \cup_{n=1}^{\infty} \Lambda_n$ and the $\Lambda_n$ are finite sets defined by recursion:

\begin{defn}
We define $\Upsilon$ by the recursion $\Lambda_1 = \{ 0, 1, \dots (k-1) \}$,
\begin{align*}
\Lambda_{n+1} &=\bigcup_{j=1}^{n+1} \left\{(n+1,p,m_j^{-1},b^*):  0\leq p < n,\ b^*\in B_{p,n}\right\}\\
&\cup \bigcup_{p=1}^{n-1}
\bigcup_{j=1}^p\left\{(n+1,\xi,m_j^{-1},b^*): \xi\in \Lambda_p,
\weight \xi = m_j^{-1},\ \age\xi<n_j,\ b^*\in B_{p,n}\right\}.
\end{align*}
\end{defn}

\begin{rem} \label{notationBpn}
\begin{enumerate}
\item Note that the cardinality of $\Lambda_1$ depends on the choice of $k \in \N$. In this way, $\Upsilon$ (and consequently the BD space $X(\Upsilon)$) really depend on the chosen $k \in \N$. In an attempt to avoid even more complicated notation, we consider $k \in \N, k \geq 2$ to be fixed for the remainder of the chapter. 
\item Later on, we will want to take a suitable subset of $\Upsilon$. To avoid any ambiguity in notation, in the above definition, and throughout the rest of the chapter, $B_{p,n}$ will denote the set of all linear combinations $b^* = \sum_{\eta \in \Upsilon_n \setminus \Upsilon_p} a_{\eta}e_{\eta}^*$, where, as before, $\sum_{\eta} |a_{\eta}| \leq 1$ and each $a_{\eta}$ is a rational number with denominator dividing $N_n!$. 
\end{enumerate}
\end{rem}

Eventually, we want to have a non-compact, bounded linear operator $S$ on our space. The ideas used to construct this operator are very similar, but more involved, than those used in Section \ref{nonsepBDalgebras}. We will make use of a single element set which is disjoint from $\Upsilon$, and label the element `undefined'. We obtain the following theorem.

\begin{thm} \label{R^*andGConstruction}
There is a map $G: \Upsilon \to \Upsilon \cup \{ \text{undefined} \}$ (we say \emph{$G(\gamma)$ is undefined} if $G(\gamma) = \text{undefined}$, otherwise we say \emph{$G(\gamma)$ is defined}) and a norm 1, linear mapping $R^* \colon \ell_1(\Upsilon) \to \ell_1(\Upsilon)$ satisying:
\begin{enumerate}
\item $G(j) = j-1$ for $1\leq j \leq k-1$ and $G(0)$ is undefined (recall $\Lambda_1 = \{ 0, 1, \dots k-1 \}$).
\item For elements $\gamma \in \Upsilon\setminus\Lambda_1$ such that $G(\gamma)$ is defined, $\rank\gamma = \rank G(\gamma)$ and  $\weight \gamma = \weight G(\gamma)$ (i.e. G preserves weight and rank). Moreover, $\age G(\gamma) \leq \age \gamma$ ($G$ doesn't increase age).
\item \begin{equation*} R^*(e_{\gamma}^*) = \begin{cases} e_{G(\gamma)}^* & \text{ if $G(\gamma)$ is defined} \\ 0 & \text{ otherwise.} \end{cases} \end{equation*}
\item \begin{equation*} R^*(d_{\gamma}^*) = \begin{cases} d_{G(\gamma)}^* & \text{ if $G(\gamma)$ is defined} \\ 0 & \text{ otherwise.} \end{cases}\end{equation*}
\end{enumerate}
\end{thm}

\begin{proof}
We will construct the maps $G$ and $R^*$ inductively. We note that since $R^*$ will be a linear operator on $\ell_{1}(\Upsilon)$, in order to ensure it is bounded, we only need to be able to control $\|R^*(e_{\gamma}^*)\|$ (for $\gamma\in\Upsilon$). More precisely, if there is some $M \geq 0$ s.t. $\|R^*(e_{\gamma}^*)\| \leq M$ for every $\gamma \in \Upsilon$, then it is elementary to check that $R^*$ is bounded with norm at most $M$. In particular, if property (3) of Theorem \ref{R^*andGConstruction} holds, it follows that $\|R^*\| = 1$.

To begin the inductive constructions of $R^*$ and $G$ we define $G \colon \Lambda_{1} \to \Lambda_{1}$ by setting $G(j) = j-1$ for $1 \leq j \leq k-1$ and declaring that $G(0)$ is undefined. Noting that $e_{\gamma}^* = d_{\gamma}^*$ for $\gamma \in \Lambda_1$, we define $R^*(e_{0}^*) = R^*(d_{0}^*) = 0$ and $R^*(e_{j}^*) = R^*(d_{j}^*) = e_{j-1}^* = d_{j-1}^*$ for $j \geq 1$. We observe that this definition is consistent with the properties (1) - (4) above.

Suppose that we have defined $G: \Upsilon_n \to \Upsilon_n$ and $R^* \colon \ell_1(\Upsilon_n) \to \ell_1(\Upsilon_n)$ satisfying properties (1) - (4) above. We must extend $G$ to $\Upsilon_{n+1}$ and $R^*$ to a map on $\ell_1(\Upsilon_{n+1})$. We consider a $\gamma \in \Lambda_{n+1}$ and recall that (see Theorem \ref{BDThm}) $e_{\gamma}^* = c_{\gamma}^* + d_{\gamma}^*$. We wish to define $R^* d_{\gamma}^*$ and $R^*e_{\gamma}^*$. By linear independence, we are free to define $R^*d_{\gamma}^*$ however we like. However, since $R^*c_{\gamma}^*$ is already defined ($c_{\gamma}^* \in \ell_{1}(\Upsilon_n))$, and we want $R^*$ to be linear, once we have defined $R^*d_{\gamma}^*$, in order to have linearity we must have $R^*e_{\gamma}^* = R^*c_{\gamma}^* + R^*d_{\gamma}^*$.

Let us consider $R^*c_{\gamma}^*$. We suppose first that $\age\gamma = 1$ so that we can write $\gamma = (n+1, p, \beta, b^*)$ where $b^* \in \ell_{1}(\Upsilon_n \setminus \Upsilon_p)$. Consequently \[
c_{\gamma}^* = \beta\left( I - P^*_{(0,p]}b^*\right) = \beta P^*_{(p,\infty)}b^*.
\]
We claim that $R^*P^*_{(p, \infty)}b^* = P^*_{(p,\infty)}R^*b^*$. Indeed, we can write $b^* = \sum_{\delta \, \in \, \Upsilon_n} \alpha_{\delta}d_{\delta}^*$ (for a unique choice of $\alpha_{\delta}$). It follows from property (4) and the inductive hypothesis that \[
R^*P^*_{(p, \infty)}b^* = R^*\big(\sum_{\delta \in \Upsilon_n\setminus\Upsilon_p} \alpha_{\delta}d_{\delta}^*\big	) = \sum_{\substack{ \delta \in \Upsilon_n \setminus \Upsilon_p \cap \\ \{\eta \, \in \, \Upsilon_n \, : \, G(\eta) \text{ is defined} \}} } \alpha_{\delta}d^*_{G(\delta)}
\]
and it is easily checked (by a similar calculation) that we obtain the same expression for $P^*_{(p,\infty)}R^*b^*$. It follows that \[
R^*c_{\gamma}^* = \beta P^*_{(p,\infty)} R^*b^*.
\]
We define $G(\gamma)$ by \[
G(\gamma) = \begin{cases} \text{undefined} & \text{ if $P^*_{(p,\infty)} R^*b^* = 0$} \\ (n+1, p, \beta, R^*b^*) & \text{ otherwise} \end{cases}
\]
where it is a simple consequence of the facts that $R^*\colon \ell_1(\Upsilon_n) \to \ell_1(\Upsilon_n)$ has norm 1 and satisfies property (3) that the element $(n+1,p,\beta, R^*b^*) \in \Upsilon$. In the case where $P^*_{(p,\infty)} R^*b^* \neq 0$, $G(\gamma)$ is defined (with the definition as above) and it is evident that $R^*c_{\gamma}^* = c_{G(\gamma)}^*$. We can define $R^*d_{\gamma}^* = d_{G(\gamma)}^*$ and it follows by linearity (and the fact that $e_{\gamma}^* = c_{\gamma}^* + d_{\gamma}^*$) that we have $R^*e_{\gamma}^* = e_{G(\gamma)}^*$ as required. Otherwise, when $P^*_{(p,\infty)} R^*b^* = 0$, $R^*c_{\gamma}^* = 0$, so we can set $R^*d_{\gamma}^* = 0$ and again by linearity we get that $R^*e_{\gamma}^* = 0$.

Now if $\gamma$ has age $>1$, we can write $\gamma = (n+1, \xi, \beta, b^*)$ and $c_{\gamma}^* = e_{\xi}^* + \beta P^*_{(\rank\xi,\infty)}b^*$. Making use of the inductive hypothesis yet again and the preceding argument, we have that 
\begin{align*}
R^*c_{\gamma}^* &= R^*(e_{\xi}^*) + \beta P^*_{(\rank\xi,\infty)}R^*b^* \\
&= \begin{cases} e_{G(\xi)}^* + \beta P^*_{(\rank\xi,\infty)}R^*b^* & \text{ if $G(\xi)$ is defined} \\ \beta P^*_{(\rank\xi,\infty)}R^*b^* & \text{ otherwise} \end{cases}
\end{align*}
It follows that if $G(\xi)$ is undefined and $P^*_{(\rank\xi,\infty)}R^*b^* = 0$ then $R^*c_{\gamma}^* = 0$. In this case we declare $G(\gamma)$ to be undefined. Otherwise, there are two remaining possiblities
\begin{enumerate}[(i)]
\item $G(\xi)$ is undefined but $P^*_{(\rank\xi,\infty)}R^*b^* \neq 0$. In this case, it is easily verified that the element $G(\gamma) := (n+1, \rank\xi, \beta, R^*b^*) \in \Upsilon$. 
\item $G(\xi)$ is defined. It is again easily checked that the element $G(\gamma):= (n+1, G(\xi), \beta, R^*b^*) \in \Upsilon$ (here we note that in addition to the above arguments, we also need the inductive hypothesis that $G$ does not increase the age of an element).
\end{enumerate}
In either of these cases, we see that $R^*c_{\gamma}^* = c_{G(\gamma)}^*$. We can define $R^*d_{\gamma}^*$ to be $d_{G(\gamma)}^*$ and as before, we then necessarily have $R^*e_{\gamma}^* = e_{G(\gamma)}^*$ (in order that $R^*$ be linear).

We have thus succeeded in extending the maps $G$ and $R^*$. By induction, we therefore obtain maps $G: \Upsilon \to \Upsilon \cup \{ \text{undefined} \}$ and $R^* \colon (c_{00}(\Upsilon), \| \cdot \|_{1}) \to (c_{00}(\Upsilon), \| \cdot \|_{1})$ satisfying the four properties above (here $(c_{00}(\Upsilon), \| \cdot \|_{1})$ is the dense subspace of $\ell_1(\Upsilon)$ consisting of all finitely supported vectors). It follows from property (3), and the argument above, that $R^*$ is continuous, with $\|R^*\|=1$. It therefore extends (uniquely) to a bounded linear map $R^*\colon\ell_1(\Upsilon) \to \ell_1(\Upsilon)$ also having norm $1$. This completes the proof.
\end{proof}
We make some important observations about the mappings $G$ and $R^*$. 
\begin{lem}\label{DualOfR^*RestrictsProperly}
The dual operator of $R^*$, which we denote by $(R^*)' \colon \ell_1(\Upsilon)^* \to \ell_{1}(\Upsilon)^*$ restricts to give a bounded linear operator $R:= (R^*)'|_{X(\Upsilon)} \colon X(\Upsilon) \to X(\Upsilon)$ of norm at most $1$.
\end{lem}
\begin{proof}
It is a standard result that the dual operator $(R^*)'$ is bounded with the same norm as $R^*$. It follows that the restriction of the domain to $X(\Upsilon)$ is a bounded, linear operator into $\ell_1(\Upsilon)^*$ with norm at most $1$. It only remains to see that this restricted mapping actually maps into $X(\Upsilon)$. Since the family $(d_{\gamma})_{\gamma\in\Upsilon}$ is a basis for $X(\Upsilon)$, it is enough to see that the image of $d_{\gamma}$ under $(R^*)'$ lies in $X(\Upsilon)$. It is therefore sufficient for us to prove that \[
(R^*)'d_{\delta} = \sum_{\gamma \in G^{-1}(\{ \delta \})} d_{\gamma} \]
where of course $G^{-1}(\{ \delta \} )$ denotes the pre-image of $\{ \delta \}$ under $G$; we remark that $G$ fails to be injective and so is certainly not invertible. Since $(d_{\gamma}^*)_{\gamma\in\Upsilon}$ is a basis for $\ell_{1}(\Upsilon)$ it is enough to show that for every $\theta \in \Upsilon$ \[
\big((R^*)'d_{\delta}\big)d_{\theta}^* =  \big( \sum_{\gamma \in G^{-1}(\{ \delta \} )} d_{\gamma} \big) d_{\theta}^*. \]

The right hand side of this expression is easy to evaluate; it is only non-zero if $G(\theta) = \delta$, in which case it is equal to $1$. In particular, if $G(\theta)$ is undefined, then the right hand side of the expression is certainly 0, as is $\big((R^*)'d_{\delta}\big)d_{\theta}^* = d_{\delta}(R^*d_{\theta}^*)  = d_{\delta}(0)$. If $G(\theta)$ is defined, the left hand side of the expression is $d_{\delta}(d_{G(\theta)}^*)$ which is clearly $1$ if $G(\theta) = \delta$ and $0$ otherwise. So the expressions are indeed equal, as required.
\end{proof}
\begin{lem}\label{G^kalwaysundefined}
For every $\gamma \in \Upsilon$, there is a unique $l, 1\leq l \leq k$ such that $G^j(\gamma)$ is defined whenever $1\leq j < l$ but $G^l(\gamma)$ is undefined.
\end{lem}
\begin{proof}
The uniqueness is easy; if $G(\gamma)$ is defined, $l$ is the maximal $j \in \N$ such that $G^{j-1}(\gamma)$ is defined. Otherwise we must have $l=1$. So we only have to prove existence of such $l$.

We prove by induction on $n$ that if $\rank\gamma = n$ there is some $1\leq l \leq k$ such that $G^{l}(\gamma)$ is undefined, but $G^j(\gamma)$ is defined if $j < l$. The case where $n=1$ is clear from the construction of the map $G$. So, inductively, we assume the statement holds whenever $\rank \gamma = m \leq n$. Let $\gamma \in \Lambda_{n+1}$ and consider 2 cases:
\begin{enumerate}[(i)]
\item $\age\gamma = 1$. We write $\gamma = (n+1,p, \beta, b^*)$. Now $b^* \in \ell_1(\Upsilon_n)$ and by the inductive hypothesis, for every $\theta \in \Upsilon$ with $\rank\theta \leq n$, there is some $l \leq k$ such that $G^l(\theta)$ is undefined. It follows that we must have $(R^*)^lb^* = 0$ for some $1\leq l \leq k$. So it is certainly true that $P^*_{(p,\infty)}(R^*)^lb^* = 0$ for some $1 \leq l \leq k$. It follows by construction of the map $G$ that the `$l$' we seek is the minimal $l$ ($1\leq l \leq k$) such that $P^*_{(p,\infty)}(R^*)^lb^* = 0$.
\item $\age\gamma > 1$. We write $\gamma = (n+1, \xi, \beta, b^*)$. If $G(\gamma)$ is undefined we are done; we must have $l=1$. Otherwise, $G(\gamma) \in \Upsilon$ and it follows from the construction of the map $G$ (see the proof of Theorem \ref{R^*andGConstruction}) that there are two possibilities to consider; either $G(\gamma) = (n+1, \rank\xi, \beta, R^*b^*)$ or $G(\gamma) = (n+1, G(\xi), \beta, R^*b^*)$. In the first of these two possibilities, the same argument as in the previous case shows that the $l$ we seek is the minimal $l$ ($2 \leq l \leq k$) such that $P^*_{(p,\infty)}(R^*)^lb^* = 0$. In the latter case, $G(\xi)$ is defined. But, since $\rank \xi < n$, we know by the inductive hypothesis that there exists some $l_0 \in \N$ with  $2\leq l_0 \leq k$ and $G^{l_0}(\xi)$ undefined but $G^j(\xi)$ defined for $j<l_0$. Now, if $P^*_{(\rank\xi, \infty)} (R^*)^{l_0}b^* = 0$, then it follows from construction of $G$ that $G^{l_0}(\gamma)$ is undefined and $G^j(\gamma)$ is defined for $j < l_0$ so we are done. Otherwise, it follows from an argument above that $l_0 < k$, and there is some (minimal) $l$, $l_0 < l \leq k$ with $P^*_{(\rank\xi, \infty)} (R^*)^l b^* = 0$. Once again, this is the desired $l$.
\end{enumerate}
\end{proof}
\begin{cor}\label{R^k=0}
The maps $R^* \colon \ell_1(\Upsilon) \to \ell_1(\Upsilon)$ and $R \colon X(\Upsilon) \to X(\Upsilon)$ satisfy $(R^*)^k = 0$ and $R^k = 0$.
\end{cor}
\begin{proof}
It is clear from Lemma \ref{G^kalwaysundefined} and construction that the restriction of $(R^*)^k$ to $c_{00}(\Upsilon)$ is the zero map. It follows by density and continuity that $(R^*)^k=0$. The other claims are immediate from the definition of $R$ as the restriction of the dual operator of $R^*$.
\end{proof} 

To obtain the extra constraints that we place on  ``odd-weight'' elements we need to introduce a `coding function' $\sigma \colon \Upsilon \to \N$. This is similar to the coding function used in the Argyros-Hadyon construction, \cite{AH}, and analogous to the coding function used in the Gowers-Maurey construction, \cite{GowMau}. We shall demand that $\sigma$ satisfies the following properties:
\begin{enumerate}[(1)] \label{SigmaAssumptions}
\item $\sigma$ is injective
\item $\sigma(\gamma) > \rank\gamma \quad \forall \, \gamma \in \Upsilon$
\item for $\gamma \in \Lambda_{n+1}$ (i.e. $\rank\gamma = n+1$), $\sigma(\gamma) > \max \{ \sigma(\xi) : \xi \in \Upsilon_n \} $.
\end{enumerate}
Such a $\sigma$ can be constructed recursively as $\Upsilon$ is constructed. Now, for each $\gamma \in \Upsilon$ we can well-define a finite set $\Sigma(\gamma)$ by \[
\Sigma(\gamma) := \{ \sigma (\gamma) \} \cup \bigcup_{j=1}^{k-1} \bigcup_{\delta\in G^{-j} (\gamma) } \{ \sigma(\delta) \}
\]
where $G^{-j}(\gamma) := \{ \theta \in \Upsilon : G^{j}(\theta) = \gamma \}$. (In particular, if $\theta \in G^{-j}(\gamma)$ then $G^l(\theta) \in \Upsilon$ for every $l \leq j$.)

Before giving our main construction, we document some important observations.

\begin{lem}\label{SigmaGammaContainedinSigmaGGamma}
If $\gamma \in \Upsilon$ is such that $G(\gamma)$ is defined then $\Sigma (\gamma) \subseteq \Sigma (G(\gamma))$.
\end{lem}
\begin{proof}
Certainly $\gamma \in G^{-1} (G(\gamma))$ so $\sigma(\gamma) \in \Sigma (G(\gamma))$. Suppose $\sigma(\delta) \in \Sigma(\gamma), \delta \neq \gamma$. So, there is some $ 1 \leq j \leq k-1$ with $\delta \in G^{-j}(\gamma)$, i.e. there is some $\delta$ such that $G^{j}(\delta) = \gamma$. Since $G(\gamma)$ is defined, we must have $G^{j+1}(\delta) = G(\gamma) \in \Upsilon$. In particular, by Lemma \ref{G^kalwaysundefined}, we must in fact have had $j < k-1$ so that $j+1 \leq k-1$. Thus $\delta \in G^{-(j+1)} (G(\gamma))$ and $\sigma(\delta) \in \Sigma (G(\gamma))$, as required.
\end{proof}

\begin{lem} \label{MonotonicityOfSigmaSets}
If $\gamma, \gamma' \in \Upsilon$, $\rank \gamma > \rank\gamma'$ then $\Sigma(\gamma) > \Sigma(\gamma')$, i.e. $\max \{ k : k \in \Sigma(\gamma') \} < \min \{ k : k \in \Sigma(\gamma) \}$.
\end{lem}
\begin{proof}
This follows immediately from the definition of the sets $\Sigma (\gamma)$ and $\Sigma(\gamma')$, the fact that $G$ is rank preserving and the assumption that for $\gamma \in \Lambda_{n+1}$, $\sigma(\gamma) > \max \{ \sigma(\xi) : \xi \in \Upsilon_n \}$.
\end{proof}

\begin{lem} \label{IntersectionPropertyofSigmaSets}
Suppose $\gamma, \delta \in \Upsilon$. If $\sigma(\gamma) \in \Sigma(\delta)$ then either $\gamma = \delta$ or there is some $1 \leq j \leq k-1$ such that $G^j(\gamma) = \delta$.
\end{lem}
\begin{proof}
Since $\sigma(\gamma) \in \Sigma(\delta)$ there are two possibilities. Either $\sigma(\gamma) = \sigma(\delta)$ or there is some $1\leq j \leq k-1$ and $\theta \in \Upsilon$ with $G^j(\theta) = \delta$ and $\sigma(\gamma) = \sigma(\theta)$. By injectivity of $\sigma$, this implies that either $\gamma = \delta$, or that $\theta = \gamma$ and $G^j(\gamma) = \delta$ as required.
\end{proof}

We are finally in a position to describe the main construction. We will take a subset $\Gamma \subset \Upsilon$ by placing some restrictions on the elements of odd weight we permit. As a consequence of imposing these additional odd weight restrictions, we are also forced to (roughly speaking) remove those elements $(n+1, p, \beta, b^*)$ and $(n+1, \xi, \beta, b^*)$ of $\Upsilon$ for which the support of $b^*$ is not contained in $\Gamma$, in order that we can apply the Bourgain-Delbaen construction to obtain a space $X(\Gamma)$. Note that the subset $\Gamma$ will also be constructed inductively, so it will be well defined and consistent with the Bourgain-Delbaen method previously described.  

We will denote by $\Delta_n$ the set of all elements in $\Gamma$ having rank $n$, and denote by $\Gamma_n$ the union $\Gamma_n = \cup_{j\leq n} \Delta_j$. The permissible elements of odd weight will be as follows. For an age 1 element of odd weight, $\gamma = (n+1, p, m_{2j - 1}^{-1}, b^*)$, we insist that either $b^* = 0$ or $b^* = e_{\eta}^*$ where $\eta \in \Gamma_n \setminus \Gamma_p$ and $\weight\eta = m_{4i}^{-1} < n_{2j-1}^{-2}$. 
For an odd weight element of age $> 1$, $\gamma = (n+1, \xi, m_{2j - 1}^{-1}, b^*)$, we insist that either $b^* = 0$ or $b^* = e_{\eta}^*$ where $\eta \in \Gamma_n \setminus \Gamma_{\rank\xi}$ and $\weight\eta = m_{4k}^{-1} < n_{2j-1}^{-2}$, $k \in \Sigma(\xi)$. Let us be more precise:

\begin{defn} \label{DefnOfGammaAndSpace}

We define recursively sets $\Delta_{n} \subseteq \Lambda_n$. Then $\Gamma_n := \cup_{j \leq n} \Delta_j$ and $\Gamma := \cup_{n\in\N}\Delta_n \subseteq \Upsilon$. To begin the recursion, we set $\Delta_1 = \Lambda_1$. Then

\begin{align*}
\Delta_{n+1} &= \bigcup_{j=1}^{\lfloor(n+1)/2\rfloor}
\left\{(n+1,p, m_{2j}^{-1},b^*): 0\leq p<n,\ b^*\in B_{p,n}\cap \ell_1(\Gamma_n) \right\}\\ &\cup
\bigcup_{p=1}^{n-1}\bigcup_{j=1}^{\lfloor
p/2\rfloor}\left\{(n+1,\xi,m_{2j}^{-1},b^*): \xi\in \Delta_p,
\text{w}( \xi) = m_{2j}^{-1},\ \age\xi<n_{2j},\
b^*\in B_{p,n}\cap \ell_1(\Gamma_n) \right\}\\
&\cup\bigcup_{j=1}^{\lfloor(n+2)/2\rfloor}
\left\{(n+1,p, m_{2j-1}^{-1},b^*):
0\leq p < n,\  b^*=0 \text{ or } b^*=e_{\eta}^* ,\ \eta\in \Gamma_n\setminus\Gamma_p \right. \\
& \left. \hspace{8cm}\text{ and } \text{w}(\eta)= m_{4i}^{-1}<n_{2j-1}^{-2}\right\}\\
&\cup  \bigcup_{p=1}^{n-1}\bigcup_{j=1}^{\lfloor(
p+1)/2\rfloor}\left\{(n+1,\xi,m_{2j-1}^{-1},b^*): \xi\in
\Delta_p, \text{w}( \xi) = m_{2j-1}^{-1},\ \age\xi<n_{2j-1} \right.,\\  &\left.
\qquad\qquad\qquad \, \, \, b^*=0 \text{ or } b^*=e_{\eta}^* \text{ with } \eta\in\Gamma_n\setminus\Gamma_p,\ \text{w}(\eta)=
m_{4k}^{-1}<n_{2j-1}^{-2},\ k \in \Sigma(\xi) \right\}.
\end{align*}
Here the $B_{p,n}$ are defined as in Remark \ref{notationBpn}, and, for brevity, we temporarily write $\text{w} (\xi)$ for $\weight\xi$. We define $\X_k$ to be the Bourgain-Delbaen space $X(\Gamma)$ where $\Gamma$ is the subset of $\Upsilon$ just defined.
\end{defn}

For the rest of this chapter we will work with the space $\X_k$. As in \cite{AH}, the structure of the space $\X_k$ is most easily
understood in terms of the basis $(d_\gamma)_{\gamma\in\Gamma}$ and the biorthogonal
functionals $d_\gamma^*$. However, we will need to work with the evaluation functionals
$e^*_\gamma$ in order to estimate norms.  To this end, we have the following proposition.

\begin{prop} \label{EvalAnal}Let $n$ be a positive integer and
let $\gamma$ be an element of $\Delta_{n+1}$ of weight $m_j^{-1}$
and age $a \le n_j$. Then there exist natural numbers
$p_0<p_1<\cdots<p_a=n+1$, elements $\xi_1,\dots,\xi_a=\gamma$ of
weight $m_j^{-1}$ with $\xi_r\in \Delta_{p_r}$ and functionals
$b^*_r\in \ball\ell_1\left(\Gamma_{p_r-1}\setminus
\Gamma_{p_{r-1}}\right)$ such that
 \begin{align*} e^*_\gamma &= \sum_{r=1}^a d^*_{\xi_r} +
m_j^{-1}\sum_{r=1}^a P^*_{(p_{r-1},\infty)}b^*_r\\
&=\sum_{r=1}^a d^*_{\xi_r} + m_j^{-1}\sum_{r=1}^a
P^*_{(p_{r-1},p_r)}b^*_r.
\end{align*}
If $1\le t<a$ we have
$$
e_\gamma^*= e^*_{\xi_t} + \sum_{r=t+1}^ad^*_{\xi_r} + m_j^{-1}
\sum_{r=t+1}^a P^*_{(p_{r-1},\infty)}b^*_r\,.
$$
\end{prop}
\begin{proof}
The proof is an easy induction on the age $a$ of $\gamma$. We omit the details because the argument is the same as in \cite{AH} except our $p_0$ is not necessarily 0.
\end{proof}
\begin{rem}
As in \cite{AH}, we shall refer to any of the above identities as the evaluation analysis of the element $\gamma$, and the data $(p_0, (p_r, b_r^*, \xi_r)_{1\leq r \leq a} )$ as the analysis of $\gamma$.  We will omit the $p_0$ when $p_0 = 0$.
\end{rem}

We will now construct the operator $S: \X_k \to \X_k$. We need:
\begin{prop} \label{S^*construction}
$\Gamma$ is invariant under $G$. More precisely, if $\, \gamma \in \Gamma \subseteq \Upsilon$ and $G(\gamma)$ is defined, then $G(\gamma) \in \Gamma$. It follows that the map $G\colon \Upsilon \to \Upsilon\cup\{\text{undefined}\}$ defined in Theorem \ref{R^*andGConstruction} restricts to give $F \colon \Gamma \to \Gamma\cup\{\text{undefined}\}$. Consequently the map $R^* \colon \ell_1(\Upsilon) \to \ell_1(\Upsilon)$ (also defined in \ref{R^*andGConstruction}) can be restricted to the subspace $\ell_1(\Gamma) \subseteq \ell_1(\Upsilon)$ giving $S^*\colon\ell_1(\Gamma) \to \ell_1(\Gamma)$. $S^*$ is a bounded linear map on $\ell_1(\Gamma)$ of norm 1 which satisfies 
\[
S^*e_{\gamma}^* = \begin{cases} 0 & \text{ if $F(\gamma)$ is undefined} \\ e_{F(\gamma)}^* & \text{ otherwise} \end{cases}
\]
for every $\gamma \in \Gamma$. Moreover, the dual operator of $S^*$ restricts to $\X_k$ to give a bounded linear operator $S \colon \X_k \to \X_k$ of norm at most 1, satisfying $S^j \neq 0$ for $1\leq j \leq k-1$, $S^k = 0$.
\end{prop}
\begin{proof}
It is enough to show that when $\gamma \in \Gamma$ and $G(\gamma)$ is defined, then $G(\gamma) \in \Gamma$. The claims about the operator $S^*$ follow immediately from the definition of $S^*$ as the restriction of $R^*$ and the definition of $R^*$. The fact that $S\colon \X_k \to \X_k$ is well defined follows by the same argument as in \ref{DualOfR^*RestrictsProperly}; indeed it is seen that for $\delta \in \Gamma$, \[
Sd_{\delta} = \sum_{\gamma \in F^{-1}(\delta)}d_{\gamma}.
\]
Moreover, by Corollary \ref{R^k=0}, we see that $(S^*)^k= 0$ and therefore that $S^k=0$. That $S^j \neq 0$ for $1\leq j \leq k-1$ is clear from the above formula and consideration of the elements $d_{m}$ for $m \in \Delta_1$.

We use induction on the rank of $\gamma$ to prove that if $\gamma \in \Gamma$ and $G(\gamma)$ is defined, then $G(\gamma) \in \Gamma$. This is certainly true when $\rank\gamma = 1$. Suppose by induction, that whenever $\gamma \in \Gamma$, $\rank\gamma \leq n$ and $G(\gamma)$ is defined, $G(\gamma) \in \Gamma$ and consider a $\gamma \in \Gamma$, of rank $n+1$ such that $G(\gamma)$ is defined. Let us suppose first that this $\gamma$ has age 1. We can write $\gamma = (n+1, p, \beta, b^*)$ where $\supp b^* \subseteq \Gamma_n \setminus \Gamma_p$, and we write $b^* = \sum_{\eta \in \Gamma_n \setminus \Gamma_p} a_{\eta}e_{\eta}^*$. Since $G(\gamma)$ is defined, we have $G(\gamma) =  (n+1, p, \beta, S^*b^*)$. We consider further sub-cases. If $\beta = m_{2j}^{-1}$ for some $j$ (i.e. $\gamma$ is an ``even weight'' element), we see that the only way $G(\gamma)$ fails to be in $\Gamma$ is if $\supp S^*b^* \nsubseteq \Gamma_n \setminus \Gamma_p$. But \[
S^*b^* = \sum_{\substack { \eta \in \Gamma_n\setminus\Gamma_p \, \cap \\ \{\eta  \, : \, F(\eta) \text{ is defined} \} } } a_{\eta}e_{F(\eta)}^*
\]
and by the inductive hypothesis all the $F(\eta)$ in this sum are in $\Gamma_n \setminus \Gamma_p$. So $\supp S^*b^* \subseteq \Gamma_n \setminus \Gamma_p$ and $G(\gamma) = F(\gamma) \in \Gamma$ as required. In the case where $\beta = m_{2j-1}^{-1}$ (i.e. $\gamma$ is an ``odd-weight'' element) we must also check that the odd weight element $G(\gamma)$ is of a permissible form. Since $G(\gamma)$ is defined, in particular we must have that $P^*_{(p,\infty)} S^*b^* \neq 0$. This of course implies that $S^*b^* \neq 0$ and $b^* \neq 0$. As $\gamma \in \Gamma$, $b^* = e_{\eta}^*$ for some $\eta \in \Gamma_n \setminus \Gamma_p$ where $\weight\eta = m_{4i}^{-1} < n_{2j-1}^{-2}$. Since $S^*b^* \neq 0$, we must have $S^*e_{\eta}^* = e_{F(\eta)}^*$ where in particular, $F(\eta)$ is defined and lies in $\Gamma$ by the inductive hypothesis. Since $G$ is weight preserving, we have $\weight F(\eta) = \weight\eta = m_{4i}^{-1} < n_{2j-1}^{-2}$. Moreover, since $G$ preserves rank, we can now conclude that $F(\gamma) = G(\gamma) = (n+1, p, m_{2j-1}^{-1}, S^*b^*) = (n+1, p, m_{2j-1}^{-1}, e_{F(\eta)}^*) \in \Gamma$ as required.

When $\age\gamma > 1$, we can write $\gamma = (n+1, \xi,  \beta, b^*)$ where (in particular) $\supp b^* \subseteq \Gamma_n$ and $\xi \in \Gamma_p$ ($p<n)$. If this $\gamma$ is of even weight, it follows easily from the inductive hyptohesis and arguments similar to the `age 1' case that $G(\gamma) = F(\gamma) \in \Gamma$ (when $G(\gamma)$ is defined). So we consider only the case when $\weight \gamma = \beta = m_{2j-1}^{-1}$ (for some $j$) and $G(\gamma)$ is defined. If $G(\xi) = F(\xi)$ is undefined, then since $G(\gamma)$ is defined, we must have $G(\gamma) = (n+1, \rank\xi, m_{2j-1}^{-1}, S^*b^*)$ and $P^*_{(\rank\xi, \infty)} S^*b^* \neq 0$. So in particular, $b^* \neq 0$ and $S^*b^* \neq 0$. Again by the restrictions on elements of odd weight, we must have $b^* = e_{\eta}^*$ where $\weight\eta = m_{4k}^{-1} < n_{2j-1}^{-2}$, some $k \in \Sigma(\xi)$. Since $S^*b^* \neq 0$, we must have $G(\eta) = F(\eta)$ defined and $S^*b^* = S^*e_{\eta}^* = e_{F(\eta)}^*$. Since $G$ preserves rank, we see as a consequence of the inductive hypothesis, that $F(\eta) \in \Gamma_n \setminus \Gamma_{\rank\xi}$. Furthermore, $\weight F(\eta) = \weight \eta = m_{4k}^{-1} < n_{2j-1}^{-2}$. We conclude that $G(\gamma) = F(\gamma) \in \Gamma$ as required.

In the case where $F(\xi)$ is defined, $G(\gamma) = (n+1, F(\xi), m_{2j-1}^{-1}, S^*b^*)$. If $S^*b^* = 0$ then by the inductive hypothesis, we certainly have $G(\gamma) = F(\gamma) \in \Gamma$ and we are done. Otherwise, we again must have $b^* = e_{\eta}^*$ where $\weight\eta = m_{4k}^{-1} < n_{2j-1}^{-2}$, some $k \in \Sigma(\xi)$ and $S^*b^* = e_{F(\eta)}^*$. Now, $\weight F(\eta) = \weight\eta = m_{4k}^{-1} < n_{2j-1}^{-2}$ and $k \in \Sigma(\xi) \subseteq \Sigma (G(\xi))$ by Lemma \ref{SigmaGammaContainedinSigmaGGamma}. 
\end{proof}

Later, we will need the following lemma about the elements of odd weight in $\Gamma$. 
\begin{lem} \label{monotonicity_of_odd_weights}
Let $\gamma \in \Gamma$ be an element of odd weight and $\age\gamma > 1$. Let $\big( p_0, (p_i , \xi_i, b_i^* ) \big)$ be the analysis of $\gamma$, where we know each $b_i^*$ is either $0$ or $e_{\eta_i}^*$ for some suitable $\eta_i$. If there are $i, j$, $1 \leq i < j \leq \age\gamma := a$ with $b_i^* = e_{\eta_i}^*$ and $b_j^* = e_{\eta_j}^*$ then $\weight\eta_j \lneq \weight\eta_i$.
\end{lem}
\begin{proof}
We consider the cases when $i = 1$ and $i > 1$ separately. Suppose $b_1^* = e_{\eta_1}^*$ where $\weight\eta_1 = m_{4i}^{-1}$. By construction, elements of rank $p$ are only allowed to have weights $m_j^{-1}$ where $1 \leq j \leq p$. So $\weight\eta_1 = m_{4i}^{-1} \implies 4i \leq \rank\eta_1 < \rank\xi_1 = p_1$. By the strict monotonicity of the sequence $m_j$, we get that $m_{\rank\xi_1}^{-1} < m_{4i}^{-1} = \weight\eta_1$.

Now for any $j > 1$, if $b_j^* \neq 0$, $b_j^* = e_{\eta_j}^*$ where $\weight\eta_j = m_{4k}^{-1}$ for some $k \in \Sigma (\xi_{j-1})$. Since the mapping $F$ preserves rank of elements, and $\sigma(\theta) > \rank\theta$ $\forall \theta \in \Gamma$ (Assumption (2) of the $\sigma$ mapping), it is immediate from the definition of $\Sigma(\xi_{j-1})$ that $k \in \Sigma(\xi_{j-1}) \implies k > \rank\xi_{j-1} \geq \rank\xi_1$. Now \[
\weight\eta_j = m_{4k}^{-1} < m_{4\rank\xi_{j-1}}^{-1} \leq m_{4\rank\xi_1}^{-1} \leq \weight\eta_1
\]
thus concluding the proof for the case $i=1$. The case when $i > 1$ is easy. In this case we suppose $b_i^* = e_{\eta_i}^* $, $\weight\eta_i = m_{4l}^{-1}$ for some $l \in \Sigma(\xi_{i-1})$. Now $j > i$, so (by Lemma \ref{MonotonicityOfSigmaSets}) $\Sigma (\xi_{j-1}) > \Sigma (\xi_{i-1})$ so that $l < k$ and therefore $\weight\eta_j = m_{4k}^{-1} < m_{4l}^{-1} = \weight\eta_i$, completing the proof.
\end{proof}

\section{Rapidly Increasing Sequences and the operator $S\colon\X_k \to\X_k$}
We recall from \cite{AH} that special classes of block sequences, namely the \emph{rapidly increasing sequences} admit good upper estimates. This class of block sequences will also be useful in our construction. We recall the definition:
\begin{defn}\label{RISDef}
Let $I$ be an interval in $\N$ and let $(x_k)_{k\in I}$ be a
block sequence (with respect to the FDD $(M_n)$). We say that
$(x_k)$ is a {\em rapidly increasing sequence}, or RIS, if there
exists a constant $C$ such that the following hold:
\begin{enumerate}
\item
$\|x_k\|\le C$ for all $k\in I$,
\end{enumerate} and there is an increasing sequence $(j_k)$ such
that, for all $k$,
\begin{enumerate}\setcounter{enumi}{1}
\item $j_{k+1} > \max\, \text{ran }\,x_k$\,,
\item $|x_k(\gamma)| \le Cm_i^{-1}$ whenever $\weight\gamma = m_i^{-1}$ and
$i<j_k$\,.
\end{enumerate}
If we need to be specific about the constant, we shall refer to a
sequence satisfying the above conditions as a $C$-RIS.
\end{defn}

\begin{rem} \label{SofRISisRIS}
We make the following important observation. If $(x_i)_{i\in\N}$ is a $C$-RIS, then so also is the sequence $(Sx_i)$. We omit the very easy proof.
\end{rem}

We also note that the estimates of Lemmas and Propositions 5.2 - 5.6 and 5.8 of \cite{AH} all still hold. The same proofs go through, with only minor modifications to take account of the fact that $p_0$ doesn't need to be $0$ in the evaluation analysis of an element $\gamma$ in our $\Gamma$ (see Proposition \ref{EvalAnal}). For convenience, we state Proposition 5.6 of \cite{AH} as we shall be making repeated use of it throughout this chapter:- 

\begin{prop}\label{AHProp5.6}
Let $(x_k)_{k=1}^{n_{j_0}}$ be a $C$-RIS.  Then
\begin{enumerate}
\item
For every $\gamma\in \Gamma$ with $\weight \gamma=m_h^{-1}$ we have
$$
|n_{j_0}^{-1}\sum_{k=1}^{n_{j_0}} x_k(\gamma)|\le \begin{cases}
16Cm_{j_0}^{-1}m_h^{-1} &\text{ if }h<j_0\\
4Cn_{j_0}^{-1} + 6Cm_h^{-1} &\text{ if } h\ge j_0 \end{cases}
$$
In particular,
$$
|n_{j_0}^{-1}\sum_{k=1}^{n_{j_0}} x_k(\gamma)|\le 10Cm_{j_0}^{-2},
$$
if $h>j_0$ and
$$
 \|n_{j_0}^{-1}\sum_{k=1}^{n_{j_0}}x_k\|\le 10Cm_{j_0}^{-1}.
$$
\item
If $\lambda_k$ ($1\le k\le n_{j_0}$) are scalars with
$|\lambda_k|\le 1$ and having the  property that
$$
|\sum_{k\in J} \lambda_kx_k(\gamma)|\le C\max_{k\in J}|\lambda_k|,
$$
for every $\gamma$ of weight $m_{j_0}^{-1}$ and every interval
$J\subseteq\{1,2,\dots,n_{j_0}\}$, then
$$
 \|n_{j_0}^{-1}\sum_{k=1}^{n_{j_0}}\lambda_kx_k\|\le 10Cm_{j_0}^{-2}.
$$
\end{enumerate}
\end{prop}

Another result of particular importance to us will be the following proposition of \cite{AH}:

\begin{prop}\label{RIStoBlock}
Let $Y$ be any Banach space and $T:\X_k \to Y$ be a bounded
linear operator.  If $\|T(x_k)\|\to0$ for every RIS $(x_k)_{k\in
\N}$ in $\X_k$  then $\|T(x_k)\|\to 0$ for every bounded
block sequence in $\X_k$.
\end{prop}

The proof exploits the explicit finite dimensional subspace structure provided by the $\mathcal{L}_{\infty}$ spaces of Bourgain-Delbaen type; exactly the same argument used by Argyros and Haydon in \cite{AH} works for the spaces $\X_k$ and we refer the reader to the cited paper for further details. This proposition will be particularly important when we come to prove the operator representation property of the main theorem.

We can also make use of Propositions \ref{AHProp5.6} and \ref{RIStoBlock} to prove the $\ell_1$ duality of the space $\X_k$.
\begin{prop}\label{Xhasl1Dual}
The dual of $\X_k$ is $\ell_1(\Gamma)$. More precisely the map \[ 
\varphi \colon \ell_{1}(\Gamma) \to \X_k^*
\]
defined by
\[
 \varphi(x^*)x := \langle x , x^* \rangle
 \]
(where $x \in \X_k \subseteq \ell_{\infty}(\Gamma) = \ell_{1}^*(\Gamma)$, $x^* \in \ell_1(\Gamma)$) is an isomorphism.
\end{prop}

\begin{proof}
The argument here can be found in \cite{AH}, though we write out the details more explicitly. We note that by Theorem \ref{boundedlycompleteshrinkingbasis} it is enough to show that the basis $(d_{\gamma})_{\gamma \in \Gamma}$ of $\X_k$ is shrinking. Using Proposition \ref{Kalton3.2.7} it is enough to show that every bounded block basic sequence with respect to this basis is weakly null. By Proposition \ref{RIStoBlock}, it is sufficient to see that every RIS $(x_i)_{i \in \N}$ is weakly null. To this end, suppose $(x_i)_{i \in \N}$ is a $C$-RIS which fails to be weakly null; without loss of generality we may assume that there exists $\delta > 0$ and an $x^* \in \X_k^*$ such that $\| x^* \| = 1$ and $x^*(x_i) \geq \delta$ for all $i$. It follows that\[
\frac{1}{n_{j_0}} \sum_{i=1}^{n_{j_0}} x^*(x_i) \geq \delta
\] for any $j_0 \in \N$. On the other hand, we can choose $j_0 \in \N$ such that $10C m^{-1}_{j_0} < \delta$. It follows from Proposition \ref{AHProp5.6} that for this choice of $j_0$ \[
\frac{1}{n_{j_0}} \sum_{i=1}^{n_{j_0}} x^*(x_i) \leq \| \frac{1}{n_{j_0}} \sum_{i=1}^{n_{j_0}} x_i \| \leq 10Cm^{-1}_{j_0} < \delta.
\]
This contradiction completes the proof.
\end{proof}

This canonical identification of $\X_k^*$ with $\ell_1(\Gamma)$ allows us to prove some important properties of the operator $S\colon \X_k\to\X_k$. 

\begin{lem} \label{imSClosed} 
The dual of the operator $S: \X_k \to \X_k$ is precisely the operator $S^* : \ell_1(\Gamma) \to \ell_1(\Gamma)$ under the canonical identification, $\varphi$, of $\X_k^*$ with $\ell_1(\Gamma)$. Moreover, $S^j$ has closed range for every $1\leq j \leq k-1$, and consequently, $S^j: \X_k \to \im S^j$ is a quotient operator.
\end{lem}

\begin{proof}
We will temporarily denote the dual map of $S$ by $S' : \X_k^*  \to \X_k^*$ so that we don't confuse it with the $S^*$ mapping on $\ell_{1}(\Gamma)$. By continuity and linearity, the maps $\varphi ^{-1} S' \varphi$ and $S^*$ are completely determined by their action on the vectors $e_{\gamma}^*$ for $\gamma \in \Gamma$. For $x \in \X_k$, \[
\left( S' \varphi (e_{\gamma}^*)\right) (x) = \varphi (e_{\gamma}^*) Sx = \langle Sx, e_{\gamma}^* \rangle = \langle x, S^*e_{\gamma}^* \rangle = \begin{cases} x(F(\gamma) & \text{ if $F(\gamma)$ is defined} \\ 0 & \text{ otherwise} \end{cases}
\]
It follows from this that 
\begin{align*}
\varphi^{-1} S' \varphi (e_{\gamma}^*) &= \begin{cases} e_{F(\gamma)}^* & \text{ if $F(\gamma)$ is defined} \\ 0 & \text{ otherwise} \end{cases} \\[2pt]
&= S^*(e_{\gamma}^*)
\end{align*}
as required.

To see that the image of $S^j$ is closed, we make use of Banach's Closed Range Theorem, Theorem \ref{ClosedRangeThm}; it follows that it is enough to show that the image of $(S^*)^j$ is closed. This is easy, as it is clear that the image of $(S^*)^j$ is just $\ell_1(\Gamma \cap \im F^j) \subseteq \ell_{1}(\Gamma)$ (with the obvious embedding), so certainly $(S^*)^j$ has closed image. 

Since $\im S^j$ is closed, an easy application of the Inverse Mapping Theorem yields that $S^j \colon \X_k \to \im S^j$ is a quotient operator.

\end{proof}

\begin{cor} \label{corSisQuotientOp}
For $1 \leq j \leq k-1$there are $a_j, b_j \in \R$, $a_j, b_j > 0$ such that whenever $x \in \X_k$, \[
a_j\|S^jx\| \leq \dist (x , \Ker S^j ) \leq b_j \| S^jx \|
\]
\end{cor}

\begin{proof}
This is immediate from the fact that $S^j: \X_k \to \im S^j$ is a quotient operator.
\end{proof}

\begin{cor} \label{SNOTCOMPACT}
Suppose $\lambda_i \in \R$ ($0\leq i \leq k-1$) are such that $\sum_{i=0}^{k-1} \lambda_i S^i$ is compact. Then $\lambda_i = 0$ for every $i$. Consequently, there does not exist  $j$ ($1\leq j \leq k-1$) such that $S^j$ is a compact perturbation of some linear combination of the operators $S^l, l \neq j$. (Equivalently, $\{ I + \mathcal{K}(\X_k) , S^j + \mathcal{K}(\X_k) : 1\leq j\leq k-1\} $ is a linearly independent set of vectors in $\mathcal{L}(\X_k)/ \mathcal{K}(\X_k)$.)
\end{cor}

\begin{proof}
It follows by a standard result that $\sum_{i=0}^{k-1} \lambda_i S^i$ is compact if, and only if, the dual operator is compact. But the dual operator is just $T: = \sum_{i=0}^{k-1} \lambda_i (S^*)^i \colon \ell_{1}(\Gamma) \to \ell_{1}(\Gamma)$ and it is now easily seen that this is not compact unless all the $\lambda_i$ are $0$. Indeed, suppose $T$ is compact and consider first the sequence $(e_{\gamma^0_n}^*)_{n=1}^{\infty} \subseteq B_{\ell_1(\Gamma)}$ where $\gamma^0_n = (n+1, m_2^{-1}, 0, e_{0}^*) \in \Gamma$. Then, for $m\neq n$ (observing that $F(\gamma^0_n)$ is undefined for every $n$) we have \[
\| T e_{\gamma^0_n}^* - T e_{\gamma^0_m}^* \|_1 = |\lambda_0 | \| e_{\gamma^0_n}^* - e_{\gamma^0_m}^* \|_1 = 2|\lambda_0|
\]
We must therefore have that $\lambda_0 = 0$ in order that the sequence $(T e_{\gamma^0_n}^*)$ has a convergent subsequence. Then, considering in turn the sequences  $(e_{\gamma^j_n}^*)_{n=1}^{\infty} \subseteq B_{\ell_1(\Gamma)}$ where $\gamma^j_n = (n+1, m_2^{-1}, 0, e_{j}^*) \in \Gamma$ (for $1 \leq j \leq k-1$), we see by the same arguments that all the $\lambda_i$ must be $0$ as claimed.
\end{proof}

To complete the proof of the main theorem, it still remains for us to see that $\X_k$ is HI, the operator $S$ is strictly singular, and that every bounded linear operator $T \in \mathcal{L}(\X_k)$ is uniquely expressible in the form $T = \sum_{i=0}^{k-1} \lambda_i S^i + K$ for some compact operator $K \colon \X_k \to \X_k$. We note that the claimed uniqueness of this operator representation follows immediately from the above corollary. We focus now on proving the existence of the representation and aim to prove:

\begin{thm}\label{EssMainTheorem}
Let $T \colon \X_k \to \X_k$ be a bounded linear operator on $\X_k$ and $(x_i)_{i \in \N}$ a RIS in $\X_k$. Then $\dist (Tx_i, \langle x_i, Sx_i \dots ,S^{k-1}x_i \rangle_{\R}) \to 0$ as $i \to \infty$. 
\end{thm}

The proof is similar to that given in \cite{AH}. We will need slight modifications to the definitions of exact pairs and dependent sequences. We find it convenient to define both the $0$ ($\delta = 0$ in the definitions that follow) and $1$ ($\delta = 1$ in the definitions that follow) exact pairs and dependent sequences below. However, initially, we will only be concerned with the $0$ exact pairs and dependent sequences. The $1$ exact pairs and dependent sequences will only be needed to establish that the space $\X_k$ is hereditarily indecomposable. We also introduce the new, but related notions, of `weak exact pairs' and `weak dependent sequences' which will be useful to us later in establishing strict singularity of $S$. 

\begin{defn} \label{SpecialExactPair}

Let $C>0, \, \delta \in \{0 ,1\}$. A pair $(x,\eta)\in
\X_k\times \Gamma$ is said to be a $(C,j,\delta )$-{\em special exact pair}
if
 \begin{enumerate}
 \item $\|x\| \leq C$
 \item $|\langle d^*_\xi,x\rangle|\le Cm_{j}^{-1}$ for all $\xi\in \Gamma$
 \item $\weight \eta = m_{j}^{-1}$
 \item $x(\eta) = \delta$ and $S^l x(\eta) = 0$ for every $1 \leq l \leq k-1 $
 \item for every element $\eta'$ of
$\Gamma$ with $\weight\eta'= m_{i}^{-1}\ne m_j^{-1}$, we have
$$
|x(\eta')| \le \begin{cases} Cm_i^{-1} &\text{ if } i<j\\
                 Cm_{j}^{-1} &\text{ if }i>j.
                            \end{cases}
$$
\end{enumerate}
Given also an $\ve > 0$ we will say a pair $(x,\eta)\in
\X_k\times \Gamma$ is a $(C,j, 0, \ve)$-{\em weak exact pair}
if condition (4) is replaced by the following (weaker) condition:
\begin{enumerate}
\item[(4$'$)] $ |S^l x ( \eta) | \leq C \ve$ for $ 0 \leq l \leq k-1$.
\end{enumerate}
We will say a pair $(x, \eta) \in \X_k \times \Gamma$ is a $(C, j, 1, \ve)$-{\em weak exact pair} if condition (4) is replaced by condition: 
\begin{enumerate}
\item[(4$''$)] $x(\eta) = 1$ and $|S^lx(\eta) | \leq C\ve$ for $1 \leq l \leq k-1$.
\end{enumerate}
\end{defn}
We note that a $(C, j, \delta)$ special exact pair is a $(C, j, \delta, \ve)$ weak exact pair for any $\ve > 0$. Moreover, the definition of a $(C,j,\delta)$-special exact pair is the same as the definition of a $(C, j, \delta)$ exact pair given in \cite{AH} but with the additional requirement that $S^jx(\eta) = 0$ for all $j\geq 1$. The remark made in \cite{AH} (following the definition of exact pairs) is therefore still valid. In fact it is easily verified that the same remark in fact holds for weak exact pairs.  For convenience, we state the remark again as it will be useful to us later:
\begin{remark}
A $(C,j, \delta, \ve)$ weak exact pair also satisfies the estimates
$$
|\langle e^*_{\eta'},P_{(s,\infty)}x\rangle| \le \begin{cases} 6Cm_i^{-1} &\text{ if } i<j\\
                 6Cm_{j}^{-1} &\text{ if }i>j
                            \end{cases}
$$
for elements $\eta'$ of $\Gamma$ with $\weight \eta'=m_i^{-1}\ne
m_j^{-1}$.
\end{remark}

We will need the following method for constructing special exact pairs.
\begin{lem}\label{RISZeroSpecialExact}
Let $(x_k)_{k=1}^{n_{2j}}$ be a skipped-block $C$-RIS, and let
$q_0<q_1<q_2<\cdots<q_{n_{2j}}$ be  natural numbers such that $\ran
x_k\subseteq (q_{k-1},q_k)$ for all $k$. Let $z$ denote the weighted
sum $z=m_{2j}n_{2j}^{-1}\sum_{k=1}^{n_{2j}}x_k$. For each $k$ let
$b^*_k$ be an element of $B_{q_{k-1},q_{k}-1}$ with $\langle
b^*_k,x_k\rangle=0$ and $\langle
(S^*)^l b^*_k,x_k\rangle= \langle b^*_k , S^l x_k \rangle = 0$ for all $l$.Then there exist $\zeta_i \in \Delta_{q_i}$
($1\le i\le n_{{2j}}$) such that the element $\eta=\zeta_{n_{2j}}$
has analysis $(q_i,b^*_i,\zeta_i)_{1\le i\le n_{2j}}$ and $
(z,\eta)$ is a $(16C,{2j},0)$-special exact pair.
\end{lem}

\begin{proof}
The proof is the same as in \cite{AH}. We only need to show that $S^lz(\eta) = 0$ for $1\leq l \leq k-1$. This is easy. 
\[
S^lz(\eta) = \langle S^lz, e_{\eta}^* \rangle = \langle S^lz, \sum_{k=1}^{n_{2j}} d_{\zeta_k}^* + m_{2j}^{-1} P_{(q_{k-1},q_k)}^* b_k^* \rangle
\]
It is clear from the definition of $S$ that $\ran x_k \subseteq (q_{k-1},q_k) \implies \ran S^lx_k \subseteq (q_{k-1},q_k)$ and since $\rank\zeta_k = q_k$ for every $k$, it follows that $\langle S^lz, \sum_{k=1}^{n_{2j}} d_{\zeta_k}^* \rangle = 0$. We thus see that \[
S^lz(\eta) = n_{2j}^{-1} \sum_{k=1}^{n_{2j}} \langle S^lx_k , b_k^* \rangle = 0
\]
as required.
\end{proof}

\begin{defn}\label{DepSeq}Consider the space $\X_k$. We
shall say that a sequence $(x_i)_{i\le n_{2j_0-1}}$ is a $(C,
2j_0-1,\delta)$-{\em special dependent sequence} if there exist
$0=p_0<p_1<p_2<\cdots<p_{n_{2j_0-1}}$, together with
$\eta_i\in\Gamma_{p_i-1}\setminus \Gamma_{p_{i-1}}$ and $\xi_i\in
\Delta_{p_i}$ ($1\le i\le n_{2j_0-1}$) such that
\begin{enumerate}
\item for each $k$, $\ran x_k\subseteq (p_{k-1},p_k)$
 \item the
element $\xi=\xi_{n_{2j_0-1}}$ of $\Delta_{p_{n_{2j_0-1}}}$ has weight
$m^{-1}_{2j_0-1}$ and analysis
$(p_i,e^*_{\eta_i},\xi_i)_{i=1}^{n_{2j_0-1}}$
\item $(x_1,\eta_1)$ is a $(C,4j_1,\delta)$-special exact pair \item
for each $2\le i\le n_{2j_0-1}$, $(x_i,\eta_i)$ is a
$(C,4\sigma(\xi_{i-1}),\delta)$-special exact pair.
\end{enumerate}
If we instead ask that 
\begin{enumerate}
\item[(3$'$)] $(x_1, \eta_1)$ is a $(C, 4j_1, \delta, n_{2j_0-1}^{-1})$ weak exact pair
\item[(4$'$)]  $(x_i , \eta_i)$ is a $(C, 4\sigma(\xi_{i-1}), \delta, n_{2j_0-1}^{-1})$ weak exact pair for $2 \leq i \leq n_{2j_0-1}$
\end{enumerate}
we shall say the sequence $(x_i)_{i=0}^{n_{2j_0-1}}$ is a {\em weak $(C, 2j_0-1, \delta)$ dependent sequence}.

In either case, we notice that, because of the special odd-weight conditions, we necessarily have $m^{-1}_{4j_1} = \weight \eta_1
<n^{-2}_{2j_0-1}$, and $\weight \eta_{i+1} =m^{-1}_{4\sigma(\xi_i)} < n^{-2}_{2j_0-1}$,
by Lemma \ref{monotonicity_of_odd_weights} for $1\le i<n_{2j_0-1}$. 
\end{defn}
We also observe that a $(C, 2j_0 - 1, \delta)$ special dependent sequence is certainly a weak $(C, 2j_0 - 1, \delta)$ dependent sequence.

\begin{lem}\label{ZeroDepSeqLem}
Let $(x_i)_{i\le n_{2j_0-1}}$ be a weak $(C, 2j_0-1,0)$-{dependent
sequence} in $\X_k$ and let $J$ be a sub-interval of
$[1,n_{2j_0-1}]$. For any $\gamma'\in \Gamma$ of weight $m_{2j_0-1}$
we have
$$
|\sum_{i\in J} x_i(\gamma')| \le 7C.
$$
\end{lem}
\begin{proof}
Let $\xi_i,\eta_i,p_i,j_1$ be as in the definition of a dependent
sequence and let $\gamma$ denote $\xi_{n_{2j_0-1}}$, an element of
weight $m_{2j_0-1}$. Let $\big( p_0', (p_i',b'^*_i,\xi'_i)_{1\le i\le
a'} \big)$ be the analysis  of $\gamma'$ where each $b'^*_i$ is either $0$ or $e_{\eta'_i}^*$ for some suitable $\eta'_i$. 

The proof is easy if all the $b'^*_r$ are $0$, or if \[
\{ \weight\eta'_r : 1\leq r \leq a', b'^*_r = e_{\eta'_r}^* \} \cap \{ \weight\eta_i : 1\leq i \leq n_{2j_0 -1} \} = \varnothing.
\]

So we may suppose that there is some $1 \leq r \leq a'$ s.t. $b'^*_r = e_{\eta'_r}^*$ with $\weight\eta'_r = \weight\eta_i$ for some $i$. We choose $l$ maximal such that there exists $i$ with $b'^*_i = e_{\eta'_i}^*$ and $\weight\eta_i'= \weight\eta_l$. Clearly we can estimate as follows\[
|\sum_{k\in J} x_k(\gamma')| \leq |\sum_{k\in J,\ k<l} x_k(\gamma')|
+ |x_l(\gamma')| +
\sum_{k\in J,\ k>l}|x_k(\gamma')|.
\]
We now estimate the three terms on the right hand side of the inequality separately. $|x_l(\gamma')| \leq \|x_l\| \leq C$. Also

\begin{align*}
\sum_{k\in J,\ k>l}|x_k(\gamma')| &= \sum_{k\in J,\ k>l}\sum_{i\le a'}|\langle d^*_{\xi'_i},x_k\rangle + m_{2j_0-1}^{-1}\langle b'^*_i, P_{(p'_{i-1},\infty)}x_k\rangle| \\
&\leq n_{2j_0-1}^2\max_{l<k\in J,\ i\le a'}|\langle d^*_{\xi'_i},x_k\rangle + m_{2j_0-1}^{-1} \langle b'^*_i,
P_{(p'_{i-1},\infty)}x_k\rangle|.
\end{align*}

Now each $b'^*_i$ is $0$ or $e_{\eta'_i}^*$ where $\weight\eta'_i \neq \weight \eta_k$ for any $k > l$ (or else we would contradict maximality of $l$). If $b'^*_i = 0$, then
\[
|\langle d^*_{\xi'_i},x_k\rangle + m_{2j_0-1}^{-1} \langle b'^*_i, P_{(p'_{i-1},\infty)}x_k\rangle| = |\langle d^*_{\xi'_i},x_k\rangle | \leq C\weight\eta_k \leq C n_{2j_0 - 1}^{-2}
\]
where the penultimate inequality follows from the definition of a (weak) exact pair, and the final inequality follows from Lemma \ref{monotonicity_of_odd_weights}. Otherwise $b'^*_i = e_{\eta'_i}^*$ where in particular (by restrictions on elements of odd weight) $\weight\eta'_i < n_{2j_0 - 1}^{-2}$. By the definition of (weak) exact pair and the remark following it we have
\begin{align*}
|\langle d^*_{\xi'_i},x_k\rangle +m_{2j_0-1}^{-1}P_{(p_{i-1},\infty)}x_k(\eta'_i)| &\le
 C\weight \eta_k + 6Cm_{2j_0-1}^{-1}\max\{\weight \eta'_{i}, \weight \eta_{k}\} \\
 &\leq 3Cn_{2j_0 -1}^{-2}.
\end{align*}
Finally we consider $|\sum_{k\in J,\ k<l} x_k(\gamma')|$. Obviously if $l=1$ this sum is zero, and the lemma is proved. So we can suppose $l>1$. By definition of $l$, there exists some $i$ such that $b'^*_i = e_{\eta'_i}^*$ and $\weight\eta_l = \weight\eta'_i$. Now either $i = 1$ or $i > 1$. We consider the 2 cases separately.

Suppose first that $i=1$.
\begin{align*}
|\sum_{k \in J,\ k<l}x_k(\gamma')| &= |\langle \sum_{k\in J,\ k<l}x_k , \sum_{r=1}^{a'} d^*_{\xi'_r} + m_{2j_0-1}^{-1}  P^*_{(p'_{r-1},\infty)}b'^*_r \rangle | \\
&\leq n_{2j_0 - 1}^{2} \max_{J \ni k<l,\ r \leq a'} |\langle x_k,d^*_{\xi'_r} \rangle| + m_{2j_0-1}^{-1} | \langle \sum_{k\in J,\ k<l}x_k , \sum_{\substack{r\leq a' \text{s.t.} \\ b'^*_r = e_{\eta'_r}^*}} P^*_{(p'_{r-1},\infty)}e_{\eta'_r}^* \rangle | \\
&\leq C + m_{2j_0-1}^{-1} | \langle \sum_{k\in J,\ k<l}x_k , \sum_{\substack{r\leq a' \text{s.t.} \\ b'^*_r = e_{\eta'_r}^*}} P^*_{(p'_{r-1},\infty)}e_{\eta'_r}^* \rangle |
\end{align*}
where the last inequality follows once again from the definition of exact pair. Suppose that for some $k \in J$, $k < l$, there is an $r$ in $\{1, 2, \dots a' \}$ with $b'^*_r = e_{\eta'_r}^*$ and $\weight\eta'_r = \weight\eta_k$. By Lemma \ref{monotonicity_of_odd_weights}, we get $\weight\eta'_r = \weight\eta_k > \weight\eta_l = \weight\eta'_1$, i.e. $\weight\eta'_r > \weight\eta'_1$. But since $\gamma'$ also has odd weight, this clearly contradicts Lemma \ref{monotonicity_of_odd_weights} applied to $\gamma'$. Thus there does not exist $r$ in $\{1, 2, \dots a' \}$ with $b'^*_r = e_{\eta'_r}^*$ and $\weight\eta'_r = \weight\eta_k$ for some $k \in J$, $k<l$. Using an argument similar to the above, we finally deduce that 
\[
m_{2j_0-1}^{-1} | \langle \sum_{k\in J,\ k<l}x_k , \sum_{\substack{r\leq a' \text{s.t.} \\ b'^*_r = e_{\eta'_r}^*}} P^*_{(p'_{r-1},\infty)}e_{\eta'_r}^* \rangle | \leq 2C \]
and so we get the required result.

Finally it remains to consider what happens when $i > 1$. Recall we are also assuming $l > 1$ and $\weight\eta'_i = \weight\eta_l$. But by definition of a special exact pair, we have $\weight\eta_l = m_{4\sigma(\xi_{l-1})}^{-1}$, and by restrictions on elements of odd weights, $\weight\eta'_i = m_{4\omega}^{-1}$ with $\omega \in \Sigma(\xi'_{i-1})$. By strict monotonicity of the sequence $m_j$, we deduce that $\omega = \sigma(\xi_{l-1}) \in \Sigma(\xi'_{i-1})$. By Lemma \ref{IntersectionPropertyofSigmaSets} there are now two possibilites. Either $\xi_{l-1} = \xi'_{i-1}$ or, if not, there is some $j$, $1\leq j \leq k-1$, such that $F^j(\xi_{l-1}) = \xi'_{i-1}$. In either of these cases, we note that in particular this implies $p_{l-1} = p'_{i-1}$ since $F$ preserves rank and we can write the evaluation analysis of $\gamma'$ as \[
e_{\gamma'}^* = (S^*)^j(e^*_{\xi_{l-1}}) + \sum_{r=i}^{a'} d_{\xi'_r}^* + m_{2j_0-1}^{-1} P_{(p'_{r-1},p'_r)}^*b'^*_r
\]
for some $0 \leq j \leq k-1$. Now, for $k<l$, since $\ran x_k \subseteq (p_{k-1}, p_k) \subseteq (0,p_{l-1}) = (0, p'_{i-1})$, we see that 
\begin{align*}
|\langle x_k, e_{\gamma'}^* \rangle| &= |\langle x_k, (S^*)^je_{\xi_{l-1}}^* \rangle| \\
&= |\langle S^jx_k, \sum_{r=1}^{l-1} d_{\xi_r}^* + m_{2j_0-1}^{-1}P^*_{(p_{r-1}, p_r)} e_{\eta_r}^* \rangle| \\
&= m^{-1}_{2j_0-1} |\langle S^jx_k, e_{\eta_k}^* \rangle| \\
&\leq Cn_{2j_0-1}^{-1} 
\end{align*} by definition of a weak exact pair. It follows that \[
|\sum_{k \in J,\ k<l}x_k(\gamma ')| \leq n_{2j_0-1} \max_{k \in J,\  k<l} |x_k(\gamma')| \leq C.
\]
This completes the proof.
\end{proof}

As a consequence of the above lemma and Proposition \ref{AHProp5.6} we obtain the following upper norm estimate for the averages of weak special dependent sequences.

\begin{prop} \label{0DepSeqUpperEst}
Let $(x_i)_{i \leq n_{2j_0-1}}$ be a weak $(C, 2j_0-1, 0)$ dependent sequence in $\X_k$. Then \[
\| n_{2j_0-1}^{-1} \sum_{i=1}^{n_{2j_0-1}} x_i \| \leq 70C m_{2j_0-1}^{-2}
\]
\end{prop}
\begin{proof}
We apply the second part of Proposition \ref{AHProp5.6}, with $\lambda_i = 1$ and $2j_0-1$ playing the role of $j_0$. Lemma \ref{ZeroDepSeqLem} shows that the extra hypothesis of the second part of Proposition \ref{AHProp5.6} is satisfied, provided we replace $C$ by $7C$. We deduce that $\| n_{2j_0-1}^{-1} \sum_{i=1}^{n_{2j_0-1}} x_i \| \leq 70C m_{2j_0-1}^{-2}$ as claimed.
\end{proof}

We will need the following lemma from elementary linear algebra in the proof of Theorem \ref{EssMainTheorem}.

\begin{lem} \label{algebralem}
Suppose for $j = 1, \dots, k$, $u_j \in \Q^n \subseteq \R^n$ (where $k \leq n$) and moreover, the $u_j$ are linearly independent over $\R$. Let $V = \{ v \in \R^n : \langle v, u_j \rangle = 0 \text{ for } j = 1, \dots, k \}$. Then $V \cap \Q^n$ is dense in $V$.
\end{lem}
  
\begin{proof}
Consider the linear map $L \colon \Q^n \to \Q^k$, defined by $\Q^n \ni v \mapsto \left( \langle v, u_1 \rangle, \langle v, u_2 \rangle, \dots, \langle v, u_k \rangle \right)$. Thinking of $\Q^n$ and $\Q^k$ as vector spaces over $\Q$, and by the assumed linear independence of the $u_j$, we can apply the Rank-Nullity Theorem to $L$ to conclude that $V \cap \Q^n = \ker L$ is a $\Q$-vector space of dimension $n-k$.

Now, $\overline{V\cap\Q^n}$ (where the closure is taken in $\R^n$), is easily seen to be a real vector space. Moreover, $\overline{V\cap\Q^n} \subseteq \overline{V} = V$, and it is easily seen (e.g. by a similar rank-nullity argument as before) that $V$ is an $n-k$ dimensional real vector space. So $\overline{V\cap\Q^n}$ is a real vector space of dimension at most $n-k$. 

On the other hand, if $q_1, \dots q_{l}$ ($l \leq n$) are vectors in $\Q^n$ which are linearly independent over $\Q$, they must in fact be linearly independent over $\R$. Indeed, without loss of generality, we may assume $l=n$; if not, working in the $\Q$-vector space, $\Q^n$, we may extend the $(q_i)_{i\leq l}$ to a basis, $(q_i)_{1\leq i \leq n}$, of $\Q^n$ and our claim we will be proved if we show this extended set (which is obviously linearly independent over $\Q$) is really linearly independent over $\R$. Now, since $q_1, \dots, q_n$ is a basis of $\Q^n$, the standard basis of $\R^n, (e_j)_{j=1}^n$ is in the real (in fact, rational) span of $q_1, \dots, q_n$. Consequently, $\R^n$ is the real span of the $q_1, \dots q_n$, and so by basic linear algebra, the $q_1, \dots q_n$ must in fact be linearly independent over $\R$ which is what we wanted.

Since we have seen that $V\cap\Q^n$ is an $n-k$ dimensional $\Q$-vector space, there are $n-k$, $\Q$-linearly independent vectors in $V\cap\Q^n \subseteq \overline{V\cap\Q^n}$. By the preceding paragraph, these $n-k$ vectors are linearly independent over $\R$ and so $\overline{V\cap\Q^n}$ is a real vector space of dimension at least $n-k$. Combined with our earlier observation, we conclude that $\overline{V\cap\Q^n}$ is a real vector space of dimension exactly $n-k$. Now $\overline{V\cap\Q^n} \subseteq V$ and both are real vector spaces of the same dimension, so they must in fact be equal, as required.
\end{proof}

We can now prove the following analogous result to Lemma 7.1 of \cite{AH}.

\begin{lem}\label{RatVec}
Let $m<n$ be natural numbers and let $x\in \X_k \cap \Q^\Gamma$,
$y\in \X_k$ be such that $\ran x, \ran y$ are both contained in the
interval $ (m,n]$. Suppose that $\dist(y,\langle x, Sx \dots ,S^{k-1}x \rangle_{\R})>\delta$.  Then
there exists $b^*\in \ball\ell_1(\Gamma_n\setminus \Gamma_m)$, with
rational coordinates, such that $\langle b^*, S^jx \rangle=0$ for $j = 0, 1, \dots k-1$ and
$b^*(y)>\frac12\delta$.
\end{lem}
\begin{proof}
It is easily seen (from the construction of $S$) that since $x \in \X_k \cap \Q^\Gamma$, with $\ran x \subseteq (m, n]$, we have $S^jx \in \X_k \cap \Q^\Gamma$ and $\ran S^jx \subseteq (m, n]$ for every $j = 0, 1, \dots k-1$. For each $j$, let $u_j\in \ell_\infty(\Gamma_n\setminus \Gamma_m)$ be the
restriction of $S^jx$ and let $v \in \ell_\infty(\Gamma_n\setminus \Gamma_m)$ be the restriction of $y$.  Then $S^j x=i_n u_j$ for every $j, y=i_nv$ and so,
for any scalars $(\lambda_j)_{j=0}^{k-1}$, $\|y-\sum_{j=0}^{k-1} \lambda_j S^j x\|\le \|i_n\|\|v-\sum_{j=0}^{k-1}\lambda_j
u_j\|$.  Hence $\dist(v,\langle u_0, u_1, \dots, u_{k-1} \rangle_{\R})>\frac12\delta$ and so, by the
Hahn--Banach Theorem in the finite dimensional space
$\ell_\infty(\Gamma_n\setminus \Gamma_m)$, there exists $a^*\in
\ball\ell_1(\Gamma_n\setminus \Gamma_m)$ with $a^*(u_j)=0$ for every $j$ and
$a^*(v)>\frac12\delta$. Since $S^jx$ has rational coordinates for every $j$, our
vectors $u_j$ are in $ \Q^{\Gamma_n\setminus \Gamma_m}$.  It follows from Lemma \ref{algebralem}
that we can approximate $a^*$ arbitrarily well with $b^*\in
\Q^{\Gamma_n\setminus \Gamma_m}$ retaining the condition
$b^*(u_j)=0$ for every $j$.
\end{proof}

The proof of Theorem \ref{EssMainTheorem} is now easy. We obtain the following minor variation of \cite[Lemma 7.2]{AH}. We omit the proof since exactly the same argument as in  \cite{AH} works:
\begin{lem}\label{DistExact}
Let $T$ be a bounded linear operator on $\X_k$, let $(x_i)$ be a
$C$-RIS in $\X_k \cap \Q^\Gamma$ and assume that
$\dist (Tx_i, \langle x_i, Sx_i \dots ,S^{k-1}x_i \rangle_{\R}) >\delta>0$ for all $i$.  Then, for
all $j,p\in \N$, there exist $z\in [x_i:i\in \N]$, $q>p$ and
$\eta\in \Delta_{q}$ such that
 \begin{enumerate}
 \item $(z,\eta)$ is a $(16C,2j,0)$-special exact pair with $\ran z \subset (p, q)$;
 \item $(Tz)(\eta)>\frac7{16}\delta$;
 \item $\|(I-P_{(p,q)})Tz\|<m_{2j}^{-1}\delta$;
 \item $\langle P^*_{(p,q]}e^*_{\eta},Tz\rangle >\frac{3}{8}\delta$.
 \end{enumerate}
\end{lem}
Theorem \ref{EssMainTheorem} is now proved in exactly the same way as that of \cite[Proposition 7.3]{AH}.

\section{Operators on the Space $\X_k$}

In this section, we see that all operators on the space $\X_k$ are expressible as $\sum_{j=0}^{k-1}\lambda_j S^j +K$ for suitable scalars $\lambda_j$ and some compact operator $K$ on $\X_k$. We choose to give an elementary proof of this fact in this chapter. However, in the following chapter, we will construct a space $\X_{\infty}$ which has a similar operator representation; $S$ is no longer nilpotent in this case, so we obtain an infinite sum in the powers of $S$ of the form $\sum_{j=0}^{\infty} \lambda_j S^j$, where the scalars $(\lambda_j)_{j=0}^{\infty}$ lie in $\ell_1(\N_0)$. The proof of the operator representation in Chapter \ref{l1Calk} will be much more technical, relying on a fixed-point-theorem due to Kakutani. We remark that the proof of the operator representation in Chapter \ref{l1Calk} would also work (with only obvious changes) to prove the operator representation we require in this section. With that said, we now proceed to give a more elementary proof. We will need the following lemmas.
\begin{lem}\label{PerturbedRIS}
Let $1\leq j \leq k-1$ and $(x_i)_{i\in\N}$ be a C-RIS in $\X_k$. Suppose there are $\widetilde{x_i}'$ such that $\| \widetilde{x_i}' - x_i \| \to 0$ as $i \to \infty$ and $S^j \widetilde{x_i}' = 0$ for every $i$. Then there is a subsequence $(x_{i_k})_{k\in\N}$ of $(x_i)$ and vectors $x_k'$ satisfying
\begin{enumerate}[(1)]
\item $\ran x_{i_k} = \ran x_k'$
\item $S^jx_k' = 0$ for every $k$
\item $\|x_k' - x_{i_k} \| \to 0$
\item $(x_k')_{k\in\N}$ is a $2C$-RIS.
\end{enumerate}
We note that in particular, if $(x_i)_{i\in\N}$ is a $C$-RIS with $S^jx_i \to 0$, then the above hypothesis are satisfied as a consequence of Corollary \ref{corSisQuotientOp}.
\end{lem}

\begin{proof}
Let $\ran x_i = (p_i, q_i)$ and set $y_i = P_{(p_i, q_i)} \widetilde{x_i}'$. Certainly then $\ran y_i = \ran x_i$ for every $i$. Note that $\left( I - P_{(p_i, q_i)} \right) x_i = 0$ and consequently \[
\left\| \left( I - P_{(p_i, q_i)} \right) \widetilde{x_i}' \right\| = \left\| \left( I - P_{(p_i, q_i)} \right) ( \widetilde{x_i}' - x_i ) \right\| \leq \left\| I - P_{(p_i, q_i)} \right\| \left\| \widetilde{x_i}' - x_i \right\| \leq 5 \left\| \widetilde{x_i}' - x_i \right\| \to 0.
\]
It follows that
\begin{align*}
\| y_i - x_i \| &= \left\| \widetilde{x_i}' - \big( ( I - P_{(p_i, q_i)} ) \widetilde{x_i}' \big) - x_i \right\| \\
&\leq \left\| \widetilde{x_i}' - x_i \right\| + \left\| \left( I - P_{(p_i, q_i)} \right) \widetilde{x_i}' \right\| \to 0.
\end{align*}
Note also that $S^jy_i = 0$ for every $i$. Indeed, for $\gamma \in \Gamma$,
\begin{align*}
S^jy_i (\gamma) = \langle S^jy_i , e_{\gamma}^* \rangle &= \langle S^jP_{(p_i, q_i)} \widetilde{x_i}' , e_{\gamma}^* \rangle  
= \langle \widetilde{x_i}' , P_{(p_i, q_i)}^* (S^*)^j e_{\gamma}^* \rangle \\
&=  \langle \widetilde{x_i}' ,  (S^*)^j P_{(p_i, q_i)}^* e_{\gamma}^* \rangle = \langle S^j\widetilde{x_i}' , P_{(p_i, q_i)}^* e_{\gamma}^* \rangle = 0
\end{align*}
since $S^j\widetilde{x_i}' = 0$ for every $i$. (We also made use of the fact that $P_{(p_i, q_i)}^* (S^*)^j = (S^*)^j P_{(p_i, q_i)}$. This follows from the construction of the operator $S^*$ and the fact that $F$ is rank-preserving. A more detailed argument of this fact is presented in the proof of Theorem \ref{R^*andGConstruction}.)

So far we have managed to achieve (1) - (3) of the above. We show we can extract a subsequence of the $y_i$, $(y_{i_k})_{k\in\N}$ say, such that $(y_{i_k})_{k\in\N}$ is a $2C$-RIS. The proof will then be complete if we set $x_k' = y_{i_k}$ and take the subsequence $(x_{i_k})$ of the $x_i$. 

Since $\| x_i \| \leq C $ for every $i$ and $\| y_i - x_i \| \to 0$, we can certainly assume (by ignoring some finite number of terms at the beginning of the sequence) that $\| y_i \| \leq 2C $ for every $i$. Let $(j_k)$ be the increasing sequence corresponding to the $C$-RIS $(x_i)$, i.e.
\begin{enumerate}[(i)]
\item $j_{k+1} > \text{max } \ran x_k$
\item $|x_k(\gamma)| \leq Cm_{i}^{-1} \text{ when $\weight\gamma = m_{i}^{-1}$ and $i < j_k$}$
\end{enumerate}
Set $l_1 = j_1$. We can certainly find an $i_1 \geq 1$ such that $\| y_{i_1} - x_{i_1} \| \leq Cm_{l_1}^{-1}$. So if $\gamma \in \Gamma$, $\weight \gamma = m_{w}^{-1}$ with $w < l_1$, then certainly $w < l_1 = j_1 \leq j_{i_1}$ so \[
|y_{i_1} (\gamma) | \leq |(y_{i_1} - x_{i_1})(\gamma) | + |x_{i_1}(\gamma)| \leq Cm_{l_1}^{-1} + Cm_{w}^{-1} \leq 2Cm_{w}^{-1}
\]
Now set $l_2 = j_{i_1+1}$. So $l_2 > \text{max } \ran x_{i_1} = \text{max } \ran y_{i_1}$. 

Inductively, suppose we have defined natural numbers $l_1 \leq l_2 \leq \dots \leq l_n$ and $i_1 < i_2 < \dots < i_n$ such that 
\begin{enumerate}[(i')]
\item $l_{k+1} > \text{max } \ran y_{i_k}$ for all $1\leq k \leq n-1$, and
\item for all $1\leq k \leq n$, $|y_{i_k} (\gamma) | \leq 2Cm_{w}^{-1}$ whenever $\gamma \in \Gamma$, $\weight \gamma = m_{w}^{-1}$ with $w < l_k$
\end{enumerate}
Set $l_{n+1} = j_{i_n + 1}$. It is easily seen from the inductive construction that $l_{n+1} \geq l_n$ and moreover (by choice of $j_k$), $l_{n+1} > \text{max } \ran x_{i_n} = \text{max } \ran y_{i_n}$. Now we can certainly find $i_{n+1} > i_n$ such that $\|y_{i_{n+1}} - x_{i_{n+1}} \| \leq Cm_{l_{n+1}}^{-1}$. So suppose $\gamma \in \Gamma$, $\weight \gamma = m_{w}^{-1}$ with $w < l_{n+1}$ In particular $w < j_{i_n + 1} \leq j_{i_{n+1}}$, so by choice of $i_{n+1}$ and the fact that $(x_i)$ is a RIS, we see that   \[
|y_{i_{n+1}} (\gamma) | \leq |(y_{i_{n+1}} - x_{i_{n+1}})(\gamma)| + |x_{i_{n+1}}(\gamma)| \leq Cm_{l_{n+1}}^{-1} + Cm_{w}^{-1} \leq 2Cm_{w}^{-1}.
\]
Inductively we obtain a subsequence $(y_{i_k})_{k\in\N}$ which is evidently a $2C$-RIS (with the sequence $(l_k)_{k\in\N}$ satisfying the RIS definition), as required.
\end{proof}
We will also need the following technical lemma.
\begin{lem} \label{TechnicalLemma}
Suppose $(x_i)_{i\in \N}$ is a normalised sequence in $\X_k$ and $\lambda_j \in \R$ ($0\leq j \leq k-1$) are scalars such that $\sum_{j=0}^{k-1} \lambda_j S^j x_i \to 0$. If $S^{k-1} x_i \arrownot\to 0$ then $\lambda_j = 0$ for every $j$. Otherwise, there is $1 \leq m \leq k-1$ such that $S^m x_i \to 0$ but $S^j x_i \arrownot\to 0$ if $j < m$, in which case, we must have that $\lambda_j = 0$ for all $j < m$.
\end{lem}

\begin{proof}
We consider first the case where $S^{k-1}x_i \to 0$ and choose $m \in \{ 1, \dots k-1 \}$ minimal such that $S^mx_i \to 0$ (noting such an $m$ obviously exists). We must observe that $\lambda_j =  0$ for all $j<m$. Since $S^j x_i \to 0$ for all $j \geq m$ we in fact know that $\sum_{j=0}^{m-1} \lambda_j S^j x_i \to 0$. If $m = 1$, this of course implies that $\lambda_0 =  0$ since the sequence $(x_i)$ is normalised and we are done. 

Otherwise, we apply the operator $S^{m-1}$ to the previous limit. Once again making use of the fact that $S^jx_i \to 0$ when $j \geq m$, we find that $\lambda_0 S^{m-1} x_i \to 0$. Since, by choice of $m$, $S^{m-1} x_i \arrownot\to 0$, we must again have $\lambda_0 = 0$, and moreover, $\sum_{j=1}^{m-1} \lambda_j S^j x_i \to 0$. If $m = 2$, then this implies $\lambda_1 Sx_i \to 0$ which implies that $\lambda_1 = 0$ (since $Sx_i \arrownot\to 0$). Otherwise, we apply the operator $S^{m-2}$. A similar argument concludes once again that we must have $\lambda_1 = 0$. Continuing in this way, we get that $\lambda_j = 0$ for all $j < m$ as required.

In the case where $S^{k-1}x_i \arrownot\to 0$, we notice that in particular this implies $S^jx_i \arrownot\to 0$ for every $1 \leq j \leq k-1$. Applying the operators $S^{k-1}, S^{k-2}, \dots S$ sequentially to the limit, $\sum_{j=0}^{k-1} \lambda_j  S^jx_i \to 0$, first yields that $\lambda_0 = 0$, then $\lambda_1 = 0$ etc. So $\lambda_j = 0$ for every $j$ as required.
\end{proof}

\begin{thm} \label{OperatorRepnTheorem}
Let $T \colon \X_k \to \X_k$ be a bounded linear operator on $\X_k$. Then there are $\lambda_j \in \R$ ($0 \leq j \leq k-1$) and a compact operator $K: \X_k \to \X_k$ such that $T = \sum_{j=0}^{k-1} \lambda_j S^j + K$.
\end{thm}

\begin{proof}
We will show that there exist scalars $\lambda_j$ such that whenever $(x_i)_{i \in \N}$ is a RIS, $Tx_i - \sum_{j=0}^{k-1} \lambda_j S^j x_i \to 0$ as $i \to \infty$. By Proposition \ref{RIStoBlock}, this implies $Tx_i - \sum_{j=0}^{k-1} \lambda_j S^j x_i \to 0$ for every block sequence $(x_i)$ which, by Proposition \ref{CompactLemma}, implies that $T-\sum_{j=0}^{k-1} \lambda_j S^j$ is compact. We note that it is enough to show that there are $\lambda_j \in \R$ such that whenever $(x_i)_{i\in\N}$ is a RIS, we can find some subsequence $(x_{i_l})$ with $Tx_{i_l} - \sum_{j=0}^{k-1}\lambda_j S^j x_{i_l} \to 0$. 

\begin{nclaim} \label{RISSxiNotTo0}
Suppose $(x_i)_{i\in\N}$ is a normalised RIS and that $S^{k-1} x_i \arrownot\to 0$ (noting in particular that this implies that $S^jx_i \arrownot\to 0$ for every $1 \leq j \leq k-1$). Then there are $\lambda_j \in \R$ ($0 \leq j \leq k-1$) and a subsequence $(x_{i_l})$ of $(x_i)$ such that $Tx_{i_l} - \sum_{j=0}^{k-1} \lambda_j S^j x_{i_l} \to 0$. 
\end{nclaim}
By passing to a subsequence, we may assume that there is some $\ve > 0$ such that $\| S^jx_i \| \geq \ve$ for every $i\in \N$ and every $1\leq j \leq k-1$. By Theorem \ref{EssMainTheorem}, there are $\lambda^j_i  \in \R$ such that $\| Tx_i - \sum_{j=0}^{k-1} \lambda^j_i S^j x_i \| \to 0$. We first show that the $\lambda^0_i$ must converge. The argument is similar to that of \cite{AH}. If not, by passing to a subsequence, we may assume that $|\lambda^0_{2i+1} - \lambda^0_{2i}| \ge \delta > 0$ for some $\delta$. Since $y_i := x_{2i-1} + x_{2i}$ is a RIS, we deduce from Theorem \ref{EssMainTheorem} that there are $\mu^j_i  \in \R$ with $\|Ty_i - \sum_{j=0}^{k-1} \mu^j_i S^j y_i \| \to 0$. Now
\begin{align*}
\| \sum_{j=0}^{k-1} & \left( \lambda^j_{2i} -  \mu^j_i \right) S^j x_{2i} + \sum_{j=0}^{k-1} \left( \lambda_{2i-1}^j - \mu^j_i \right) S^j x_{2i-1} \|  \\
& \leq \| \sum_{j=0}^{k-1} \lambda_{2i}^{j} S^jx_{2i} - Tx_{2i} \|+ \|  \sum_{j=0}^{k-1} \lambda_{2i-1}^{j} S^jx_{2i-1} - Tx_{2i-1} \| + \| Ty_i - \sum_{j=0}^{k-1} \mu_i^j S^j y_i \| 
\end{align*}
and so we deduce that both sides of the inequality converge to $0$. Since the sequence $(x_i)$ is a block sequence, there exist $l_k$ such
that $P_{(0,l_k]}y_k = x_{2k-1}$ and $P_{(l_k,\infty)}y_k = x_{2k}$. Recalling that if $x \in \X_k$ has $\ran x = (p,q] $ then $\ran S^jx \subseteq (p,q]$, we consequently have \[
\| \sum_{j=0}^{k-1} \left( \lambda^j_{2i} - \mu^j_i \right) S^j x_{2i} \| \leq  \| P_{(l_i, \infty)} \| \| \sum_{j=0}^{k-1} \left( \lambda^j_{2i} - \mu^j_i \right) S^j x_{2i} + \sum_{j=0}^{k-1} \left( \lambda_{2i-1}^j - \mu_i^j \right) S^j x_{2i-1} \|  \to 0
\]
and similarly
\[
\| \sum_{j=0}^{k-1} \left( \lambda^j_{2i-1} - \mu^j_i \right) S^j x_{2i-1} \| \to 0.
\]
By continuity of $S$ and the fact that $S^k = 0$, applying $S^{k-1}$ to both limits above (and recalling that $\| S^j x_i \| \geq \varepsilon$ for every $1 \leq j \leq k-1$) we obtain 
\[
|\lambda^0_{2i} - \mu^0_{i} | \leq \frac{1}{\varepsilon} \| \left( \lambda^0_{2i} - \mu^0_{i} \right) S^{k-1} x_{2i} \| \leq \frac{1}{\varepsilon} \| S^{k-1} \| \| \sum_{j=0}^{k-1} \left( \lambda_{2i}^j - \mu^j_i \right) S^j x_{2i} \| \to 0
\]
and similarly we find that $|\lambda^0_{2i-1} - \mu^0_{i} | \to 0$. It follows that $| \lambda^0_{2i} - \lambda^0_{2i-1} | \to 0$ contrary to our assumption. So the $\lambda^0_i$ converge to some $\lambda_0$ as claimed. It follows that $\| Tx_i - \lambda_0 x_i - \sum_{j=1}^{k-1} \lambda^j_i  S^jx_i \| \to 0$. We observe that, since $(x_i)$ is normalised and $\| S^{k-1} x_i \| \geq \varepsilon$, applying $S^{k-2}$ to the previous limit, we see that \[
|\lambda^1_i | \leq \frac{1}{\varepsilon} \| \lambda^1_i S^{k-1}x_i \| \leq \frac{1}{\varepsilon} \big( \| S^{k-2} \| \| Tx_i - \lambda_0 x_i - \sum_{j=1}^{k-1} \lambda^j_i S^jx_i \| + \| S^{k-2}T - \lambda_0 S^{k-2} \| \big)
\]
so that in particular the $\lambda_i^1$ are bounded. Consequently there is some convergent subsequence $\lambda_{i_l}^{1}$ (limit $\lambda_1$) of the $\lambda^1_i$. It follows that the corresponding subsequence $(x_{i_l})$ satisfies \[
Tx_{i_l} - \lambda_0 x_{i_l} - \lambda_1 Sx_{i_l} - \sum_{j=2}^{k-1}\lambda^j_{i_l} S^j x_{i_l} \to 0
\]
Now, if $k=2$ we are done (the last sum is empty). Otherwise, we can apply $S^{k-3}$ to the previous limit and use the same argument to conclude that $(\lambda^2_{i_l})_{l=1}^{\infty}$ is a bounded sequence of scalars. Continuing in this way, we eventually find (after passing to further subsequences which we relabel as $x_{i_l}$) that there are $\lambda_j$ with $(T - \sum_{j=0}^{k-1} \lambda_j S^j )x_{i_l} \to 0$ as required.
\begin{nclaim} \label{RISwithSxiTo0}
Suppose $(x_i)_{i \in \N}$ is a normalised $C$-RIS and that $S^m x_i \to 0$ for some $1 \leq m \leq k-1$. Let $m_0 \geq 1$ be minimal such that $S^{m_0} x_i \to 0$. Then there are $\lambda_j \in \R$ ($0 \leq j < m_0$) and a subsequence $(x_{i_l})$ of $(x_i)$ such that $(T-\sum_{j=0}^{m_0 - 1} \lambda_j S^j) x_{i_l} \to 0$. 
\end{nclaim}
By minimality of $m_0$, we can assume (by passing to a subsequence if necessary) that $\| S^j x_i \| \geq \ve$ for all $i\in \N$ and all $j < m_0$.  By Lemma \ref{PerturbedRIS}, (with $j = m_0$) there is a $2C$-RIS $(x'_l)_{l \in \N}$ and some subsequence $(x_{i_l})$ of $(x_i)$ such that $x'_l \in \Ker S^{m_0} \subseteq \Ker S^j$ for $j \ge m_0$ and every $l$. Moreover,  $\|x_{i_l} - x'_l \| \to 0$ (as $l \to \infty$). By Theorem \ref{EssMainTheorem}, there are $\lambda^j_l$ s.t \[
\| Tx'_l - \sum_{j=0}^{k-1}\lambda^j_l S^j x'_l  \| = \| Tx'_l - \sum_{j=0}^{m_0-1} \lambda^j_l S^j x'_l \| \to 0.
\]
We claim the $\lambda^0_l$ must converge to some $\lambda_0$. The argument is the same as that used in Claim \ref{RISSxiNotTo0}, except now we obtain \[
\| \sum_{j=0}^{m_0-1} \left( \lambda^j_{2l-1} - \mu^j_l \right) S^j x'_{2l-1} \| \to 0 \text{ and } \| \sum_{j=0}^{m_0-1} \left( \lambda^j_{2l} - \mu^j_l \right) S^j x'_{2l} \| \to 0.
\] 
(We note there are no terms of the form `$S^jy$' when $j \geq m_0$ since the RIS $(y_k)$ defined by $y_k := x_{2k-1}' + x_{2k}'$ also lies in $\Ker S^{m_0}$). We apply $S^{m_0-1}$, noting that $S^j x_l' = 0$ for every $l$ and $j \geq m_0$. We reach the same contradiction as in the previous argument, i.e. that $| \lambda^0_{2j} - \lambda_{2j-1}^0 | \to 0$ and consequently it follows that $\lambda^0_l$ must converge to some $\lambda_0$ as claimed. It easily follows that $Tx'_k - \lambda_0 x'_l - \sum_{j=1}^{m_0-1} \lambda^j_l S^j x'_l \to 0$. We use the same argument as above to show that the sequences $(\lambda^j_l)_{l=1}^{\infty}$ for $j = 1,2 \dots , m_0-1$ all have convergent subsequences and consequently that we can find some subsequence $x'_{l_r}$ with $(T - \sum_{j=0}^{m_0 - 1}\lambda_j S^j )x'_{l_r} \to 0$. It follows that
\begin{align*}
\| Tx_{i_{l_r}} - \sum_{j=0}^{m_0-1} \lambda_j S^j x_{i_{l_r}} \| &\leq \| \big( T - \sum_{j=0}^{m_0-1} \lambda_j S^j  \big)(x_{i_{l_r}} - x'_{l_r})\| + \| \big( T - \sum_{j=0}^{m_0-1} \lambda_j S^j \big)(x'_{l_r}) \| \\
&\leq \| T - \sum_{j=0}^{m_0-1}\lambda_j S^j  \| \|x_{i_{l_r}} - x'_{l_r} \| + \| \big( T - \sum_{j=0}^{m_0-1} \lambda_j S^j \big)(x'_{l_r}) \| \to 0.
\end{align*}
A priori, the $\lambda_j$ found in Claims \ref{RISSxiNotTo0} and \ref{RISwithSxiTo0} may depend on the RIS. We see now that this is not the case.
\begin{nclaim} \label{sameLambdaAndMuWork}
There are $\lambda_j \in \R$ ($0 \leq j \leq k-1$) such that whenever $(x_i)_{i\in\N}$ is a RIS, there is a subsequence $(x_{i_l})$ of $(x_i)$ such that $Tx_{i_l} - \sum_{j=0}^{k-1}\lambda_j S^j x_{i_l} \to 0$.
\end{nclaim}
Note that if $(x_i)_{i\in\N}$ is a RIS with some some subsequence converging to $0$ then any $\lambda_j$ can be chosen satisfying the conclusion of the claim. So it is sufficient to only consider normalised RIS. We begin by choosing a normalised RIS, $(x_i)_{i \in \N}$, with $S^{k-1}x_i \arrownot\to 0$. Note that such a RIS must exist. Indeed, if not, then $S^{k-1} x_i \to 0$ whenever $(x_i)_{i \in \N}$ is a RIS and by the argument in the first paragraph of the proof, this would imply that $S^{k-1}$ is compact, contradicting Corollary \ref{SNOTCOMPACT}. It follows from Claim \ref{RISSxiNotTo0}, after passing to a subsequence and relabelling if necessary, that there are $\lambda_j \in \R$ with 
\begin{equation}\label{eqn1}
Tx_i - \sum_{j=0}^{k-1} \lambda_jS^j x_i \to 0. \tag{1}
\end{equation}

Now suppose $(x_i')_{i\in\N}$ is any normalised RIS in $\X_k$. We will show that there is some subsequence, $(x'_{i_l})_{l \in \N}$, of $(x'_i)$ such that Equation (\ref{eqn1}) holds when $x_i$ is replaced by $x'_{i_l}$, thus proving the claim.  It follows from Claims \ref{RISSxiNotTo0} and \ref{RISwithSxiTo0}, after passing to a subsequence and relabelling if necessary, that there are $\lambda_j'  \in \R$ with 
\begin{equation} \label{eqn2}
Tx_i' - \sum_{j=0}^{k-1} \lambda'_j S^j x_i' \to 0. \tag{2}
\end{equation}
We pick natural numbers $i_1 < i_2 < \dots$ and $j_1 < j_2 \dots$ such that $\text{max } \ran x_{i_k} < \text{min } \ran x'_{j_k} \leq \text{max } \ran x'_{j_k} < \text{min } \ran x_{i_{k+1}}$ for every $k$ and such that the sequence $(x_{i_k} + x'_{j_k} )_{k\in\N}$ is again a RIS. For notational convenience, we (once again) relabel the subsequences $(x_{i_k})$, $(x_{j_k}')$ by $(x_i)$ and $(x_i')$. So (by choice of the subsequences) $(x_i + x'_i)_{i \in \N}$ is a RIS and there are natural numbers $l_i$ such that $P_{(0, l_i]} (x_i+ x'_i) = x_i$ and $P_{(l_i, \infty)} (x_i + x'_i) = x'_i$. It follows again from Claims  \ref{RISSxiNotTo0} and \ref{RISwithSxiTo0} that there are $\mu_j$ and a subsequence $(x_{i_m} + x_{i_m}')$ such that 
\begin{equation} \label{eqn3}
T(x_{i_m} +x'_{i_m}) - \sum_{j=0}^{k-1} \mu_j S^j (x_{i_m} + x'_{i_m} ) \to 0. \tag{3}
\end{equation}
We note also that 
\begin{align*}
P_{(0, l_i] }Tx_i' &= P_{(0, l_i]}\big( Tx_i' - \sum_{j=0}^{k-1} \lambda_j' S^j x_i' \big) + \sum_{j=0}^{k-1} \lambda_j' P_{(0, l_i]} S^jx_i'  \\
&= P_{(0, l_i]}\big( Tx_i' - \sum_{j=0}^{k-1} \lambda_j' S^jx_i'  \big)  \to 0
\end{align*}
and similarly, $P_{(l_i, \infty)} Tx_i \to 0$. Passing to the appropriate subsequences of Equations \eqref{eqn1} and \eqref{eqn2} and substracting them from Equation \eqref{eqn3} we see that 

\begin{equation} \label{eqn4}
Tx_{i_m}' - \sum_{j=0}^{k-1}(\mu_j - \lambda_j) S^j x_{i_m} - \sum_{j=0}^{k-1} \mu_j S^j x_{i_m}'  \to 0 \tag{4}
\end{equation}
and
\begin{equation} \label{eqn5}
Tx_{i_m} - \sum_{j=0}^{k-1}(\mu_j - \lambda'_j) S^j x'_{i_m} - \sum_{j=0}^{k-1} \mu_j S^j x_{i_m}  \to 0. \tag{5}
\end{equation}
Finally we apply the projections $P_{(0, l_{i_m}]}$ and $P_{(l_{i_m}, \infty)}$ to Equations \eqref{eqn4} and \eqref{eqn5} respectively to obtain (using the above observations) that \[
 \sum_{j=0}^{k-1}(\mu_j - \lambda_j) S^j x_{i_m} \to 0 \text{ and } \sum_{j=0}^{k-1}(\mu_j - \lambda'_j) S^j x'_{i_m}  \to 0.
\]
Since $(x_i)_{i\in \N}$ was chosen such that $S^{k-1}x_{i} \arrownot\to 0$, Lemma \ref{TechnicalLemma} and the first of the above limits implies that we must have $\lambda_j = \mu_j$ for all $j$. We now consider two cases:
\begin{enumerate}[(i)]
\item $S^{k-1} x'_{i_m} \arrownot\to 0$. By Lemma \ref{TechnicalLemma} and the second of the above limits, we see that we must have $\lambda '_j = \mu_j = \lambda_j$ for every $j$. So, we can replace $\lambda'_j$ by $\lambda_j$ in Equation (\ref{eqn2}) and we see that Equation (\ref{eqn1}) does indeed hold with $(x_i)$ replaced by $(x'_{i})$.
\item There is some $1 \leq r \leq k-1$ such that $S^r x'_{i_m} \to 0$, but $S^j x'_{i_m} \arrownot\to 0$ for any $j < r$. Again, by Lemma \ref{TechnicalLemma}, we must have $\lambda'_j = \mu_j = \lambda_j$ for all $j < r$. So, replacing $\lambda'_j$ by $\lambda_j$ for $j < r$ in Equation (\ref{eqn2}), we find that $Tx_i' - \sum_{j=0}^{r-1} \lambda_j S^j x_i' - \sum_{j=r}^{k-1} \lambda'_j S^j x'_i \to 0$. Since $S^l x'_{i_m} \to 0$ for all $l \geq r$, it is now clear that $Tx_{i_m}' - \sum_{j=0}^{k-1} \lambda_j S^j x_{i_m}' \to 0$, so that Equation (1) holds with $x_i$ replaced by $x'_{i_m}$.
\end{enumerate}
This completes the proof of the claim and thus (as noted earlier), the proof. 
\end{proof}

\section{Strict Singularity of $S\colon \X_k \to \X_k$}

In this section we will prove that $S$ is strictly singular.  By Proposition \ref{SSiffSSonblocks}, it is enough to see that $S$ is not an isomorphism when restricted to any infinite dimensional block subspace $Z$ of $\X_k$. We begin by stating a result taken from the paper of Argyros and Haydon, \cite[Corollary 8.5]{AH}. The reader can check that the same proofs as given in \cite{AH} will also work for the spaces $\X_k$ constructed here.

\begin{lem}\label{ExistRIS}
Let $Z$ be a block subspace of $\X_k$, and let $C>2$ be a real number.
Then $Z$ contains a normalized $C$-RIS.
\end{lem}

We will need a variation of Lemma \ref{RISZeroSpecialExact} to be able to construct weak dependent sequences. We first observe that the lower norm estimate for skipped block sequences given in Proposition 4.8 of \cite{AH} also holds in the space $\X_k$ and exactly the same proof works. We state it for here for convenience:

\begin{lem} \label{AHProp4.8}
Let $(x_r)_{r=1}^a$ be a skipped block
sequence in $\X_k$.  If $j$ is a
positive integer such that $a\le n_{2j}$ and $2j<\min\ran x_2$, then
there exists an element $\gamma$ of weight $m_{2j}^{-1}$ satisfying
\begin{align*}
\sum_{r=1}^a x_r(\gamma)&\ge \half m_{2j}^{-1} \sum_{r=1}^a
\|x_r\|.\\
\intertext{Hence} \|\sum_{r=1}^a x_r\|&\ge \half m_{2j}^{-1}
\sum_{r=1}^a \|x_r\|.
\end{align*}
\end{lem}

\begin{lem} \label{WeakExactPairs}
Let $\ve > 0$, $j$ be a positive integer and let $(x_i)_{i=1}^{n_{2j}}$ be a
skipped-block $C$-RIS, such that $\min\ran x_2 > 2j$. Suppose further that one of the following hypotheses holds:
\begin{enumerate}[(i)]
\item $\|S^{k-1} x_i\|\ge \delta$ for all $i$ (some $\delta > 0$)
\item There is some $2 \leq m \leq k-1$ (where we are, of course, assuming here that $k > 2$) such that $\| S^{m-1} x_i \| \geq \delta $ for all $i$ (some $\delta > 0$) and $\| S^m x_i \| \leq Cm_{2j}^{-1} \ve$.
\end{enumerate}
Then there exists $\eta\in \Gamma$ such that $x(\eta) \geq \frac{\delta}{2}$ where $x$ is the weighted
sum
$$
 x=m_{2j}n_{2j}^{-1}\sum_{i=1}^{n_{2j}}x_i.
$$
Moreover, if hypothesis (i) above holds, the pair $(Sx, \eta)$ is a $(16C, 2j, 0)$-special exact pair. Otherwise, hypothesis (ii) holds and $(Sx, \eta)$ is a $(16C, 2j, 0, \ve)$ weak exact pair.
\end{lem}
\begin{proof}
Let us consider first the case where hypothesis (i) holds. Since $(S^{k-1}x_i)_{i=1}^{n_{2j}}$ is a skipped block sequence, it follows from Lemma \ref{AHProp4.8} that there is an element $\gamma \in \Gamma$ of weight $m_{2j}^{-1}$ satisfying \[
m_{2j}n_{2j}^{-1} \sum_{i=1}^{n_{2j}} S^{k-1}x_i (\gamma) \ge \half n_{2j}^{-1} \sum_{i=1}^{n_{2j}} \|S^{k-1}x_i\| \ge \frac{\delta}{2}
\]
Consequently, we must have $F^{k-1}(\gamma)$ being defined, and $x(F^{k-1}(\gamma)) \ge \frac{\delta}{2}$. We set $\eta = F^{k-1}(\gamma) \in \Gamma$. Certainly $S^jSx(\eta) = \langle x, (S^*)^{j+1} e_{\eta}^* \rangle = 0$ for any $j \geq 0$ (since, by Lemma \ref{G^kalwaysundefined}, we must have $F(\eta) = F^{k}(\gamma)$ being undefined). So conditions 3 and 4 are satisfied for $(Sx, \eta)$ to be a $(16C, 2j, 0)$-special exact pair. The other conditions are satisfied since a careful examination of the corresponding argument in \cite{AH}, i.e. the proof of \cite[Lemma 6.2]{AH},  reveals that the remaining conditions will hold for the weighted sum $x$, regardless of the specific element $\eta$ of weight $m_{2j}^{-1}$ chosen. We then simply make use of the fact that for any $\theta \in \Gamma$ \[
 \langle Sx , e_{\theta}^* \rangle = \begin{cases} \langle x, e_{F(\theta)}^*\rangle & \text{ if $F(\theta)$ is defined} \\  0 & \text{ otherwise } \\ \end{cases}
 \]
 and similarly
 \[
 \langle Sx , d_{\theta}^* \rangle = \begin{cases} \langle x, d_{F(\theta)}^*\rangle & \text{ if $F(\theta)$ is defined } \\ 0 & \text{ otherwise } \end{cases}
 \]
 In the case where hypothesis (ii) holds, we find by the same argument as above that there is a $\gamma \in \Gamma$ of weight $m_{2j}^{-1}$ with $F^{m-1}(\gamma)$ being defined and $x(F^{m-1}(\gamma)) \geq \frac{\delta}{2}$. We now set $\eta = F^{m-1}(\gamma)$. Now, for any $0 \leq j \leq k-1$, either $S^j Sx (\eta) = 0$ (if $F^{j+1}(\eta) = F^{j+m}(\gamma)$ is undefined) or $|S^j Sx(\eta)| = |x(F^{j+m}(\gamma))| = |S^{j+m}x (\gamma)| \leq \| S^{j+m}x \| \leq m_{2j}n_{2j}^{-1}\sum_{i=1}^{n_2j} \| S^{j+m} x_i \| \leq C\ve$. The final inequality here is obtained using the fact that $\|S^{j+m} x_i \| \leq Cm_{2j}^{-1}\ve$ for all $i$ and $j \geq 0$; indeed, when $j = 0$, this is the hypothesis in the lemma, and when $j > 0$, we simply use the fact that $S$ has norm at most $1$. We have shown conditions (3) and (4$'$) hold for $(Sx, \eta)$ to be a $(16C, 2j, 0, \ve)$ weak exact pair. The remaining conditions hold once again because they hold for $x$; the argument is the same as before.
\end{proof}
\begin{thm} \label{SisStrictlySingular}
The operator $S\colon\X_k \to\X_k$ is strictly singular.
\end{thm}
\begin{proof}
We suppose by contradiction that $S$ is not strictly singular. It follows that there is some infinite dimensional block subspace $Y$ of $\X_k$ on which $S$ is an isomorphism, i.e. there is some $0<\delta \leq 1$ such that whenever $y \in Y, \|Sy\| \geq \delta \|y\|$. By Lemma \ref{ExistRIS}, $Y$ contains a normalised skipped block $3$-RIS, $(x_i)_{i\in \N} \subseteq Y$. We note that certainly $Sx_i \arrownot\to 0$ and consider two possibilities. Either $S^{k-1} x_i \to 0$ or it does not. In the latter of these possibilities, passing to a subsequence, we can assume without loss of generality that $\|S^{k-1}x_i\| \geq \nu > 0$ for every $i$ (and some $\nu$). Thus, we see by Lemma \ref{WeakExactPairs} that we can construct $(48, 2j, 0)$ special exact pairs $(Sx, \eta)$ for any $j \in \N$, with $\text{min }\ran Sx$ arbitrarily large and $x(\eta) \geq \frac{\nu}{2}$.

Otherwise, we must have $k > 2$ and there is an $m \in \{ 2, \dots, (k-1) \}$ with $S^{m-1}x_i \arrownot\to 0$ but $S^m x_i \to 0$. By passing to a subsequence, we can assume that $\| S^{m-1}x_i \| \geq \nu$ for all $i$. Moreover, for a fixed $j_0 \in \N$, since $S^mx_i \to 0$, given any $j \in \N$, we can find an $N_j \in \N$ such that $\| S^mx_i \| \leq Cm_{2j}^{-1} n_{2j_0-1}^{-1}$ for every $i \geq N_j$. So by Lemma \ref{WeakExactPairs}, we can construct weak $(48, 2j, 0, n_{2j_0 - 1}^{-1})$ exact pairs $(Sx, \eta)$ for any $j \in \N$, with $\text{min }\ran Sx$ arbitrarily large and $x(\eta) \geq \frac{\nu}{2}$.

Now, we choose $j_0, j_1$ with $m_{2j_0-1} > 6720\delta^{-1}\nu^{-1}$ and $m_{4j_1} > n_{2j_0-1}^{2}$. By Lemma \ref{WeakExactPairs}, and the argument above, there is a $y_1 \in Y, \eta_1 \in \Gamma$ such that $(Sy_1, \eta_1)$ is a $(48, 4j_1, 0, n_{2j_0-1}^{-1})$-weak exact pair and $y_1(\eta_1) \ge \frac{\nu}{2}$. We let $p_1 > \rank \eta_1 \vee \text{max }\ran y_1$ and define $\xi_1 \in \Delta_{p_1}$ to be $(p_1, 0, m_{2j_0-1}, e_{\eta_1}^*)$.

Now set $j_2 = \sigma (\xi_1)$. Again by Lemma \ref{WeakExactPairs} and the argument above, there is $y_2 \in Y, \eta_2 \in \Gamma$ with $\text{min }\ran y_2 > p_1, y_2(\eta_2) \ge \frac{\nu}{2}$ and $(Sy_2, \eta_2)$ a $(48, 4j_2, 0, n_{2j_0-1}^{-1})$-weak exact pair. We pick $p_2 > \rank \eta_2 \vee \text{max } \ran y_2$ and take $\xi_2$ to be the element $(p_2, \xi_1, m_{2j_0-1}, e_{\eta_2}^*)$, noting that this tuple is indeed in $\Delta_{p_2}$. 

Continuing in this way, we obtain a $(48, 2j_0-1,0, n_{2j_0-1}^{-1})$-weak dependent sequence $(Sy_i)$.  By Proposition \ref{0DepSeqUpperEst} we see that \[
\| m_{2j_0-1}n_{2j_0-1}^{-1} \sum_{i=1}^{n_{2j_0-1}} Sy_i \| \le 70\times48 m_{2j_0-1}^{-1} < \frac{\delta\nu}{2} \]
with the final inequality following by the choice of $j_0$. On the other hand, \[
\sum_{i=1}^{n_{2j_0-1}} y_i (\xi_{n_{2j_0-1}}) = m_{2j_0-1}^{-1} \sum_{i=1}^{n_{2j_0-1}} y_i(\eta_i) \geq m_{2j_0-1}^{-1}n_{2j_0-1} \frac{\nu}{2}
\]
So,
\begin{align*}
\| m_{2j_0-1}n_{2j_0-1}^{-1} \sum_{i=1}^{n_{2j_0-1}} Sy_i \| &\geq \delta \|  m_{2j_0-1}n_{2j_0-1}^{-1} \sum_{i=1}^{n_{2j_0-1}} y_i \| \\
&\geq \delta m_{2j_0-1}n_{2j_0-1}^{-1} \sum_{i=1}^{n_{2j_0-1}} y_i (\xi_{n_{2j_0-1}}) \geq \frac{\delta\nu}{2}
\end{align*}
This contradiction completes the proof.
\end{proof}
The previous theorem combined with Theorem \ref{OperatorRepnTheorem} and Corollary \ref{SSareIdeal} immediately yield the following corollary.
\begin{cor}
The operators $S^j \colon \X_k \to \X_k$ ($j \geq 1$) are strictly singular. The spaces $\X_k$ have few but not very few operators. 
\end{cor}

\section{The HI Property}
It only remains to see that the spaces $\X_k$ are hereditarily indecomposable. The proof is sufficiently close to the corresponding proof of \cite{AH} that we will omit most of the details. We first observe we have the following generalisations of Lemmas 8.8 and 8.9 of \cite{AH}.
\begin{lem} \label{1DepSeqLem}
Let $(x_i)_{i\leq n_{2j_0-1}}$ be a $(C, 2j_0-1, 1)-$weak dependent sequence in $\X_k$ and let $J$ be a sub-interval of $[1, n_{2j_0-1}]$. For any $\gamma' \in \Gamma$ of weight $m_{2j_0-1}$ we have \[
\left| \sum_{i\in J} (-1)^{i}x_i(\gamma') \right| \leq 7C.
\]
It follows that \[
\| n_{2j_0-1}^{-1} \sum_{i=1}^{n_{2j_0-1}} x_i \| \geq m_{2j_0-1}^{-1} \quad \text{ but } \quad \| n_{2j_0-1}^{-1}\sum_{i=1}^{n_{2j_0-1}} (-1)^i x_i \| \leq 70C m_{2j_0-1}^{-2}.
\]
\end{lem}
\begin{proof}
The proof of the first claim is sufficiently close to the proof of Lemma \ref{ZeroDepSeqLem} that we omit any more details. The upper bound \[
\| n_{2j_0-1}^{-1}\sum_{i=1}^{n_{2j_0-1}} (-1)^i x_i \| \leq 70C m_{2j_0-1}^{-2}
\]
then follows using exactly the same argument as in Proposition \ref{0DepSeqUpperEst}. The lower bound is proved in the same way as Lemma 8.9 of \cite{AH}; one simply observes that using the notation of Definition~\ref{DepSeq}, 
\begin{align*}
\sum_{i=1}^{n_{2j_0-1}} x_i(\xi_{n_{2j_0 -1}}) &= \langle \sum_{i=1}^{n_{2j_0}-1} x_i , \sum_{i=1}^{n_{2j_0}-1} d_{\xi_i}^* + m_{2j_0-1}^{-1} \sum_{i=1}^{n_{2j_0}-1} P^*_{(p_{i-1}, p_i)} e_{\eta_i}^* \rangle \\
&= m_{2j_0 - 1} \sum_{i=1}^{n_{2j_0-1}} x_i(\eta_i) = n_{2j_0-1} m_{2j_0-1}^{-1}.
\end{align*}
We immediately obtain
$$
\|n_{2j_0-1}^{-1}\sum_{i=1}^{n_{2j_0-1}} x_i\|\ge
n_{2j_0-1}^{-1}\sum_{i=1}^{n_{2j_0-1}} x_i(\xi_{n_{2j_0-1}}) =
m_{2j_0-1}^{-1}.
$$

\end{proof}

To see the spaces $\X_k$ are HI, we claim it will be enough to see that we have the following lemma:
\begin{lem} \label{1WeakPair}
Let $Y$ be a block subspace of $\X_k$. There exists $\delta > 0$ such that whenever $j, p \in \N, \ve > 0$, there exists $q \in \N, x \in Y, \eta \in \Gamma$ with $\ran x \subseteq (p,q)$ and $(x, \eta)$ a $(96\delta^{-1}, 2j, 1, \ve)$ weak exact pair.
\end{lem}
We omit the proof of Lemma \ref{1WeakPair}. It is essentially the same as Lemma \ref{WeakExactPairs} combined with the proof of \cite[Lemma 8.6]{AH}. 

\begin{prop}
$\X_k$ is hereditarily indecomposable.
\end{prop}
\begin{proof}
The argument is the same as in \cite{AH}; by Proposition \ref{equivHIcondition} it is enough to show that if $Y$ and $Z$ are block subspaces of $\X_k$, then for each $\ve > 0$ there exists a $y \in Y$ and $z \in Z$ with $\| y - z \| < \ve \|y + z\|$.

By Lemma \ref{1WeakPair}, given two block subspaces $Y$ and $Z$ of $\X_k$ there exists some $\delta > 0$ such that for all $j_0 \in \N$, we can construct $(96\delta^{-1}, n_{2j_0-1}, 1)-$weak dependent sequences, $(x_i)_{i \leq n_{2j_0-1}}$ with $x_i \in Y$ when $i$ is odd and $x_i \in Z$ when $i$ is even. We choose $j_0 \in \N$ such that $m_{2j_0 - 1} > 6720\delta^{-1}\ve^{-1}$ and obtain a $(96\delta^{-1}, n_{2j_0-1}, 1)-$weak dependent sequence as in the preceding sentence. We define \[
y = \sum_{i \text{ odd}} x_i \text{ and } z = \sum_{i \text{ even} } x_i
\]
and observe that by Lemma \ref{1DepSeqLem} 
\[
\| y + z \| = \| \sum_{i=1}^{n_{2j_0-1}} x_i \| \geq n_{2j_0-1}m_{2j_0-1}^{-1}
\]
while
\[
\| y - z \| = \| \sum_{i=1}^{n_{2j_0-1}} (-1)^{i} x_i \| \leq 70\cdot 96\delta^{-1} n_{2j_0-1} m_{2j_0-1}^{-2}
\]
so that $\| y - z \| < \ve \|y+z\| $ as required.
\end{proof}

%% file: MainResultExt.tex
\chapter{A Banach space with $\ell_1$ Calkin algebra} \label{l1Calk}

\section{The Main Theorem}

In this chapter, we push the results of the preceding chapter a little further. We construct another $\ell_1$ predual, $\X_{\infty}$, which shares some of the properties of the previously constructed spaces $\X_k$; our new space is still an $\mathcal{L}_{\infty}$ space of Bourgain-Delbaen type and we will once again obtain a representation formula for all bounded linear operators on $\X_{\infty}$. However, the space $\X_{\infty}$ also has some radically different properties from the spaces in the previous chapter - the Calkin algebra is isomorphic as a Banach algebra to $\ell_1(\N_0)$,  the space fails to have few operators and moreover, the space is indecomposable, but is not hereditarily indecomposable. Our main result is the following:

\begin{thm} \label{C4TheMainTheorem}
There exists a separable $\mathscr{L}_{\infty}$ space with a Schauder basis, $\X_{\infty}$, that possesses the following properties:
\begin{enumerate}
\item $\X_{\infty}^* = \ell_1$.
\item There exists a non-trivial, non-compact operator `$S$' on $\X_{\infty}$. (By `non-trivial', we simply mean that the operator $S$ is not a scalar multiple of the identity.) The sequences $(S^j)_{j=1}^{\infty} \subset \mathcal{L}(\X_{\infty})$ and $\left( [S^j]\right)_{j=0}^{\infty} \subset \mathcal{L}(\X_{\infty}) / \mathcal{K}(\X_{\infty})$ are basic sequences isometrically equivalent to the canonical basis of $\ell_1(\N_0)$. 
\item If $T \in \mathcal{L}(\X_{\infty})$ then there are unique scalars $(\lambda_i)_{i=0}^{\infty}$ and a compact operator $K \in \mathcal{L}(\X_{\infty})$ with $\sum_{i=0}^{\infty} |\lambda_i| < \infty$ and \[
T = \sum_{i=0}^{\infty} \lambda_i S^i + K.
\]
The operator $S$ appearing in the above sum is the same operator on $\X_{\infty}$ described in Property 2 above. 
\item The Calkin algebra $\mathcal{L}(\X_{\infty})/ \mathcal{K}(\X_{\infty})$ is isometric as a Banach algebra to the convolution algebra $\ell_1(\N_0)$. 
\end{enumerate}
\end{thm}

For clarity, we remark that when we talk of $\ell_1(\N_0)$ as a Banach algebra, we are thinking of the Banach space $\ell_1(\N_0)$ equipped with the usual multiplication coming from convolution. That is, if $(a_n)_{n=0}^{\infty}, (b_n)_{n=0}^{\infty} \in \ell_1(\N_0)$ then the product $(a_n)_{n=0}^{\infty} * (b_n)_{n=0}^{\infty}$ is defined by
\[
\big((a_n) * (b_n)\big) (m) := \sum_{j=0}^m a_j b_{m-j}.
\]

It turns out that possessing $\ell_1$ Calkin algebra is a rather remarkable property. Before proving the main theorem, we note some corollaries that demonstrate this.

\begin{prop} \label{SSimpliesCompactInXinfinity}
If $X$ is a Banach space for which the Calkin algebra is isomorphic (as a Banach algebra) to $\ell_1(\N_0)$ then the strictly singular and compact operators on $X$ coincide. \end{prop}

A particular consequence of the proposition is, of course, that the strictly singular and compact operators on $\X_{\infty}$ coincide. Consequently $\X_{\infty}$ fails to have few operators; the operator $S$ cannot be written as a compact  perturbation of the identity by property (2) of the theorem, and since the compact and strictly singular operators coincide, $S$ cannot be written as a strictly singular perturbation of the identity.  

Before giving the proof of Proposition \ref{SSimpliesCompactInXinfinity}, we document the following easy lemma:

\begin{lem}
Let $a = (a_j)_{j=0}^{\infty}$ and $b =(b_j)_{j=0}^{\infty}$ be non-zero vectors in $\ell_1(\N_0)$ and choose $k, l \geq 0$ minimal such that $a_k \neq 0$ and $b_l \neq 0$. Then $a * b (j)$ is equal to $0$ whenever $j < k + l$. When $j = k+l$, $a*b(j) = a_kb_l$. In particular, $\| a * b \|_{\ell_1(\N_0)} \geq |a_kb_l |$ and consequently $\|a^n \|_{\ell_1(\N_0)} \geq |a_k|^n$. 
\end{lem}

\begin{proof}
The proof is an easy computation using the definition of the convolution product.
\end{proof}

\begin{proof}[Proof of \ref{SSimpliesCompactInXinfinity}]
If $T$ is strictly singular then by Theorem \ref{SSimpliesspecradius0} \[
\lim_{n \to \infty} \| [T]^n \|^{\frac1n} =  \lim_{n \to \infty} \| T^n + \mathcal{K}(X) \|^{\frac1n} = 0
\]Since the Calkin algebra is isomorphic as a Banach algebra to $\ell_1(\N_0)$, we can identify $[T]$ with an element $a = (a_n)_{n=0}^{\infty} \in \ell_1(\N_0)$ and the above observation yields that $\lim_{n\to\infty} \| a^n \|^{\frac1n}  = 0$. By the previous lemma, if $a \neq 0$, then letting $k$ be minimal such that $a_k \neq 0$, $\|a^n\|_{\ell_1} \geq |a_k|^n$, so that $\lim_{n\to\infty} \| a^n \|^{\frac1n} \geq |a_k| > 0$. So we must have that $a = 0$, and therefore that $[T] = 0$. This is the same as saying that $T$ is compact. 

We have therefore shown that every strictly singular operator is compact. The converse is always true. 
\end{proof}

We can also easily deduce from the previous proposition that the space $\X_{\infty}$ fails to be hereditarily indecomposable.

\begin{prop}
The space $\X_{\infty}$ is not hereditarily indecomposable.
\end{prop}

\begin{proof}
It was shown in \cite{Fer} that for a real HI space $X$, the quotient algebra $\mathcal{L}(X) / \mathcal{S}\mathcal{S} (X) $ has dimension at most $4$ (where, as usual, $\mathcal{S}\mathcal{S} (X) $ denotes the ideal of strictly singular operators on $X$). By the previous proposition, $\mathcal{S}\mathcal{S} (\X_{\infty}) = \mathcal{K}(\X_{\infty})$, and so we see that $\mathcal{L}(\X_{\infty}) / \mathcal{S}\mathcal{S} (\X_{\infty}) $ is isomorphic to $\ell_1(\N_0)$ by the final property of the main theorem. Combining these observations, it is clear that $\X_{\infty}$ is not HI. 
\end{proof}

On the other hand, the fact that the Calkin algebra is $\ell_1(\N_0)$ forces the space to be indecomposable. 

\begin{prop}
If $X$ is a Banach space for which the Calkin algebra is isomorphic (as a Banach algebra) to $\ell_1(\N_0)$ then $X$ is indecomposable. 
\end{prop}

\begin{proof}
If $P$ is a projection on $X$, then $[P] $ is an idempotent in the Calkin algebra. We identify $[P]$ with $(a_i)_{i=0}^{\infty} \in \ell_1(\N_0)$ It follows that $(a_i)_{i=0}^{\infty}$ is an idempotent in the Banach algebra $\ell_1(\N_0)$. An easy computation yields that the only idempotents in $\ell_1(\N_0)$ are $0$ and $(1, 0, 0, \dots)$. Thus either $[P] = 0$ or $[I-P] = 0$, i.e. either $P$ or $I-P$ is compact. In either case, $P$ is certainly a trivial projection. 
\end{proof}

\section{The Construction}

The construction of $\X_{\infty}$ is almost identical to that used in the previous chapter to obtain the spaces $\Xk$, though the proof of the operator representation for the new space is somewhat more involved. We continue to work with the fast increasing sequences of  natural numbers $(m_j)_{j=1}^{\infty}$ and $(n_j)_{j=1}^{\infty}$ from the previous chapter and refer the reader back to Assumption \ref{mnAssump} for the precise assumptions on these sequences. 

We begin by working with a Bourgain-Delbaen space generated by a countable set $\Upsilon = \cup_n \Lambda_n$ which is essentially the same as in the construction of the previous chapter except for the addition of some new elements which will have degenerate corresponding BD functionals. Continuing with the terminology from the previous chapter, we will define these additional elements as having age and weight 0 and refer to them as `trivial elements' since their corresponding BD functionals will just be $0$. Let us be more precise.

\begin{defn}
We define $\Upsilon$ by the recursion $\Lambda_1 = \{ 0 \}$,
\begin{align*}
\Lambda_{n+1} &=\bigcup_{j=1}^{n+1} \left\{(n+1,p,m_j^{-1},b^*):  0\leq p < n,\ b^*\in B_{p,n}\right\}\\
&\cup \bigcup_{p=1}^{n-1}
\bigcup_{j=1}^p\left\{(n+1,\xi,m_j^{-1},b^*): \xi\in \Lambda_p,
\weight \xi = m_j^{-1},\ \age\xi<n_j,\ b^*\in B_{p,n}\right\} \\
&\cup \{ (n+1, j) : 0 \leq j \leq n \}
\end{align*}
\end{defn}

The notation is exactly as before, whereby elements of $\Upsilon$ code the corresponding BD functional. For a trivial element, $\gamma = (n+1, j) \in \Lambda_{n+1}$,  we define the corresponding $c^*_{\gamma}$ to be $0$ and so $d_{\gamma}^* = e_{\gamma}^*$. It is readily checked that Theorem \ref{BDThm} still holds with the addition of this degenerate case (indeed, within the notation of the theorem, we could just as well consider $\tau(\gamma)$ to be the Type 0 tuple, $\tau(\gamma) = (1,0, 0)$, where the first coordinate sets $\ve = 1$, the second coordinate denotes the 0 scalar, and the third coordinate denotes the element $0 \in \Lambda_1$). Consequently we obtain a Bourgain-Delbaen space $X(\Upsilon)$. 

As in the previous chapter, we bring to the attention of the reader the same remark that was made in the previous chapter (see below). We once again ask the reader to consult Chapter \ref{MainResult} for an explanation of the notation being used.

\begin{rem} \label{C4notationBpn}
\item Later on, we will want to take a suitable subset of $\Upsilon$. To avoid any ambiguity in notation, in the above definition, and throughout the rest of the chapter, $B_{p,n}$ will denote the set of all linear combinations $b^* = \sum_{\eta \in \Upsilon_n \setminus \Upsilon_p} a_{\eta}e_{\eta}^*$, where, as before, $\sum_{\eta} |a_{\eta}| \leq 1$ and each $a_{\eta}$ is a rational number with denominator dividing $N_n!$.
\end{rem}

We have the following theorem, which is the analogue of Theorem \ref{R^*andGConstruction} and Lemma \ref{DualOfR^*RestrictsProperly}.

\begin{thm} \label{C4R^*andG}
There is a map $G: \Upsilon \to \Upsilon \cup \{ \text{undefined} \}$ (we say \emph{$G(\gamma)$ is undefined} if $G(\gamma) = \text{undefined}$, otherwise we say \emph{$G(\gamma)$ is defined}) and a norm 1, linear mapping $R^* \colon \ell_1(\Upsilon) \to \ell_1(\Upsilon)$ satisfying:
\begin{enumerate}
\item $G(0)$ is undefined (where we recall $\Lambda_1 = \{ 0\}$).
\item $G\left( (n+1, j) \right) = (n+1, j-1)$ whenever $1 \leq j \leq n$ and $G\left( (n+1, 0) \right)$ is undefined.
\item For non-trivial elements $\gamma \in \Upsilon\setminus\Lambda_1$, if $G(\gamma)$ is defined, $\rank\gamma = \rank G(\gamma)$ and  $\weight \gamma = \weight G(\gamma)$ (i.e. G preserves weight and rank). Moreover, $\age G(\gamma) \leq \age \gamma$ ($G$ doesn't increase age).
\item \begin{equation*} R^*(e_{\gamma}^*) = \begin{cases} e_{G(\gamma)}^* & \text{ if $G(\gamma)$ is defined} \\ 0 & \text{ otherwise.} \end{cases} \end{equation*}
\item \begin{equation*} R^*(d_{\gamma}^*) = \begin{cases} d_{G(\gamma)}^* & \text{ if $G(\gamma)$ is defined} \\ 0 & \text{ otherwise.} \end{cases}\end{equation*}
\end{enumerate}

Moreover, the dual operator of $R^*$ restricts to give a bounded linear operator $R\colon X(\Upsilon) \to X(\Upsilon)$ of norm at most 1.
\end{thm}

We omit the proof, since it is essentially the same as the proof given for the corresponding results in the previous chapter. The following observation about the operator $R^*$ will be essential later in the proof.

\begin{lem} \label{nilpotentproperty}
For every $\gamma \in \Upsilon$, there is a unique $l\in\N$, $1\leq l \leq \rank \gamma$ such that $G^j(\gamma)$ is defined whenever $1\leq j < l$ but $G^{l}(\gamma)$ undefined. Consequently, $(R^*)^j e_{\gamma}^* = (R^*)^j d_{\gamma}^* = 0$ whenever $j \in \N, j \geq \rank \gamma$. 
\end{lem}

\begin{proof}
The proof is an easy modification of the proof of Lemma \ref{G^kalwaysundefined}.
\end{proof}

As with the constructions in \cite{AH} and the previous chapter, for our main construction, we will place some restrictions on the elements of odd-weight via the use of a coding function $\sigma\colon \Upsilon \to \N$. We make the same assumptions on this coding function as in the previous chapter; we refer the reader back to Section \ref{SigmaAssumptions} for the details. As before, the coding function can be constructed recursively as $\Upsilon$ is constructed. For each $\gamma \in \Upsilon$, we define a finite set $\Sigma(\gamma)$ by \[
\Sigma(\gamma) := \{ \sigma (\gamma) \} \cup \bigcup_{j=1}^{\infty} \bigcup_{\delta\in G^{-j} (\gamma) } \{ \sigma(\delta) \}
\]
where $G^{-j}(\gamma) := \{ \theta \in \Upsilon : G^{j}(\theta) = \gamma \}$. (In particular, if $\theta \in G^{-j}(\gamma)$ then $G^l(\theta) \in \Upsilon$ for every $l \leq j$.) 

It is easily checked that Lemmas \ref{SigmaGammaContainedinSigmaGGamma} and \ref{MonotonicityOfSigmaSets} still hold with (essentially) the same proofs and that we have the following modification of Lemma \ref{IntersectionPropertyofSigmaSets}.

\begin{lem}\label{C4IntersectionPropertyofSigmaSets}

Suppose $\gamma, \delta \in \Upsilon$. If $\sigma(\gamma) \in \Sigma(\delta)$ then either $\gamma = \delta$ or there is some $1 \leq j < \infty$ such that $G^j(\gamma) = \delta$.
\end{lem}

We are now in a position to describe the main construction.  The ideas are exactly the same as in Chapter \ref{MainResult}. We take a subset $\Gamma \subset \Upsilon$ by placing restrictions on the elements of odd weight we permit. As a consequence of imposing these additional odd weight restrictions, we are also forced to (roughly speaking) remove those elements $(n+1, p, \beta, b^*)$ and $(n+1, \xi, \beta, b^*)$ of $\Upsilon$ for which the support of $b^*$ is not contained in $\Gamma$, in order that we can apply the Bourgain-Delbaen construction to obtain a space $X(\Gamma)$. Exactly as was the case in Chapter \ref{MainResult}, the subset $\Gamma$ is also constructed inductively, so it will be well defined and consistent with the Bourgain-Delbaen method.  

We will denote by $\Delta_n$ the set of all elements in $\Gamma$ having rank $n$, and denote by $\Gamma_n$ the union $\Gamma_n = \cup_{j\leq n} \Delta_j$. The permissible elements of odd weight will be as follows. For an age 1 element of odd weight, $\gamma = (n+1, p, m_{2j - 1}^{-1}, b^*)$ we insist that either $b^* = 0$ or $b^* = e_{\eta}^*$ where $\eta \in \Gamma_n \setminus \Gamma_p$ and $\weight\eta = m_{4i}^{-1} < n_{2j-1}^{-2}$. 
For an odd weight element of age $> 1$, $\gamma = (n+1, m_{2j - 1}^{-1}, \xi, b^*)$ we insist that either $b^* = 0$ or $b^* = e_{\eta}^*$ where $\eta \in \Gamma_n \setminus \Gamma_{\rank\xi}$ and $\weight\eta = m_{4k}^{-1} < n_{2j-1}^{-2}$, $k \in \Sigma(\xi)$. Let us be more precise:

\begin{defn} \label{C4DefnOfGammaAndSpace}

We define recursively sets $\Delta_{n} \subseteq \Lambda_n$. Then $\Gamma_n := \cup_{j \leq n} \Delta_j$ and $\Gamma := \cup_{n\in\N}\Delta_n \subseteq \Upsilon$. To begin the recursion, we set $\Delta_1 = \Lambda_1$. Then

\begin{align*}
\Delta_{n+1} &= \bigcup_{j=1}^{\lfloor(n+1)/2\rfloor}
\left\{(n+1,p, m_{2j}^{-1},b^*): 0\leq p<n,\ b^*\in B_{p,n}\cap \ell_1(\Gamma_n) \right\}\\ &\cup
\bigcup_{p=1}^{n-1}\bigcup_{j=1}^{\lfloor
p/2\rfloor}\left\{(n+1,\xi,m_{2j}^{-1},b^*): \xi\in \Delta_p,
\text{w}( \xi) = m_{2j}^{-1},\ \age\xi<n_{2j},\
b^*\in B_{p,n}\cap \ell_1(\Gamma_n) \right\}\\
&\cup\bigcup_{j=1}^{\lfloor(n+2)/2\rfloor}
\left\{(n+1,p, m_{2j-1}^{-1},b^*):
0\leq p < n,\  b^*=0 \text{ or } b^*=e_{\eta}^* ,\ \eta\in \Gamma_n\setminus\Gamma_p \right. \\
& \left. \hspace{8cm}\text{ and } \text{w}(\eta)= m_{4i}^{-1}<n_{2j-1}^{-2}\right\}\\
&\cup  \bigcup_{p=1}^{n-1}\bigcup_{j=1}^{\lfloor(
p+1)/2\rfloor}\left\{(n+1,\xi,m_{2j-1}^{-1},b^*): \xi\in
\Delta_p, \text{w}( \xi) = m_{2j-1}^{-1},\ \age\xi<n_{2j-1} \right.,\\  &\left.
\qquad\qquad\qquad\qquad b^*=0 \text{ or } b^*=e_{\eta}^* \text{ with } \eta\in\Gamma_n\setminus\Gamma_p,\ \text{w}(\eta)=
m_{4k}^{-1}<n_{2j-1}^{-2},\ k \in \Sigma(\xi) \right\}  \\
&\cup \{ (n+1, j) : 0 \leq j \leq n \}.
\end{align*}
Here the $B_{p,n}$ are defined as in Remark \ref{C4notationBpn}, and, for brevity, we temporarily write $\text{w} (\xi)$ for $\weight\xi$. We define $\X_{\infty}$ to be the Bourgain-Delbaen space $X(\Gamma)$ where $\Gamma$ is the subset of $\Upsilon$ just defined.
\end{defn}

Once again, to compute norms we need to work with the evaluation functionals $e_{\gamma}^*$. We note that this is no harder than in the previous constructions.

\begin{prop}
For a $\gamma \in \Gamma$, either it is a trivial, age 0 element ($\gamma = (n+1, j)$) in which case we simply have $e_{\gamma}^* = d_{\gamma}^*$, or $\age \gamma \geq 1$ in which case we once again have the evaluation analysis given by Proposition \ref{EvalAnal}. 
\end{prop}

\begin{proof}
If $\gamma$ is a trivial, age 0 element, there is nothing to prove; the statement follows from our earlier definitions. In the case where $\gamma \in \Gamma$ and $\age \gamma \geq 1$, the same proof as before goes through. 
\end{proof}

We still need to show that there exists an operator `$S$' on $\X_{\infty}$ having all the properties described in the main theorem. We first define the operator. It is constructed in exactly the same way as in Proposition \ref{S^*construction}.

\begin{prop} \label{C4S^*construction}
$\Gamma$ is invariant under $G$. More precisely, if $\, \gamma \in \Gamma \subseteq \Upsilon$ and $G(\gamma)$ is defined, then $G(\gamma) \in \Gamma$. It follows that the map $G\colon \Upsilon \to \Upsilon\cup\{\text{undefined}\}$ defined in Theorem \ref{C4R^*andG} restricts to give $F \colon \Gamma \to \Gamma\cup\{\text{undefined}\}$. Consequently the map $R^* \colon \ell_1(\Upsilon) \to \ell_1(\Upsilon)$ (also defined in \ref{C4R^*andG}) can be restricted to the subspace $\ell_1(\Gamma) \subseteq \ell_1(\Upsilon)$ giving $S^*\colon\ell_1(\Gamma) \to \ell_1(\Gamma)$. $S^*$ is a bounded linear map on $\ell_1(\Gamma)$ of norm 1 which satisfies 
\[
S^*e_{\gamma}^* = \begin{cases} 0 & \text{ if $F(\gamma)$ is undefined} \\ e_{F(\gamma)}^* & \text{ otherwise} \end{cases}
\]
for every $\gamma \in \Gamma$. Moreover, the dual operator of $S^*$ restricts to $\X_{\infty}$ to give a bounded linear operator $S \colon \X_{\infty} \to \X_{\infty}$ of norm at most 1, satisfying $S^j \neq 0$ for all $j \in \N$.
\end{prop}

\begin{proof}
We once again omit the details, as the proof is only a very minor modification of the proof of the original result. Note that $S^j d_{(n+1,0)} = d_{(n+1, j)}$ whenever $0\leq j \leq n$ so that $S^j \neq 0$ for any $j \in \N$. 
\end{proof}

The reader can easily verify that Lemma \ref{monotonicity_of_odd_weights} of the previous chapter still holds for the odd-weight elements of our new construction. Moreover, we can define a $C$-RIS in the space $\X_{\infty}$ using {\em exactly} the same definition as in Chapter \ref{MainResult}; we refer the reader to Definition \ref{RISDef}. It is then an easy exercise to check that Lemmas and Propositions \ref{AHProp5.6}, \ref{RIStoBlock} and \ref{Xhasl1Dual} all still hold, with only minor modifications required to the original proofs to take into account the new, weight and age 0 elements. In particular, we note that the basis $(d_{\gamma})_{\gamma \in \Gamma}$ of the space $\X_{\infty}$ is shrinking and $\X_{\infty}^*$ is naturally isomorphic to $\ell_1$. Moreover, having established that $\X_{\infty}^*$ is naturally isomorphic to $\ell_1$, the same proof as in Lemma \ref{imSClosed} works to show that the dual operator of $S \colon \X_{\infty} \to \X_{\infty}$ is precisely the operator $S^* \colon \ell_1(\Gamma) \to \ell_1(\Gamma)$ of Proposition \ref{C4S^*construction}.

The strategy used to prove our main result will be to follow the arguments of Gowers and Maurey in \cite{GM}. We will thus be interested in the seminorm $\| |\cdot  |\|$  on $\mathcal{L}(\X_{\infty})$, defined by \[
\| | T | \| := \sup_{(x_n)_{n\in\N} \in \mathcal{L}} \limsup_n \| Tx_n \|
\]
where $\mathcal{L}$ denotes the set of all 1-RIS, $(x_n)_{n\in\N}$.

We first document some observations about this seminorm. 

\begin{lem} \label{C4sn=0iffcompact}
Given $T \in \mathcal{L}(\X_{\infty})$, $T$ is compact $\iff \| | T | \| = 0$.
\end{lem}

\begin{proof}
Suppose first that $T$ is compact. Since the basis $(d_{\gamma})_{\gamma\in\Gamma}$ is shrinking, it follows by Proposition \ref{CompactLemma} that if $(x_n)_{n\in\N} \in \mathcal{L}$, $\|Tx_n\| \to 0$. It follows immediately that $\| |T|\| = 0$.

Conversely, suppose $\| |T| \| = 0$. If $(x_n)_{n\in\N}$ is a C-RIS, then $(\frac{x_n}{C})_{n\in\N} \in \mathcal{L}$ so $\|Tx_n\| = C \|\frac{Tx_n}{C} \| \to 0$. It follows from Proposition \ref{RIStoBlock} (and our earlier observation that this result still holds for the space $\X_{\infty}$), that whenever $(y_n)_{n\in\N}$ is a block sequence in $\X_{\infty}$, $\|Ty_n\| \to 0$. This implies that $T$ is compact, as we saw in Proposition \ref{CompactLemma}.

\end{proof}

\begin{lem}\label{snofSlookslikel1}
Let $T = \sum_{j=0}^{N} \lambda_j S^j$. Then $\| T\| = \sum_{j=0}^N |\lambda_j |$ and moreover $\| | T | \| \geq \frac{1}{6} \sum_{j=0}^N |\lambda_j |$.
\end{lem}

\begin{proof}
We already know $\| S \| \leq 1$; it follows from the triangle inequality that $\| T \| \leq \sum_{j=0}^{N} |\lambda_j|$. Moreover, it is an elementary fact that $\|T \| = \| T^* \| = \| \sum_{j=0}^N \lambda_j (S^*)^j \|$. Since $(N+1, N) \in \Gamma$, we thus have,
\begin{align*}
\|T \| = \| \sum_{j=0}^N \lambda_j (S^*)^j \| &\geq \|\sum_{j=0}^N \lambda_j (S^*)^j e_{(N+1, N)}^* \|_{\ell_1} \\
&= \| \sum_{j=0}^N \lambda_j e_{(N+1, N-j)}^* \|_{\ell_1} = \sum_{j=0}^N |\lambda_j|
\end{align*}
and we obtain the first equality stated in the lemma.

To obtain the final inequality stated in the lemma, note that whenever $A \subseteq \Delta_q$, and $\ve_{\gamma} \in \{ 1, -1 \}$ (for $\gamma \in A$),  $\sum_{\gamma \in A}\ve_{\gamma}  d_{\gamma} = i_q \left( \sum_{\gamma \in A} \ve_{\gamma} e_{\gamma} \right)$; this follows immediately from Lemma \ref{i_nlem}. Consequently, $\| \sum_{\gamma\in A} \ve_{\gamma} d_{\gamma} \| \leq 2$. Note that $(N+n, N-j) \in \Gamma$ whenever $n \geq 1$ and $0 \leq j \leq N$. We consider the sequence $(x_n)_{n=1}^{\infty}$ defined by \[
x_n = \frac{1}{6} \sum_{j=0}^N \text{sign}\, \lambda_j \, d_{(N+n, N-j)}
\]
We claim $(x_n) \in \mathcal{L}$. By our first observation, $\| x_n \| \leq \frac13 \leq 1$ for all $n$. Moreover, we see that the local support (see \cite[Definition 5.7]{AH}) of any $x_n$ contains no element of weight $> 0$. Thus by Lemma 5.8 of \cite{AH}, whenever $\gamma \in \Gamma$ and $\weight \gamma = m_{h}^{-1}$ for some $h$, $|x_n(\gamma)| \leq m_h^{-1}$. It follows that $(x_n)_{n \in \N}$ is a 1-RIS and is therefore in $\mathcal{L}$. (For the sequence $(j_k)$ appearing in the definition of a RIS, we take $j_1 = 1$ and $j_{k+1} = \max \ran x_k + 1$ for $k \geq 1$.) 

Finally, we note that
\begin{align*}
\| Tx_n \| &\geq \langle e_{(N+n, N)}^* , \sum_{j=0}^N \lambda_j S^j x_n \rangle = \langle d_{(N+n, N)}^* , \sum_{j=0}^N \lambda_j S^j x_n \rangle = \langle \sum_{j=0}^N \lambda_j (S^*)^j d_{(N+n, N)}^* , x_n \rangle  \\
&= \frac{1}{6} \sum_{j=0}^N \sum_{k=0}^N \lambda_j \text{sign}\, \lambda_k \, \langle d_{(N+n, N-j)}^* , d_{(N+n, N-k)} \rangle \\
&= \frac{1}{6}\sum_{j=0}^N |\lambda_j|
\end{align*}
from which it follows that $\| | T | \| \geq \limsup_n \|Tx_n\| \geq \frac{1}{6}\sum_{j=0}^N |\lambda_j|$ as required.
\end{proof}

We aim to prove the following theorem:

\begin{thm} \label{snclosureLem}
Every operator in $\mathcal{L}(\X_{\infty})$ is in the $\| | \cdot | \|$-closure of the set of operators in the algebra $\mathcal{A}$, generated by the Identity and $S$.
\end{thm}

As the proof is long and technical, we first see that our main result, Theorem \ref{C4TheMainTheorem}, almost completely follows from Theorem \ref{snclosureLem} and the above lemmas. 

%\begin{thm}\label{C4MainThm}
%The sequence $(S^j)_{j=0}^{\infty}$ is a basic sequence in $\mathcal{L}(\X)$ which is 1-equivalent to the canonical basis of $\ell_1 = \ell_1(\N_0)$. Moreover, $\mathcal{L}(\X) = \mathcal{K}(\X) \oplus [S^j : j \in \N_0]$. It follows that $\X$ has Calkin algebra isomorphic to $\ell_1$. 
%\end{thm}

\begin{proof}[Proof of Theorem \ref{C4TheMainTheorem}]
We have already observed that $\X_{\infty}^*$ is isomorphic to $\ell_1$; indeed, the proof is almost identical to the proof for the spaces $\X_k$.

We note that the first part of Lemma \ref{snofSlookslikel1} is equivalent to saying that $(S^j)_{j=0}^{\infty}$ is a basic sequence in $\mathcal{L}(\X_{\infty})$, 1-equivalent to the canonical basis of $\ell_1 = \ell_1(\N_0)$. Moreover, it follows from the second part of Lemma \ref{snofSlookslikel1} and Lemma \ref{C4sn=0iffcompact} that $\mathcal{K}(\X_{\infty}) \oplus [S^j : j\in \N_0]$ is an algebraic direct sum in $\mathcal{L}(\X_{\infty})$. We first show that this direct sum is the entire operator algebra. The proof of this is identical to the arguments used in \cite{GM} but for completeness we include it here.

We write $\mathcal{G}$ for the Banach space obtained by taking the $\| | \cdot | \|$-completion of $\mathcal{A}$, noting that by the second part of Lemma \ref{snofSlookslikel1}, $\| | \cdot | \|$ actually defines a norm on $\mathcal{A}$. (Here $\mathcal{A}$ is defined as in Theorem \ref{snclosureLem}.) Given $T \in \mathcal{L}(\X_{\infty})$, by Theorem \ref{snclosureLem} we can find a $\| | \cdot | \|$-Cauchy sequence $(T_n)_{n=1}^{\infty}$ of operators in $\mathcal{A}$ such that $\| | T - T_n | \| \to 0$. Let $\phi(T)$ be the limit of $(T_n)_{n=1}^{\infty}$ in $\mathcal{G}$. It is easily checked that, in this way, we get a well-defined, bounded linear map $\phi \colon \mathcal{L}(\X_{\infty}) \to \mathcal{G}$ (of norm at most $1$). Note that the restriction of $\phi$ to $\mathcal{A}$ is the identity (or more accurately, the embdedding of $\mathcal{A}$ into $\mathcal{G}$). It is also easily seen that the kernel of $\phi$ is the set of $T$ such that $\| | T | \| = 0$, which by Lemma \ref{C4sn=0iffcompact} is precisely the compact operators, $\mathcal{K}(\X_{\infty})$. 

Now, we note that we can define a linear isomorphic embedding $\mathcal{I} \colon \mathcal{G} \to \mathcal{L}(\X_{\infty})$, with $\mathcal{I}(\mathcal{G}) = [S^j : j\in \N_0]$. Precisely, we define $\mathcal{I}$ by defining it on the dense subspace of vectors in $\mathcal{G}$ of the form  $ [\sum_{j=0}^N \lambda_j S^j] $, where it is to be understood that $  [\sum_{j=0}^N \lambda_j S^j] $ is the vector in $\mathcal{G}$ obtained under the natural embedding $\mathcal{A} \hookrightarrow \mathcal{G}$ of the vector $\sum_{j=0}^N \lambda_j S^j$ (recalling that $\mathcal{G}$ is simply the $\| | \cdot | \|$-completion of $\mathcal{A}$). We can therefore well-define $\mathcal{I}$ in the obvious way, i.e. $[\sum_{j=0}^N \lambda_j S^j] \mapsto \sum_{j=0}^{N} \lambda_j S^j$.  By the second part of Lemma \ref{snofSlookslikel1}, this is a continuous linear map, so we may take the unique continuous linear extension to $\mathcal{G}$. Since $\| | \cdot | \| \leq \| \cdot \|$, it is easily seen that this map is an isomorphic embedding onto the closed linear span $[S^j : j \in \N_0] $, as required.

%Now, we note that by the second part of Lemma \ref{snofSlookslikel1} (and that $\| | \cdot | \| \leq \| \cdot \|$), we can define a linear isomorphic embedding $\mathcal{I} \colon \mathcal{G} \to \mathcal{L}(\X_{\infty})$, with $\mathcal{I}(\mathcal{G}) = [S^j : j\in \N_0]$. Precisely, we define $\mathcal{I}$ by defining it on the dense subspace of equivalence classes in $\mathcal{G}$ of the form  $ [ \sum_{j=0}^N \lambda_j S^j ] $, noting that Lemma \ref{snofSlookslikel1} implies that this representation of these equivalence classes is unique. We can therefore well-define $\mathcal{I}$ in the obvious way, i.e. $[\sum_{j=0}^N \lambda_j S^j] \mapsto \sum_{j=0}^{N} \lambda_j S^j$, and then take the unique continuous linear extension to $\mathcal{G}$. 

We observe that $\mathcal{I} \circ \phi (T) - T$ is in the kernel of $\phi$ and is therefore compact. Indeed, it is clear from the definitions of $\phi$ and $\mathcal{I}$ that $\phi \circ \mathcal{I}$ is the identity on the dense subspace of vectors in $\mathcal{G}$ of the form  $ [ \sum_{j=0}^N \lambda_j S^j ] $. Using this observation and continuity, we see that
\begin{align*}
\phi \left( \mathcal{I}\circ\phi(T) - T \right) &= \phi \circ \mathcal{I} \left( \phi(T) \right) - \phi (T) \\
&= \phi(T) - \phi(T) = 0
\end{align*}
as required. Consequently, if $T \in \mathcal{L}(\X_{\infty})$ there exists a compact operator $K$ such that $T = \mathcal{I} \circ \phi(T) +K$ and we have \[
\mathcal{L}(\X_{\infty}) = [S^j : j \in \N_0 ] \oplus \mathcal{K}(\X_{\infty})
\]
as claimed.

Property (3) of Theorem \ref{C4TheMainTheorem} now follows immediately from the above and the fact that $[S^j]_{j=0}^{\infty}$ is a basic sequence in $\mathcal{L}(\X_{\infty})$. Moreover, if we let $\textfrak{e} \colon \ell_1(\N_0) \to \mathcal{L}(\X_{\infty})$ denote the isometric embedding $(\lambda_j)_{j=0}^{\infty} \mapsto \sum_{j=0}^{\infty} \lambda_j S^j$ and let $\pi \colon \mathcal{L}(\X_{\infty}) \to \mathcal{L}(\X_{\infty}) / \mathcal{K}(\X_{\infty})$ denote the canonical projection, then the map $\textfrak{I} \colon \ell_1(\N_0) \to \mathcal{L}(\X_{\infty}) / \mathcal{K}(\X_{\infty})$ defined by $\textfrak{I} = \pi \circ \textfrak{e}$ is now easily seen to be a continuous bijection; it is therefore an isomorphism by the Inverse Mapping Theorem. It is readily checked that this map is a Banach algebra isomorphism. The fact that the sequence $([S^j])_{j=0}^{\infty} \subset \mathcal{L}(\X_{\infty}) / \mathcal{K}(\X_{\infty})$ is a basic sequence also follows from the fact that $\textfrak{I}$ is an isomorphism.

We have very nearly completed the proof. However, it is claimed in property (4) that $\mathcal{L}(\X_{\infty}) / \mathcal{K}(\X_{\infty})$ is actually {\em isometric} to $\ell_1(\N_0)$. We will show in the next section that the sequence $[S^j]_{j=0}^{\infty}$ is really $1$-equivalent to the canonical basis of $\ell_1$, from which it follows that $\textfrak{I}$ is really an isometry.
\end{proof}

It remains for us to prove Theorem \ref{snclosureLem}, i.e. that every operator in $\mathcal{L}(\X_{\infty})$ is in the $\| | \cdot | \|-$closure of the set of operators in the algebra $\mathcal{A}$ generated by the identity operator and $S$. In order to do so, we will need to make further modifications to the definition of a $(C,j,0, \ve)$ weak exact pair (see Definition \ref{SpecialExactPair}) and also to the definition of weak dependent sequences (see Definition \ref{DepSeq}). Since these definitions are quite complicated, we will include all the details in our new definitions, although most of the conditions we demand are as before. Moreover, unlike in Chapter \ref{MainResult}, we won't need $1$-exact pairs and dependent sequences in what follows. Consequently we make a very small change in notation to the previous chapter.

\begin{defn} \label{C4SpecialExactPair}

Let $C, \ve >0$ and $\omega \in \N\cup\{ \infty \}$. A pair $(x,\eta)\in
\X_{\infty} \times \Gamma$ is said to be a $(C,j, \ve, \omega)$-{\em weak exact pair of order $\omega$}
if
 \begin{enumerate}
 \item $\|x\| \leq C$
 \item $|\langle d^*_\xi,x\rangle|\le Cm_{j}^{-1}$ for all $\xi\in \Gamma$
 \item $\weight \eta = m_{j}^{-1}$
 \item $|S^l x(\eta)| \leq C\ve$ for every $0 \leq l \leq \omega$ if $\omega < \infty$ and $|S^l x(\eta)| \leq C\ve$ for every $0 \leq l < \omega$ if $\omega = \infty$ 
 \item for every element $\eta'$ of
$\Gamma$ with $\weight\eta'= m_{i}^{-1}\ne m_j^{-1}$, we have
$$
|x(\eta')| \le \begin{cases} Cm_i^{-1} &\text{ if } i<j\\
                 Cm_{j}^{-1} &\text{ if }i>j.
                            \end{cases}
$$
\end{enumerate}
\end{defn}
As before, we observe that we have the following estimates for weak exact pairs (for exactly the same reasons as in \cite{AH} and as in the previous chapter).
\begin{remark}
A $(C,j, \ve, \omega)$ weak exact pair also satisfies the estimates
$$
|\langle e^*_{\eta'},P_{(s,\infty)}x\rangle| \le \begin{cases} 6Cm_i^{-1} &\text{ if } i<j\\
                 6Cm_{j}^{-1} &\text{ if }i>j
                            \end{cases}
$$
for elements $\eta'$ of $\Gamma$ with $\weight \eta'=m_i^{-1}\ne
m_j^{-1}$.
\end{remark}
We will need the following method for constructing 0 weak exact pairs of order $N$.
\begin{lem}\label{C4RISZeroSpecialExact}
Let $\omega \in \N\cup\{ \infty \}, (x_k)_{k=1}^{n_{2j}}$ be a skipped-block $C$-RIS, and let
$q_0<q_1<q_2<\cdots<q_{n_{2j}}$ be  natural numbers such that $q_1 \geq 2j$ and $\ran
x_k\subseteq (q_{k-1},q_k)$ for all $k$. Let $z$ denote the weighted
sum $z=m_{2j}n_{2j}^{-1}\sum_{k=1}^{n_{2j}}x_k$. For each $k$ let
$b^*_k$ be an element of $B_{q_{k-1},q_{k}-1}$ with $|\langle
(S^*)^l b^*_k,x_k\rangle|= |\langle b^*_k , S^l x_k \rangle| \leq C \ve$ for all $0\leq l \leq \omega$ (or $0 \leq l < \infty$ if $\omega = \infty$).Then there exist $\zeta_i \in \Delta_{q_i}$
($1\le i\le n_{{2j}}$) such that the element $\eta=\zeta_{n_{2j}}$
has analysis $(q_i,b^*_i,\zeta_i)_{1\le i\le n_{2j}}$ and $
(z,\eta)$ is a $(16C,{2j},\ve, \omega)$-special exact pair.
\end{lem}
The proof of the lemma is sufficiently close to the proof of Lemma \ref{RISZeroSpecialExact} that we omit it.

\begin{defn}\label{C4DepSeq} We
shall say that a sequence $(x_i)_{i\le n_{2j_0-1}}$ in $\X_{\infty}$ is a $(C,
2j_0-1,0)$-{\em weak dependent sequence (of finite order)} if there exist sequences of natural numbers
$0=p_0<p_1<p_2<\cdots<p_{n_{2j_0-1}}$ and $1 = N_0 < N_1 < N_2 < \cdots < N_{n_{2j_0-1}} <\infty$ together with
$\eta_i\in\Gamma_{p_i-1}\setminus \Gamma_{p_{i-1}}$ and $\xi_i\in
\Delta_{p_i}$ ($1\le i\le n_{2j_0-1}$) such that
\begin{enumerate}
\item for each $k$, $\ran x_k\subseteq (p_{k-1},p_k)$
 \item for each $k$, $S^j x_k (\eta_k) = 0$ whenever $j \geq N_k$
 \item the
element $\xi=\xi_{n_{2j_0-1}}$ of $\Delta_{p_{n_{2j_0-1}}}$ has weight
$m^{-1}_{2j_0-1}$ and analysis
$(p_i,e^*_{\eta_i},\xi_i)_{i=1}^{n_{2j_0-1}}$
\item $(x_1,\eta_1)$ is a $(C,4j_1, n_{2j_0-1}^{-1}, N_{0})$-weak exact pair \item
for each $2\le i\le n_{2j_0-1}$, $(x_i,\eta_i)$ is a
$(C,4\sigma(\xi_{i-1}), n_{2j_0-1}^{-1}, N_{i-1})$-weak exact pair.
\end{enumerate}
In certain applications, we can remove the more complicated conditions involving the sequence $(N_j)_{j=0}^{n_{2j_0-1}}$ Specifically, we will say the sequence $(x_i)_{i\le n_{2j_0-1}}$ is a $(C, 2j_0-1,0)$-{\em weak dependent sequence (of infinite order)} if we remove condition (2) above, and ask that all the weak exact pairs occurring in (4) and (5) have infinite order.

In either case, we notice that, because of the special odd-weight conditions, we necessarily have $m^{-1}_{4j_1} = \weight \eta_1
<n^{-2}_{2j_0-1}$, and $\weight \eta_{i+1} =m^{-1}_{4\sigma(\xi_i)} < n^{-2}_{2j_0-1}$,
by Lemma \ref{monotonicity_of_odd_weights} for $1\le i<n_{2j_0-1}$. 
\end{defn}

The following two lemmas are fundamental to proving our main result.

\begin{lem}\label{C4ZeroDepSeqLem}
Let $(x_i)_{i\le n_{2j_0-1}}$ be a weak $(C, 2j_0-1,0)$-{dependent
sequence} (of finite or infinite order) in $\X_{\infty}$ and let $J$ be a sub-interval of
$[1,n_{2j_0-1}]$. For any $\gamma'\in \Gamma$ of weight $m_{2j_0-1}$
we have
$$
|\sum_{i\in J} x_i(\gamma')| \le 7C.
$$
\end{lem}
\begin{proof}
The proof is almost the same as the proof of Lemma \ref{ZeroDepSeqLem}. We recall the notation used; we let $\xi_i,\eta_i,p_i, N_i$ (if we are considering a dependent sequence of finite order)$ , j_1$ be as in the definition of a dependent
sequence and let $\gamma$ denote $\xi_{n_{2j_0-1}}$, an element of
weight $m_{2j_0-1}$. We denote by $\big( p_0', (p_i',b'^*_i,\xi'_i)_{1\le i\le
a'} \big)$ the analysis  of $\gamma'$ where each $b'^*_i$ is either $0$ or $e_{\eta'_i}^*$ for some suitable $\eta'_i$.

We recall that we can estimate as follows
\[
|\sum_{k\in J} x_k(\gamma')| \leq |\sum_{k\in J,\ k<l} x_k(\gamma')|
+ |x_l(\gamma')| +
\sum_{k\in J,\ k>l}|x_k(\gamma')|,
\]
where we are assuming we have chosen $l$ maximal such that there exists $i$ with $b'^{*}_{i} = e_{\eta'_i}^*$ and $\weight\eta_i'= \weight\eta_l$. (We refer the reader back to the proof of Lemma \ref{ZeroDepSeqLem} for the details on reducing to this case.) The terms  $|x_l(\gamma')|$  and $\sum_{k\in J,\ k>l}|x_k(\gamma')|$ can be estimated in exactly the same way as in the proof of Lemma \ref{ZeroDepSeqLem}. 

The argument for estimating the remaining term, $|\sum_{k\in J,\ k<l} x_k(\gamma')|$, is a little more involved. Obviously if $l=1$, the sum we are trying to estimate is just zero, and the lemma is proved. So we can suppose $l>1$. By definition of $l$, there exists some $i$ such that $b'^*_i = e_{\eta'_i}^*$ and $\weight\eta_l = \weight\eta'_i$. Now either $i = 1$ or $i > 1$. When $i =1$, we can use exactly the same argument as used for the corresponding case in the proof of Lemma \ref{ZeroDepSeqLem} to conclude that $|\sum_{k\in J,\ k<l} x_k(\gamma')| \leq 3C$ and we are done.

It only remains to consider what happens when $i > 1$. Recall we are also assuming $l > 1$ and $\weight\eta'_i = \weight\eta_l$. But by definition of an exact pair, we have $\weight\eta_l = m_{4\sigma(\xi_{l-1})}^{-1}$, and by restrictions on elements of odd weights, $\weight\eta'_i = m_{4\omega}^{-1}$ with $\omega \in \Sigma(\xi'_{i-1})$. By strict monotonicity of the sequence $m_j$, we deduce that $\omega = \sigma(\xi_{l-1}) \in \Sigma(\xi'_{i-1})$. By Lemma \ref{C4IntersectionPropertyofSigmaSets}, there is some $j$, $0 \leq j < \infty$, such that $F^j(\xi_{l-1}) = \xi'_{i-1}$. We note that in particular this implies $p_{l-1} = p'_{i-1}$ since $F$ preserves rank and we can write the evaluation analysis of $\gamma'$ as \[
e_{\gamma'}^* = (S^*)^j(e^*_{\xi_{l-1}}) + \sum_{r=i}^{a'} d_{\xi'_r}^* + m_{2j_0-1}^{-1} P_{(p'_{r-1},p'_r)}^*b'^*_r.
\]
Now, for $k<l$, since $\ran x_k \subseteq (p_{k-1}, p_k) \subseteq (0,p_{l-1}) = (0, p'_{i-1})$, we see that 
\begin{align*}
|\langle x_k, e_{\gamma'}^* \rangle| &= |\langle x_k, (S^*)^je_{\xi_{l-1}}^* \rangle| \\
&= |\langle S^jx_k, \sum_{r=1}^{l-1} d_{\xi_r}^* + m_{2j_0-1}^{-1}P^*_{(p_{r-1}, p_r)} e_{\eta_r}^* \rangle| \\
&= m^{-1}_{2j_0-1} |\langle S^jx_k, e_{\eta_k}^* \rangle| \leq  |\langle S^jx_k, e_{\eta_k}^* \rangle|.
\end{align*}
Now, if $(x_i)_{i \leq n_{2j_0-1}}$ is a dependent sequence of infinite order, we have $|S^lx_k(\eta_k)| \leq Cn_{2j_0-1}^{-1}$ for every $l \in \N_0$, so certainly $|\langle S^jx_k, e_{\eta_k}^* \rangle| \leq Cn_{2j_0-1}^{-1}$ and consequently, 
\[
|\sum_{k \in J,\ k<l}x_k(\gamma ')| \leq n_{2j_0-1} \max_{k \in J,\  k<l} |x_k(\gamma')| \leq C.
\]
Otherwise $(x_i)_{i\leq n_{2j_0-1}}$ is a dependent sequence of finite order, and we estimate by considering three possibilities. If $j < N_1 \leq N_{k-1}$ for all $k\geq 2$, then by definition of weak dependent sequence, $|\langle S^j x_k, \eta^*_k\rangle| \leq Cn_{2j_0-1}^{-1}$ for $k \geq 2$. So, 
\[
|\sum_{k \in J,\ k<l}x_k(\gamma ')|  \leq |x_1(\gamma')| +  n_{2j_0-1} \max_{k \in J,\  2\leq k<l} |x_k(\gamma')| \leq 2C.
\]
If $j > N_{n_{2j_0-1}} \geq N_k$ for all $k$, then condition (2) of weak dependent sequence implies that $\langle S^jx_k, e_{\eta_k}^* \rangle = 0$ and so $|\sum_{k \in J,\ k<l}x_k(\gamma ')| = 0$.

The only remaining possibility is that there is some $r \in \{ 1,2, \dots, n_{2j_0-1}-1 \}$ with $N_r \leq j \leq N_{r+1}$. Now if $k \leq r$, it follows that $j \geq N_r \geq N_k$ and so $S^jx_k (\eta_k) = 0$ by condition (2) of weak dependent sequence. If $k \geq r+2$, $j \leq N_{r+1} \leq N_{k-1}$ and so $|\langle S^j x_k, e_{\eta_k}^* \rangle | \leq Cn_{2j_0-1}^{-1}$.  It follows that
\begin{align*}
|\sum_{k \in J,\ k<l}x_k(\gamma ')|  & \leq |x_{r+1} (\gamma')| + |\sum_{k \in J,\ k<l,\ k\leq r }x_k(\gamma ')|  +   |\sum_{k \in J,\ k<l,\ k\geq r+2 }x_k(\gamma ')|  \\
& \leq \|x_{r+1}\| + 0 + n_{2j_0-1} \max_{k \in J,\ k<l,\ k\geq r+2}|\langle S^jx_k, e_{\eta_k}^* \rangle| \leq 2C.
\end{align*}

This completes the proof.

\end{proof}

\begin{lem}\label{C4estimatesofdepsequences}
If $(x_i)_{1 \leq i \leq n_{2j_0-1}}$ is a $(C, 2j_0-1, 0)$ weak dependent sequence (of finite or infinite order), then 
$$
 \|n_{2j_0-1}^{-1}\sum_{k=1}^{n_{2j_0-1}}x_k\|\le 70Cm_{2j_0-1}^{-2}.
$$
and moreover, if $\gamma \in \Gamma, \weight (\gamma) = m_{h}^{-1}$, then \[
|n_{2j_0-1}^{-1}\sum_{k=1}^{n_{2j_0-1}} x_k(\gamma)|\le
\begin{cases} 28Cn_{2j_0-1}^{-1}+84Cm_{2j_0-1}^{-2}m_h^{-1} &\text{ if }h<2j_0-1\\
              28Cn_{2j_0-1}^{-1} + 42Cm_h^{-1} &\text{ if } h> 2j_0-1.
\end{cases}
\]
\end{lem}
\begin{proof}
The proof follows from Lemma \ref{C4ZeroDepSeqLem} (readjusting the RIS constant to 7C), and by the proof of \cite[Proposition 5.6]{AH}.
\end{proof}

In the spirit of \cite{GM}, we aim to prove the following proposition. 

\begin{prop} \label{GM1stLem}
Let $T\colon\X_{\infty} \to \X_{\infty}$ be a bounded linear operator. Then for every $\delta > 0$, there exists an $l\in \N$ such that whenever $x \in \X_{\infty}$ is a block vector satisfying $\|x\| \leq 1, \min \ran x > l$ and $|x(\gamma)| \leq m_{i}^{-1}$ whenever $\weight \gamma = m_{i}^{-1}$ with $i < l$, 
\[
\dist (Tx, l\,\conv \{ \lambda S^jx : j \in \{0, 1, \dots, l \}, |\lambda| = 1 \} ) \leq \delta.
\]
\end{prop}

In order to prove the above proposition, we will need the following lemmas,  which are analogous to Lemma 7.1 of \cite{AH} and Lemma \ref{DistExact}.

\begin{lem} \label{C4RatVecs}
Let $m, n$ and $l$ be natural numbers with $m < n$ and let $x, y \in \X_{\infty}$ be such that $\ran x, \ran y$ are both contained in the
interval $ (m,n]$. Suppose that $\dist(y, l\,\conv \{\lambda S^jx : j \in \{0, 1, \dots, l \}, |\lambda| = 1 \}  )>\delta$.  Then
there exists $b^*\in \ball\ell_1(\Gamma_n\setminus \Gamma_m)$, with
rational coordinates, such that $|\langle S^jx , b^* \rangle | \leq \frac{\|y\|}{l}$ for every $j \in \{0, 1, \dots, l \}$ and
$b^*(y)>\frac12\delta$.
\end{lem}
\begin{proof}
For $0 \leq j \leq l$, let $u_j \in \ell_\infty(\Gamma_n\setminus \Gamma_m)$ be the restriction of $S^jx$, and let  $v\in \ell_\infty(\Gamma_n\setminus \Gamma_m)$ be the
restriction of $y$.  Then $S^jx=i_nu_j$ for $0\leq j \leq l$, $y=i_nv$ and
for any scalars $(\mu_j)_{j=0}^{l}$, with $\sum_{j=0}^{l} |\mu_j| \leq l$, $\delta < \|y-\sum_{j=0}^{l}\mu_j S^jx\|\le \|i_n\|\|v-\sum_{j=0}^{l}\mu_j
u_j\|$.  Hence \[
\dist(v, l\, \conv\{\lambda u_j : j\in \{0,1,\dots, l\}, | \lambda | = 1\} )>\frac12\delta.
\]
Let $\mathcal{C}_l  = \conv\{\lambda u_j : j\in \{0,1,\dots, l\}, | \lambda | = 1 \}$. By the Hahn--Banach Theorem (see Theorem \ref{HBSeparationThm}) in the finite dimensional space
$\ell_\infty(\Gamma_n\setminus \Gamma_m)$, there exists $a^*\in
\text{sphere}\,\ell_1(\Gamma_n\setminus \Gamma_m)$ with \[
 \sup \left\{ \langle a^*, z \rangle : z \in C_l + \frac12\delta B(\ell_{\infty}(\Gamma_n \setminus \Gamma_m) \right\} < \langle a^*, v \rangle.
\]
It follows that $\langle a^*, v \rangle > \frac12\delta$ and $\sup |a^*(\mathcal{C}_l)| < \sup |a^*(\mathcal{C}_l)| +\frac12\delta < \langle a^*, v \rangle \leq \|v\| \leq \|y \|$.  Consequently, $|\langle a^*, u_j \rangle | < \frac{\|y\|}{l}$ for every $0 \leq j \leq l$.  To complete the proof, we note that by an elementary continuity argument, we may approximate $a^*$ arbitrarily well by $b^* \in \ball\ell_1(\Gamma_n \setminus \Gamma_m)\cap\Q^{\Gamma_n\setminus \Gamma_m}$ retaining the desired conditions. 
\end{proof}
The proof of Proposition \ref{GM1stLem} will be easy once we have proved the following lemma. 
\begin{lem}\label{GM1stLemFails}
Let $T$ be a bounded linear operator on $\X_{\infty}$. Suppose for this $T$, there exists some $\delta > 0$ for which the conclusion of Proposition \ref{GM1stLem} fails.  Then, given any $\ve >0$ and natural numbers, $j, p, N$, there exists a block vector $z\in \X_{\infty}$, a natural number $q > p$, and an $\eta \in \Delta_q$ such that $\ran z \subseteq (p, q)$ and 
\begin{enumerate}[(1)]
 \item $(z,\eta)$ is a $(16,2j, \ve, N)$-exact pair;
 \item $(Tz)(\eta)>\frac7{16}\delta$;
 \item $\|(I-P_{(p,q)})Tz\|<m_{2j}^{-1}\delta$;
 \item $\langle P^*_{(p,q]}e^*_{\eta},Tz\rangle >\frac{3}{8}\delta$.
 \end{enumerate}
\end{lem}

\begin{proof}
Note that under the hypothesis of the lemma, we can obtain a skipped block 1-RIS $(x_k)_{k\in\N}$ with the property that for every $k$ \[
\dist (Tx_k, k\,\conv \{ \lambda S^jx_k : j \in \{0, 1, \dots, k \}, |\lambda| = 1 \} > \delta.
\]
We now fix $L \in \N$ such that $L \geq N$ and $L > 4\|T\| / \ve$. Observe that if $m < n$ and $x \in \X_{\infty}$, then $m\,\conv \{ \lambda S^jx : j \in \{0, 1, \dots, m \}, |\lambda| = 1 \} \subseteq n\,\conv \{ \lambda S^jx : j \in \{0, 1, \dots, n \}, |\lambda| = 1 \}$ so, for $k \geq L$,  
\begin{align*}
\dist( Tx_k, L\,\conv \{ &\lambda S^jx_k : j \in \{0, 1, \dots, L \}, |\lambda| = 1 \}  ) \\
& \geq \dist (Tx_k, k\,\conv \{ \lambda S^jx_k : j \in \{0, 1, \dots, k \}, |\lambda| = 1 \} > \delta.
\end{align*}
We pass to the subsequence $(x_k)_{k \geq L}$ so that the above inequality holds. Note also that the sequence $(Tx_k)$ is weakly null (since the basis $(d_{\gamma})_{\gamma\in\Gamma}$, is shrinking).  Consequently, passing to a further subsequence if necessary, we can assume that there exist
$p<q_0<q_1<q_2<\dots$ such that, for all $k\ge 1$, $\ran
x_k\subseteq (q_{k-1},q_k)$ and
$\|(I-P_{(q_{k-1},q_k)})Tx_k\|<\frac15m_{2j}^{-2}\delta\le\frac1{80}m_{2j}^{-1}\delta\le
\frac1{1280}\delta.$ It follows from this that
$\mathrm{dist}(P_{(q_{k-1},q_k)}Tx_k, L\,\conv \{ \lambda S^jx_k : j \in \{0, 1, \dots, L \}, |\lambda| = 1 \}) >\frac{1279}{1280}\delta$. We may  apply Lemma \ref{C4RatVecs} to
obtain
$b^*_k\in\ball\ell_1(\Gamma_{q_k-1}\setminus\Gamma_{q_{k-1}})$, with
rational coordinates, satisfying
$$
|\langle b^*_k, S^jx_k\rangle | \leq \frac{\|P_{(q_{k-1},q_k)}Tx_k\|}{L} \leq \frac{4\|T\|}{L} < \ve
$$
whenever $0 \leq j \leq L$, and \[
\langle b^*_k,
P_{(q_{k-1},q_k)}Tx_k\rangle>\txtfrac{1279}{2560}\delta.
\]
Taking a further subsequence if necessary (and redefining the
$q_k$), we may assume that the coordinates of $b^*_k$ have
denominators dividing $N_{q_k-1}!$, so that $b^*_k\in
B_{q_{k-1},q_k-1}$, and we may also assume that $q_1\ge 2j$. We are
thus in a position to apply Lemma \ref{C4RISZeroSpecialExact}, getting
elements $\xi_k$ of weight $m_{2j}^{-1}$ in $\Delta_{q_k}$ such that
the element $\eta=\xi_{{n_{2j}}}$ of $\Delta_{q_{n_{2j}}}$ has
evaluation analysis
$$
e^*_\eta =
\sum_{i=1}^{n_{2j}}d^*_{\xi_i}+m_{2j}^{-1}\sum_{i=1}^{n_{2j}}P^*_{(q_{i-1},q_i)}b^*_i.
$$
and such that $(z,\eta)$ is a $(16,2j, \ve, L)$-weak exact pair, where $z$
denotes the weighted average
$$
z=m_{2j}n_{2j}^{-1}\sum_{i=1}^{n_{2j}}x_i.
$$
Since we also chose $L \geq N$, $(z, \eta)$ is certainly a $(16, 2j, \ve, N)$-weak exact pair as required. We set $q = q_{n_{2j}}$; the rest of the proof is now exactly the same as in \cite[Lemma 7.2]{AH}. 

\end{proof}

\begin{proof}[Proof of Proposition \ref{GM1stLem}]
We suppose for contradiction that there exists $\delta > 0$, such that the conclusion of the Proposition \ref{GM1stLem} fails.

The argument is the same as Proposition 7.3 of \cite{AH}. We construct a $(16, 2j_0-1,0)$-weak dependent sequence of finite order (for a suitably chosen $j_0 \in \N$) by making repeated applications of Lemma \ref{GM1stLemFails}.

We begin by choosing $j_0$ such that $m_{2j_0-1} > 4480\|T\| \delta^{-1}$ and $j_1$ such that $m_{4j_1}>n_{2j_0-1}^2$. Taking $p=p_0=0$ and $j=2j_1, N = N_0=1, \ve = n_{2j_0-1}^{-1}$ in
Lemma \ref{GM1stLemFails}, we can find $q_1$ and a $(16,4j_1,n_{2j_0-1}^{-1}, N_0)$-exact
pair $(z_1,\eta_1)$ with $\rank \eta_1=q_1$,
$(Tz_1)(\eta_1)>\frac38\delta$ and
$\|(I-P_{(0,q_1]})(Tz_1)\|<m_{4j_1-2}^{-1}\delta$. Let $p_1=q_1+1$
and let $\xi_1$ be the special Type  1 element of $\Delta_{p_1}$
given by  $\xi_1 = (p_1,m_{2j_0-1}^{-1}, e^*_{\eta_1})$.  By Lemma \ref{nilpotentproperty}, we can find an $N_1 > N_0$ such that $S^lz_1(\eta_1) = 0$ whenever $l \geq N_1$.

Now, recursively for $2\le i\le n_{2j_0-1}$, define
$j_{i}=\sigma(\xi_{i-1})$,
 and use Lemma \ref{GM1stLemFails} again to choose $q_{i}$ and a
$(16,4j_{i}, n_{2j_0-1}^{-1}, N_{i-1})$-exact pair $(z_{i},\eta_{i})$ with $\rank
\eta_i=q_i$, $\ran z_i\subseteq (p_{i-1},q_i)$, $\langle
P^*_{(p_{i-1},q_i]}e^*_{\eta_i},Tz_i\rangle>\frac38\delta$ and
$\|(I-P_{(p_i,q_i]})(Tz_i)\|<m_{4j_i}^{-1}\delta$.
 We now define $p_i=q_i+1$, choose $N_i$ such that $S^lx_i(\eta_i) = 0$ whenever $l \geq N_i$ (which is again possible by \ref{nilpotentproperty}), and let $\xi_i$ be the Type 2 element
$(p_i,\xi_{i-1},m_{2j_0-1}^{-1}, e^*_{\eta_i})$ of $\Delta_{p_i}$.

It is clear that we have constructed a $(16,2j_0-1,0)$-weak dependent
sequence $(z_i)_{1\le i\le n_{2j_0-1}}$. We set $z = n_{2j_0-1}^{-1} \sum_{i=1}^{n_{2j_0-1}} z_i$ and note that $\|z\| \leq 70\cdot 16 m_{2j_0-1}^{-2}$ by Lemma \ref{C4estimatesofdepsequences}. On the other hand, the same argument as in the proof of Proposition 7.3 of \cite{AH} yields that \[
\| Tz \| > \frac{m_{2j_0-1}^{-1} \delta}{4} \geq \frac{\delta m_{2j_0-1} \| z \| }{4\cdot 70 \cdot 16 }
\]
which implies that $\| T\| > \frac{\delta m_{2j_0-1} }{4480}$, contradicting $m_{2j_0-1} > 4480 \| T\| \delta^{-1}$.

 \end{proof}

\begin{prop} \label{GMLem2}
Let $T$ be a bounded linear operator on $\X_{\infty}$, $\delta > 0$. For this $\delta$, choose $l \in \N$ as given by Proposition \ref{GM1stLem} and let $\mathcal{A}_{l} = l\, \conv\{ \lambda S^j : j\in \{0, 1, \dots, l \}, |\lambda | = 1 \}$. Then, there exists $U \in \mathcal{A}_l$ with $\| |T - U| \| \leq 28,674\delta$.
\end{prop}

For the proof, we will need the following fixed-point theorem due to Kakutani. We refer the reader to \cite{Kak} for the proof and more details. 

\begin{thm}[Kakutani Fixed Point Theorem] \label{Kakutani}
Let $\mathcal{S}$ be an r-dimensional closed simplex and denote by $\mathcal{R}(\mathcal{S})$ the set of all non-empty, closed, convex subsets of $\mathcal{S}$. If $\Phi \colon \mathcal{S} \to \mathcal{R}(\mathcal{S})$ is upper semi-continuous, then there exists an $x_0 \in \mathcal{S}$ such that $x_0 \in \Phi(x_0)$. 
\end{thm}

We remind the reader of the following definition from \cite{Kak}:

\begin{defn}
With the notation as in the preceding theorem, a point-to-set mapping $\Phi \colon \mathcal{S} \to \mathcal{R}(\mathcal{S})$ is called upper semi-continuous if whenever $(x_n)_{n=1}^{\infty}$ is a sequence in $\mathcal{S}$ converging to $x_0 \in \mathcal{S}$, $y_n \in \Phi(x_n)$ and $y_n \to y_0 \in \mathcal{S}$, then $y_0 \in \Phi(x_0)$.
\end{defn}

\begin{proof}[Proof of \ref{GMLem2}]
The proof follows closely that of \cite[Lemma 9]{GM}. We find it convenient to introduce the following piece of notation. Given a sequence of block vectors $(x_i)_{i\in\N}$, we write $x_1 < x_2 < \dots $ if there exist natural numbers $q_0 < q_1 < \dots$ with $\ran x_i \subseteq (q_{i-1}, q_i )$ for every $i$ (with similar notation if the sequence has only finite length).

It is sufficient to show that if $T$ is an operator on $\X_{\infty}$, with matrix representation with respect to the basis $(d_{\gamma})_{\gamma\in\Gamma}$ possessing the property that there are only finitely many non-zero entries in each row and column, then there is $U \in \mathcal{A}_l$ with $\| | T - U | \| \leq 28,673 \delta$. (We temporarily call this property P so that we can refer back to this statement.) Indeed, since the basis $(d_{\gamma})_{\gamma\in\Gamma}$ is shrinking, $T$ can be perturbed by a distance at most $\delta$ in operator norm to an operator of this form, and the seminorm, $\| | \cdot | \|$, is bounded above by the operator norm.

Having performed the described perturbation, it follows that we may assume $T$ sends block vectors (with respect to the FDD) to block vectors. Note we can also assume (without loss of generality) that if $N\in\N$ is given, whenever $x\in\X_{\infty}$ is a block vector with $\min \ran x$ sufficiently large, $\min \ran Tx > N$. Indeed, if this were not the case, an elementary argument involving the fact that the matrix of $T$ is assumed to have only finitely many non-zero entries in each row and column would imply that $T$ is in fact a finite-rank, hence compact, operator. We could thus take $U = 0$ and the proof would be complete by Lemma \ref{C4sn=0iffcompact}. 

Now, if property P is false, then for every $U \in \mathcal{A}_l$, there exists a 1-RIS $(x_U(i))_{i\in\N}$ with $\limsup_i \|(T - U)x_U(i)\| > 28,673\delta$. Since subsequences of 1-RIS are again 1-RIS, passing to a subsequence if necessary, we see that for every $U \in \mathcal{A}_l$, there exists a skipped block 1-RIS $(x_U(i))_{i\in\N}$ such that $\|(T-U)x_U(i)\| > 28,673\delta$ for all $i$.

Note that $\mathcal{A}_l$ is compact. Indeed,  consider the (continuous) linear map $h: \ell_1^{l+1} \to \mathcal{L}(\X_{\infty})$, defined by $(\lambda_0, \lambda_2, \dots, \lambda_l) \mapsto \sum_{j=0}^{l} \lambda_jS^j$. Then the restriction of $h$ to $l\, B(\ell_1^{l+1})$ is a homeomorphism onto $\mathcal{A}_l$. We can thus choose a covering $(\mathcal{U}_j)_{j=1}^{k}$ of $\mathcal{A}_l$ by open sets of diameter less than $\delta$. Let $(\phi_j)_{j=1}^{k}$ be a partition of unity on $\mathcal{A}_l$ with $\phi_j$ supported inside $\mathcal{U}_j$ for each $j$.

For every $j = 1, \dots, k$, let $U_j \in \mathcal{U}_j$ and let $(x_{j,i})_{i\in\N}$ be a skipped block 1-RIS with the above property with $U = U_j$. By the condition on the diameter of $\mathcal{U}_j$, we have $\|(T - U) x_{j,i} \| > 28,672\delta$ for all $i \in \N$ and $U \in \mathcal{U}_j$. Moreover, by the remarks made earlier (and the fact that $\ran x \subset (p,q) \implies \ran Sx \subseteq (p,q)$), passing to a subsequence of $(x_{j,i})_{i\in\N}$ if necessary, we may also assume that $\left( (T-U)x_{j,i} \right)_{i\in\N}$ is a skipped block sequence of successive vectors for every $U \in \mathcal{A}_l$.

For the rest of the proof,  we work with $j_0 \in \N$ chosen large enough that
\begin{enumerate}[(1)]
\item $2j_0 - 1 > l$
\item $m_{2j_0}n_{2j_0-1}^{-1} \leq m_{l}^{-1}$
\item $kn_{2j_0-1} \leq n_{2j_0}$.
\end{enumerate}
Note that it is certainly possible to choose such a $j_0$. Indeed, by Assumptions (2) and (4) on the sequences $(m_j)$ and $(n_j)$ (see Assumption \ref{mnAssump}), $ m_{2j}n_{2j-1}^{-1} = m_{2j-1}^{2} n_{2j-1}^{-1} \leq n_{2j-2}^{-1} \to 0$ as $j \to \infty$ and $ n_{j+1}n_j^{-1} \geq m_{j+1}^2 \to \infty$ as $j \to \infty$. 

Now, suppose we are given $\ve > 0$, $p, r \in \N$. We can select a skipped block 1-RIS of length $n_{2r}$, $(x_{j, i})_{i = M+1}^{M+n_{2r}}$ ($M$ suitably chosen) from the 1-RIS $(x_{j,i})_{i \in \N}$, such that there are natural numbers $p < q_0 < q_1 < \dots < q_{n_{2r}}$ with $q_1 \geq 2r$ and $\ran x_{j, M+i} \subseteq (q_{i-1}, q_i)$.  Now, setting $b_i^* = 0 \in B_{q_{i-1}, q_i}$ for all $i$, we can apply Lemma \ref{C4RISZeroSpecialExact} to see that there exists $q = q_{n_{2r}}$ and a $(16, 2r, \ve, \infty)$-weak exact pair $(z, \eta)$ where $z$ is the weighted sum \[
z = m_{2r}n_{2r}^{-1} \sum_{i=1}^{n_{2r}} x_{j, M + i}
\]
and $\eta \in \Delta_q$. Note also that $\ran z \subseteq (p, q)$ and whenever $U \in \mathcal{U}_j$,
\begin{align*}
\| (T-U) z \| &= m_{2r}n_{2r}^{-1} \| \sum_{i=1}^{n_{2r}} (T-U)x_{j, M + i} \| \\
&\geq \frac12 m_{2r}^{-1} m_{2r}n_{2r}^{-1} \sum_{i=1}^{n_{2r}} \|(T-U)x_{j, M + i}\| > 14,336 \delta
\end{align*}
where the penultimate inequality follows from Lemma \ref{AHProp4.8} and the assumption we made earlier about $\left( (T-U)x_{j,i} \right)_{i\in\N}$ being a skipped block sequence of successive vectors for every $U \in \mathcal{A}_l$.  

Using this observation, for each $j = 1, \dots, k$, we may inductively construct a $(16, 2j_0-1, 0)$ weak dependent sequence of infinite order, $(z_{j,i})_{i=1}^{n_{2j_0-1}}$, with $\min \ran z_{j, 1} > l$ and $\| (T - U)z_{j,i}  \| > 14,336 \delta$ whenever $U \in \mathcal{U}_j, 1\leq i \leq n_{2j_0-1}$. Moreover, once again making use of the assumptions on $T$ discussed at the beginning of the proof, we may arrange that \[
(T-U)z_{j, 1} < (T-U)z_{j, 2} , \dots < (T-U)z_{j, n_{2j_0-1}} 
\]
for every $j$ and $U \in \mathcal{A}_l$ and moreover that \[
(T-U)z_{j, n_{2j_0-1}} <  (T-U)z_{j+1, 1}
\]
for every $1 \leq j < k$, and $U \in \mathcal{A}_l$. Now, for $j = 1, \dots, k$, we set $y_j = \frac{1}{1792} m_{2j_0}n_{2j_0-1}^{-1}\sum_{i=1}^{n_{2j_0-1}} z_{j, i}$. Observe that $\min \ran y_j > l$. Moreover, by Lemma \ref{C4estimatesofdepsequences}, we have $\| y_j \| \leq \frac{1}{1792} m_{2j_0} \times 70\times 16 m_{2j_0-1}^{-2} = \frac{1120}{1792}$ and if $\gamma \in \Gamma, \weight \gamma = m_{h}^{-1}$ with $h <  2j_0-1$, 
\begin{align*}
|y_j(\gamma)| &\leq \frac{28\times16}{1792}m_{2j_0}n_{2j_0-1}^{-1} + \frac{84\times16}{1792}m_{2j_0}m_{2j_0-1}^{-2}m_{h}^{-1} \\
&\leq \frac{448}{1792} m_{l}^{-1} + \frac{1344}{1792} m_{h}^{-1}
\end{align*}
by the choice of $j_0$. Since $j_0$ was chosen so that $l < 2j_0-1$, if $h < l <  2j_0-1$, $m_{l}^{-1} \leq m_{h}^{-1}$ and we have $|y_j(\gamma)| \leq m_{h}^{-1}$. 

For $U \in \mathcal{A}_l$, let $y(U) = \sum_{j=1}^{k} \phi_j (U)y_j$. We have $\min \ran y(U) > l, \|y(U) \| \leq \frac{1120}{1792} \leq 1$ and $|y(U)(\gamma) | \leq m_h^{-1}$ whenever $\gamma \in \Gamma$ and $ \weight \gamma = m_h^{-1}$ with $h < l$ (since we have just observed these facts are true when $y(U)$ is replaced by any $y_j$). We show that $y(U)$ is also a `bad vector' for $U$, by showing that $\|(T-U)y(U) \| > 4\delta$. 

To see this, let $U \in \mathcal{A}_l$ be fixed, and let $J = \{ j : \phi_j(U) > 0 \}$, noting that $\| (T-U)z_{j, i} \| > 14,336 \delta$ whenever $j \in J$ and $1 \leq i \leq n_{2j_0-1}$. Observe that $(T-U)y(U) = \sum_{j \in J} \phi_j (U) (T-U) y_j = \frac{1}{1792} m_{2j_0}n_{2j_0-1}^{-1} \sum_{j\in J}\sum_{i=1}^{n_{2j_0-1}} \phi_j (U) (T-U)z_{j,i}$. By the way in which we constructed the $z_{j,i}$, the final sum is a sum of at most $kn_{2j_0-1} \leq n_{2j_0}$ skipped, successive vectors. It follows that 
\begin{align*}
 \| \sum_{j\in J}\sum_{i=1}^{n_{2j_0-1}} \phi_j (U) (T-U)z_{j,i} \| &\geq \frac12 m_{2j_0}^{-1}  \sum_{j\in J}\sum_{i=1}^{n_{2j_0-1}} \| \phi_j (U) (T-U)z_{j,i} \| \\
 &> m_{2j_0}^{-1} n_{2j_0-1} 7168\delta \sum_{j \in J} |\phi_j (U)| \\
 &= m_{2j_0}^{-1} n_{2j_0-1} 7168\delta
 \end{align*}
 and consequently, $\|(T - U)y(U) \| > 4\delta$ as claimed.
 
The proof of the proposition now follows by the Kakutani Fixed Point Theorem, Theorem \ref{Kakutani}. Indeed, for each $U \in \mathcal{A}_l$, let $\Phi(U)$ be the set of $V \in \mathcal{A}_l$ such that $\| (T-V)y(U) \| \leq 4\delta$. Clearly $\Phi(U)$ is a closed, convex subset of $\mathcal{A}_l$, Moreover, by Proposition \ref{GM1stLem}, and our earlier observations about $y(U)$, $\Phi(U)$ is non-empty for every $U \in \mathcal{A}_l$. Clearly, the function $\mathcal{A}_l \to \X_{\infty}$ defined by $U \mapsto y(U)$ is continuous, and this ensures that $\Phi$ is upper semi-continuous. Indeed, suppose $U_n \in \mathcal{A}_l$, $U_n \to U$, $V_n \in \Phi(U_n)$ and $V_n \to V$. Then \[
\| (T - V)y(U) \| \leq \| (T - V_n)\big(y(U) - y(U_n)\big)\| + \| (T-V_n) y(U_n) \| + \| (V_n - V)y(U) \|.
\]
The first term on the right-hand-side of the inequality converges to $0$ by continuity of $U \mapsto y(U)$ and the fact that $U_n \to U$. The third term converges to $0$ since $V_n \to V$. The second term is at most $4\delta$ since $V_n \in \Phi(U_n)$. It follows that $\| (T - V)y(U) \| \leq 4\delta$ so that $V \in \Phi(U)$ as required. Consequently, by the fixed point theorem, there exists some $U \in \mathcal{A}_l$ with $U \in \Phi(U)$; this contradicts the fact that $y(U)$ is a `bad' vector for $U$, as was shown earlier.
\end{proof}

We conclude this section by making the obvious observation that the proof of Theorem \ref{snclosureLem} is now an immediate corollary of Proposition \ref{GMLem2}. 

\section{$\mathcal{L}(\X_{\infty}) / \mathcal{K}(\X_{\infty})$ is isometric to $\ell_1(\N_0)$}

We have already seen that the Calkin algebra, $\mathcal{L}(\X_{\infty}) / \mathcal{K}(\X_{\infty})$, is isomorphic to $\ell_1(\N_0)$. In the final section of this chapter, we complete the proof of the main theorem and show that the algebras are in fact isometric. As observed in the proof of Theorem \ref{C4TheMainTheorem}, it will be enough for us to show that the sequence $([S^j])_{j=0}^{\infty}$ is $1$-equivalent to the canonical basis of $\ell_1$. 

\begin{lem} \label{Spointwiseto0}
For any $x \in \X_{\infty}$, $\| S^n x \| \to 0$ as $n \to \infty$.
\end{lem}

\begin{proof}
Observe that if $x$ is a block vector with $\ran x \subseteq (p, q)$ then $S^n x = 0$ whenever $n \geq q$. Indeed, since the vectors $(d_{\gamma}^*)_{\gamma \in \Gamma}$ form a basis for the dual of $\X_{\infty}$, it is enough to show $\langle d_{\gamma}^*,  S^n x \rangle = 0$ whenever $\gamma \in \Gamma$ and $n \geq q$. Since $\ran x \subseteq (p,q) \implies \ran S^n x \subseteq (p, q)$ for every $n$, we certainly have $\langle d_{\gamma}^*,  S^n x \rangle = 0$ for every $n \in \N$ and $\gamma \in \Gamma$ with $\rank \gamma \geq q$. For $\gamma \in \Gamma$ with $\rank \gamma < q$, we have $(S^*)^n d_{\gamma}^* = 0$ whenever $n \geq q$ by Lemma \ref{nilpotentproperty}, and so $\langle d_{\gamma}^*,  S^n x \rangle = \langle (S^*)^n d_{\gamma}^*, x \rangle =0$ as required.

Now for a general $x \in \X_{\infty}$, fix $\ve > 0$ and choose $q \in \N$ such that $\| (I - P_{(0, q)}) x \| \leq \ve$. Then for $n \geq q$, \[
\| S^n x \| \leq \|S^n P_{(0,q)}x \| + \| S^n (I-P_{(0,q)})x \| \leq \| (I - P_{(0, q)}) x \| \leq \ve
\]
so $\| S^n x \| \to 0 $ as claimed.
\end{proof}

\begin{lem} \label{SKuniformto0}
Let $K \in \mathcal{K}(\X_{\infty})$. Then $\|S^n K \| \to 0$ as $n \to \infty$.
\end{lem}

\begin{proof}
Fix $\ve > 0$ and choose a finite $\frac{\ve}{2}$-net, $x_1, \dots x_m$ of $\overline{K(B_{\X_{\infty}})}$. By Lemma \ref{Spointwiseto0}, there exists $N \in \N$ such that $\| S^n x_j \| \leq \frac12\ve$ whenever $n \geq N, j = 1, \dots, m$. Now suppose $x \in B_{\X_{\infty}}$ and let $j \in \{ 1, \dots, m \}$ be such that $\| Kx - x_j \| \leq \frac12 \ve$. For $n \geq N$ we have $\|S^nK x \| \leq \| S^n (Kx - x_j) \| + \| S^n x_j \| \leq \| Kx - x_j \| + \frac12\ve \leq \ve$. Thus $\| S^n K \| \to 0$ as $n \to \infty$ as required.
\end{proof}

\begin{prop}  \label{normofS+K}
Let $(\lambda_j)_{j=0}^{\infty} \in \ell_1$. Then $\| \sum_{j=0}^{\infty} \lambda_j S^j + \mathcal{K}(\X_{\infty}) \| = \sum_{j=0}^{\infty} |\lambda_j |$. Consequently, the sequence $([S^j])_{j=0}^{\infty}$ is $1$-equivalent to the canonical basis of $\ell_1$.
\end{prop}

\begin{proof}
Certainly $\| \sum_{j=0}^{\infty} \lambda_j S^j + \mathcal{K}(\X_{\infty}) \| \leq \| \sum_{j=0}^{\infty} \lambda_j S^j \|  = \sum_{j=0}^{\infty} |\lambda_j |$ by Lemma \ref{snofSlookslikel1}. To prove the converse inequality, fix $\ve > 0$ and suppose $K \in \mathcal{K}(\X_{\infty})$. By Lemma \ref{SKuniformto0}, we can find $m \in \N$ with $\|S^mK \| \leq \ve$. It follows that 
\begin{align*}
\| K + \sum_{j=0}^{\infty} \lambda_j S^j \| &\geq \| S^m K + \sum_{j=0}^{\infty} \lambda_j S^{j+m} \| \\
&\geq \|\sum_{j=0}^{\infty} \lambda_j S^{j+m} \| - \| S^m K \| \\
& \geq \sum_{j=0}^{\infty} |\lambda_j| - \ve.
\end{align*}
It follows that $\| \sum_{j=0}^{\infty} \lambda_j S^j + \mathcal{K}(\X_{\infty}) \| \geq \sum_{j=0}^{\infty} |\lambda_j | - \ve$, and since $\ve > 0$ was arbitrary, $\| \sum_{j=0}^{\infty} \lambda_j S^j + \mathcal{K}(\X_{\infty}) \| \geq \sum_{j=0}^{\infty} |\lambda_j | $.
\end{proof}

%% file: ShiftInvariant.tex
\chapter{Shift invariant $\ell_1$ preduals} \label{Conclusion}

\section{Introduction} 

We have seen in Chapters \ref{MainResult} and \ref{l1Calk} that there exist a multitude of non-isomorphic Banach spaces with $\ell_1$ dual. Moreover, these spaces all possess some interesting Banach space structure; the spaces in Chapter \ref{MainResult} are easily seen to be saturated with infinite dimensional reflexive subspaces and have the property that their operator algebras and Calkin algebras can be considered `small', whereas the interesting property of the space in Chapter \ref{l1Calk} is the fact it has $\ell_1$ Calkin algebra.

In addition to the results in this thesis, we  also note that there are a number of other exotic $\ell_1$ preduals discussed in the existing literature. For example, it was shown in \cite{BenLin} that there exist isometric preduals of $\ell_1$ which are not isomorphic to a complemented subspace of any $C(K)$ space and, more recently, it has been shown (\cite{Freeman}) that any Banach space with separable dual can be embedded isomorphically into an $\ell_1$ predual. 

In this chapter, we briefly look at a problem considered by Haydon, Daws, Schlumprecht and White in \cite{Shift}, where $\ell_1$ preduals satisfying an additional `shift-invariant' condition are studied. For completeness, we begin by discussing the motivation behind this property and define precisely what we mean by a shift-invariant predual of $\ell_1$. We then show that one of the spaces defined in \cite{Shift} (which was shown to be isomorphic to $c_0$) can in fact be considered (in some sense) as being obtained from a specific Bourgain-Delbaen construction. Thus, whilst none of the results in this chapter are new, our approach is somewhat different to that taken by the authors of \cite{Shift}. Given the success of constructing exotic Banach spaces via the Bourgain-Delbaen method, the author had hoped that this new approach might lead to new shift invariant preduals of $\ell_1$ possessing some interesting Banach space properties. Unfortunately, at the time of writing, this remains an open problem and the author has been unable to make any further progress. 

The motivation behind \cite{Shift} comes from the study of dual Banach algebras. If a Banach algebra, $A$, is also a dual space of some Banach space, $E$, i.e. there exists a (Banach space) isomorphism $T \colon A \to E^*$, we can induce a weak* topology on $A$, simply by demanding that $T$ should be a weak*-weak* homeomorphism. This is, of course, equivalent to saying that a net $(x_{\alpha})_{\alpha \in I}$ converges weak* to $x$ in $A$ if and only if $Tx_{\alpha}$ converges weak* to $Tx$ in $E^*$.

We then have the following definition:

\begin{defn}
A Banach algebra $A$ is a {\em dual Banach algebra} if it is a dual space of some Banach space, $E$, and the product is separately weak*-continuous (with respect to the induced weak* topology just described). 
\end{defn}

We note that the induced weak* topology, and hence the above definition, depends on the choice of isomorphism $T \colon A \to E^*$. Therefore, one should be more precise and define a predual of $A$ as a Banach space $E$ together with an isomorphism $T \colon A  \to E^*$. We shall see shortly that we can forget about this technicality for the purposes of this chapter as we will be working with so-called concrete preduals of $\ell_1$ (see Definition \ref{concretepredual}).

As noted in \cite{Shift}, it is well known that a $C^*$-algebra, $M$, which is {\em isometric} to a dual space is a von Neumann algebra. In this case, the product and involution are automatically weak* continuous, so that $M$ is a dual Banach algebra. Moreover, the (isometric) predual is unique. However, we observe that {\em isomorphic} preduals need not be unique, even in the case of $C^*$ algebras. Indeed, Pelczy\'{n}ski proved that the Banach spaces $\ell_{\infty}$ and $L_{\infty}[0,1]$ are isomorphic. It follows that both $\ell_1$ and $L_1[0,1]$ are isomorphic preduals of the $C^*$ algebra, $\ell_{\infty}$, while of course $\ell_1$ is not isomorphic to $L_1[0,1]$. 

The authors of \cite{Shift} are specifically concerned with the Banach algebra $\ell_1(\Z)$ equipped with the natural convolution product, \[
(f * g)(n) = \sum_{k\in \Z} f(k) g(n-k), \quad \quad f,g \in \ell_1(\Z), n \in \Z
\]
and the isomorphic preduals of $\ell_1(\Z)$ with respect to which $\ell_1(\Z)$ is a dual Banach algebra. We should really be a little more precise. Observe that if we were to equip $\ell_1(\Z)$ with the zero product, {\em any} isomorphic predual of $\ell_1(\Z)$ makes $\ell_1(\Z)$ equipped with the zero product into a dual Banach algebra.  We are only interested in a predual $E$ of $\ell_1(\Z)$ with respect to which $\ell_1(\Z)$ {\em equipped with the convolution product} (defined above) is a dual Banach algebra.  Consequently, in what follows, we are always thinking of $\ell_1(\Z)$ as the `usual' Banach algebra equipped with the multiplication being the convolution just defined. We also remark that the canonical isometry of $c_0(\Z)^*$ with $\ell_1(\Z)$ makes $\ell_1(\Z)$ into a dual Banach algebra, so the sorts of preduals we are interested in certainly exist. Whilst this happens to be an isometric predual, we emphasise that we are not restricting our attention to just isometric preduals. In \cite{Shift}, the authors are looking for preduals $E$ for which $E$ has some interesting Banach space structure or for which the isomorphism between $E$ and $\ell_1(\Z)$ induces an exotic weak* topology.

Before continuing any further we remark that we shall frequently be making the canonical identification of $\ell_1(\Z)^*$ with $\ell_{\infty}(\Z)$ and find it convenient to denote by $\theta \colon \ell_1(\Z)^* \to \ell_{\infty}(\Z)$ the usual isometry.
 
Observe that every predual, $E$, of $\ell_1(\Z)$ can be canonically embedded (isomorphically) into $\ell_{\infty}(\Z)$. Indeed, if $J_E$ denotes the canonical isometric embedding of $E$ into its second dual and $T \colon \ell_1(\Z) \to E^*$ is an isomorphism, the composition $\theta \circ T^* \circ J_E \colon E \to \ell_{\infty}(\Z)$ is an isomorphic embedding of $E$ into $\ell_{\infty}(\Z)$. It turns out (see Lemmas \ref{arbtoconcrete} and \ref{shiftinvlem}) that the preduals of $\ell_1(\Z)$ which make the multiplication separately weak* continuous, i.e. the preduals with respect to which $\ell_1(\Z)$ is a dual Banach algebra, are precisely those which, when thought of as a subspace of $\ell_{\infty}(\Z)$ under this embedding, are invariant under the bilateral shift on $\ell_{\infty}(\Z)$. Consequently, we refer to such preduals as {\em shift-invariant preduals of $\ell_1(\Z)$}.

If was shown in \cite{Shift} that, when considering shift-invariant preduals of $\ell_1(\Z)$, we lose no generality by considering so called {\em concrete preduals}. As well as being a considerable simplification, the $\ell_1$ duality arising from concrete preduals is precisely that which arises naturally from the embedding of $\ell_1(\Z)$ into the dual of a space of Bourgain-Delbaen type. Consequently, concrete preduals are fundamental, both to the work of the authors in \cite{Shift} and to enable us to connect the work of this thesis with the results in the aforementioned paper. For this reason, we chose to include the details of concrete preduals now; the proofs and results largely follow the work of the authors in \cite{Shift}. 
 
Given a closed subspace $F \subseteq \ell_{\infty}(\Z)$, it is well known that the dual space $F^*$ is canonically isometric to $\ell_{\infty}(\Z)^* / F^{\perp}$, where $F^{\perp} = \{ \Phi \in \ell_{\infty}(\Z)^* : \langle \Phi, x \rangle = 0 \, \, \, \forall x \in F \}$. Let $\mathcal{I} \colon \ell_{\infty}(\Z)^* / F^{\perp} \to F^*$ denote this identification and  $\pi \colon \ell_{\infty}(\Z)^* \to \ell_{\infty}(\Z)^* / F^{\perp}$ be the obvious projection. 

\begin{defn} \label{concretepredual} 
Given a closed subspace $F \subseteq \ell_{\infty}(\Z)$, we define the map $\iota_F \colon \ell_1(\Z) \to F^*$ to be the composition $\iota_F = \mathcal{I} \pi (\theta^{-1})^* J_{\ell_1(\Z)} $ and say $F$ is a {\em concrete predual for $\ell_1(\Z)$} if the map $\iota_F$ is an isomorphism. 
\end{defn}

It is an easy computation to see that if $a \in \ell_1(\Z)$ and $x \in F \subseteq \ell_{\infty}(\Z)$, then $\langle \iota_F (a), x \rangle = \langle x , a \rangle$ (where we are of course thinking of $x$ as an element of $\ell_1(\Z)^*$ in the right-hand-side of the equality). Consequently, if $F$ is a concrete predual, the identification of $\ell_1$ with $F^*$ under $\iota_F$ is the same type of identification of $\ell_1$ as a subspace of $X^*$ when $X$ is a Bourgain-Delbaen space (see, for example, Theorem \ref{BDThm}, Proposition \ref{Xhasl1Dual}).

As stated earlier, we lose no generality in working with concrete preduals for $\ell_1(\Z)$ and this is the contents of the next lemma (which is proved in \cite[Lemma 2.1]{Shift}). We remark that the lemma states both the Banach space structure and induced weak* topologies on $\ell_1(\Z)$ are preserved when changing to the setup of a concrete predual. This is fundamental given we are concerned with those preduals which turn $\ell_1(\Z)$ into a dual Banach algebra.

\begin{lem} \label{arbtoconcrete}
Let $E$ be a Banach space and $T \colon \ell_1(\Z) \to E^*$ be an isomorphism. Then the map $\theta T^* J_{E} \colon E \to \ell_{\infty}(\Z)$ is an isomorphism onto its range, say $F \subseteq \ell_{\infty}(\Z)$. Furthermore, $\iota_F$ is an isomorphism so that $F$ is a concrete predual for $\ell_1(\Z)$ and the weak* topologies induced by $T \colon \ell_1(\Z) \to E^*$ and $\iota_F \colon \ell_1(\Z) \to F^*$ agree. That is, given a net $(a_{\alpha})$ in $\ell_1(\Z)$, we have that $\lim_{\alpha} \langle T( a_{\alpha}), x \rangle = 0$ for all $x \in E$ if and only if $\lim_{\alpha} \langle \iota_F(a_{\alpha}), y \rangle = \lim_{\alpha} \langle y, a_{\alpha} \rangle = 0$ for all $y \in F$.
\end{lem}

\begin{proof}
The first statement is obvious. We prove that $\iota_F$ is an isomorphism and the desired properties of the weak* topologies. Let $R$ denote the isomorphism obtained by thinking of the operator $\theta T^* J_E$ as a map onto its image, that is, $R = \theta T^* J_E | \colon E \to F$ (where $F \subseteq \ell_{\infty}(\Z)$ is defined as in the lemma). Note that for $x \in E, a \in \ell_1(\Z)$, 
\begin{align*}
\langle R^* \iota_F (a) , x \rangle &= \langle \iota_F (a), Rx \rangle = \langle \theta^{-1} Rx, a \rangle \\
&= \langle T^*J_E x, a \rangle = \langle Ta, x \rangle.
\end{align*}
In other words, $R^*\iota_F = T$. It follows that $\iota_F = (R^*)^{-1}T$ and thus $\iota_F$ is an isomorphism from $\ell_1(\Z)$ to $F^*$ as required.

Next, suppose $(a_{\alpha})$ is a net in $\ell_1(\Z)$. We have $(a_{\alpha})$ weak* convergent to $0$ with respect to the topology induced by the isomorphism $\iota_F$ if and only if $\langle \iota_F (a_{\alpha}) , y \rangle \to 0$ for all $y \in F$. Since $\iota_F = (R^{-1})^*T$, this happens if and only $ \langle (R^{-1})^*T (a_{\alpha}) , y \rangle = \langle Ta_{\alpha}, R^{-1}y \rangle \to 0$ for all $y \in F$. Since $R \colon E \to F$ is an isomorphism, this happens if and only if $\langle Ta_{\alpha} , x \rangle \to 0$ for all $x \in E$, that is, if and only if $(a_{\alpha})$ is weak* convergent to 0 with respect to the weak* topology induced by the isomorphism $T \colon \ell_1(\Z) \to E^*$ as required.
\end{proof}

For completeness, we state and prove two more lemmas. The results and proofs are taken from \cite{Shift} and the author of this thesis makes no claim to originality of the two results that follow. The first lemma proves that the shift invariant condition discussed at the beginning of this chapter is indeed equivalent to requiring the predual make $\ell_1$ with its natural convolution product into a dual Banach algebra. Of course, since we now know that we lose no generality in working with concrete preduals, we choose to  work with these in the lemma. The second lemma provides further evidence that concrete preduals do indeed provide a simplification to the study of the problem originally motivating the authors of \cite{Shift}. Indeed, we see that concrete preduals allow one to detect if two concrete preduals $F_1, F_2 \subseteq \ell_{\infty}(\Z)$ induce the same weak* topology on $\ell_1(\Z)$; this happens if and only if $F_1 = F_2$.

\begin{lem} \label{shiftinvlem}
Let $F \subseteq \ell_{\infty}(\Z)$ be a concrete predual for $\ell_1(\Z)$. The following are equivalent:
\begin{enumerate}[(1)]
\item The bilateral shift on $\ell_1(\Z)$ is weak*-continuous with respect to $F$.
\item The subspace $F$ is invariant under the bilateral shift on $\ell_{\infty}(\Z)$. 
\item $\ell_1(\Z)$ is a dual Banach algebra with respect to $F$. 
\end{enumerate}
\end{lem}

\begin{proof}
We show that (1) and (3) are equivalent and (2) and (3) are equivalent. We denote by $e_n$ the usual vector in $\ell_1(\Z)$ (i.e. $e_n = (\dots, 0, 0, 1, 0, 0 \dots) \in \ell_1(\Z)$ where the $1$ occurs in the n'th coordinate) and, for $y \in \ell_1(\Z)$, we denote by $*_y \colon \ell_1(\Z) \to \ell_1(\Z)$ the linear operator defined by $x \mapsto x * y$.

If (3) holds then, in particular, the maps $*_{e_1}$ and $*_{e_{-1}}$ on $\ell_1(\Z)$ are weak* continuous with respect to $F$. Note that if $n \in \Z$, \[
x * e_1 (n) =  \sum_{k \in \Z} x(k) e_1(n-k) = x(n-1)
\]
so the map $*_{e_1}$ is just the right shift on $\ell_1(\Z)$. It is just as easy to see that $*_{e_{-1}}$ is the left shift operator. So we conclude that the bilateral shifts on $\ell_1(\Z)$ are weak* continuous with respect to $F$. 

Conversely, suppose (1) holds, i.e. the right and left shifts on $\ell_1(\Z)$ are weak* continuous. We must see that for a fixed $y \in \ell_1(\Z)$ the map $*_y$ on $\ell_1(\Z)$ is weak* continuous with respect to $F$. 

A similar computation to the one above shows that for $j \in \N$, the map $*_{e_j}$ is just the right shift operator on $\ell_1(\Z)$ applied $j$ times and the map $*_{e_{-j}}$ is the left shift on $\ell_1(\Z)$ applied $j$ times. Thus if there is some $N \in \N$ such that $y = \sum_{|j| \leq N} a_j e_j $ then the map $*_y$ is weak* continuous since it is just the finite sum of weak* continuous maps; $*_y = \sum_{|j| \leq N} a_j *_{e_j}$. For an arbitrary $y \in \ell_1(\Z)$, we can write $y = \lim_{N\to \infty} \sum_{|j| \leq N} a_j e_j$. We write $P_N y = \sum_{|j| \leq N } a_j e_j$. Observe that \[
\| x * y - x* P_N y \| = \| x * (y - P_Ny) \| \leq \|x \| \| y - P_Ny \| 
\]
so that $*_y$ is the the limit (in the operator norm topology) of the (weak* continuous) operators $*_{P_Ny}$. It follows by the second part of Lemma \ref{w*tow*impliesdual} that $*_y$ is weak* continuous, as required.

Now suppose $F$ is invariant under the left shift operator on $\ell_{\infty}(\Z)$. We claim the right shift, $\sigma$, on $\ell_1(\Z)$ is weak*-continuous. Indeed, let $(x_{\alpha})$ be a net in $\ell_1(\Z)$ with $x_{\alpha} \wsto x$, i.e. $\langle \theta^{-1} f, x_{\alpha} \rangle  \to \langle \theta^{-1} f, x \rangle$ for all $f \in F \subseteq \ell_{\infty}(\Z)$.  We are required to see $\sigma x_{\alpha} \wsto \sigma x$. 

To see this, let $f \in F \subseteq \ell_{\infty}(\Z) $ and note that 
\begin{align*}
\langle \theta^{-1} f, \sigma x_{\alpha}\rangle &= \langle \sigma^* \theta^{-1} f, x_{\alpha} \rangle \\
&= \langle \theta^{-1} \circ (\theta\sigma^*\theta^{-1})f, x_{\alpha} \rangle.
\end{align*}
An easy computation shows that the operator $\theta \sigma^* \theta^{-1} \colon \ell_{\infty}(\Z) \to \ell_{\infty}(\Z)$ is just the left shift on $\ell_{\infty}(\Z)$. Indeed, if $(a_m)_{m\in\Z}$ is in $\ell_{\infty}(\Z)$ and $e_m$ denotes the m'th standard basis unit vector of $\ell_1(\Z)$ as before, we have \[
\left( \theta \sigma^* \theta^{-1} \left( (a_m)_{m\in\Z} \right)\right)(n) = \langle \sigma^*\left( \theta^{-1} (a_m) \right) , e_n \rangle = \langle \theta^{-1}(a_m), \sigma e_n \rangle = a_{n+1}
\]
so $\theta\sigma^*\theta^{-1}$ is the left shift on $\ell_{\infty}$ as required. 

Now, since we are assuming that $F$ is invariant under the left shift, $f' := \theta \sigma^* \theta^{-1} f \in F$. Thus, from our previous calculation, and that $x_{\alpha} \wsto x$ we have, 
\begin{align*}
\langle \theta^{-1} f, \sigma x_{\alpha} \rangle = \langle \theta^{-1} f', x_{\alpha} \rangle &\to \langle \theta^{-1} f' , x \rangle \\
&= \langle \sigma^* \theta^{-1} f, x \rangle \\
&= \langle \theta^{-1}f, \sigma x \rangle
\end{align*}
as required. A similar argument yields that if $F$ is invariant under the right shift operator on $\ell_{\infty}(\Z)$, then the left shift on $\ell_1(\Z)$ is weak* continuous. We have thus proved that (2) $\implies$ (3).

Conversely, suppose we know the right shift $\sigma$ on $\ell_1(\Z)$ is weak* continuous. We will show $F$ is left shift invariant. We continue with the notation used previously, recalling that $\theta \sigma^* \theta^{-1} \colon \ell_{\infty}(\Z) \to \ell_{\infty}(\Z)$ is the left shift on $\ell_{\infty}(\Z)$.

Now, if our claim is false, there exists an element $f' \in F$ such that it's image under the left shift, $\theta \sigma^* \theta^{-1} f' \in \ell_{\infty}(\Z) \setminus F$. By the Hahn-Banach Theorem, we may find a $\gamma^* \in \ell_{\infty}(\Z)^*$ such that 
\begin{enumerate}[(i)]
\item $\langle  \gamma^*, \theta \sigma^* \theta^{-1}  f'  \rangle = 1$
\item $\gamma^*(F)  = \{ 0 \}$ or, equivalently, $\langle \gamma^* , f \rangle = 0$ for all $f \in F$. 
\end{enumerate}

Note that $\theta^*\gamma^* \in \ell_1(\Z)^{**}$. We can rewrite (i) above as $\langle (\theta^{-1})^* \theta^*\gamma^* , \theta \sigma^* \theta^{-1} f' \rangle = \langle \theta^* \gamma^*, \sigma^* \theta^{-1} f' \rangle = 1$, and similarly, we can rewrite (ii) as $\langle \theta^* \gamma^* , \theta^{-1} f \rangle = 0$ for all $f \in F$. By Goldstine's Theorem, there is a net $(x_{\alpha})_{\alpha\in I}$ in $\ell_1(\Z)$ with $\| x_{\alpha} \|_{\ell_1} \leq \| \theta^* \gamma^* \|$ for every $\alpha$ and $J_{\ell_1}x_{\alpha} \wsto \theta^*\gamma^*$. We conclude that $\lim_{\alpha} \langle \sigma^*\theta^{-1} f' , x_{\alpha} \rangle = \lim_{\alpha} \langle \theta^{-1}f' , \sigma x_{\alpha} \rangle =1$ and $\lim_{\alpha} \langle \theta^{-1} f, x_{\alpha} \rangle = 0$ for all $f \in F$.

Now $(x_{\alpha})$ is a net in some closed ball of $\ell_1(\Z) \simeq F^*$. By the weak* compactness of closed balls in $F^*$, $(x_{\alpha})$ has a subnet $(x_{g(\beta)})_{\beta}$ which is weak* convergent to some $x \in \ell_1(\Z)$. Since $\langle \theta^{-1} f' , \sigma x_{g(\beta)} \rangle $ is a subnet of $\langle \theta^{-1} f' , \sigma x_{\alpha} \rangle$, the latter of which we know converges to $1$, we have $\lim_{\beta} \langle \theta^{-1} f' , \sigma x_{g(\beta)} \rangle = 1$. Moreover, by the assumed weak* continuity of $\sigma$, we have $\lim _{\beta} \langle \theta^{-1} f' , \sigma x_{g(\beta)} \rangle = \langle \theta^{-1} f' , \sigma x \rangle$ and therefore, $\langle \theta^{-1} f' , \sigma x \rangle = 1$. 

On the other hand, by the very definition of the weak* topology induced on $\ell_1(\Z)$, we have $ \langle \theta^{-1} f, x \rangle = \lim_{\beta} \langle \theta^{-1} f , x_{g(\beta)} \rangle = \lim_{\alpha} \langle \theta^{-1} f , x_{\alpha} \rangle = 0$ for all $f \in F$. This is equivalent to saying $\langle \iota_F(x) , f \rangle = 0$ for all $f \in F$, which of course tells us that $\iota_F(x)=0$, and hence $x = 0$ ($\iota_F$ is an isomorphism). This clearly contradicts $\langle \theta^{-1} f' , \sigma x \rangle = 1$. Again, one can argue similarly to prove that if the left shift on $\ell_1(\Z)$ is weak* continuous, then $F$ must be right shift invariant. This completes the proof. 
\end{proof}

\begin{lem} \label{w*topspreduals}
Let $E_1$ and $E_2$ be preduals of $\ell_1(\Z)$, and use these to induce concrete preduals $F_1, F_2 \subseteq \ell_{\infty}(\Z)$ as described in Lemma \ref{arbtoconcrete}. Then $E_1$ and $E_2$ induce the same weak* topology on $\ell_1(\Z)$ if and only if $F_1 = F_2$. In particular, two concrete preduals $F_1$ and $F_2$ induce the same weak* topology on $\ell_1(\Z)$ if and only if $F_1 = F_2$. 
\end{lem}

\begin{proof}
If $F_1 = F_2$ then it follows immediately from Lemma \ref{arbtoconcrete} that $E_1$ and $E_2$ induce the same weak* topology on $\ell_1(\Z)$. Conversely, for i=1,2, let $T_i \colon \ell_1(\Z) \to E_i^*$ be an isomorphism, and suppose that these induce the same weak* topology on $\ell_1(\Z)$. Suppose for a contradiction that there exists an $x \in F_2 \setminus F_1$. By the Hahn-Banach Theorem, there exists $\Lambda \in \ell_{\infty}(\Z)^*$ with $\langle \Lambda , x \rangle = 1$ and $\langle \Lambda , y \rangle = 0$ for all $y \in F_1$. By Goldstine's Theorem, there is a bounded net $(a_{\alpha})$ in $\ell_1(\Z)$ with $J_{\ell_1(\Z)}a_{\alpha} \wsto \theta^*\Lambda$. It follows that, for all $y \in F_1$, we have \[
0 = \langle \Lambda , y \rangle = \langle \theta^* \Lambda , \theta^{-1} y \rangle = \lim_{\alpha} \langle \theta^{-1} y , a_{\alpha} \rangle = \lim_{\alpha} \langle \iota_{F_1} (a_{\alpha}) , y \rangle
\]
so that $(T_1 (a_{\alpha}) )$ is weak* null in $E_1^*$ by Lemma \ref{arbtoconcrete}. By assumption, it follows that $(T_2 (a_{\alpha}))$ is weak* null in $E_2^*$, but this contradicts that \[
1 = \langle \Lambda , x \rangle = \lim_{\alpha} \langle \theta^{-1} x , a_{\alpha} \rangle = \lim_{\alpha} \langle \iota_{F_2} (a_{\alpha}) , x \rangle.
\]  
This shows that $F_2 \subseteq F_1$, and an identical argument shows $F_1 \subseteq F_2$, as required. 
\end{proof}

\section{Connection to the BD construction} 

Having described the motivation behind the work in \cite{Shift} and understood the importance of concrete preduals, we now move on to discuss a specific class of shift invariant preduals that were introduced in \cite{Shift}. We follow the same notation used in \cite{Shift}; for $n \geq 0$ in $\Z$, we denote by $b(n)$ the number of ones in the binary expansion of $n$. (So $b(0) = 0, b(1) = 1, b(2) = 1, b(3) = 2$ and so on.) 

\begin{defn}
Let $\lambda \in \R, |\lambda| > 1$. Define the vector $x^{\lambda}_0 \in \ell_{\infty}(\Z)$ by $x^{\lambda}_0(n) = \lambda^{-b(n)}$ when $n \geq 0$ and $0$ otherwise. Thus $x_0^{\lambda}$ is given by 
\[
x_0 = ( \dots 0, 0, 1, \lambda^{-1}, \lambda^{-1}, \lambda^{-2}, \lambda^{-1}, \lambda^{-2}, \lambda^{-2}, \lambda^{-3}, \lambda^{-1}, \dots ).
\]
$F_{\lambda} \subseteq \ell_{\infty}(\Z)$ is defined to be the closed, shift-invariant subspace of $\ell_{\infty}(\Z)$ generated by the vector $x_0^{\lambda}$, i.e. $F_{\lambda}$ is the closed linear span of the bilateral shifts of $x_0^{\lambda}$. 
\end{defn}

It was shown in \cite{Shift} that the collection of subspaces $(F_{\lambda})_{|\lambda| > 1}$ of $\ell_{\infty}(\Z)$ are all concrete preduals of $\ell_1(\Z)$. Clearly, by construction, they are shift invariant and so by Lemma \ref{shiftinvlem}, they are preduals of $\ell_1(\Z)$ that make $\ell_1(\Z)$ with the convolution product a dual Banach algebra. The authors of \cite{Shift} used a somewhat technical argument involving the Szlenk index to show that as Banach spaces, the spaces $F_{\lambda}$ are all in fact isomorphic to $c_0$. The interest in the spaces $F_{\lambda}$ is therefore due to the weak* topologies that they induce on $\ell_1(\Z)$. Since the $F_{\lambda}$ are pairwise distinct subspaces of $\ell_{\infty}(\Z)$, and distinct from $c_0$, it is immediate from Lemma \ref{w*topspreduals} that these spaces induce an uncountable family of distinct weak*-topologies making $\ell_1(\Z)$ into a dual Banach algebra. 

The final result of this thesis, is an observation about the space $F_{2}$. We will show that we can obtain a `one-sided' version of the space $F_{2}$ using the Bourgain Delbaen construction. When viewed in this framework, it is very easy to see that the one-sided version of $F_{2}$ is isomorphic as a Banach space to $c_0$ and that its dual space is naturally isomorphic to $\ell_1$. In what follows $\N_0$ denotes the set of natural numbers including $0$. We first set up the notation. 

We define the vector $y_0 \in \ell_{\infty}(\N_0)$ by $y_0(n) = 2^{-b(n)}$ where, as before, $b(n)$ denotes the number of ones in the binary expansion of $n$. Thus $y_0$ is given by \[
(1, 2^{-1}, 2^{-1}, 2^{-2}, 2^{-1}, \dots )
\]
We define $R_{2}$ to be the closed, right-shift-invariant subspace of $\ell_{\infty}(\N_0)$ generated by $y_0$, that is, $R_{2}$ is the closed linear span of the vector $y_0$ and all its right shifts.

We also use Theorem \ref{BDThm} to construct a space of Bourgain-Delbaen type. Comparing with the notation from Theorem \ref{BDThm}, we take $\Gamma = \N_0$ and define finite subsets $\Delta_n \subset \N_0$ as follows. We set $\Delta_0 = \{ 0 \}$, and for $n \geq 1$, we set $\Delta_n = \{ 2^{n-1} + j : 0 \leq j  < 2^{n-1} \}$. For each $k \geq 1$, if $n \in \Delta_k$, we can uniquely write $n = 2^{k-1} + m$ where $m < 2^{k-1}$ and $m,k \in \N_0$. Again comparing with the notation from Theorem \ref{BDThm} and making using of the previous observation, we define a $\tau$ mapping by setting $\tau(n)$ to be the Type 0 tuple $(1, \frac12, m)$. With this setup, it is easily seen that the Bourgain-Delbaen Theorem (Theorem \ref{BDThm}) applies. We get corresponding vectors $c_n^* = \frac12 e_m^* \in \ell_1(\N_0)$ and $d_n^* = e_n^* - c_n^* \in \ell_1(\N_0)$. For clarity, the first few $d^*$ vectors are as follows:

\begin{align*}
d_0^* &= (1, 0, 0, 0, 0, 0, \dots ) \\
d_1^* &= (-1/2,1, 0, 0,0 ,0, \dots ) \\
d_2^* &= (-1/2, 0, 1, 0, 0, 0, \dots ) \\
d_3^* &= (0, -1/2, 0, 1, 0, 0, \dots).
\end{align*}
We obtain a $\mathscr{L}_{\infty}$ space of Bourgain-Delbaen-type, $\mathcal{R}_{\text{BD}} := X(\N_0, \tau)$. We recall that  $\mathcal{R}_{\text{BD}}$ is nothing other than the closed linear span of the biorthogonal vectors $[d_n : n \in \N_0 ] \subseteq \ell_{\infty}(\N_0)$. 

The main observation of this chapter is the following lemma:

\begin{lem}
The spaces $R_{2}$ and $\mathcal{R}_{\text{BD}}$ defined above are the same subspace of $\ell_{\infty}(\N_0)$. The vectors $(d_n^*)_{n=1}^{\infty}$ are in fact a basic sequence in $\ell_1(\N_0)$ equivalent to the canonical basis of $\ell_1(\N_0)$. Consequently, $\mathcal{R}_{\text{BD}}$ is a right-shift-invariant subspace of $\ell_{\infty}$ which is isomorphic to $c_0$ and has dual space naturally isomorphic to $\ell_1(\N_0)$. 
\end{lem}

\begin{proof}
We first show that $R_{2}$ and $\mathcal{R}_{\text{BD}}$ are the same subspace of $\ell_{\infty}(\N_0)$. The proof involves nothing other than a number of lengthy computations so we proceed by proving a number of smaller claims. It is convenient for us to re-introduce the $\tau$ operator that featured in $\cite{Shift}$. Precisely, $\tau \colon \ell_{\infty}(\N_0) \to \ell_{\infty}(\N_0)$ is the bounded linear map defined by \[
\tau (x) (n) = \begin{cases} x(\frac{n}{2}) & n \text{ even} \\
					0 & n \text{ odd.} 
		\end{cases}
\]
We denote by $\sigma \colon \ell_{\infty}(\N_0) \to \ell_{\infty}(\N_0)$ the right shift, i.e. $\sigma x (n) = x(n-1)$. It was shown in \cite[Lemma 3.1]{Shift} that $\tau^k(y_0) \in R_2$ for all $k \in \N_0$ and moreover, $\tau \sigma = \sigma^2 \tau$. 

We also define a bounded operator $\beta \colon \ell_{\infty}(\N_0) \to \ell_{\infty}(\N_0)$ by \[
\beta (x) (n) = \begin{cases} x(\frac{n-1}{2}) & n \text{ odd} \\
					0 & n \text{ even.} 
		\end{cases}
\]
We note that $\sigma \tau = \beta$ so that $\beta(y_0) \in R_2$. 
\begin{Step}
$d_0 = y_0$.
\end{Step}
We use induction on $n \in \N_0$ to prove that $d_0(n) = \langle d_0 , e_n^* \rangle = 2^{-b(n)} = y_0(n)$. Clearly $ \langle d_0, e_0^* \rangle = \langle d_0, d_0^* \rangle = 1 = 2^{-b(0)}$ as required. Inductively,  assume that $\langle d_0 , e_j^* \rangle = 2^{-b(j)}$ for all $j \leq n$. Then $\langle d_0, e_{n+1}^* \rangle =  \langle d_0, d_{n+1}^* + c_{n+1}^* \rangle = \frac12 \langle d_0^* , e_{m}^* \rangle$ where $n +1= 2^k + m$ with $m < 2^k$. In particular, $m \leq n$, so by the inductive hypothesis, $\langle d_0 , e_{n+1}^* \rangle = \frac12 2^{-b(m)}$. Note that since $m < 2^k$, we can write $n+1 = 2^k + \sum_{j=0}^{k-1} \ve_j 2^j$ where $\ve_j \in \{0, 1\}$ for all $j$. It follows that $b(n+1) = b(m) +1$ and so $\langle d_0, e_{n+1}^* \rangle = 2^{-b(n+1)}$, completing the induction.  
\begin{Step}
For all $j \in \N_0, d_{2j+1} = \beta d_j = \sigma \tau d_j$.
\end{Step}
We use induction on $n \in \N_0$ to prove that $d_{2j+1}(n) = \beta d_j (n)$. It is easily seen that for $n < 2j+1$, $d_{2j+1}(n) = 0$. Moreover, for such $n$, $\beta d_j (n) = 0$ if $n$ is even. When $n$ is odd and $n < 2j + 1$, $(n-1)/2 < j$ so that $\beta d_j (n) = d_j ( (n-1)/2 ) = 0$. Consequently the inductive statement holds when $n < 2j+1$. When $n = 2j+1$ it is easily verified that $d_{2j+1}(n) = \beta d_j(n) = 1$. 

We now suppose that the inductive statement holds for all $k \leq n$ and must show that $d_{2j+1}(n+1) = \beta d_j (n+1)$. By the preceding argument, we may as well assume that $n \geq 2j+1$, so $n+1 > 2j+1 \geq 1$. We write $n+1 = 2^k + m$ where $0 \leq m < 2^k$ and since $n+1 \geq 2$ we must have $k \geq 1$. Now $d_{2j+1}(n+1) = \langle d_{2j+1} , c_{n+1}^* \rangle = \frac12 \langle d_{2j+1} , e_m^* \rangle$. Certainly $m \leq n$, so we may apply the apply inductive hypothesis which yields \[
d_{2j+1}(n+1) = \frac12 \beta d_j (m) = \begin{cases} \frac12 d_j (\frac{m-1}{2}) & m \text{ odd } \\ 0 & \text{ otherwise.} \end{cases}
\]
Finally, we consider cases where $n+1$ is even or odd separately, noting that $n+1$ is even if and only if $m$ is even, since we have seen that $k \geq 1$. Consequently, if $n+1$ is even, then our above calculation yields that $d_{2j+1}(n+1) = 0 = \beta d_j(n+1)$.  It remains only to consider the possibility that $n+1$ is odd, in which case, $m$ is odd and $d_{2j+1}(n+1) = \frac12 d_j((m-1)/2)$. We need to see that this is equal to $\beta d_j (n+1) = d_j ( n / 2)$. Since we assumed $n \geq 2j +1$, we have $n/2 \geq j + 1/2$ and since $n/2 \in \Z$, in fact, $n/2 \geq j + 1 > j$. Consequently, $d_j(n/2) = \langle d_j , c_{n/2}^* \rangle$. Recalling that $n+1 = 2^k + m$, where $m$ must now be odd, we see that $n/2 = 2^{k-1} + \frac{m-1}{2}$ and consequently $\langle d_j , c_{n/2}^* \rangle = \frac12 \langle d_j, e_{\frac{m-1}{2}}^* \rangle = \frac12 d_j (\frac{m-1}{2})$ as required.
\begin{Step}
For all $j \geq 1$, $d_{2j} = \tau d_j$.
\end{Step}
Again, we use induction on $n$ and prove that $d_{2j}(n) = \tau d_j (n)$ for all $n \in \N_0$. The calculations are very similar to the previous step. When $n < 2j$ it's easy to see that $d_{2j}(n) = 0 = \tau d_j (n)$ and when $n = 2j$, $d_{2j}(n) = 1 = \tau d_j (n)$. So we suppose that the inductive statement holds for all $k \leq n$ and prove the statement must hold for $n+1$. By the previous argument, we can assume $n \geq 2j \geq 2$. Then $n+1 > 2j$ and $d_{2j}(n+1) = \langle d_{2j}, c_{n+1}^* \rangle = \frac12 \langle d_{2j} , e_m^* \rangle$ where $n +1 = 2^k + m$ and we must have $k \geq 1$ since $n+1 \geq 3$. By the inductive hypothesis we have $d_{2j}(n+1) = \frac12 \tau d_j (m)$. Now, if $n+1$ is odd, $m$ must also be odd, so $d_{2j}(n+1) = \frac12 \tau d_j (m ) = 0 = \tau d_j (n+1)$. So we assume now that $n+1$ is even, so $m$ is even and $d_{2j}(n+1) = \frac12 d_j (m/2)$. In this case, $\tau d_j (n+1) = d_j (\frac{n+1}{2}) = \langle d_j , c_{\frac{n+1}{2}}^* \rangle$ since we recall that $n + 1 > 2j$. Also $n +1 = 2^k + m$ so that $(n+1) / 2  = 2^{k-1} + m/2$. It follows that $\tau d_j (n+1) = \langle d_j , c_{\frac{n+1}{2}}^* \rangle = \frac 12 d_j (m/2) = d_{2j}(n+1)$ as required.
\begin{Step} 
$d_{2^j + m} = \sigma^m d_{2^j}$ whenever $j \geq 1$, $1 \leq m < 2^j$.
\end{Step}
To see this, we will show by induction on $j \in \N$ that $d_{2^j + m} = \sigma^m d_{2^j}$ whenever $1 \leq m < 2^j$. The base case, $j= 1$ follows easily from the previous two steps. Indeed, the only permissible value of $m$ is $1$ and $d_{2^1 +1} = \sigma \tau d_{1} = \sigma d_{2^1}$ as required. 

So, suppose by induction, that for some $j > 1$, $d_{2^{j-1} + m} = \sigma^m d_{2^{j-1}}$ whenever $1 \leq m < 2^{j-1}$. We must show that $d_{2^{j} + m} = \sigma^m d_{2^{j}}$ whenever $1 \leq m < 2^{j}$. We will do this by an induction on $m$. The base case, $m =1$, is again easy. By Steps 2 and 3, we have $d_{2^j + 1} = \sigma \tau d_{2^{j-1}} = \sigma d_{2^j}$ as needed. So suppose that $d_{2^j + m} = \sigma ^m d_{2^j}$ and that $m+1 < 2^j$. If $m = 2k$, we have $d_{2^j + m + 1} = d_{2(2^{j-1} + k) + 1} = \sigma \tau d_{2^{j-1} + k} = \sigma d_{2^j + 2k} = \sigma d_{2^j + m} = \sigma^{m+1} d_{2^j}$. Here we have appealed to the previous two steps and the inductive hypothesis on $m$. 

It remains to consider the case $m = 2k+1$. Again, we will make repeated use of the previous two steps. We have $d_{2^j + m +1} = d_{2^j + 2k +2} = d_{2(2^{j-1} + k + 1)} = \tau d_{2^{j-1} + k + 1}$. %%If $k$ is even, then the final term is equal to
%%\begin{align*}
%%\tau d_{2(2^{j-2} + \frac{k}{2}) + 1} &= \tau \sigma \tau d_{2^{j-2} + \frac{k}{2}} = \tau  \sigma d_{2^{j-1} + k} \\
%%&= \sigma( \sigma \tau d_{2^{j-1} + k } ) = \sigma d_{2^j + 2k + 1} = \sigma d_{2^j + m}  = \sigma^{m+1} d_{2^j}.
%%\end{align*}
%%If $k$ is odd, we have $d_{2^j + m + 1} = \tau d_{2^{j-1} + k + 1}$. 
Moreover, $m+1 = 2(k+1) < 2^j$ so that $k+1 < 2^{j-1}$. Thus by our inductive hypothesis on $j$, $d_{2^{j-1} + k + 1} = \sigma^{k+1} d_{2^{j-1}}$ and so $d_{2^j + m + 1} = \tau \sigma^{k+1} d_{2^{j-1}} = \sigma^{2k+2} \tau d_{2^{j-1}} = \sigma^{2k+2} d_{2^j} = \sigma^{m+1} d_{2^j}$. (Here we made use of the fact that $\tau \sigma = \sigma ^2 \tau$.) This completes both inductive steps. 
\begin{Step}
$\sigma d_n = d_{n+1}$ except if $n$ is of the form $n = 2^j -1$, with $j \geq0$. 
\end{Step}
This is almost immediate from the previous step. Indeed, suppose $j \geq 1$ and $1 \leq m < 2^j  -1$. Then if $n = 2^j + m$, $\sigma d_n = \sigma d_{2^j + m} = \sigma^{m+1} d_{2^j} = d_{2^j + m +1} = d_{n+1}$. We made use of the fact that $m+1 < 2^j$ so that we could apply the previous step to obtain the penultimate equality. If $n = 2^j$ with $j \geq 1$ then $d_{n+1} = d_{2(2^{j-1}) + 1} = \sigma \tau d_{2^{j-1}}$ by Step 2, and $\sigma \tau d_{2^{j-1}} = \sigma d_{2^j} = \sigma d_n$ by Step 3. 

We therefore have that $\sigma d_n = d_{n+1}$ whenever $n$ can be written in the form $2^j + m$ with $j \geq 1$ and $0 \leq m < 2^j -1$. The only values of $n$ which aren't of the form $n = 2^j + m$ with $j \geq 1$ and $0 \leq m < 2^{j} -1$ are precisely those $n$ of the form $2^k  -1$ for some $k \in \N_0$. 
\begin{Step}
For $j \geq 0$, \[
d_{2^j}(n) = \begin{cases} d_0(l) & \text{ if } n = 2^j(2l +1) \text{ (for some } l \in \N_0) \\ 0 & \text{ otherwise. } \end{cases}
\] In particular the vectors $\{ d_{2^j} : j \in \N_0 \}$ have disjoint supports.  
\end{Step}
We use induction on $j \in \N_0$. When $j = 0$ we have $d_1(n) = \sigma \tau d_0 (n)$ by Step 2. So\[
d_1(n) = \tau d_0 (n-1) = \begin{cases} d_0 (\frac{n-1}{2}) & \text{ if } n-1 \text{ even} \\ 0 & \text{ otherwise.} \end{cases}
\] It follows that $d_1(n) = 0$ unless $n = 2l +1$ in which case $d_1(n) = d_0(l)$ as required. 

Now assume the claim holds for some $j \geq 0$. We have $d_{2^{j+1}} (n) = d_{2\cdot 2^j} (n) = \tau d_{2^j}(n)$ by Step 3. So \[
d_{2^{j+1}}(n) = \begin{cases} d_{2^j}(\frac{n}{2}) & \text{ if } n  \text{ even } \\ 0 & \text{ otherwise. } \end{cases}
\] From this, and the inductive hypothesis, we see that $d_{2^{j+1}}(n) = 0$ unless $n/2 = 2^j (2l +1)$ for some $l \in \N_0$, i.e. if $n = 2^{j+1}(2l+1)$, in which case we have $d_{2^{j+1}}(n) = d_0(l)$. The statement about disjoint supports is obvious.
\begin{Step}
For all $j \geq 1$, \[
d_{2^j - 1} (n) = \begin{cases} d_0(k) & \text{ if there is } k \in \N_0 \text{ such that } n = 2^j(k+1) - 1 \\ 0 & \text{otherwise.} \end{cases}
\]
\end{Step}
We proceed by induction on $j$. When $j =1$, $d_1(n) = \sigma \tau d_0 (n)$ by Step 2. So $d_1(n) = \tau d_0 (n-1)$; this equals $d_0(k)$ if $n = 2k +1 = 2(k+1) -1$ for some $k \in \N_0$ and $0$ otherwise by the definition of the $\tau$ operator.

Assuming the claim holds for some $j \geq 1$, $d_{2^{j+1} - 1} (n) = d_{2(2^j - 1) +1}(n) = \sigma\tau d_{2^j-1} (n)$ by Step 2. So \[
d_{2^{j+1} -1}(n) = \tau d_{2^j - 1} (n-1) = \begin{cases} d_{2^j - 1} (\frac{n-1}{2}) & \text{ if } n-1 \text{ even } \\ 0 & \text{otherwise.} \end{cases}
\]
From this, and the inductive hypothesis, we see that $d_{2^{j+1} -1} (n) \neq 0$ if and only if $(n-1)/2 = 2^j(k+1) - 1$ for some $k \in \N_0$ in which case $d_{2^{j+1}-1}(n) = d_0(k)$ as required. 
\begin{Step}
For all $j \geq 0$, $\sigma d_{2^j - 1} = \sum_{k=0}^{\infty} 2^{-k} d_{2^{k+j}}$.
\end{Step}
We first treat the case $j = 0$. Since $e_0^* = d_0^*$ we easily see that $\sigma d_0 (0) = 0 = \sum_{k=0}^{\infty} 2^{-k} d_{2^{k}}(0)$. Now, for every $n \in \N$, we can write $n = 2^m(2l+1)$ for unique $m, l \in \N_0$. If $m = 0$, then $\sigma d_0(n) = \sigma d_0(2l + 1) = d_0 (2l) = d_0(l)$. (We have made use of Step 1 here and the fact that $b(2l) = b(l)$.) By Step 6, $\sum_{k=0}^{\infty} 2^{-k} d_{2^k} (n) = \sum_{k=0}^{\infty} 2^{-k} d_{2^k} (2l+1) = d_1(2l+1) = d_0(l) = \sigma d_0 (n)$ where the penultimate equality follows by Step 7. Otherwise, $n = 2^m(2l+1)$ and $m \geq 1$. In this case, $\sum_{k=0}^{\infty} 2^{-k} d_{2^k} (n) = 2^{-m} d_0(l)$ by Step 6. On the other hand, we can write $n = \sum_{k=m}^N \ve_k 2^k$ for some $N \geq m$, where $\ve_k \in \{ 0, 1 \}$ for $k > m$ and $\ve_m = 1$. It follows that $n-1 = 2^m - 1 + \sum_{k = m+1}^N \ve_k 2^k = \sum_{k=0}^{m-1} 2^k + \sum_{k=m+1}^{N} \ve_k 2^k$. We see from this that $b(n-1) = b(n) + m-1$. Consequently, $\sigma d_0(n) = d_0(n-1) = 2^{-b(n-1)}$ (by Step 1) $= 2^{-b(n) - m + 1}$. Finally we note that $b(n) = b(2l+1) = b(l) + 1$ so, $\sigma d_0(n) = 2^{-b(l) - m } = 2^{-m} d_0(l) = \sum_{k=0}^{\infty} 2^{-k} d_{2^k} (n)$ as required.

The calculation for when $j \geq 1$ is only slightly more involved. It is again easily seen that $\sigma d_{2^j - 1} (0) = 0 = \sum_{k=0}^{\infty} 2^{-k} d_{2^{k+j}}(0)$. If $n \in \N$, we again write $n = 2^m(2l+1)$ and $\sigma d_{2^j -1}(n) = d_{2^j - 1} (n-1) = d_{2^{j} - 1} (2^m(2l+1) -1 )$. By Step 7, this will be non-zero if and only if $m \geq j$. Moreover, if $m = j$, i.e. $n = 2^m (2l +1)$, then we see from the previous calculation and Step 7 that $\sigma d_{2^j -1}(n) = d_0(2l)$. By Step 1, this equals $d_0(l)$. If $m > j$, then $n= 2^j( 2^{m-j}(2l+1))$ and by Step 7, we have $\sigma d_{2^j -1}(n) = d_0(2^{m-j}(2l+1) - 1)$. We note that $2^{m-j}(2l+1) = \sum_{a = m-j}^{N} \ve_a 2^a$ for some $N \geq m-j$, where $\ve_a \in \{0, 1 \}$ for $ a > m-j$ and $\ve_{m-j} = 1$. It follows that $2^{m-j}(2l+1) -1 = 2^{m-j} - 1 + \sum_{a = m-j + 1}^N \ve_a 2^a = \sum_{a=0}^{m-j -1} 2^a + \sum_{a=m-j+1}^N \ve_a 2^a$. This implies that $b(2^{m-j}(2l +1) -1 ) = b(2^{m-j}(2l+1)) + m - j -1$. So, $\sigma d_{2^j -1}(n) = d_0(2^{m-j}(2l+1) - 1) = 2^{-b(2^{m-j}(2l+1) - 1)} = 2^{-m+j+1}2^{-b(2^{m-j}(2l+1))} = 2^{-m +j+1} 2^{-b(2l+1)} = 2^{-m+j}2^{-b(l)} = 2^{-m+j} d_0(l)$. 

It remains to compute $\sum_{k=0}^{\infty} 2^{-k} d_{2^{k+j}}(n)$ and see that the formulae we obtain agree with those just found for $\sigma d_{2^j - 1} (n)$. We note that if $n = 2^m(2l+1)$ then $d_{2^{k+j}}(n)$ will only be non-zero if $k+j = m$ by Step 6. In particular, if $n$ is such that $m < j$ then we get $\sum_{k=0}^{\infty} 2^{-k} d_{2^{k+j}}(n) = 0 = \sigma d_{2^j -1}(n)$. Otherwise, if $n = 2^m(2l+1)$ is such that $m \geq j$ we have $\sum_{k=0}^{\infty} 2^{-k} d_{2^{k+j}}(n) = 2^{-k} d_0 (l)$ where $k$ is such that $k + j = m$, by Step 6. The latter expression is therefore equal to $2^{-m + j} d_0(l) = \sigma d_{2^j -1} (n)$, as found in the previous paragraph. 

We are finally in a position to show that $R_2 = \mathcal{R}_{\text{BD}}$. By Step 5 and Step 8, we clearly have that $\sigma \big( \mathcal{R}_{\text{BD}} \big) = \sigma \big( [d_n : n \in \N_0] \big) \subseteq  \mathcal{R}_{\text{BD}}$. Since, by Step 1, $d_0 = y_0$, we conclude that $ \mathcal{R}_{\text{BD}}$ is a closed, right-shift invariant subspace of $\ell_{\infty}(\N_0)$ containing $y_0$; we must therefore have $R_2 \subseteq  \mathcal{R}_{\text{BD}}$. 

On the other hand, since $\tau^k (y_0) = \tau^k (d_0) \in R_2$ for all $k$, and $\sigma^ 2\tau = \tau \sigma$, it is easily seen by Steps 2 and 3 that for each $n \in \N$ there are $l, m \in \N_0$ with $d_n = \sigma^m \tau^l d_0 \in R_2$. Consequently $\mathcal{R}_{\text{BD}} \subseteq R_2$ and so the two subspaces are indeed equal. 

Finally, we observe that in this construction, $\|c_n^* \| \leq \frac{1}{2}$ for all $n \in \N_0$. Appealing to Lemma \ref{BDNeedLargeNormPerturbation}, we see that $\mathcal{R}_{\text{BD}}$ is isomorphic to $c_0$. Moreover, the sequence of vectors $(d_n^*)$ in $\ell_1(\N_0)$ is equivalent to $(e_n^*) \subseteq \ell_1(\N_0)$. Since $(e_n^*)_{n=1}^{\infty}$ is a boundedly complete basis for $\ell_1$ so also is $(d_n^*)$. It follows by Theorem \ref{boundedlycompleteshrinkingbasis} that $(d_n)$ is a shrinking basis for $\mathcal{R}_{\text{BD}}$. Consequently by the Bourgain-Delbaen Theorem (Theorem \ref{BDThm}), the dual space of $\mathcal{R}_{\text{BD}}$ is naturally isomorphic to $\ell_1$, i.e. the $\ell_1$ duality obtained here is the same as that which defines the $\iota_F$ mapping discussed earlier. 
\end{proof}

\section{Concluding Remarks}

We have obtained a right-shift invariant $\ell_1$ predual using the Bourgain-Delbaen construction which has a striking resemblance to the space $F_2$ constructed in \cite{Shift}. However, the space $R_2$ in the above lemma does not quite do what we want it to. Tracing through the proof of Lemma \ref{shiftinvlem}, we see that for the natural convolution on $\ell_1(\N_0)$ to be weak* continuous with respect to the topology induced by the natural isomorphism of $\ell_1$ with $R_2^*$, we would need $R_2$ to be left shift invariant. Alas, our shift is in the wrong direction! 

At the time of writing, it is unclear to the author whether or not it is possible to modify the construction in some way so as to obtain a one-sided version of $F_2$ that is left-shift invariant as a subspace of $\ell_{\infty}(\N_0)$. Less clear still, is whether or not it is possible to obtain 2-sided (bilateral) versions of these spaces via the Bourgain-Delbaen construction; indeed, the recursive nature of the Bourgain-Delbaen construction makes the one-sided object a much more natural object to obtain. 

We make one final remark on shift-invariant preduals in relation to the Bourgain-Delbaen construction. We have seen in Chapters \ref{MainResult} and \ref{l1Calk} that the Bourgain-Delbaen construction is much more naturally understood with respect to a countable set $\Gamma$ whereby elements of $\Gamma$ code the form of the corresponding $c^*$ vector. Of course, when working with shift invariant preduals of $\ell_1$, one needs to make some enumeration of $\Gamma$ so that the shift operator can be defined. This provides yet another obstruction to finding exotic $\ell_1$ preduals via the Bourgain-Delbaen construction. 

Whilst the work contained in this thesis has allowed us to gain an insight into the structure of Calkin algebras, we note that there are a number of open problems that arise from the work here. Indeed, it is unclear exactly what Banach algebras can be realised as a Calkin algebra. If it is too much to hope to solve this problem, a useful partial result might be to obtain some kind of obstruction that would prevent a Banach algebra being obtained as a Calkin algebra.  